\documentclass{amsart}
\usepackage[utf8]{inputenc}
\usepackage{amssymb}
\usepackage{amsmath}
\usepackage{xcolor}
\usepackage{amscd}
\usepackage{comment}
\usepackage{accents}
\usepackage{enumitem}
\usepackage{tikz}
\usepackage{tikz-cd}
\usepackage{graphicx}
\usepackage{mathrsfs}
\usepackage{mathtools}

\usepackage{xr-hyper}
\usepackage[
pdftex,
bookmarks=false,
colorlinks=true,
debug=true,
pdfnewwindow=true]{hyperref}

\newcommand{\CA}{\mathcal{A}}
\newcommand{\uCA}{\underline{\mathcal{A}}}
\newcommand{\uH}{\underline{H}}
\newcommand{\CB}{\mathcal{B}}
\newcommand{\BB}{\mathsf{B}}
\newcommand{\bU}{{\mathbf{U}}}
\newcommand{\bN}{{\mathbf{N}}}
\newcommand{\CI}{\mathcal{I}}
\newcommand{\CF}{\mathcal{F}}
\newcommand{\CK}{\mathcal{K}}

\newcommand{\CR}{\mathcal{R}}

\newcommand{\g}{\mathfrak{g}}
\newcommand{\ssl}{\mathfrak{sl}}
\newcommand{\h}{\mathfrak{h}}
\newcommand{\fp}{\mathfrak{p}}
\newcommand{\fF}{\mathfrak{F}}
\newcommand{\fJ}{\mathfrak{J}}
\newcommand{\dU}{\dot{U}}
\newcommand{\rU}{\mathring{U}}
\newcommand{\rc}{\mathring{c}}
\newcommand{\tE}{\tilde{E}}
\newcommand{\tF}{\tilde{F}}

\newcommand{\ad}{\operatorname{ad}}
\newcommand{\Fr}{\operatorname{Fr}}
\newcommand{\fr}{\operatorname{fr}}
\newcommand{\Img}{\operatorname{Im}}

\newcommand{\BZ}{{\mathbb{Z}}}
\newcommand{\BN}{{\mathbb{N}}}
\newcommand{\BQ}{{\mathbb{Q}}}
\newcommand{\BC}{{\mathbb{C}}}
\newcommand{\BR}{{\mathbb{R}}}
\newcommand{\BF}{\mathbb{F}}
\newcommand{\sd}{{\mathsf{d}}}
\newcommand{\Dyn}{\mathrm{Dyn}}
\newcommand{\Or}{\mathsf{Or}}
\newcommand{\sA}{\mathsf{p}}

\newcommand{\Hom}{\operatorname{Hom}}
\newcommand{\End}{\operatorname{End}}
\newcommand{\Id}{\mathrm{Id}}
\newcommand{\Lprod}{\overset{\leftarrow}{\prod}}
\newcommand{\Rprod}{\overset{\rightarrow}{\prod}}

\makeatletter
\DeclareRobustCommand{\cev}[1]{%
  {\mathpalette\do@cev{#1}}%
}
\newcommand{\do@cev}[2]{%
  \vbox{\offinterlineskip
    \sbox\z@{$\m@th#1 x$}%
    \ialign{##\cr
      \hidewidth\reflectbox{$\m@th#1\vec{}\mkern4mu$}\hidewidth\cr
      \noalign{\kern-\ht\z@}
      $\m@th#1#2$\cr
    }%
  }%
}
\makeatother

\newcommand\iso{\,\vphantom{j^{X^2}}\smash{\overset{\sim}{\vphantom{\rule{0pt}{0.20em}}\smash{\longrightarrow}}}\,}

\newcommand{\ul}[1]{\underline{\smash{#1}}}

\newcommand{\wt}{\widetilde}
\newcommand{\mm}{\mathrm{m}}

\newcommand{\Rep}{\operatorname{Rep}}

\title{On De Concini-Kac forms of quantum groups}
\author{Ivan Losev, Alexander Tsymbaliuk and Trung Vu}

\newtheorem{Thm}{Theorem}[section]
\newtheorem{Prop}[Thm]{Proposition}
\newtheorem{Cor}[Thm]{Corollary}

\newtheorem{Lem}[Thm]{Lemma}
\newtheorem{PropDef}[Thm]{Proposition-Definition}
\theoremstyle{definition}

\newtheorem{defi}[Thm]{Definition}
\newtheorem{Rem}[Thm]{Remark}

\numberwithin{equation}{section}

\newcommand*{\dt}[1]{%
  \accentset{\mbox{\large\bfseries .}}{#1}}
\newcommand{\dx}{\dt{x}}
\newcommand{\dy}{\dt{y}}
\newcommand{\dlambda}{\dt{\lambda}}
\newcommand{\dmu}{\dt{\mu}}
\newcommand{\dnu}{\dt{\nu}}

\def\a{\alpha}
\def\b{\beta}
\def\c{\gamma}

\def\e{\epsilon}
\def\th{\theta}
\def\w{\omega}
\def\s{\sigma}

\def\g{\mathfrak{g}}
\def\h{\mathfrak{h}}
\def\p{\mathfrak{p}}
\def \n{\mathfrak{n}}
\def \t{\mathfrak{t}}
\def \m{\mathfrak{m}}
\def \fu{\mathfrak{u}}

\def \uchi{\underline{\chi}}
\def \hchi{\hat{\chi}}
\def \r{\textnormal{\textbf{r}}}
\def\k{\textnormal{\textbf{k}}}
\def \q{\textnormal{\textbf{q}}}

\def \op{\textnormal{op}}
\def\cop{\textnormal{c-op}}

\def\to{\rightarrow}

\def\x{\times}
\def\<{\langle}
\def\>{\rangle}

\def \({\left(}
\def \){\right)}
\newcommand{\ds}{\displaystyle}

\newcommand{\tl}{\tilde{l}}
\newcommand{\te}{\tilde{e}}
\newcommand{\tf}{\tilde{f}}
\newcommand{\he}{\hat{e}}
\newcommand{\hf}{\hat{f}}
\newcommand{\hb}{\hat{b}}

\newcommand{\tZ}{\tilde{Z}}
\newcommand{\tFr}{\tilde{\textnormal{Fr}}}
\newcommand{\hFr}{\widehat{\textnormal{Fr}}}
\newcommand{\cFr}{\check{\textnormal{Fr}}}

\newcommand{\hW}{\widehat{\textnormal{W}}}
\newcommand{\hL}{\widehat{\textnormal{L}}}

\newcommand{\Ker}{\mathrm{Ker}}
\newcommand{\Dim}{\text{dim}}
\newcommand{\hU}{\widehat{U}}
\newcommand{\oU}{\mathring{U}}
\newcommand{\oR}{\mathring{\mathcal{R}}}
\newcommand{\tU}{\tilde{U}}

\newcommand{\res}{\textnormal{res}}
\newcommand{\St}{\textbf{St}}

\def\Mod{\textnormal{-mod}}
\newcommand{\Rmod}{\textnormal{-Rmod}}
\newcommand{\rmod}{\textnormal{-rmod}}
\newcommand{\Lmod}{\textnormal{-Lmod}}
\newcommand{\lmod}{\textnormal{-lmod}}
\newcommand{\Bimod}{\textnormal{-Bimod}}
\newcommand{\bimod}{\textnormal{-bimod}}
\newcommand{\weight}{\textnormal{wt}}
\newcommand{\gr}{\textnormal{gr}}
\def\het{\textnormal{ht}}
\newcommand{\Irr}{\textnormal{Irr}}
\newcommand{\Spec}{\textnormal{Spec}}
\newcommand{\dsim}{\overset{\raisebox{-1 ex}{\textnormal{s}}}{\sim}}

\newcommand{\mU}{\mathcal{U}}
\newcommand{\dmU}{\dot{\mathcal{U}}}
\newcommand{\hmU}{\widehat{\mathcal{U}}}

\newcommand{\hiota}{\hat{\iota}}

\newcommand{\Ext}{\textnormal{Ext}}
\newcommand{\du}{\dot{\mathfrak{u}}}
\newcommand{\K}{\mathbb{K}}
\newcommand{\F}{\mathbb{F}}
\newcommand{\Dist}{\operatorname{Dist}}

\newcommand{\ccop}{\textnormal{cop}}

\newcommand{\zlambda}{\zeta^>}
\newcommand{\zmu}{\zeta^<}
\newcommand{\sN}{\mathsf{N}}
\newcommand{\sF}{\mathsf{F}}

\newcommand{\bmat}[1]{\begin{bmatrix}#1\\\end{bmatrix}}
\newcommand{\sslash}{\mathbin{/\mkern-6mu/}}

\begin{document}

\begin{abstract}
Quantum groups of semisimple Lie algebras at roots of unity admit several different forms. Among them is the De Concini-Kac form, which is the easiest to define but, perhaps, hardest to study.
In this paper, we propose a suitable modification to the De Concini-Kac form, namely the {\em even part algebra}, which has some appealing features. Notably, it behaves uniformly with respect to the order of the roots of unity and admits an adjoint action of the Lusztig form. We revisit several results due to De Concini-Kac-Procesi and Tanisaki for the even part algebra. Namely, we give conceptual definitions of the Frobenius and Harish-Chandra centers and describe the entire center in terms of these two subalgebras getting a complete quantum analog of the Veldkamp theorem on the center of the universal enveloping algebras in positive characteristic. We investigate the Azumaya locus of the even part algebra over its center. We also show that the locally finite part of the even part algebra under the adjoint  action of the Lusztig form is isomorphic to the reflection equation algebra, which is the quantized coordinate algebra with the product twisted by $R$-matrix. Some results on Lusztig forms at roots of unity are revisited and proved in greater generality including Kempf vanishing theorem and good filtrations on the quantized coordinate algebra. 
\end{abstract}

\maketitle

\tableofcontents


\section{Introduction}


\subsection{Quantum groups}\label{SS:intro_generic}\

Quantum groups introduced by Drinfeld and Jimbo in the 80's \cite{dr, ji} form, arguably, the most important class of Hopf algebras studied to this date. They play a crucial role in several fields of mathematics and physics. Let us briefly explain key features of the definition in a somewhat simplified setting. Let $R$ be a field and $q$ a nonzero element of $R$. Let $G$ be a simple algebraic group over $R$, $\g$ be its Lie algebra, and $\h$ be a Cartan subalgebra of $\g$. The quantum group $\bU_q(\g)$ is the $R$-algebra generated by the elements $E_i,F_i,K^\mu$, where $i$ runs over the indexing set of the simple roots and $\mu$ over the elements of a suitable lattice $X$ in $\h^*$. Two most conventional choices of the lattice $X$ are the weight and root lattices of~$\g$. The elements $E_i,F_i,K^\mu$ are supposed to satisfy certain relations (including the quantum Serre relations on the $E_i$'s and on the $F_i$'s) that will be recalled in Section \ref{ssec: quantum DJ}. It is classically known that $\bU_q(\g)$ carries a Hopf algebra structure, in fact, this structure depends on a number of choices obtained one from another through a ``Cartan twist''. 

While the classical presentation of $\bU_q(\g)$ is very elegant it comes with a number of structural deficiencies that we are going to list below. In this section, we will restrict to the (easier) case when $q$ is not a root of unity.

The first deficiency one encounters can be described as:
\begin{itemize}

\item[(I)]  
$2^{\operatorname{rk}(\g)}$-cover phenomenon.  

\end{itemize}
For example, while $\g$ has only one $1$-dimensional representation, the quantum group associated to the root lattice has $2^{\operatorname{rk}(\g)}$ of such. This leads to a somewhat unpleasant technicality in the description of the center of $\bU_q(\g)$. This description is easier when the lattice $X$ is the weight lattice $P$. Here we have the quantum Harish-Chandra isomorphism between the center and $\operatorname{Span}_R(K^{\mu} \,|\, \mu\in 2P)^W$, where $W$ is the Weyl group of $\g$ and the action of $W$ on the group algebra of $2P$ is suitably twisted. 

Independently, there is the following issue:
\begin{itemize}

\item[(II)] 
Absence of non-degenerate characters: the positive part of $\bU_q(\g)$, i.e., the subalgebra generated by the elements $E_i$, admits no homomorphism to $R$ that is nonzero on all~$E_i$.

\end{itemize}
Such homomorphisms exist for the positive subalgebra of the usual universal enveloping algebra $U(\g)$ and play a crucially important role in various aspects of the representation theory of $\g$ and $G$, as realized already in \cite{k2}.

The above issue~(II) was resolved by Sevostyanov~\cite{s1}. To do so, one replaces the generators $E_i$ with their twisted versions $\tE_i=E_iK^{\nu_i}$ for suitable choices of elements $\nu_i$, which only lie in $\frac{1}{2}P$ in general, and so $\tE_i$ are not elements in the version of the quantum group discussed above.


\subsection{Roots of unity case}\label{SS:intro_roots}\

While our discussion above centered on the easier case when $q$ is not a root of $1$, we now assume that $q$ is a root of unity of order that is not too small. We also assume that the base field for the quantum group is $\BC$ as some aspects become more complicated in positive characteristic. 

One can still define $\bU_q(\g)$ by the same generators and relations as before, getting the so called {\it De Concini-Kac form}. The structure of this algebra was extensively studied in \cite{dck, dckp, dckp2}.

One can also consider different forms of $\bU_q(\g)$: one replaces the base field $R$ with the field of rational functions $R(v)$, sets $q:=v$, considers an $R[v^{\pm 1}]$-subalgebra of $\bU_v(\g)$ and then mods out the ideal $(v-q)$. Of particular interest is the \emph{Lusztig form}, i.e., the form with divided powers of $E_i,F_i$, see Section \ref{ssec Lus form} for the definition, which has also been studied extensively in \cite{l0}--\cite{l-book}.

A remarkable insight of Lusztig is that the Lusztig form should behave roughly as the distribution algebra $\operatorname{Dist}_1(G_{\F})$ of  $G_{\F}$ over a positive characteristic field $\F$, cf.~\cite{l1}. With the same logic, the usual De Concini-Kac form should behave as a partly multiplicative version of $U(\g_{\F})$.

With this analogy in mind, let us list some results and constructions from the study of $G_{\F}$ and $\g_{\F}$ whose analogs for the quantum groups are desirable to have:
\begin{enumerate}

\item 
There is the adjoint action of $\Dist_1(G_\F)$ on $U(\g)$ and the algebra of regular functions $\F[G]$.

\item 
There is a natural adjoint invariant pairing  $\operatorname{Dist}_1(G_\F)\times \F[G]\rightarrow \F$ that is non-degenerate in the second argument, i.e., the induced map $\F[G]\rightarrow \operatorname{Dist}_1(G_\F)^*$ is an embedding.

\item 
There is a Frobenius epimorphism $\operatorname{Dist}_1(G_\F)\rightarrow \operatorname{Dist}_1(G^{(1)}_\F)$ induced by the Frobenius epimorphism for $G_\F$. Here the superscript (1) indicates the Frobenius twist.

\item 
The center $Z$ of $U(\g_\F)$ is described in terms of its two subalgebras $Z_{HC}$ and $Z_{Fr}$:
\begin{itemize}

\item[-] 
the {\it Harish-Chandra center} $Z_{HC}:=U(\g_\F)^{G_\F}$ is described in the same manner as the center of $U(\g_{\BC})$, unless the characteristic of $\F$ is too small;

\item[-] 
the Frobenius center $Z_{Fr}$, a.k.a.\ the $p$-center, is naturally identified with the symmetric algebra $S(\g_\F^{(1)})$.

\end{itemize}
By Veldkamp's theorem~\cite{ve}, the entire center $Z$ is identified with the tensor product of these two central subalgebras over their intersection $Z_{HC}\cap Z_{Fr} \simeq S(\g_\F^{(1)})^{G_\F^{(1)}}$.

\item 
The algebra $U(\g_\F)$ is a free rank $p^{\dim\g}$-module over $Z_{Fr}$, hence a finitely generated module over $Z$. One can thus ask about the Azumaya locus of $U(\g_\F)$ in $\operatorname{Spec}(Z)$, i.e., the locus where the module $U(\g_\F)$ is projective and the fibers are matrix algebras. As follows from the Kac-Weisfeiler conjecture proved in~\cite{Pr}, this locus contains the preimage of the locus of regular elements under the projection $\operatorname{Spec}(Z)\rightarrow \operatorname{Spec}(Z_{Fr})=\g^{(1),*}_{\F}$.

\end{enumerate}

Items (3)--(5) are less basic and their analogs in the quantum case attracted more attention, so we start with them. Lusztig constructed a quantum analog of the Frobenius epimorphism, see \cite{l2} and \cite[$\mathsection 35$]{l-book}. There are two interesting features. First, the target is almost the classical distribution algebra. Second, if the order of $q$ is divisible by $4$, then the distribution algebra is for the Langlands dual group \cite[$\mathsection 35$]{l-book}.

(4) and (5) were investigated by De Concini, Kac, and Procesi in \cite{dckp, dckp2} in the case when the order of $q$ is odd (and coprime to 3 for $G_2$). They defined analogs of $Z_{HC}$ and $Z_{Fr}$, the latter in a somewhat ad hoc way \cite[$\mathsection 3$, $\mathsection 6.2$]{dckp}. Their definition of the Harish-Chandra center was made so that $Z_{HC}$ is the same as in the case when $q$ is not a root of unity. The algebra $Z_{Fr}$ turns out to be the algebra of regular functions on a $2^{\operatorname{rk}\g}$-fold cover of $G^0$, the open Bruhat cell in $G$. In \cite[$\mathsection 6.4$]{dckp}, the authors proved that the entire center is still the tensor product of $Z_{HC}$ and $Z_{Fr}$ over their intersection, an analog of the Veldkamp's theorem. And in \cite[Theorem 5.1]{dckp2}, a direct analog of (5) was established.

In the case of even roots of unity an analog of (4) is known, see \cite[Theorem 5.4]{t2}, but it is considerably more complicated and technical. (5) is not in the literature to the best of our knowledge.

Now we get to more basic questions (1) and (2). We still have the adjoint action of $\bU_v(\g)$ on itself as well as a non-degenerate adjoint pairing $\bU_v(\g)\times \bU_v(\g)\rightarrow \BC(v^{1/2})$ \cite[$\mathsection 6$]{j}. When the order of $q$ is odd (and coprime to 3 for $G_2$) the action  descends to the adjoint action of the Lusztig form on the De Concini-Kac form, while the pairing descends to that of the Lusztig form and the De Concini-Kac form. However, one can check that this is no longer the case for even roots of $1$. This is yet another deficiency of the usual definition of the quantum group.


\subsection{Main results}\

We suggest a modification of $\bU_q(\g)$. First, we replace the usual generators $E_i,F_i$ with the Sevostyanov generators $\tE_i,\tF_i$ and only take $K^\mu$ with $\mu\in 2P$, $P$ being the weight lattice of $\g$ (morally, addressing issues (I) and (II) from Section \ref{SS:intro_generic}). We call this the \emph{even part algebra} and denote it by $U^{ev}_q(\g)$. We note that $U^{ev}_q(\g)$ is not closed under the usual coproduct, so we twist the latter as well using a Cartan type twist, see Section \ref{ssec twisted DJ} for details. 

An easy but computational argument shows that the analogs of (1) and (2) hold for all $q$ that we consider (recall that we exclude roots of unity of very small order). Lusztig's construction of the quantum Frobenius still applies in our setting with some mild modifications. We now explain our versions of analogs of (4) and (5) from Section \ref{SS:intro_roots}.

Let $R$ be an arbitrary field. First, for all $q$ we can define the Harish-Chandra center as the subalgebra of invariants of the (twisted) Lusztig form acting on $U^{ev}_q(\g)$ mirroring the definition for $U(\g_\F)$. We show in Section \ref{sec: HC center} that it admits the usual isomorphism with the $W$-invariants in $\operatorname{Span}_R(K^\mu \,|\, \mu\in 2P)$.  

A more interesting is our description of the Frobenius center $Z_{Fr}$ (here $q$ is a root of unity). Recall that we have the adjoint-invariant pairing between the Lusztig form and $U^{ev}_q(\g)$. We define $Z_{Fr}\subset U^{ev}_q(\g)$ as the orthogonal complement to the kernel of the quantum Frobenius homomorphism, which is easily seen to be a central subalgebra. Let $G^d$ be the simply connected simple algebraic group such that the target of the quantum Frobenius homomorphism is $\operatorname{Dist}_1(G^d)$, and let $G^d_0$ be the open Bruhat cell in $G^d$. We produce an identification between $Z_{Fr}$ and $R[G^d_0]$ which is equivariant under the action of $\operatorname{Dist}_1(G^d)$.

Before we get to the discussion of our versions of (4) and (5), we would like to highlight an important construction, which does not have a classical analog and is very useful, for instance in our computation of the Harish-Chandra center. The algebra of regular functions $R[G]$ admits two well-known $q$-deformations. One, more immediate, is the \emph{quantized coordinate algebra}: it comes with the rational actions of two copies of the Lusztig form (corresponding to the left and right actions of the classical group on itself) but it is not an algebra object in the category of rational representations for either of these two actions or for the adjoint action. There is however a way to twist the product on the quantized coordinate algebra using the universal $R$-matrix to get a different algebra, called the \emph{reflection equation algebra} (REA) and denoted by $O_q[G]$, see \cite{KV} and references therein. If $q$ is not a root of unity this algebra admits an embedding into $\bU_q(\g)$ such that the image coincides with the locally finite part of $\bU_q(\g)$, i.e., the subspace of all elements lying in finite dimensional subrepresentations for the adjoint action \cite[Proposition 2.8]{KV}. We carry out this construction in our situation and establish an isomorphism between $O_q[G]$ and the locally finite part for all $q$. An advantage of our setup is that $U_q^{ev}(\g)$ actually becomes a localization of $O_q[G]$, while in the usual setting one also needs to pass to covers.

\begin{Rem}
We would like to explain a reason why we care about the algebra $O_q[G]$ and its relation to $U_q^{ev}(\g)$. One can define the category of Harish-Chandra bimodules for the quantum group as the category of right modules over $O_q[G]$ together with a (weakly) equivariant structure for the action of the Lusztig form that turns a module into a rational representation (which makes sense because $O_q[G]$ is an algebra in the category of rational representations), see \cite[$\mathsection 1.4$]{gjv}, \cite[$\mathsection 6.2$]{bbj}. Many results in this paper, including the connection between $O_q[G]$ and $U_q^{ev}(\g)$, will be used to study this category in an upcoming paper by the third-named author. 
\end{Rem}

Now let us state our version of (4) that we only obtain in the case when $\operatorname{char}R=0$ (the case of positive characteristic presents some complication but we believe the result is still true). Let $T^d$ be the maximal torus in the group $G^d$. Recall that the adjoint quotient of $G^d$ is identified with the quotient $T^d/W$ giving us maps $G^d_0\hookrightarrow G^d\twoheadrightarrow T^d/W$. We identify the center of $U^{ev}_q(\g)$ with the algebra of regular functions on the fiber product $G^d_0\times_{T^d/W}T/W$ getting a direct analog of the Veldkamp's theorem. Finally, we derive a description of the Azumaya locus in a complete analogy with \cite[Theorem 5.1]{dckp2}.


\subsection{Outline}\label{ssec:outline}\

In Section~\ref{sec Setting}, we recall the basic facts about quantum groups as well as their De Concini-Kac and Lusztig forms. In Section~\ref{sec twisted coproduct}, we introduce the even part algebra $U_q^{ev}(\g)$ and endow it with an adjoint action of Lusztig form $\dU_q(\g)$. In Section~\ref{S_inv_pair}, we construct the adjoint invariant pairing between $U_q^{ev}(\g)$ and $\dU_q(\g)$, which provides the main tool to study $U_q^{ev}(\g)$ in the present paper. In Section~\ref{Quantum Frobenius}, we adapt Lusztig's quantum Frobenious homomorphism to our modified setup. In Section~\ref{sec: Frobenius kernel}, we study the Frobenius center $Z_{Fr}$ of the even algebra in the root of unity case. In Section~\ref{sec rational reps}, we revisit rational representations of the Lusztig form and prove some old results in the greater generality as well as adapt them to the present setup.
In Section~\ref{sec: REA}, we recall the construction of reflection equation algebra $O_q[G]$. In Section~\ref{sec:REA as locfin}, we identify $O_q[G]$ with $U_q^{fin}$, the locally finite part of $U_q^{ev}(\g)$ under the adjoint action of $\dU_q(\g)$. In Section~\ref{sec:PoissonGeometry},  we study the center of the even algebra $U_q^{ev}(\g)$ and its Azumaya locus, in a close analogy with~\cite{dckp, dckp2}. In Appendix~\ref{sec:equiv}, we recall basic facts about equivariant modules over Hopf algebras.


\subsection{Acknowledgment}\label{ssec:acknowl}
\

We are very grateful to P.~Etingof and D.~Jordan for many useful discussions on the subject. The work of I.~Losev and T.~Vu was supported by an NSF Grant DMS-$2001139$, the work of A.~Tsymbaliuk was supported by an NSF Grant DMS-$2302661$.


\section{Setting}\label{sec Setting}

In this section, we establish notations and recall facts about the Drinfeld-Jimbo quantum groups as well as 
their Lusztig and De~Concini-Kac forms.

Let $\g$ be a semisimple Lie algebra with simple positive roots $\alpha_1,\ldots,\alpha_r$ and fundamental 
weights $\omega_1,\ldots,\omega_r$. Let $P:=\bigoplus_{i=1}^r \BZ\omega_i$ be the weight lattice and 
$Q:=\bigoplus_{i=1}^r \BZ\alpha_i$ be the root lattice. Let $P_+:=\bigoplus_{i=1}^r \BZ_{\geq 0} \omega_i$ 
be the set of all dominant weights in $P$, and $Q_+:=\bigoplus_{i=1}^r \BZ_{\geq 0} \alpha_i$. We fix a 
non-degenerate invariant bilinear form $(\ ,\ )$ on the Cartan subalgebra $\h\subset \g$, and identify 
$\h^*$ with $\h$ using $(\ ,\ )$. We set $\sd_i:=\frac{(\alpha_i,\alpha_i)}{2}$. The choice of $(\ ,\ )$ 
is such that $\sd_i=1$ for short roots $\alpha_i$, in particular, $\sd_i\in \{1,2,3\}$ for any $i$. Define 
$\omega^\vee_i:=\frac{\omega_i}{\sd_i}$ and $\alpha^\vee_i:=\frac{\alpha_i}{\sd_i}$, so that 
$(\alpha_i,\omega^\vee_j)=(\omega_i,\alpha^\vee_j)=\delta_{i,j}$. Thus, we shall identify the coweight and 
the coroot lattices of $\g$ with $P^\vee:=\bigoplus_{i=1}^r \BZ\omega_i^\vee$ and 
$Q^\vee:=\bigoplus_{i=1}^r \BZ\alpha_i^\vee$, respectively. Let $(a_{ij})_{i,j=1}^r$ be the Cartan matrix of $\g$:
\begin{equation*}
  a_{ij}=(\alpha_i^\vee,\alpha_j) = 2(\alpha_i,\alpha_j)/(\alpha_i,\alpha_i) \,.
\end{equation*}
We also consider the symmetrized Cartan matrix $B=(b_{ij})_{i,j=1}^r$ defined via:  
\begin{equation}\label{eq:b-matrix}
  b_{ij}:=\sd_ia_{ij}=(\alpha_i,\alpha_j) \,.
\end{equation}

Let $v$ be a formal variable and let us consider $\BZ[v,v^{-1}]$ localized at 
$\{v^{2k}-1\}_{1\leq k\leq \max\{\sd_i\}_{i=1}^r}$:
\begin{equation*}
  \CA = \BZ[v,v^{-1}] \Big[\Big\{ \tfrac{1}{v^{2k}-1} \Big\}_{1\leq k\leq \max\{\sd_i\}} \Big] \,,
  \footnote{We note that~\cite{l2,l-book} rather use $\CA$ for $\BZ[v,v^{-1}]$, while~\cite{dck,dckp} use $\CA$ for $\BC[v,v^{-1}]$.}
\end{equation*}
with the quotient field $\BQ(v)$. Given $m\in \BZ$, $s\in \BN$, define 
$[s]_v, [s]_v!, \bmat{m\\s}_v\in \BZ[v,v^{-1}]\subset \CA$ via:
\begin{equation*}
  [s]_v:=\frac{v^s-v^{-s}}{v-v^{-1}} \,, \quad
  [s]_v!:=\prod_{c=1}^s [c]_v=[1]_v \cdots [s]_v \,, \quad 
  \bmat{m\\s}_v:=\, \prod_{c=1}^s \frac{v^{m-c+1}-v^{-m+c-1}}{v^c-v^{-c}} \,.
\end{equation*}
Note that $[0]_v!=1$. If $m\geq s\geq 0$, then the $v$-binomial coefficient equals  
$\bmat{m\\s}_v=\frac{[m]_v!}{[s]_v!\cdot [m-s]_v!}$. For $1\leq i\leq r$, we set $v_i:=v^{\sd_i}$, 
and define $[s]_{v_i}, [s]_{v_i}!, \bmat{m\\s}_{v_i}$ accordingly.

Let $q$ be an invertible element of a Noetherian ring $R$ such that the elements $q^{2k}-1\in R$ 
are invertible for all $1\leq k \leq \max\{\sd_i\}_{i=1}^r$. Define an algebra homomorphism: 
\begin{equation}\label{sigma homom}
  \sigma\colon \CA\longrightarrow R \qquad \mathrm{via}\qquad v\mapsto q \,.
\end{equation}
For $m\in \BZ, s\in \BN$, let  $[s]_q, [s]_q!, \bmat{m\\s}_q\in R$ be the images of 
$[s]_v, [s]_v!, \bmat{m\\s}_v\in \CA$ under~(\ref{sigma homom}).


\subsection{The Drinfeld-Jimbo quantum group $\bU(\g)$}\label{ssec: quantum DJ}
\

Let $\bU(\g)$ denote the Drinfeld-Jimbo quantum group of $\g$ over $\BQ(v)$, that is, the associative unital $\BQ(v)$-algebra 
generated by $\{E_i,F_i,K^{\mu}\}_{1\leq i\leq r}^{\mu\in Q}$ with the following defining relations: 
\begin{equation}\label{DJ eqn 1}
  K^{\mu} K^{\mu'}=K^{\mu+\mu'}\,, \qquad K^0=1 \,,
\end{equation}
\begin{equation}\label{DJ eqn 2}
  K^{\mu} E_i K^{-\mu}=v^{(\alpha_i,\mu)}E_i \,, \qquad K^{\mu} F_i K^{-\mu}=v^{-(\alpha_i,\mu)}F_i \,,
\end{equation}
\begin{equation}\label{DJ eqn 3}
  [E_i,F_j]=\delta_{ij}\frac{K_i-K_i^{-1}}{v_i-v_i^{-1}} \,,
\end{equation}
\begin{equation}\label{DJ eqn 4}
  \sum_{m=0}^{1-a_{ij}}(-1)^m \bmat{1-a_{ij}\\ m}_{v_i} E_i^{1-a_{ij}-m}E_jE_i^{m}=0 \quad (i\ne j) \,,
\end{equation}
\begin{equation}\label{DJ eqn 5}
  \sum_{m=0}^{1-a_{ij}}(-1)^m \bmat{1-a_{ij}\\m}_{v_i} F_i^{1-a_{ij}-m}F_jF_i^{m}=0 \quad (i\ne j) \,,
\end{equation}
with the standard notation $K_i:=K^{\alpha_i}$.

Let $\bU^<, \bU^>, \bU^0$ denote the $\BQ(v)$-subalgebras of $\bU(\g)$ generated by $\{F_i\}_{i=1}^r, \{E_i\}_{i=1}^r$, 
and $\{K^\mu\}_{\mu\in Q}$, respectively. The following is standard (see~\cite[Theorem 4.21]{j}):

\begin{Lem}\label{triangular_DJ}
(a) (Triangular decomposition of $\bU(\g)$) The multiplication map
\begin{equation}\label{triang_isom_DJ}
  \mm\colon \bU^< \otimes_{\BQ(v)} \bU^0 \otimes_{\BQ(v)} \bU^> \longrightarrow \bU(\g)    
\end{equation}
is an isomorphism of $\BQ(v)$-vector spaces.

\noindent
(b) The subalgebras $\bU^<, \bU^>, \bU^0$ are isomorphic to the algebras generated by $\{F_i\}_{i=1}^r$, $\{E_i\}_{i=1}^r$, 
and $\{K^\mu\}_{\mu\in Q}$, with the defining relations~(\ref{DJ eqn 5}),~(\ref{DJ eqn 4}), and~(\ref{DJ eqn 1}), respectively.
\end{Lem}

The algebra $\bU(\g)$ is $Q$-graded via: 
\begin{equation}\label{grading of DJ}
  \deg(E_i)=\alpha_i \,, \qquad \deg(F_i)=-\alpha_i \,, \qquad \deg(K^\mu)=0 \,.
\end{equation}
Following~\cite[Proposition 4.11]{j}, we also endow $\bU(\g)$ with the standard Hopf algebra structure:
\begin{equation}\label{standard_Hopf_DJ}
\begin{split}
  & \mathrm{coproduct}\ \ \Delta\colon 
    E_i\mapsto E_i\otimes 1 + K_i\otimes E_i \,, \ \
    F_i\mapsto F_i\otimes K_i^{-1} + 1\otimes F_i \,,\ \ 
    K^{\mu}\mapsto K^{\mu}\otimes K^{\mu} \,, \\
  & \mathrm{antipode}\ \ \ \ S\colon 
    E_i\mapsto -K_i^{-1}E_i \,,\ \ 
    F_i\mapsto -F_iK_i \,,\ \ 
    K^{\mu}\mapsto K^{-\mu} \,, \\
  & \mathrm{counit}\ \ \ \ \ \ \ \ \varepsilon\colon 
    E_i\mapsto 0 \,,\ \ 
    F_i\mapsto 0 \,,\ \ 
    K^{\mu}\mapsto 1 \,.    
\end{split}
\end{equation}

\begin{Rem}\label{rem: enlarge Cartan part}
One often needs a more general version of $\bU(\g)$: obtained by enlarging the Cartan part $\bU^0$ of $\bU(\g)$ and 
getting an algebra over $\BQ(v^{1/\bN})$ for a suitable positive integer $\bN$. Let $X$ be a $\BZ$-lattice in 
$\BZ Q\otimes_{\BZ} \BQ$ containing the root lattice $Q$. Let $\bN$ be a positive integer such that 
$(X,Q) \in \frac{1}{\bN} \BZ$; then for any $\mu\in X$ and $\lambda\in Q$, we define 
$v^{(\mu, \lambda)} :=(v^{1/\bN})^{\bN(\mu, \lambda)}$. Let $\bU(\g, X)$ be an associative $\BQ(v^{1/\bN})$-algebra 
generated by $\{E_i, F_i, K^\mu\}_{1\leq i \leq r}^{\mu \in X}$ with the defining relations 
similar to the defining relations \eqref{DJ eqn 1}--\eqref{DJ eqn 5}. We also endow $\bU(\g, X)$ with 
the Hopf algebra structure as in \eqref{standard_Hopf_DJ} and the $Q$-grading as in~\eqref{grading of DJ}.
\end{Rem}

Let us recall the $\BQ$-algebra anti-involution $\tau\colon \bU(\g)\rightarrow \bU(\g)$ defined by 
\begin{equation}\label{eq: anti-involution map}  
  \tau(E_i)=F_i \,, \quad \tau(F_i)=E_i \,, \quad \tau(K^\lambda)=K^{-\lambda} \,, \quad \tau(v)=v^{-1} 
  \qquad \forall\, 1\leq i \leq r, \lambda\in Q \,.
\end{equation}
There is Lusztig's braid group action on $\bU(\g)$ defined via (see~\cite[Theorem 3.1]{l2},~\cite[Part VI]{l-book}, 
cf.~\cite[\S8.14--8.15]{j}):
\begin{equation}\label{eq:braid group}
\begin{split}
   T_i(K^\mu)&=K^{s_{i}\mu} \,, \qquad T_i(E_i)=-F_iK^{\a_i} \,, \qquad T_i(F_i)=-K^{-\a_i}E_i \,, \\
   T_i(E_j)&=\sum_{k=0}^{-a_{ij}} (-1)^k\frac{v_i^{-k}}{[-a_{ij}-k]_{v_i}![k]_{v_i}!}E_i^{-a_{ij}-k}E_jE_i^k \,, \\
   T_i(F_j)&=\sum_{k=0}^{-a_{ij}} (-1)^k\frac{v_i^{k}}{[-a_{ij}-k]_{v_i}![k]_{v_i}!}F_i^kF_jF_i^{-a_{ij}-k} \,,
\end{split}
\end{equation}
where $s_i=s_{\a_i}$ is a simple reflection. 
The operators $T_i$ are easily seen to commute with the map $\tau$ of~\eqref{eq: anti-involution map}. 
With this braid group action, Lusztig defined the PBW basis of $\bU(\g)$ as we recall next 
(see~\cite[Part VI]{l-book}, cf.~\cite[\S8]{j}). Pick a reduced decomposition of the longest element 
$w_0=s_{i_1}s_{i_2} \cdots s_{i_N}$ of the Weyl group $W$ of $\g$ (where $N$ is the cardinality of the positive root 
system $\Delta_+$). Then the set of roots $\beta_k=s_{i_1}\cdots s_{i_{k-1}} \alpha_{i_k}\ (1\leq k\leq N)$ provides 
a labeling of all positive roots $\Delta_+$ of $\g$, and using Lusztig's braid group action we define root vectors 
\begin{equation}\label{eq:root-generator}
  E_{\beta_{k}} =\, T_{i_1}\cdots T_{i_{k-1}} E_{i_k} \,, \qquad 
  F_{\beta_{k}} =\, T_{i_1}\cdots T_{i_{k-1}} F_{i_k} =\tau(E_{\b_k})
  \qquad \forall\, 1\leq k\leq N \,.
\end{equation}
For $\vec{k}=(k_1,\dots, k_N) \in \BZ_{\geq 0}^N$, consider the ordered monomials:
\[
  F^{\vec{k}}:=F_{\b_1}^{k_1} \dots F_{\b_N}^{k_N} \,, \quad  
  E^{\vec{k}}:=E_{\b_1}^{k_1} \dots E_{\b_N}^{k_N} \,, \quad 
  F^{\cev{k}}:=F_{\b_N}^{k_N} \dots F_{\b_1}^{k_1} \,, \quad 
  E^{\cev{k}}:=E_{\b_N}^{k_N} \dots E_{\b_1}^{k_1} \,. 
\]
The following result follows from~\cite[\S8.24]{j} and the triangular decomposition of Lemma~\ref{triangular_DJ}:

\begin{Lem}\label{lem:PBW-DJ} 
(a) The sets $\{E^{\vec{k}}\}_{\vec{k}\in \BZ_{\geq 0}^N}, \{E^{\cev{k}}\}_{\vec{k}\in \BZ_{\geq 0}^N}$ are 
$\BQ(v)$-bases of $\bU^>$.

\noindent
(b) The sets $\{F^{\vec{k}}\}_{\vec{k}\in \BZ_{\geq 0}^N}, \{F^{\cev{k}}\}_{\vec{k}\in \BZ_{\geq 0}^N}$ are 
$\BQ(v)$-bases of $\bU^<$.

\noindent 
(c) The set $\{K^\mu\}_{\mu \in Q}$ is a $\BQ(v)$-basis of $\bU^0$.

\noindent
(d) The sets 
  $\{F^{\vec{k}}K^\mu E^{\vec{r}}\}, \{F^{\cev{k}}K^\mu E^{\vec{r}}\}, \{F^{\vec{k}}K^\mu E^{\cev{r}}\}, \{F^{\cev{k}}K^\mu E^{\cev{r}}\}$, 
with $\vec{k}, \vec{r}\in \BZ_{\geq 0}^N$ and $\mu \in Q$, form $\BQ(v)$-bases of $\bU(\g)$.
\end{Lem}


\subsection{The Lusztig form}\label{ssec Lus form} 
\

For $1\leq i\leq r$ and $s\in \BN$, define the \textbf{divided powers} $E^{[s]}_i, F^{[s]}_i\in \bU(\g)$ 
via:\footnote{Our superscript $[s]$ differs from $(s)$ used in~\cite{l2,l-book}, as we use the latter 
for a modified version in Section~\ref{ssec: new Lusztig form}.} 
\begin{equation*}
  E^{[s]}_i:=\frac{E^s_i}{[s]_{v_i}!} \,, \qquad  F^{[s]}_i:=\frac{F^s_i}{[s]_{v_i}!} \,.
\end{equation*}
For $1\leq i\leq r,\, a\in \BZ,\, n\in \BN$, define the $v$-binomial coefficients $\bmat{K_i; a \\ n}\in \bU(\g)$ via: 
\begin{equation*}
  \bmat{K_i; a \\ n}:=\prod_{c=1}^n \frac{K_iv_i^{a-c+1}-K_i^{-1}v_i^{-a+c-1}}{v_i^{c}-v_i^{-c}} \,.
\end{equation*}
Following~\cite[\S1.3]{l2}, define the \textbf{Lusztig form} $\dmU_\CA(\g)$ as the $\CA$-subalgebra 
of $\bU(\g)$ generated by $\{E^{[s]}_i, F^{[s]}_i, K^{\mu}\}_{1\leq i\leq r}^{s\in \BN,\mu\in Q}$. 
Note that all elements $\bmat{K_i; a \\ n}$ lie in $\dmU_\CA(\g)$, due to a recursion 
\begin{equation*}
  \bmat{K_i; a \\ n}-v_i^{-n}\bmat{K_i; a+1 \\ n} = -v^{-a-1}_i K_i^{-1} \bmat{K_i; a \\n-1} \,,       
\end{equation*}
combined with the following important equality (cf.~\cite[\S4.3]{l0}):
\begin{equation}\label{EF-formula}
  E^{[p]}_i F^{[s]}_i = \sum_{c=0}^{\min(p,s)} F^{[s-c]}_i \bmat{K_i; 2c-p-s \\ c} E^{[p-c]}_i \,.    
\end{equation}
We also note that $\dmU_\CA(\g)$ is clearly invariant under the map $\tau$ of~\eqref{eq: anti-involution map}.

Evoking the standard Hopf algebra structure~(\ref{standard_Hopf_DJ}) on $\bU(\g)$, we find (see~\cite[\S4.9]{j}):
\begin{equation*}
\begin{split}
  & \Delta(E^{[s]}_i)=\sum_{c=0}^{s} v_i^{c(s-c)} E^{[s-c]}_iK^c_i\otimes E^{[c]}_i \,, \qquad
    S(E^{[s]}_i)=(-1)^s v_i^{s(s-1)}K_i^{-s}E^{[s]}_i \,, \\
  & \Delta(F^{[s]}_i)=\sum_{c=0}^{s} v_i^{-c(s-c)} F^{[c]}_i\otimes K^{-c}_iF^{[s-c]}_i \,, \qquad
    S(F^{[s]}_i)=(-1)^s v_i^{-s(s-1)}F^{[s]}_iK_i^{s} \,,
\end{split}  
\end{equation*}
so that $\dmU_\CA(\g)$ is actually a Hopf $\CA$-subalgebra of $\bU(\g)$.     
    
Let $\dmU^<_\CA, \dmU^>_\CA, \dmU^0_\CA$ denote the $\CA$-subalgebras of $\dmU_\CA(\g)$ generated by 
$\{F^{[s]}_i\}, \{E^{[s]}_i\}$, $\left\{\bmat{K_i;m\\s},K^\mu\right\}$, respectively. Evoking the construction 
of root vectors~\eqref{eq:root-generator}, we define 
\begin{equation}\label{eq:old divided root-generator}
  F_{\b_k}^{[s]}:=\frac{F_{\b_k}^s}{[s]_{v_{i_k}}!} \,, \qquad E_{\b_k}^{[s]}:=\frac{E_{\b_k}^s}{[s]_{v_{i_k}}!} 
  \qquad \forall\, 1\leq k\leq N,\, s\in \BN \,.
\end{equation}
According to~\cite[Theorem 6.7]{l2}, we have:

\begin{Lem}\label{triangular_Lus}
(a) Both subalgebras $\dmU^>_\CA$ and $\dmU^<_\CA$ are $Q$-graded via~(\ref{grading of DJ}), 
and each of their degree components is a free $\CA$-module of finite rank.

\noindent
(b) (Triangular decomposition of $\dmU_\CA(\g)$) The multiplication map
\begin{equation}\label{triang_isom_Lus}
  \mm\colon \dmU^<_\CA \otimes_\CA \dmU^0_\CA \otimes_\CA \dmU^>_\CA \longrightarrow \dmU_\CA(\g)    
\end{equation}
is an isomorphism of free $\CA$-modules.

\noindent
(c) The elements
\begin{equation*}
  \left\{ \prod_{i=1}^r \left(K_i^{\delta_i}\bmat{K_i; 0 \\ t_i}\right) \,\Big|\, t_i\geq 0, \delta_i\in\{0,1\} \right\}
\end{equation*}
form an $\CA$-basis of $\dmU^0_\CA$.

\noindent
(d) The elements $E^{[\vec{k}]}:=E_{\b_1}^{[k_1]} \dots E_{\b_N}^{[k_N]}$ with all $\vec{k}\in \BZ_{\geq 0}^N$ 
form an $\CA$-basis of $\dmU^>_\CA$. Similarly, the elements $E^{[\cev{k}]}:=E_{\b_N}^{[k_N]} \dots E_{[\b_1]}^{[k_1]}$ 
with all $\vec{k}\in \BZ^N_{\geq 0}$ form an $\CA$-basis of $\dmU^>_\CA$.

\noindent
(e) The elements $F^{[\vec{k}]}:=F_{\b_1}^{[k_1]}\dots F_{\b_N}^{[k_N]}$ with all $\vec{k}\in \BZ_{\geq 0}^N$ 
form an $\CA$-basis of $\dmU^<_\CA$. Similarly, the elements $F^{[\cev{k}]}:=F_{\b_N}^{[k_N]}\dots F_{\b_1}^{[k_1]}$ 
with all $\vec{k}\in \BZ^N_{\geq 0}$ form an $\CA$-basis of $\dmU^<_\CA$.
\end{Lem}

\begin{Rem}
While~\cite{l2} treats only the first basis in (d) and the second basis in (e), we obtain the other two bases 
from (e) and (d) by applying the $\BQ$-algebra anti-involution $\tau$ of~\eqref{eq: anti-involution map}.   
\end{Rem}

Given $q\in R$ as in the paragraph preceding~(\ref{sigma homom}), we define the \textbf{Lusztig form} $\dmU_q(\g)$ as the 
base change of $\dmU_\CA(\g)$ with respect to $\sigma\colon \CA\to R$ of~(\ref{sigma homom}):
\begin{equation}\label{Lusztig base changed}
  \dmU_q(\g):= \dmU_\CA(\g)\otimes_\CA R \,.
\end{equation}


\subsection{The De~Concini-Kac form}\label{ssec DCK form}
\

Following~\cite[\S1.5]{dck}, define the \textbf{De~Concini-Kac form} $\mU_\CA(\g)$ as the $\CA$-subalgebra 
of $\bU(\g)$ generated by $\{E_i,F_i,K^{\mu}\}_{1\leq i\leq r}^{\mu\in Q}$. We note that $\mU_\CA(\g)$ is 
clearly a Hopf $\CA$-subalgebra of $\bU(\g)$. 
It is also clear that $\mU_\CA(\g)$ is invariant under the map $\tau$ of~\eqref{eq: anti-involution map}.


Let $\mU^<_\CA, \mU^>_\CA, \mU^0_\CA$ denote the $\CA$-subalgebras of $\mU_\CA(\g)$ generated by 
$\{F_i\}_{i=1}^r, \{E_i\}_{i=1}^r$, and $\{K^\mu\}_{\mu\in Q}$, respectively. Since $\CA$ contains 
$(v^{2k}-1)^{-1}$ for any $1\leq k\leq \sd_i$ and any $i$, the braid group action preserves $\mU_\CA(\g)$ according 
to~\eqref{eq:braid group}. Hence the elements $\{E_{\b_k}, F_{\b_k}\}_{k=1}^N$ of \eqref{eq:root-generator} 
are contained in $\mU_A(\g)$. The following result is a standard corollary of Lemma~\ref{triangular_DJ}:

\begin{Lem}\label{triangular_DCK}
(a) Both subalgebras $\mU^>_\CA$ and $\mU^<_\CA$ are $Q$-graded via~\eqref{grading of DJ}, 
and each of their degree components is a free $\CA$-module of finite rank.

\noindent
(b) (Triangular decomposition of $\mU_\CA(\g)$) The multiplication map
\begin{equation}\label{triang_isom_DCK}
  \mathrm{m}\colon \mU^<_\CA \otimes_\CA \mU^0_\CA \otimes_\CA \mU^>_\CA \longrightarrow \mU_\CA(\g)    
\end{equation}
is an isomorphism of free $\CA$-modules.

\noindent
(c) The elements $K^\mu$ with $\mu \in Q$ form an $\CA$-basis of $\mU^0_\CA$.

\noindent
(d) The sets $\{E^{\vec{k}}\}_{\vec{k}\in \BZ^N_{\geq 0}}, \{E^{\cev{k}}\}_{\vec{k}\in \BZ^N_{\geq 0}}$ form 
$\CA$-bases of $\mU^>_\CA$. 

\noindent
(e) The sets $\{F^{\vec{k}}\}_{\vec{k}\in \BZ^N_{\geq 0}}, \{F^{\cev{k}}\}_{\vec{k}\in \BZ^N_{\geq 0}}$ form 
$\CA$-bases of $\mU^<_\CA$. 
\end{Lem}

\begin{proof}
Let us show that the second set in part (d) is an $\CA$-basis of $\mU^>_\CA$. It suffices to show that the inclusion 
$\mU^>_\CA\supseteq \bigoplus_{\vec{k}\in \BZ_{\geq 0}^N} \CA\cdot E^{\cev{k}}$ is actually an equality. To this end, 
pick any $x\in \mU^>_\CA$. Evoking the Hopf pairing $(\cdot,\cdot)$ of~\cite[\S6.12]{j}, cf.~Proposition~\ref{prop:J-pairing-6.12} 
below,  we note that $(\dmU^<_\CA,x)\in \CA$, with $\dmU^<_\CA$ defined in Section~\ref{ssec Lus form} above. On the other hand, 
writing down $x$ in the Lusztig PBW basis as  $x=\sum_{\vec{k}} c_{\vec{k}}E^{\cev{k}}$ and using the ``duality'' of 
the Lusztig basis with respect to the  pairing $(\cdot,\cdot)$, see~\cite[\S8.29--8.30]{j} and~\eqref{eq:PBW-pairing} 
below, we conclude that all $c_{\vec{k}}\in \CA$ as $\frac{1}{v_i-v_i^{-1}}\in \CA$ for all $i$. 
The proof that the second set in part (e) is an $\CA$-basis of $\mU^<_\CA$ is analogous. 

The anti-involution $\tau$ of~\eqref{eq: anti-involution map} induces an anti-isomorphism $\mU^>_\CA\rightarrow \mU^<_\CA$. 
Moreover, $\tau(E^{\cev{k}})=F^{\vec{k}}, \tau(F^{\cev{k}})=E^{\vec{k}}$. Therefore, the first sets in part (d) and (e) are 
$\tau$-images of the second sets in part (e) and (d), respectively, hence form $\CA$-bases. Part (c) follows from 
Lemma~\ref{triangular_DJ}(b). The rest of Lemma~\ref{triangular_DCK} immediately follows from the above combined 
with Lemma~\ref{triangular_DJ}.
\end{proof}

Given $q\in R$ as in the paragraph preceding~(\ref{sigma homom}), we define the \textbf{De~Concini-Kac form} $\mU_q(\g)$ 
as the base change of $\mU_\CA(\g)$ with respect to $\sigma\colon \CA\to R$ of~(\ref{sigma homom}):
\begin{equation}\label{DCK base changed}
  \mU_q(\g):=\mU_\CA(\g)\otimes_\CA R \,.       
\end{equation}


\section{Twisted coproduct and integral forms}\label{sec twisted coproduct}

For a Hopf algebra $(A,\Delta,S,\varepsilon)$, the \textbf{left adjoint action} $\ad\colon A\curvearrowright A$ 
of $A$ on itself is given~by: 
\begin{equation}\label{def adjoint}
  (\ad a)(b):=a_{(1)}\cdot b\cdot S(a_{(2)})  \qquad \forall\, a,b\in A \,. 
\end{equation}
Here we use the Sweedler's notation for the coproduct (suppressing the summation symbol): 
\begin{equation}\label{Sweedler}
  \Delta(a)=a_{(1)}\otimes a_{(2)}  \qquad \forall\, a\in A \,. 
\end{equation}
Let us record the following basic property of the adjoint action (see~\cite[Lemma 2.2(ii)]{jl}):
\begin{equation}\label{ad on products}
  (\ad a)(b\cdot c)=(\ad a_{(1)})(b)\cdot (\ad a_{(2)})(c)  \qquad \forall\, a,b,c\in A \,.    
\end{equation}

Thus we have the adjoint action  $\bU(\g) \curvearrowright \bU(\g)$. However, this action does not restrict to an action 
of the Lusztig form $\dmU_\CA(\g)$ on the De~Concini-Kac form $\mU_\CA(\g)$. To remedy this, and for other purposes, we will 
modify the coproduct of $\bU(\g)$ via the twist construction. Then we introduce the \textbf{(twisted) Lusztig form} $\dU_\CA(\g)$ 
and the \textbf{even part} subalgebra $U^{ev}_\CA(\g)$. The latter is a suitable alternative (based on a certain ``Cartan twist'' 
of generators) to the De~Concini-Kac form so that we have an adjoint action $\dU_\CA(\g) \curvearrowright U^{ev}_\CA(\g)$, 
which is of crucial importance for the rest of this paper.


\subsection{Twisting}\label{ssec twists}
\

To achieve the above goal, let us recall the standard twist construction (cf.~\cite[Theorem~1]{r}):

\begin{Prop}\label{resh twist construction}
(a) For a (topological) Hopf algebra $(A,m,\Delta,S,\varepsilon)$ and $\sF\in A\otimes A$ 
(or in an appropriate completion $A\widehat{\otimes} A$) satisfying
\begin{equation}\label{twist conditions}
  (\Delta \otimes \Id)(\sF)=\sF_{13}\sF_{23} \,, \ 
  (\Id \otimes \Delta)(\sF)=\sF_{13}\sF_{12} \,, \ 
  \sF_{12}\sF_{13}\sF_{23}=\sF_{23}\sF_{13}\sF_{12} \,, \
  \sF_{12}\sF_{21}=1 \,,
\end{equation}
the formulas
\begin{equation}\label{resh twist formulas}
  \Delta^{(\sF)}(a)=\sF\Delta(a)\sF^{-1} \,, \quad 
  S^{(\sF)}(a)=uS(a)u^{-1} \,, \quad 
  \varepsilon^{(\sF)}(a)=\varepsilon(a) 
\end{equation}
with $u:=m(\Id \otimes S)(\sF)$, endow $A$ with a new Hopf algebra structure:
$(A,m,\Delta^{(\sF)},S^{(\sF)},\varepsilon^{(\sF)})$.

\noindent
(b) If $(A,m,\Delta,S,\varepsilon)$ is a quasitriangular Hopf algebra with a universal $R$-matrix $R\in A\otimes A$ 
(or in an appropriate completion $A\widehat{\otimes} A$), then $(A,m,\Delta^{(\sF)},S^{(\sF)},\varepsilon^{(\sF)})$ 
is also a quasitriangular Hopf algebra with a universal $R$-matrix 
\begin{equation}\label{R-matrix}
  R^{(\sF)}=\sF^{-1}R\sF^{-1} \,.    
\end{equation}
\end{Prop}

The Hopf algebra $(A,m,\Delta^{(\sF)},S^{(\sF)},\varepsilon^{(\sF)})$ is called 
the \textbf{twist} of $(A,m,\Delta,S,\varepsilon)$ by $\sF$.

\begin{Rem}
(a) The inverse of $u$ is explicitly given by $u^{-1}=m(S \otimes \Id)(\sF_{21})$.

\noindent
(b) For Proposition~\ref{resh twist construction}(a) alone, conditions~(\ref{twist conditions}) may be relaxed 
(cf.~\cite[Proposition 4.2.13]{cp}):
\begin{equation*}
  \sF_{12}(\Delta \otimes \Id)(\sF)=\sF_{23}(\Id \otimes \Delta)(\sF) \,, \quad
  (\varepsilon \otimes \Id)(\sF) = 1 = (\Id \otimes \varepsilon)(\sF) \,, \quad
  \sF \ \mathrm{is\ invertible} \,.
\end{equation*}
\end{Rem}


\subsection{Twisted Hopf algebra structure on $\bU(\g)$}\label{ssec twisted DJ}
\

Following~\cite[\S2]{r}, let us apply Proposition~\ref{resh twist construction} to $A=\bU(\g)$ 
endowed with the standard Hopf algebra structure $(A,\Delta,S,\varepsilon)$ of~(\ref{standard_Hopf_DJ}) and 
a Cartan type element $\sF$:
\begin{equation}\label{formal Cartan twist}
  \sF=v^{\sum_{i,j=1}^{r} \phi_{ij}\omega^\vee_i\otimes \omega^\vee_j} \,.
\end{equation}
Such $\sF$ satisfies the conditions~(\ref{twist conditions}) iff the matrix 
$\Phi=(\phi_{ij})_{i,j=1}^r\in \mathrm{Mat}_{r\times r} (\BQ)$ is skew-symmetric.

To simplify our notations, we shall henceforth denote the corresponding twisted coproduct $\Delta^{(\sF)}$, antipode 
$S^{(\sF)}$, and counit $\varepsilon^{(\sF)}$ of~(\ref{resh twist formulas}) by $\Delta', S'$, and $\varepsilon'$, 
respectively.

\begin{Rem}\label{rem: twisted Cartan} 
We note that $\sum_{j=1}^r \phi_{ij} \w^\vee_j$ is not contained in $Q$. Thus, in fact, we will perform the twist on 
a Hopf $\BQ(v^{1/\sN})$-algebra $\bU(\g, X)$ for a $\BZ$-lattice $X$ of $\BZ Q\otimes_{\BZ} \BQ$ containing all 
$\{ \sum_{j=1}^r \phi_{ij} \w^\vee_j\}_{1 \leq i \leq r}$, see Remark \ref{rem: enlarge Cartan part}. Explicitly, 
the twisted Hopf structure is as follows:
\begin{equation}\label{eq: twisted-Hopf}
\begin{split}
  & \Delta'(K^{\mu})=K^{\mu}\otimes K^{\mu} \,, \\
  & \Delta'(E_i) = E_i\otimes K^{\sum_{j=1}^r \phi_{ij}\omega^\vee_j} + 
      K^{\alpha_i-\sum_{j=1}^r \phi_{ij}\omega^\vee_j}\otimes E_i \,, \\
  & \Delta'(F_i) = F_i\otimes K^{-\alpha_i-\sum_{j=1}^r \phi_{ij}\omega^\vee_j} + 
      K^{\sum_{j=1}^r \phi_{ij}\omega^\vee_j}\otimes F_i \,, \\
  & S'(K^\mu)=K^{-\mu} \,, \qquad S'(E_i)= -K^{-\a_i}E_i \,, \qquad S'(F_i)=-F_iK^{\a_i} \,, \\
  & \varepsilon'(K^\mu)=1 \,, \qquad \varepsilon'(E_i)=\varepsilon'(F_i)=0 \,,
\end{split}
\end{equation}
for $\mu \in X, 1\leq i \leq r$. 
\end{Rem}

%


\subsection{Special choice of $\Phi$}\label{ssec Sevostyanov phi}
\

In this section, we spell out a special choice of $\Phi$ which will be used in various constructions in the rest 
of this paper. Let $\Dyn(\g)$ denote the graph obtained from the Dynkin diagram of $\g$ by replacing all multiple edges 
by simple ones, e.g.\ $\Dyn(\mathfrak{sp}_{2r}) = \Dyn(\mathfrak{so}_{2r+1}) = \Dyn(\mathfrak{sl}_{r+1}) = A_r$. Given 
an orientation $\Or$ of the graph $\Dyn(\g)$, define the \emph{associated matrix} $(\epsilon_{ij})_{i,j=1}^r$ via:
\begin{equation}\label{epsilon matrix}
  \epsilon_{ij}:=
  \begin{cases}
     0 & \mathrm{if}\ a_{ij}\geq 0 \\
     1 & \mathrm{if}\ a_{ij}<0\ \mathrm{and}\ \Or\ \mathrm{contains\ an\ oriented\ edge}\ i\rightarrow j \\
    -1 & \mathrm{if}\ a_{ij}<0\ \mathrm{and}\ \Or\ \mathrm{contains\ an\ oriented\ edge}\ i\leftarrow j  
  \end{cases} \,.
\end{equation}
Then, we consider $\Phi=(\phi_{ij})_{i,j=1}^r$ with 
\begin{equation}\label{eq:condition on Phi}
  \phi_{ij}=\epsilon_{ij}\frac{b_{ij}}{2} \,, 
\end{equation} 
where $b_{ij}=(\a_i, \a_j)$ as in \eqref{eq:b-matrix}. With this choice, we have  
$\sum_{j=1}^r \phi_{ij}\w^\vee_j \in P/2$ for all $1\leq i \leq r$. Therefore, following Remark \ref{rem: twisted Cartan}, 
we will consider the Hopf algebra $\bU(\g,P/2)$ over $\BQ(v^{1/2})$ with the Hopf structure as in 
\eqref{eq: twisted-Hopf}. The new Lusztig form $\dU_\CA(\g)$ and the even part algebra $U^{ev}_\CA$ will be 
Hopf $\CA$-subalgebras of $\bU(\g, P/2)$. To introduce those, we first define the \textbf{modified} 
elements $\{\tE_i, \tF_i\}_{i=1}^r$ of $\bU(\g, P/2)$.

For any $1\leq i\leq r$, let 
\begin{equation}\label{tilda elements}
  \nu^>_i:=-\a_i +\sum_{j=1}^r \phi_{ij}\w^\vee_j \,, \qquad 
  \nu^<_i:=\sum_{j=1}^r \phi_{ij}\w^\vee_j=\alpha_i+\nu^>_i \,,
\end{equation}
and consider 
\begin{equation}\label{Sev twist}
  \tE_i:=E_iK^{\nu^>_i} \,, \qquad \tF_i:=K^{-\nu^<_i}F_i \,.
\end{equation}
We will also need the following elements of $\mathfrak{h}^*$:
\begin{equation}\label{lambda and mu}
  \zlambda_i:=\a_i-2\sum_{j=1}^r \phi_{ij}\w^\vee_j=-\nu^>_i-\nu^<_i \,, \qquad 
  \zmu_i:=-\a_i-2\sum_{j=1}^r \phi_{ij}\w^\vee_j \,.
\end{equation}

One can show that the $\BQ(v^{1/2})$-algebra $\bU(\g, P/2)$ is generated by 
$\{\tE_i, \tF_i, K^\mu\}_{1\leq i\leq r}^{\mu \in P/2}$ subject to the following relations: 
\begin{equation}\label{eq:gen-rel-twistedDJ}
\begin{split}
  & K^{\mu} K^{\mu'}=K^{\mu+\mu'} \,, \qquad K^0=1 \,, \\
  & K^{\mu} \tE_i K^{-\mu}=v^{(\alpha_i,\mu)}\tE_i \,, \qquad K^{\mu} \tF_i K^{-\mu}=v^{-(\alpha_i,\mu)}\tF_i \,, \\
  & \tE_i\tF_j=v^{(\a_i, -\zmu_j)}\tF_j\tE_i \quad(i\neq j) \,, \qquad 
    \tE_i \tF_i - v_i^2\tF_i\tE_i=v_i \frac{1-K_i^{-2}}{1-v_i^{-2}} \,, \\
  & \sum_{m=0}^{1-a_{ij}} 
    (-1)^m v^{m \epsilon_{ij}b_{ij}}\bmat{1-a_{ij}\\ m}_{v_i} \tE_i^{1-a_{ij}-m}\tE_j\tE_i^{m}=0 \quad (i\ne j) \,, \\
  & \sum_{m=0}^{1-a_{ij}}
    (-1)^m v^{m \epsilon_{ij}b_{ij}}\bmat{1-a_{ij}\\ m}_{v_i} \tF_i^{1-a_{ij}-m}\tF_j\tF_i^{m}=0 \quad (i\ne j) \,,
\end{split}
\end{equation}
with the standard notation $K_i:=K^{\alpha_i}$. The algebra $\bU(\g, P/2)$ is $Q$-graded via (cf.~\eqref{grading of DJ}):
\begin{equation}\label{eq: Q-grading for twisted one}
  \deg(\tE_i)=\a_i \,, \qquad \deg(\tF_i)=-\a_i \,, \qquad \deg(K^\mu)=0 \,.
\end{equation}
Moreover, we have 
\begin{equation}\label{eq: twist-Hopf-2}
\begin{split}
  & \Delta'(K^\mu)=K^\mu\otimes K^\mu \,, \ \  \Delta'(\tE_i)=1\otimes \tE_i + \tE_i\otimes K^{-\zlambda_i} \,, \ \ 
    \Delta'(\tF_i)=1\otimes \tF_i+\tF_i \otimes K^{\zmu_i} \,, \\
  & S'(K^\mu)=K^{-\mu} \,, \qquad S'(\tE_i)=-\tE_iK^{\zlambda_i} \,, \qquad S'(\tF_i)=-\tF_iK^{-\zmu_i} \,.
\end{split}
\end{equation}

Let 
\begin{equation}\label{even sublattice}
  2P:=\big\{\mu \in \h^* \,\big|\, \mu/2\in P \big\}=
  \Big\{\mu\in P^\vee \, \Big|\, \tfrac{(\alpha_i,\mu)}{2\sd_i}\in \BZ\ \mathrm{for\ all}\ 1\leq i\leq r\Big\} \,. 
\end{equation}

\begin{Lem}\label{lem: lambda and mu in 2P}
The elements $\{\zlambda_i\}_{i=1}^r$ and $\{\zmu_i\}_{i=1}^{r}$ are contained in $2P$.
\end{Lem}

\begin{proof} 
We have $(\a_j,\zlambda_i)=(\a_j,\a_i-2\sum_{k=1}^r \phi_{ik} \w^\vee_k)=(\a_j, \a_i)-2\phi_{ij}=(1-\epsilon_{ij})(\a_i,\a_j)$. 
Hence $\ds \frac{(\a_j, \zlambda_i)}{2\sd_j}=\frac{1-\epsilon_{ij}}{2}a_{ji}$. If $i=j$ then 
$\ds \frac{(\a_j, \zlambda_i)}{2\sd_j}=1$. If $i \neq j$, then $\ds \frac{(\a_j, \zlambda_i)}{2\sd_j}\in \{0, a_{ji}\}$. 
This implies that  $\zlambda_i\in 2P$. The result for $\zmu_i$ follows from $\zmu_i=\zlambda_i-2\alpha_i$.
\end{proof}

\begin{Rem} 
The elements $\tE_i,\tF_i$ appeared in \cite{s1}, where the motivation was as follows. Let $\bU^<_{\Phi}, \bU^>_{\Phi}$ be 
the $\BQ(v^{1/2})$-subalgebras of $\bU(\g, P/2)$ generated by $\{\tF_i\}, \{\tE_i\}$, respectively. Then, the assignment 
$\tE_i\mapsto 1, \tF_i \mapsto 1$ gives rise to algebra homomorphisms:
\[ 
  \chi^<\colon \bU^<_{\Phi}\longrightarrow \BQ(v^{1/2}) \,, \qquad \chi^>\colon \bU^>_{\Phi} \longrightarrow \BQ(v^{1/2}) \,,
\]
cf.~\cite[Proposition 2, Theorem 4]{s1}. We are not going to use this fact in what follows.
\end{Rem}


\subsection{The (twisted) Lusztig form $\dU_q(\g)$}\label{ssec: new Lusztig form}
\

Let 
\[  
  (n)_v=\frac{1-v^{-2n}}{1-v^{-2}}, \quad(n)_{v}!=(1)_{v}\dots (n)_{v}, \quad 
  \binom{m}{s}_v:= \prod_{c=1}^s \frac{1-v^{2(-m+c-1)}}{1-v^{-2c}}.
\] 
We note that $[n]_{v}=v^{n-1}(n)_{v}$. Let us consider the following elements:
\begin{equation*}
  \binom{K_i;a}{n} = \frac{\prod_{s=1}^n (1-K_i^{-2}v_i^{2(s-a-1)})}{\prod_{s=1}^n(1-v_i^{-2s})} \,, \qquad 
  \tE_i^{(n)}=\frac{\tE_i^n}{(n)_{v_i}!} \,, \qquad \tF_i^{(n)}=\frac{\tF_i^n}{(n)_{v_i}!} \,.
\end{equation*}
We define the \textbf{(twisted) Lusztig form} $\dU_\CA(\g)$ as  the  $\CA$-subalgebra of $\bU(\g, P/2)$ generated by 
the elements $\{\tE^{(s)}_i, \tF_i^{(s)}, K^\mu\}_{1\leq i\le r}^{s\in \BN, \mu \in 2P}$. 
Later on we shall often refer to $\dU_\CA(\g)$ as the Lusztig form, when there is no ambiguity. 
By \eqref{EF-formula}, we have
\begin{equation}\label{eq:ef-twisted-swap-in-v} 
\begin{split}
 \tE_i^{(p)}\tF_j^{(s)}&=v^{ps(\a_i, -\zmu_j)} \tF_j^{(s)}\tE_i^{(p)} \qquad \mathrm{for}\ i\ne j \,, \\
 \tE_i^{(p)}\tF_i^{(s)}&=\sum_{c=0}^{\min(p,s)} v_i^{2ps-c^2} \tF_i^{(s-c)}\binom{K_i; 2c-p-s}{c} \tE_i^{(p-c)} \,. 
\end{split}
\end{equation}
The last two relations in \eqref{eq:gen-rel-twistedDJ} can be rewritten as follows:
\begin{equation}\label{eq:gen-rel-twistedDJ-2}
\begin{split}
  & \sum_{m=0}^{1-a_{ij}}(-1)^m v_i^{ma_{ij}(\epsilon_{ij}-1)-m(m-1)}\tE_i^{(1-a_{ij}-m)}\tE_j \tE_i^{(m)}=0 \qquad  
    (i \neq j) \,, \\
  & \sum_{m=0}^{1-a_{ij}}(-1)^m v_i^{m a_{ij}(\e_{ij}-1)-m(m-1)} \tF_i^{(1-a_{ij}-m)} \tF_j \tF_i^{(m)}=0 \qquad 
    (i \neq j) \,.
\end{split}
\end{equation}

Evoking the twisted Hopf algebra structure from \eqref{eq: twisted-Hopf} and \eqref{eq: twist-Hopf-2}, we have
\begin{equation}\label{eq:twisted Hopf on divided powers}
\begin{split}
  \Delta'(\tE_i^{(s)}) &= \sum_{c=0}^s \tE_i^{(s-c)}\otimes \tE_i^{(c)}K^{-(s-c)\zlambda_i} \,,\qquad 
  S'(\tE_i^{(s)})=(-1)^sv_i^{s(s-1)}\tE_i^{(s)}K^{s\zlambda_i} \,, \\
  \Delta'(\tF_i^{(s)}) &= \sum_{c=0}^sv_i^{2c(s-c)}\tF_i^{(c)}\otimes \tF_i^{(s-c)}K^{c\zmu_i} \,, \qquad 
  S'(\tF_i^{(s)})= (-1)^sv_i^{-s(s-1)}\tF_i^{(s)}K^{-s\zmu_i} \,.
\end{split}
\end{equation}
Combining these formulas with Lemma \ref{lem: lambda and mu in 2P}, we deduce that $\dU_\CA(\g)$ is a Hopf $\CA$-subalgebra 
of $\bU(\g, P/2)$ with respect to the twisted Hopf structure of \eqref{eq: twisted-Hopf} and \eqref{eq: twist-Hopf-2}.

Given $q\in R$ as in the paragraph preceding~(\ref{sigma homom}), we define the \textbf{(twisted) Lusztig form} $\dU_q(\g)$ 
as the base change of $\dU_\CA(\g)$ with respect to $\sigma\colon \CA\rightarrow R$ of \eqref{sigma homom}:
\begin{equation}\label{eq: Lusztig form over R}
  \dU_q(\g):=\dU_\CA(\g) \otimes_\CA R \,.
\end{equation}
We denote the images of $\ds \tE_i^{(s)},\tF_i^{(s)}, K^\mu$, $\binom{K_i;a}{n}$ in $\dU_q(\g)$ by the same symbols.  

We shall often be interested in idempotented versions of the Lusztig forms, see Section~\ref{sssec:idempotented-Lus}. 
In this context, the ``Cartan twist'' disappears and we get the same algebra as~\cite{l-book} but with a different coproduct.


\subsection{The even part algebra $U^{ev}_q(\g)$} \label{ssec:even-part}
\

We define the \textbf{even part} algebra $U^{ev}_\CA$ as the $\CA$-subalgebra of $\bU(\g, P/2)$ generated by the elements 
$\{\tE_i, \tF_i, K^{\mu}\}_{1\leq i \leq r}^{\mu \in 2P}$. By \eqref{eq: twist-Hopf-2} and Lemma \ref{lem: lambda and mu in 2P}, 
it follows that $U^{ev}_\CA(\g)$ is a Hopf $\CA$-subalgebra of $\bU(\g, P/2)$ with respect to  the twisted Hopf structure of 
\eqref{eq: twisted-Hopf} and \eqref{eq: twist-Hopf-2}.

Given $q\in R$ as in the paragraph preceding~(\ref{sigma homom}), we define the \textbf{even part} algebra $U^{ev}_q(\g)$ as 
the base change of $U^{ev}_\CA(\g)$ with respect to $\sigma\colon \CA\rightarrow R$ of \eqref{sigma homom}:
\begin{equation}\label{eq: even part over R}
  U^{ev}_q(\g):=U^{ev}_\CA(\g) \otimes_\CA R \,.
\end{equation}
We denote the images of $\tE_i, \tF_i, K^\mu$ in $U^{ev}_q(\g)$ by the same symbols.


\subsection{Triangular decompositions and PBW bases for $\bU(\g, P/2), \dU_\CA(\g), U^{ev}_\CA(\g)$}
\

Let $\bU^>_{P/2}, \bU^<_{P/2}, \bU^0_{P/2}$ be the $\BQ(v^{1/2})$-subalgebras of $\bU(\g, P/2)$ generated by 
$\{E_i\}_{i=1}^r, \{F_i\}_{i=1}^r$, $ \{K^{\mu}\}_{\mu \in P/2}$, respectively. Let $\bU^>_{\Phi}, \bU^<_\Phi, \bU^0_{\Phi}$ 
be the $\BQ(v^{1/2})$-subalgebras of $\bU(\g, P/2)$ generated by $\{\tE_i\}_{i=1}^r, \{\tF_i\}_{i=1}^r, \{K^\mu\}_{\mu \in P/2}$, 
respectively. We note that $\bU^0_{P/2}=\bU^0_\Phi$. Likewise, the subalgebras generated by $\bU^>_{P/2}, \bU^0_{P/2}$ and 
$\bU^>_{\Phi}, \bU^0_{\Phi}$ (resp.\ $\bU^<_{P/2}, \bU^0_{P/2}$ and $\bU^<_{\Phi}, \bU^0_{\Phi}$) coincide, 
and will be denoted $\bU^\geqslant$ (resp.\ $\bU^\leqslant$). We also note that all these subalgebras are $Q$-graded via~\eqref{eq: Q-grading for twisted one}.

\begin{Lem} \label{lem: triangular-U(g,P/2)}
The following multiplication maps are isomorphisms of $\BQ(v^{1/2})$-vector spaces:
\begin{align}
   \label{eq: old-triangular-DJ}  
  & \mm\colon \bU^<_{P/2}\otimes_{\BQ(v^{1/2})} \bU^0_{P/2} \otimes_{\BQ(v^{1/2})} \bU^>_{P/2} \longrightarrow \bU(\g, P/2) \,, \\
   \label{eq:new-triangular-DJ}  
  & \mm\colon \bU^<_{\Phi}\otimes_{\BQ(v^{1/2})} \bU^0_{\Phi} \otimes_{\BQ(v^{1/2})} \bU^>_{\Phi} \longrightarrow \bU(\g, P/2) \,.
\end{align}
\end{Lem}

\begin{proof} 
The proof of \eqref{eq: old-triangular-DJ} is the same as that of Lemma \ref{triangular_DJ}, while 
\eqref{eq:new-triangular-DJ} follows from \eqref{eq: old-triangular-DJ}.
\end{proof}

The $\BQ$-algebra anti-involution $\tau$ of \eqref{eq: anti-involution map} can be extended to a $\BQ$-algebra 
anti-involution $\tau$ of $\bU(\g,P/2)$ via:
\begin{equation}\label{eq: extended tau map}
    \tau(E_i)=F_i \,, \qquad \tau(F_i)=E_i \,, \qquad 
    \tau(K^\mu)=K^{-\mu} \,, \qquad \tau(v^{1/2})=v^{-1/2} \,,
\end{equation}
for $1\leq i \leq r, \mu \in P/2$.

Since the lattice $P/2$ is stable under the action of the Weyl group $W$, the formulas \eqref{eq:braid group} still 
define a braid group action on $\bU(\g, P/2)$. Therefore, we can still define the root vectors $E_{\b_k}, F_{\b_k}$ 
as in \eqref{eq:root-generator}. However, this braid group action  as well as the map $\tau$ 
of~\eqref{eq: extended tau map} do not preserve $\dU_\CA(\g), U^{ev}_\CA(\g)$. We also note that the elements $E_{\b_k}, F_{\b_k}$ 
are not contained in $\bU^>_{\Phi}, \bU^<_{\Phi}$, respectively. To construct the bases of $\bU^>_{\Phi}, \bU^<_{\Phi}$ 
we shall first introduce \textbf{modified} elements $\tE_{\b_k}, \tF_{\b_k}$. To this end, 
for any $\b=\sum_{i=1}^{r} a_i \a_i \in Q_+$, define (cf.~\eqref{tilda elements}):
\begin{equation}\label{nu for alpha}
  \nu^>_\b= \sum_{i=1}^r a_i \nu^>_i \,, \qquad \nu^<_\b=\sum_{i=1}^r a_i \nu^<_i \,.
\end{equation}

\begin{Lem}\label{lem: some values are integer}
For any $\a, \b\in Q_+$, we have that $(\nu^>_\a, \b) \pm (\a, \nu^<_\b)\in \BZ$.
\end{Lem}

\begin{proof} 
We have $(\nu^>_i, \a_j) - (\a_i, \nu^<_j)= (\epsilon_{ij}-1)(\a_i, \a_j) \in \BZ$,  
$(\nu^>_i, \a_j) + (\a_i, \nu^<_j)=-(\a_i, \a_j)\in \BZ$. 
The general case of $\alpha,\beta\in Q_+$ follows immediately now. 
\end{proof}

For the positive root $\b_k =\sum_{i=1}^{r} a_i \a_i$ of $\g$, set 
\begin{equation}\label{eq: normalizer b}
  b^>_{\beta_k} = \sum_{i<j} a_ia_j(\nu^>_i, \a_j) \,,\qquad
  b^<_{\beta_k} = -\sum_{i<j}a_ia_j(\nu^<_j, \a_i) \,,
\end{equation}
and define
\begin{equation}\label{eq:twisted-root-vector}
  \tE_{\b_k}:=v^{b^>_{\beta_k}}E_{\b_k}K^{\nu^>_{\b_k}} \,, \quad \tF_{\b_k}:=v^{b^<_{\beta_k}}K^{-\nu^<_{\b_k}}F_{\b_k} \,, \qquad 
  \tE_{\b_k}^{(s)}:=\frac{\tE_{\b_k}^s}{(s)_{v_{i_k}}!} \,, \quad \tF_{\b_k}^{(s)}:=\frac{\tF_{\b_k}^s}{(s)_{v_{i_k}}!} \,.
\end{equation}
For any $\vec{k}\in \BZ_{\geq 0}^N$ (with $N=|\Delta_+|$ as before), we define the ordered monomials: 
\begin{equation*}
\begin{split}
  & \tE^{\vec{k}}:=\tE_{\b_1}^{k_1}\dots \tE_{\b_N}^{k_N} \,, \qquad \tF^{\vec{k}}:=\tF_{\b_1}^{k_1}\dots \tF_{\b_N}^{k_N} \,, \quad 
    \tE^{\cev{k}}:=\tE_{\b_N}^{k_N}\dots \tE_{\b_1}^{k_1} \,, \qquad \tF^{\cev{k}}:=\tF_{\b_N}^{k_N}\dots \tF_{\b_1}^{k_1} \,, \\
  & \tE^{(\vec{k})}:=\tE_{\b_1}^{(k_1)}\dots \tE_{\b_N}^{(k_N)} \,, \; \tF^{(\vec{k})}:=\tF_{\b_1}^{(k_1)}\dots \tF_{\b_N}^{(k_N)} \,,\;
   \tE^{(\cev{k})}:=\tE_{\b_N}^{(k_N)}\dots \tE_{\b_1}^{(k_1)} \,, \; \tF^{(\cev{k})}:=\tF_{\b_N}^{(k_N)}\dots \tF_{\b_1}^{(k_1)} \,.
\end{split}
\end{equation*}
Let $\dU^<_\CA, \dU^>_\CA, \dU^0_\CA$ denote the $\CA$-subalgebras of $\dU_\CA(\g)$ generated by 
$\{\tF_i^{(s)}\}, \{\tE_i^{(s)}\}, \{\binom{K_i;a}{s}, K^\mu\}$, respectively. Let $U^{ev\, <}_\CA, U^{ev\, >}_\CA, U^{ev\, 0}_\CA$ 
denote the $\CA$-subalgebras of $U^{ev}_\CA(\g)$ generated by $\{\tF_i\}, \{\tE_i\}$ and $\{K^\mu\}$, respectively.

\begin{Lem} \label{lem: modified elements}
For all $\vec{k}\in \BZ_{\geq 0}^N$, we have: 
\[
  \tE^{\vec{k}}, \tE^{\cev{k}}\in U^{ev\, >}_\CA \,; \quad \tF^{\vec{k}}, \tF^{\cev{k}} \in U^{ev\,  <}_\CA \,; \quad
  \tE^{(\vec{k})}, \tE^{(\cev{k})}\in \dU^>_\CA \,; \quad \tF^{(\vec{k})}, \tF^{(\cev{k})} \in \dU^<_\CA \,.
\]
\end{Lem}

\begin{proof} 
{\it Show $\tE^{\vec{k}} \in U_\CA^{ev>}$ :} Let $\mathcal{I}_{\b_k}$ denote the collection of all tuples $I=(i_1, \dots, i_m)$ with $1 \leq i_1, \dots, i_m \leq r$ 
and $\sum_{p=1}^m \a_{i_p} =\b_k$. Any tuple $I$ in $\CI_{\b_k}$ is obtained from the following tuple $I_0$ by permuting 
the indices: $I_0=(1,\dots, 1, 2, \dots, 2, \dots, r\dots, r)$ with the index $j$ repeated $a_j$ times. For any $I\in \CI_{\b_k}$, 
let $c_I=-\sum_{p=1}^{m-1}(\nu^>_{i_p}, \a_{i_{p+1}}+\ldots+\a_{i_m})$. In particular, $b^>_{\beta_k}+c_{I_0}=\sum_{i=1}^{r} \sd_i a_i(a_i-1) \in \BZ$. 
We now prove that $c_I-c_{I_0} \in \BZ$ for any $I \in \CI_{\b_k}$. Since $I$ can be obtained from $I_0$ via permuting the indices, 
it is enough to show that $c_{I_1}-c_{I_2}\in \BZ$ for two tuples $I_1, I_2\in \mathcal{I}_{\b_k}$ obtained from each other 
by permuting two consecutive indices. To this end, assume that $I_2$ is obtained form $I_1$ by permuting indices $i_p$ and $i_{p+1}$. 
Then $c_{I_1}-c_{I_2}=(\nu^>_{i_{p+1}}, \a_{i_p})-(\nu^>_{i_p}, \a_{i_{p+1}})=2\phi_{i_{p+1},i_p}\in \BZ$. We conclude 
that $b^>_{\beta_k} + c_I \in \BZ$ for any tuple $I$ in $\CI_{\b_k}$.

For any $I=(i_1, \dots, i_m)\in \CI_{\b_k}$, let $E_I=E_{i_1} \dots E_{i_m}$, $\tE_I=\tE_{i_1}\dots \tE_{i_m}$. Since 
$E_{\b_k}$ are contained in $\mU^>_A$ and have degree $\beta_k$, we have $E_{\b_k}=\sum_{I\in \CI_{\b_k}} p_I(v)E_I$ with 
$p_I(v)\in \CA$, and so 
\[
  \tE_{\b_k}=v^{b^>_{\beta_k}}E_{\b_k}K^{\nu^>_{\b_k}} = \sum_{I\in \CI_{\b_k}} v^{b^>_{\beta_k}+c_I}p_I(v)\tE_I.
\]
Since $b^>_{\beta_k}+c_I\in \BZ$, it follows that $\tE_{\b_k}\in U^{ev\, >}_\CA$, hence $\tE^{\vec{k}}\in U^{ev\,>}_\CA$ for any 
$\vec{k} \in \BZ_{\geq 0}^N$. 

{\it Show $\tE^{(\vec{k})} \in \dU_\CA^>$ :} It is enough to show that  $\tE_{\b_k}^{(n)} \in \dU^>_\CA$ for all $n \in \BN$. We have 
\[ \tE^{(n)}_{\b_k}= v^{(\nu^>_{\b_k}, \b_k)(n-1)n/2}v^{n b^>_{\b_k}} E^{(n)}_{\b_k} K^{n \nu^>_{\b_k}}, \quad \text{in which }~ (\nu^>_{\b_k}, \b_k) \in \BZ.\]
Now 
\[E_{\b_k}^{(n)}=\sum_{\textbf{a}=(a_1, \dots a_l)}p_{\textbf{a}}(v) E_{i_1}^{(a_1)}\dots E_{i_l}^{(a_l)},\]
for some tuple  $\textbf{a}=(a_1, \dots a_l)$ such that $\sum_{j=1}^l a_j \a_{i_j}=n \b_k$ and $p_{\textbf{a}}(v) \in \CA$. Then arguing as in the case of $\tE_{\b_k}$, one can show that $v^{n \b^>_k} E^{(n)}_{\b_k} K^{n \nu^>_{\b_k}} \in \dU^>_\CA$, hence $\tE_{\b_k}^{(n)} \in \dU^>_\CA$.

The proofs of the other statements are analogous. 
\end{proof}

We are now ready to establish the PBW-bases for $\bU(\g, P/2), \dU_\CA(\g), U^{ev}_\CA(\g)$. 
We start with $\bU(\g, P/2)$. We also recall the ordered monomials 
$F^{\vec{k}}, E^{\vec{k}}, F^{\cev{k}}, E^{\cev{k}}$ introduced before Lemma~\ref{lem:PBW-DJ}.

\begin{Lem}\label{lem:PBW-twisted-DJ}
(a1) The sets $\{E^{\vec{k}}\}_{\vec{k}\in \BZ^N_{\geq 0}}, \{E^{\cev{k}}\}_{\vec{k}\in \BZ^N_{\geq 0}}$ are 
$\BQ(v^{1/2})$-bases of $\bU^>_{P/2}$.

\noindent
(a2) The sets $\{F^{\vec{k}}\}_{\vec{k}\in \BZ^N_{\geq 0}}, \{F^{\cev{k}}\}_{\vec{k}\in \BZ^N_{\geq 0}}$ are 
$\BQ(v^{1/2})$-bases of $\bU^<_{P/2}$.

\noindent
(a3) The set $\{K^\mu\}_{\mu \in P/2}$ is a $\BQ(v^{1/2})$-basis of $\bU^0_{P/2}$.

\noindent
(a4) The sets 
  $\{F^{\vec{k}}K^\mu E^{\vec{r}}\}, \{F^{\cev{k}}K^\mu E^{\vec{r}}\}, 
   \{F^{\vec{k}}K^\mu E^{\cev{r}}\}, \{F^{\cev{k}}K^\mu E^{\cev{r}}\}$, 
with $\vec{k}, \vec{r}\in \BZ_{\geq 0}^N$ and $\mu \in P/2$, form $\BQ(v^{1/2})$-bases of $\bU(\g, P/2)$. 

\noindent
(b1) The sets $\{\tE^{\vec{k}}\}_{\vec{k}\in \BZ^N_{\geq 0}}, \{\tE^{\cev{k}}\}_{\vec{k}\in \BZ^N_{\geq 0}}$ are 
$\BQ(v^{1/2})$-bases of $\bU^>_\Phi$.

\noindent
(b2) The sets $\{\tF^{\vec{k}}\}_{\vec{k}\in \BZ^N_{\geq 0}}, \{\tF^{\cev{k}}\}_{\vec{k}\in \BZ^N_{\geq 0}}$ are 
$\BQ(v^{1/2})$-bases of $\bU^<_\Phi$.

\noindent
(b3) The set $\{K^\mu\}_{\mu \in P/2}$ is a $\BQ(v^{1/2})$-basis of $\bU^0_\Phi$.

\noindent
(b4)  The sets 
  $\{\tF^{\vec{k}} K^\mu \tE^{\vec{r}}\}, \{\tF^{\cev{k}} K^\mu \tE^{\vec{r}}\}, 
   \{\tF^{\vec{k}} K^\mu \tE^{\cev{r}}\}, \{\tF^{\cev{k}} K^\mu \tE^{\cev{r}}\}$, 
with $\vec{k}, \vec{r}\in \BZ_{\geq 0}^N$ and $\mu \in P/2$, form $\BQ(v^{1/2})$-bases of $\bU(\g, P/2)$. 
\end{Lem}

\begin{proof}
Parts (a1)--(a4) follow from Lemma \ref{lem:PBW-DJ}, while parts (b1)--(b4) follow from (a1)--(a4).
\end{proof}

Before stating the next result, it will be useful to record some of the relations among Cartan elements (cf.~\cite[\S2.3]{l3}):
\begin{align}
\label{eq: K-rel1}\binom{t+t'}{t}_i \binom{K_i;0}{t+t'} &= \binom{K_i;0}{t} \sum_{k=0}^{t'}(-1)^k 
    v_i^{2t'(t+k)-k(k+1)}\binom{t+k-1}{k}_i \binom{K_i;0}{t'-k}\,,\\
\label{eq: K-rel2} \binom{K_i; -c}{t} &= \sum_{0\leq k \leq t} (-1)^k 
    v_i^{2t(c+k)-k(k+1)} \binom{c+k-1}{k}_i \binom{K_i;0}{t-k} \;\; ( c\geq 1) \,, \\
\label{eq: K-rel3} \binom{K_i;c}{t} &= \sum_{0\leq k \leq t} 
    v_i^{2k(t-k)} \binom{c}{k}_i K_i^{-2k} \binom{K_i;0}{t-k} \;\; ( c\geq 0) \,,
\end{align}
for all $t, t' \geq 0$.

\begin{Lem}\label{lem:twisted-triangular-PBW}
(a1) Subalgebras $\dU^<_\CA, \dU^>_\CA, U^{ev\,<}_\CA, U^{ev\,>}_\CA$ are $Q$-graded via \eqref{eq: Q-grading for twisted one}, 
and each of their degree components is a free $\CA$-module of finite rank.

\noindent
(a2) The multiplication maps
\begin{align*}
  & \mm\colon \dU^<_\CA \otimes_\CA \dU^0_\CA \otimes_\CA \dU^>_\CA \longrightarrow \dU_\CA(\g) \,,\\
  & \mm\colon U^{ev\,<}_\CA \otimes_\CA U^{ev\, 0}_\CA\otimes_\CA U^{ev\, >}_\CA \longrightarrow U^{ev}_\CA(\g)
\end{align*}
are isomorphisms of free $\CA$-modules.

\noindent
(b1) The sets $\{\tE^{\vec{k}}\}_{\vec{k}\in \BZ^N_{\geq 0}}, \{\tE^{\cev{k}}\}_{\vec{k}\in \BZ^N_{\geq 0}}$ are 
$\CA$-bases of $U^{ev\,>}_\CA$.

\noindent 
(b2) The sets $\{\tF^{\vec{k}}\}_{\vec{k}\in \BZ^N_{\geq 0}}, \{\tF^{\cev{k}}\}_{\vec{k}\in \BZ^N_{\geq 0}}$ are 
$\CA$-bases of $U^{ev\,<}_\CA$.

\noindent 
(b3) The set $\{K^\mu\}_{\mu \in 2P}$ is an $\CA$-basis of $U^{ev\,0}_\CA$.

\noindent 
(c1) The sets $\{\tE^{(\vec{k})}\}_{\vec{k}\in \BZ^N_{\geq 0}}, \{\tE^{(\cev{k})}\}_{\vec{k}\in \BZ^N_{\geq 0}}$ are 
$\CA$-bases of $\dU^>_\CA$.

\noindent 
(c2) The sets $\{\tF^{(\vec{k})}\}_{\vec{k}\in \BZ^N_{\geq 0}}, \{\tF^{(\cev{k})}\}_{\vec{k}\in \BZ^N_{\geq 0}}$ are 
$\CA$-bases of $\dU^<_\CA$.

\noindent 
(c3) The subalgebra $\dU^0_\CA$ has the following $\CA$-basis:
\begin{equation}\label{eq:basis-Cartan-twisted-Lus}
  \left\{ K^{2\varsigma_j}\cdot \prod_{i=1}^r \left(K_i^{2\big\lfloor \tfrac{t_i}{2} \big\rfloor}\binom{K_i;0}{t_i}\right) 
  \,\Big|\, \;t_i\geq 0, 1\leq j\leq k \right\} \,,
\end{equation}
where $\varsigma_1, \dots, \varsigma_k \in P$ is a set of representatives of the left cosets $P/Q$. 
\end{Lem}

\begin{proof} 
We will present the proof of (b1), (b2), (c1), (c2) after Lemma  \ref{Lem:pairing_quantum_Borel}. Part (b3) easily follows 
from part (b3) of Lemma \ref{lem:PBW-twisted-DJ} that does not depend on this lemma. 

Let us sketch the proof of (c3). Let $\wt{\dU^0_\CA}$ denote the 
$\CA$-submodule of $\dU^0_\CA$ linearly generated by the elements in \eqref{eq:basis-Cartan-twisted-Lus}. It is not hard 
to see that elements in \eqref{eq:basis-Cartan-twisted-Lus} form a $\BQ(v^{1/2})$-basis for the  $\BQ(v^{1/2})$-subalgebra 
$\bU^0(2P)$ of $\bU(\g, P/2)$ generated by $\{K^\mu\}^{\mu \in 2P}$. Hence $\wt{\dU^0_\CA}$ is a free $\CA$-module with 
the $\CA$-basis \eqref{eq:basis-Cartan-twisted-Lus}.

{\it Step 1:} 
We will show that   $ K_i^{2k}\binom{K_i;0}{m}$ for 
$m \in \BN, k \in \BZ$ are $\CA$-linear combinations of elements 
$K_i^{2\big\lfloor \tfrac{t_i}{2} \big\rfloor}\binom{K_i;0}{t_i}$.  
Hence $K^\mu, \binom{K_i; a}{m}$ are contained in 
$\wt{\dU^0_\CA}$ for all $\mu \in 2P, a \in \BZ, m \in \BN$ by using \eqref{eq: K-rel2}--\eqref{eq: K-rel3}. 

We will prove the statement for the even number $2m$, and the proof for odd number $2m+1$ is the same. The statement holds for $K_i^{2m}\binom{K_i;0}{2m}$. Consider
\[ (1-v_i^{-4m})K^{2m}\binom{K_i;0}{2m+1}=K^{2m}\binom{K_i;0}{m}-v_i^{2m}K_i^{2m-2}\binom{K_i;0}{2m},\]
hence the statement holds for $K_i^{2m-2}\binom{K_i;0}{2m}$. Consider 
\begin{multline*} \prod_{i=0}^1(1-v_i^{-4m-i}) K_i^{2m+2} \binom{K_i;0}{2m+2}= K_i^{2m+2}\binom{K_i;0}{2m}-(v_i^{4m}+v_i^{4m+2})K_i^{2m}\binom{K_i;0}{2m}\\
+ v_i^{8m+2}K_i^{2m-2}\binom{K_i;0}{2m},
\end{multline*}
hence the statement holds for $K_i^{2m+2}\binom{K_i;0}{2m}$. Keep doing this procedure, we can show that the statement holds for $K_i^{2k} \binom{K_i;0}{2m}$ for all $k\in \BZ$.

{\it Step 2:} We will show that $\wt{\dU^0_\CA}$ is stable under the left multiplication by $ K^\mu,\binom{K_i;a}{m}$. By \eqref{eq: K-rel2}-\eqref{eq: K-rel3}, it is enough to show that $\wt{\dU^0_\CA}$ is stable under the left multiplication by $K_i^{2k}\binom{K_i;0}{m}$. On the other hand, \eqref{eq: K-rel1} (by induction) implies that $\binom{K_i;0}{t} \binom{K_i;0}{t'}$ is a $\CA$-linear combination of $\binom{K_i;0}{t''}$. Therefore, after multiplying any element of $\wt{\dU^0_\CA}$ by $K_i^{2k} \binom{K_i;0}{m}$, we get a $\CA$-linear combination of $\prod_{i=1}^r K_i^{2k_i} \binom{K_i;0}{t_i}$ for some $k_i \in \BZ$. Using Step 1, the latter $\CA$-linear combination belongs to $\wt{\dU^0_\CA}$.

Therefore, we have  $\wt{\dU^0_\CA}=\dU^0_\CA$ and part (c3) follows.


Part (a1) follows from (b1), (b2), (c1), and (c2). Let us prove the first isomorphism in part (a2). One can show that $\dU^<_\CA \dU^0_\CA\dU^>_\CA =\dU_\CA(\g)$. Indeed,  $1\in \dU^<_\CA \dU^0_\CA \dU^>_\CA$ while $\dU^<_\CA \dU^0_\CA \dU^>_\CA$ is closed under the left multiplications by $\dU_\CA(\g)$ as can be seen by using the commutation relations between $\tE^{(n)}_i, \tF^{(n)}_i, K^\mu, \binom{K_i;0}{m}$.
 
Let $\bU(\g, 2P)$ denote the $\BQ(v^{1/2})$-subalgebra of $\bU(\g, P/2)$ generated by 
$\{\tE_i, \tF_i, K^\mu\}_{1\leq i\leq r}^{\mu \in 2P}$. By Lemma \ref{lem:PBW-twisted-DJ} and the proof of (c3), 
the algebra $\bU(\g, 2P)$ has the following $\BQ(v^{1/2})$-basis:
\begin{equation}\label{eq:PBW-twisted-Lusz}
  \tF^{\vec{k}} K^{2\varsigma_j}\prod_{i=1}^r\left(K_i^{2\big\lfloor \tfrac{t_i}{2} \big\rfloor} 
  \binom{K_i;0}{t_i}\right) \tE^{\vec{r}} \,,
\end{equation}
with the suitable indices as above. Combining this with (c1)--(c3), we see that $\dU^<_\CA \dU^0_\CA \dU^>_\CA$ is a free 
$\CA$-module with the $\CA$-basis \eqref{eq:PBW-twisted-Lusz}, hence the first isomorphism of part (a) follows. The proof 
of the second isomorphism is similar. 
\end{proof}


\subsection{The adjoint action of $\dU_q(\g)$ on $U^{ev}_q(\g)$}
\

For the twisted Hopf algebra structure in \eqref{eq: twisted-Hopf}, we have the left adjoint action $\ad'$ of $\bU(\g,P/2)$ 
on itself. Explicitly, for any $1\leq i\leq r$, $\mu\in P/2$, $x\in \bU(\g, P/2)$, we have:
\begin{equation}\label{eq:explicit adjoint action}
  \ad'(K^\mu)(x)=K^{\mu}xK^{-\mu} \,, \quad \ad'(\tE_i)(x)=[\tE_i, x]K^{\zlambda_i} \,, \quad \ad'(\tF_i)(x)=[\tF_i, x]K^{-\zmu_i} \,,
  \footnote{In contrast, while $\ad(F_i)(x)=[F_i, x]K_i$, there is no similar formula relating $\ad(E_i)(x)$ to $[E_i, x]$.}
\end{equation}
due to~\eqref{eq: twist-Hopf-2}.
It turns out that this action restricts to an action $\dU_\CA(\g) \curvearrowright U^{ev}_\CA(\g)$:

\begin{Prop}\label{action of Lus on DCK}
The adjoint action $\ad'$ of $\dU_\CA(\g)\subset \bU(\g, P/2)$ preserves the even part $U^{ev}_\CA(\g)\subset \bU(\g, P/2)$, 
thus giving rise to 
\begin{equation}\label{eq:Lus-on-DCK}
  \ad'\colon \dU_{\CA}(\g)\curvearrowright U_{\CA}^{ev}(\g) \,.    
\end{equation}
\end{Prop}

Given $q\in R$ as in the paragraph preceding~(\ref{sigma homom}), we obtain 
\begin{equation}\label{eq:specialized ad-action}
  \mathrm{\textbf{(left) adjoint action}}\ \ad'\colon \dU_q(\g)\curvearrowright U_q^{ev}(\g) 
\end{equation}
after applying the base change $\sigma\colon \CA\to R$ of~\eqref{sigma homom} to~\eqref{eq:Lus-on-DCK}.

\begin{proof}[Proof of Proposition~\ref{action of Lus on DCK}]
Due to property~(\ref{ad on products}) and the explicit coproduct formulas in \eqref{eq: twist-Hopf-2} and 
\eqref{eq:twisted Hopf on divided powers}, it suffices to verify the claim for the generators of $\dU_{\CA}(\g)$ 
acting on the generators of $U^{ev}_{\CA}(\g)$. To this end, we first note that $U_{\CA}^{ev}(\g)$ is obviously 
stable under $\{\ad'(K^\beta) \,|\, \beta\in 2P\}$.


Next, let us consider $\ad'(\tE^{(m)}_i)$. First, we apply $\ad'(\tE^{(m)}_i)$ to $\{K^{\beta}\}_{\beta\in 2P}$. 
To this end, recall: 
\begin{equation*}
  \ad'(\tE_i)(x)=[\tE_i,x]K^{\zlambda_i} 
\end{equation*}  
with $\zlambda_i\in 2P$ by Lemma~\ref{lem: lambda and mu in 2P}. Combining this with 
$[\tE_i, \tE_i^{s}K^\beta] = (1-v^{(\beta,\alpha_i)})\tE^{s+1}_iK^\beta$, we obtain:
  $\ad'(\tE_i)(\tE^{s}_iK^\beta) = (1-v^{(\beta,\alpha_i)})\tE^{s+1}_iK^{\beta+\zlambda_i}$.
As $(\zlambda_i,\alpha_i)=(\alpha_i,\alpha_i)=2\sd_i$, we finally get:
\begin{equation}\label{E-acts-K}
  \ad'(\tE^{m}_i)(K^\beta) = 
  \prod_{s=0}^{m-1} \left(1-v^{(\beta,\alpha_i)+2s\sd_i}\right) \cdot \tE^{m}_iK^{\beta+m\zlambda_i} \,.
\end{equation}
Note that $(\beta,\alpha_i)\in 2\sd_i\BZ$ for $\beta\in 2P$, so that:
\begin{equation}\label{eq:divisibility 1}
  \frac{\prod_{s=0}^{m-1} (1-v^{(\beta,\alpha_i)+2s\sd_i})}{(m)_{v_i}!}\in \CA \,.
\end{equation}
Hence, we get the following equality in $\bU(\g)$: 
\begin{equation}\label{dividedE-acts-K}
  \ad'(\tE^{(m)}_i)(K^\beta) = 
  \frac{\prod_{s=0}^{m-1} (1-v^{(\beta,\alpha_i)+2s\sd_i})}{(m)_{v_i}!} \cdot \tE^{m}_iK^{\beta+m\zlambda_i} \,,
\end{equation}
the right-hand side of which lies in $U^{ev}_\CA(\g)$, due to~\eqref{eq:divisibility 1}.


\medskip
Next, we apply $\ad'(\tE_i)$ to $\tE_jK^\beta$ with a suitably chosen $\beta\in 2P$:

\noindent
$\bullet$ For $j=i$, we have $\ad'(\tE^{(m)}_i)(\tE_i)\in U^{ev}_{\CA}(\g)$, due to:
\begin{equation*}
  \ad'(\tE_i)(\tE_i)=[\tE_i,\tE_i]K^{\zlambda_i}=0 \  \Longrightarrow \
  \ad'(\tE^{(m)}_i)(\tE_i)=0 \ \mathrm{for} \ m\geq 1 \,.
\end{equation*}  

\noindent
$\bullet$ For $j\ne i$, set $\beta=\zlambda_j\in 2P$. Evoking~\eqref{eq:twisted Hopf on divided powers}, we thus obtain:
\begin{multline*}
  \ad'(\tE^{1-a_{ij}}_i)(\tE_j K^{\beta})=
  \left(\sum_{m=0}^{1-a_{ij}} (-1)^m v^{m\epsilon_{ij}b_{ij}} 
  \bmat{1-a_{ij}\\ m}_{v_i} \tE_i^{1-a_{ij}-m}\tE_j\tE_i^{m}\right)K^{\beta+(1-a_{ij})\zlambda_i} = 0 \,,
\end{multline*}
with the last equality due to~\eqref{eq:gen-rel-twistedDJ}. Therefore, we have:
\begin{equation*}
  \ad'(\tE^{(m)}_i)(\tE_j K^{\beta})=0\in U^{ev}_{\CA}(\g) \quad \mathrm{for} \quad m>-a_{ij} \,.
\end{equation*}
We also have:
\begin{equation*}
  \ad'(\tE^{(m)}_i)(\tE_j K^{\beta}) = \frac{1}{(m)_{v_i}!}\ad'(\tE^m_i)(\tE_j K^{\beta})\in U^{ev}_{\CA}(\g)
  \quad \mathrm{for} \quad 1\leq m\leq -a_{ij} \,,
\end{equation*}
as $1/(m)_{v_i}!\in \CA$ (here, we use that if $a_{ij}<0$, then either $a_{ij}=-1$ or $a_{ji}=-1$). 

Combining the inclusions $\ad'(\tE^{(m)}_i)(\tE_j K^{\beta})\in U^{ev}_{\CA}(\g)$ established above for all $m$, 
formulas~(\ref{ad on products},~\ref{eq:twisted Hopf on divided powers}), and 
the inclusion $\ad'(\tE^{(m)}_i)(K^{-\beta})\in U^{ev}_{\CA}(\g)$ (also verified above), we conclude 
that $\ad'(\tE^{(m)}_i)(\tE_j)$ is indeed an element of $U^{ev}_{\CA}(\g)$ for any $m\geq 1$.


\medskip
Finally, we apply $\ad'(\tE_i^{(m)})$ to $\tF_jK^\beta$ with $\beta:=-\zmu_j\in 2P$:

\noindent
$\bullet$ For $j\ne i$, we have
\begin{equation*}
  \ad'(\tE_i)(\tF_j K^{\beta}) =
  \left(\tE_i\tF_j-v^{(\alpha_i,-\zmu_j)}\tF_j\tE_i\right) K^{\beta+\zlambda_i} = 0 \,,
\end{equation*}
with the last equality due to the third line of~\eqref{eq:gen-rel-twistedDJ}. Thus, we get $\ad'(\tE^{(m)}_i)(\tF_jK^{\beta})=0$ 
for any $m\geq 1$. Combining this with the inclusion $\ad'(\tE^{(m)}_i)(K^{-\beta})\in U_\CA^{ev}(\g)$ verified 
above and formulas~(\ref{ad on products},~\ref{eq:twisted Hopf on divided powers}), we obtain the desired 
inclusion $\ad'(\tE^{(m)}_i)(\tF_j)\in U_\CA^{ev}(\g)$ for any $m\geq 1$.

\noindent
$\bullet$ For $j=i$, we have
\begin{align*}
  \ad'(\tE_i)(\tF_i K^\beta)&=[\tE_i, \tF_iK^{-\zmu_i}]K^{\zlambda_i}=(\tE_i \tF_i-v_i^2\tF_i\tE_i)K^{\zlambda_i-\zmu_i}\\
  &=v_i \frac{1-K^{-2\a_i}}{1-v_i^2}K^{2\a_i}=\frac{v_i}{1-v_i^{-2}}(K^{2\a_i}-1) \in U^{ev}_\CA(\g) \,,
\end{align*}
due to the third line of \eqref{eq:gen-rel-twistedDJ}. Combining this with~\eqref{E-acts-K}, we obtain for $m\geq 2$:
\begin{equation*}
  \ad'(\tE_i^m)(\tF_iK^\beta)=
  \frac{v_i}{1-v_i^{-2}}\prod_{s=0}^{m-2}(1-v^{4\sd_i+2s\sd_i})\tE_i^{m-1}K^{2\a_i +(m-1)\zlambda_i} \,.
\end{equation*}
Therefore, we have: 
\begin{equation}\label{eq:e-f-k}
  \ad'(\tE_i^{(m)})(\tF_iK^\beta) \in U^{ev}_\CA(\g) \qquad \text{for} \quad m \geq 1 \,,
\end{equation}
due to the obvious inclusion
\begin{equation*}
  \frac{v_i\prod_{s=0}^{m-2}(1-v^{4\sd_i+2s\sd_i})}{(1-v_i^{-2})(m)_{v_i}!}\in \CA \qquad \text{for}\quad m \geq 2 \,.
\end{equation*}
Combining~\eqref{eq:e-f-k}, formulas~(\ref{ad on products},~\ref{eq:twisted Hopf on divided powers}), and the inclusion 
$\ad'(\tE^{(m)}_i)(K^{-\beta})\in U^{ev}_{\CA}(\g)$ verified above, we conclude that $\ad'(\tE^{(m)}_i)(\tF_i)$ 
is indeed an element of $U^{ev}_{\CA}(\g)$ for any $m\geq 1$.

\medskip
The case of $\ad'(\tF^{(m)}_i)$ is treated analogously; we leave details to the interested reader.
\end{proof}


\section{Hopf and invariant pairings}\label{S_inv_pair}

Let $\bU^{\geqslant}, \bU^{\leqslant}$ denote the $\BQ(v^{1/2})$-subalgebras of $\bU(\g, P/2)$ generated by 
$\{E_i, K^\mu\}_{1\leq i \leq r}^{\mu \in P/2}$ and $\{F_i, K^\mu\}_{1\leq i \leq r}^{\mu \in P/2}$, respectively.  
We note that both $\bU^{\geqslant}$ and $\bU^{\leqslant}$ are Hopf $\BQ(v^{1/2})$-subalgebras of $\bU(\g,P/2)$ 
endowed with either the standard Hopf algebra structure $(\Delta,S,\varepsilon)$ of~(\ref{standard_Hopf_DJ}) or 
the twisted one $(\Delta',S',\varepsilon')$ of~(\ref{eq: twisted-Hopf}). Let $U^{ev\, \geqslant}_\CA$ and 
$U^{ev\,\leqslant}_\CA$ (respectively, $\dU^{\geqslant}_\CA$ and $\dU^{\leqslant}_\CA$) denote the 
$\CA$-subalgebras of $U^{ev}_\CA(\g)$ (respectively, of $\dU_\CA(\g)$) generated by $U^{ev\,>}_\CA,U^{ev\,0}_\CA$ 
and $U^{ev\, <}_\CA,U^{ev\,0}_\CA$ (respectively, $\dU^{>}_\CA,\dU^0_\CA$ and $\dU^{<}_\CA,\dU^0_\CA$). 
Likewise, define the $R$-algebras $U^{ev\, \geqslant}_q$, $U^{ev\, \leqslant}_q$, $\dU^{\geqslant}_q$, 
$\dU^{\leqslant}_q$.

The goal of this section is to generalize the Hopf pairing of \cite[\S6.12]{j} and the invariant pairing of 
\cite[\S6.20]{j}\footnote{As noted in~\cite[\S6.18]{j}, this pairing slightly differs from the one in~\cite[\S3.1.7-8]{l-book}.}
to our twisted setup and then construct similar pairings (now involving both the modified Lusztig form and the even part algebra):
\begin{equation*}
  \left(\ ,\ \right)'\colon \dU^{\leqslant}_q \times U^{ev\,\geqslant}_q \longrightarrow R \,, \qquad 
  \left(\ ,\ \right)'\colon U^{ev\, \leqslant}_q \x \dU^{\geqslant}_q \longrightarrow R \,,
\end{equation*}
\begin{equation*}
  \left\langle \ ,\ \right\rangle'\colon U^{ev}_q(\g)\times \dU_q(\g)\longrightarrow R \,.
\end{equation*}


\subsection{Two endomorphisms of $\h^*$}\label{ssec kappa-gamma}
\

Let us recall the  skew-symmetric matrix $\Phi=(\phi_{ij})_{i,j=1}^r$ from Section \ref{ssec Sevostyanov phi}. 
It gives rise to two endomorphisms $\kappa,\gamma\in \End(\h^*)$ defined in the basis $\{\alpha_i\}_{i=1}^r$ via:
\begin{equation}\label{kappa and gamma}
  \kappa(\alpha_i):=\alpha_i+\sum_{j=1}^r 2\phi_{ij}\omega^\vee_j=-\zmu_i \qquad \mathrm{and} \qquad 
  \gamma(\alpha_i):=\alpha_i-\sum_{j=1}^r 2\phi_{ij}\omega^\vee_j=\zlambda_i \,,
\end{equation}
where we use the notations~(\ref{lambda and mu}). These endomorphisms are adjoint to each other:
\begin{equation}\label{adjoint}
  \left(\kappa(\alpha),\beta\right) = \left(\alpha,\gamma(\beta)\right) 
  \qquad \mathrm{for\ any} \quad \alpha,\beta\in \h^* \,,
\end{equation}
with respect to the pairing $(\ ,\ )$ on $\h^*$ that is naturally induced from the one on $\h$.

\begin{Lem}\label{Q-to-2P}
We have $\kappa(Q)=\gamma(Q)=2P$. In particular, $\kappa$ and $ \gamma$ are invertible.
\end{Lem}

\begin{proof}
Since the Dynkin diagram of $\g$ has no cycles, we can reorder $\{\alpha_i\}_{i=1}^r$ (respectively $\{\omega_i\}_{i=1}^r$) 
so that $\epsilon_{ij}=-1$ for $i>j$ with $(\alpha_i,\alpha_j)\ne 0$. Then, evoking the symmetrized Cartan matrix $B$ 
of~\eqref{eq:b-matrix}, we get:
\begin{equation*}
  \begin{pmatrix} 
    \kappa(\alpha_1) \\ \vdots \\ \kappa(\alpha_r) 
  \end{pmatrix} =
  (B+2\Phi)\cdot 
  \begin{pmatrix} 
    \omega^\vee_1 \\ \vdots \\ \omega^\vee_r 
  \end{pmatrix} =
  K\cdot \mathrm{diag}(2\sd_1,\cdots,2\sd_r)\cdot 
  \begin{pmatrix} 
    \omega^\vee_1 \\ \vdots \\ \omega^\vee_r 
  \end{pmatrix} =
  K\cdot \begin{pmatrix} 2\omega_1 \\ \vdots \\ 2\omega_r \end{pmatrix} \,.
\end{equation*}
Here, $K=(k_{ij})_{i,j=1}^r$ is an upper-triangular matrix with $k_{ii}=1$ and $k_{ij}=a_{ij},\ i<j$, which thus maps 
the lattice $2P=\bigoplus_{i=1}^r \BZ\cdot 2\omega_i$ bijectively to itself. This completes the proof of $\kappa(Q)=2P$.

The proof of $\gamma(Q)=2P$ is completely analogous. 
\end{proof}


\subsection{Twisted Hopf pairing}\label{ssec twisted Hopf pairing}
\

Let us pick  $\sN \in \BN$ such that $\frac{1}{2}(P/2, P/2) \in \frac{1}{\sN}\BZ$ and $(\kappa^{-1}(P/2), P/2)\in \frac{1}{\sN} \BZ$. 
The following result is just an extension of \cite[\S6.12]{j}:

\begin{Prop}\label{prop:J-pairing-6.12}
There exists a unique $\BQ(v^{1/2})$-bilinear pairing 
\begin{equation}\label{eq:old-pairing}
  (\ , \ )\colon \bU^{\leqslant} \x \bU^{\geqslant} \longrightarrow \BQ(v^{1/\sN})
\end{equation} 
such that 
\begin{equation}\label{old pairing condition 1}
  (y,xx')=(\Delta(y), x'\otimes x) \,, \qquad (yy', x)=(y\otimes y', \Delta(x)) \,,
\end{equation}
\begin{equation}\label{old pairing condition 2}
  (F_i, E_j)=-\frac{\delta_{i,j}}{v_i-v_i^{-1}} \,, \quad 
  (K^\mu, K^{\mu'})=v^{-(\mu,\mu')} \,, \quad (F_i, K^{\mu'})=(K^\mu, E_i)=0 \,,
\end{equation}
for any $1\leq i,j \leq r,$ $ \mu, \mu'\in P/2$, $x,x'\in \bU^{\geqslant}$, and $y,y'\in \bU^{\leqslant}$.
\end{Prop}


The pairing \eqref{eq:old-pairing} involves the standard Hopf structure on $\bU(\g, P/2)$. The following result 
provides an appropriate generalization of the above proposition to our twisted setup:

\begin{Prop}\label{twisted Hopf pairing on DJ}
There exists a unique $\BQ(v^{1/2})$-bilinear pairing 
\begin{equation}\label{pairing of DJ halves}
  (\ ,\ )'\colon \bU^{\leqslant} \times \bU^{\geqslant} \longrightarrow \BQ(v^{1/\sN})
\end{equation}
such that 
\begin{equation}\label{pairing condition 1}
  (y,xx')'=(\Delta'(y),x'\otimes x)' \,, \qquad  (yy',x)'=(y\otimes y', \Delta'(x))' \,,
\end{equation} 
\begin{equation}\label{pairing condition 2}
  (\tF_i,\tE_j)'=-\frac{\delta_{i,j}}{v_i-v_i^{-1}} \,, \quad 
  (K^\mu, K^{\mu'})'=v^{-(\kappa^{-1}(\mu),\mu')} \,, \quad
  (\tF_i,K^{\mu'})'=(K^\mu,\tE_i)'=0 \,,
\end{equation}
for any $1\leq i,j\leq r$, $\mu,\mu'\in P/2$, $x,x'\in \bU^{\geqslant}$, and $y,y'\in \bU^{\leqslant}$.
\end{Prop}

\begin{proof}
Follows by applying a line-to-line reasoning as in~\cite[\S6.8--6.12]{j}. 
\end{proof}

\begin{Rem} 
We have the following equalities:
\begin{equation}\label{eq:degree-zero property}
  (\bU^\leqslant_{-\nu}, \bU^{\geqslant}_\mu) =0 \,, \quad \big(\bU^{\leqslant}_{-\nu} , \bU^{\geqslant}_{\mu}\big)' = 0 
  \qquad \mathrm{for} \ \ \nu \ne \mu \,,
\end{equation}
\begin{equation}\label{eq:old-Cartan-reduction} 
  (yK^\lambda, x K^{\lambda'}) = (y,x)v^{-(\lambda, \lambda')} 
  \quad \mathrm{for} \ y\in \bU^<_{P/2}, x \in \bU^>_{P/2},\ \lambda, \lambda'\in P/2 \,,
\end{equation}
\begin{equation}\label{eq:new-Cartan-reduction}
  (\wt{y}K^\lambda, \wt{x}K^{\lambda'})'=(\wt{y},\wt{x})'v^{(\lambda, \deg(\wt{x}))-(\kappa^{-1}(\lambda), \lambda')}
  \quad \mathrm{for} \ \wt{y} \in \bU^<_\Phi, \wt{x} \in\bU^>_\Phi,\ \lambda, \lambda' \in P/2 \,,
\end{equation}
where we write $\deg(\tilde{x})$ for the weight of $\tilde{x}$.
\end{Rem}

\begin{Cor}
For any $1\leq i\leq r$ and $n\in \BN$, we have: 
\begin{equation*}
  (\tF^n_i,\tE^n_i)' = 
  (-1)^n v_i^{-n}\frac{(n)_{v_i}!}{(1-v_i^{-2})^n} \,.
\end{equation*}
\end{Cor}

\begin{proof}
Due to~\eqref{pairing condition 1} and~\eqref{eq:twisted Hopf on divided powers}, we have:
\begin{multline*}
  (\tF^n_i,\tE^n_i)' = (\tF^{n-1}_i\otimes \tF_i ,\Delta'(\tE^n_i))' = 
  (n)_{v_i} \cdot (\tF^{n-1}_i,\tE^{n-1}_i)' (\tF_i,\tE_i K^{-(n-1)\zlambda_i})' = \\
  \frac{-(n)_{v_i}}{v_i-v_i^{-1}} (\tF^{n-1}_i,\tE^{n-1}_i)' \,.
\end{multline*}
The result now follows by induction on $n$. 
\end{proof}

The next proposition relates the pairings \eqref{eq:old-pairing} and \eqref{pairing of DJ halves}. Consider the linear 
endomorphisms $\varpi^{\leqslant}$ of $\bU^{\leqslant}$ and $\varpi^{\geqslant}$ of $\bU^{\geqslant}$ defined via:
\begin{align}
  \varpi^{\leqslant} &\colon F_{i_1}\cdots F_{i_n}K^\mu \mapsto 
  K^{-\nu^{<}_{i_1}-\ldots-\nu^{<}_{i_n}} F_{i_1}\cdots F_{i_n} K^{\kappa(\mu)} 
    \label{eq:neg-half-match} \,, \\
  \varpi^{\geqslant} &\colon E_{j_n}\cdots E_{j_1}K^\nu \mapsto 
  E_{j_n}\cdots E_{j_1} K^{\nu-\nu^{<}_{j_n}-\ldots-\nu^{<}_{j_1}} 
    \label{eq:pos-half-match} \,, 
\end{align}
for any $n\geq 0,\, 1\leq i_1,\ldots,i_n,j_1,\ldots, j_n\leq r,\, \mu,\nu\in P/2$, with 
$\nu^<_j=\sum_{p=1}^r \phi_{jp}\omega_p^\vee$ of~\eqref{tilda elements}, $\kappa\in \End(\h^*)$ 
of~\eqref{kappa and gamma}. These maps are well-defined as they preserve the $v$-Serre 
relations~(\ref{DJ eqn 4},~\ref{DJ eqn 5}).

\begin{Prop}\label{lem:pairing comparison}
For any $y\in \bU^{\leqslant}$ and $x\in \bU^{\geqslant}$, we have:
\begin{equation*}
  (y,x)=\big(\varpi^{\leqslant}(y),\varpi^{\geqslant}(x)\big)' \,.
\end{equation*}
\end{Prop}

\begin{proof} 
Let us first evaluate the pairing $(y,x)$ for any $y=F_{i_1}\cdots F_{i_n}\in \bU^<_{P/2}$ 
and $x=E_{j_n}\cdots E_{j_1}$$\in \bU^>_{P/2}$. Applying $n-1$ times the first equality 
of~\eqref{old pairing condition 1}, we obtain:  
\begin{equation*}
  (y,x)=(F_{i_1}\cdots F_{i_n},E_{j_n}\cdots E_{j_1})=
  (\Delta^{(n-1)}(F_{i_1})\cdots \Delta^{(n-1)}(F_{i_n}) , E_{j_1}\otimes \cdots \otimes E_{j_n}) \,.    
\end{equation*}
Here, the $(n-1)$-fold coproduct $\Delta^{(n-1)}(F_i)\in {\bU^{\leqslant}}^{\otimes n}$ is explicitly given by:
\begin{equation*}
  \Delta^{(n-1)}(F_i) = 
  \sum_{a=1}^n \underbrace{1\otimes \cdots \otimes 1}_{a-1 \text{ times}} \otimes F_i \otimes 
  \underbrace{K^{-\alpha_i}\otimes \cdots \otimes K^{-\alpha_i}}_{n-a \text{ times}} \,.  
\end{equation*}
Therefore, we have: 
\begin{multline*}
  \Delta^{(n-1)}(F_{i_1})\cdots \Delta^{(n-1)}(F_{i_n}) = 
  \sum_{\sigma\in S_n} \bigotimes^{\to}_{1\leq l\leq n} \prod_{m<l}^{\sigma(m)<\sigma(l)} K^{-\alpha_{i_{\sigma(m)}}} 
  \cdot F_{i_{\sigma(l)}} \cdot \prod_{m<l}^{\sigma(m)>\sigma(l)} K^{-\alpha_{i_{\sigma(m)}}} + (\cdots) \\ = 
  \sum_{\sigma\in S_n} \bigotimes^{\to}_{1\leq l\leq n} 
  v^{\sum_{1\leq m<l}^{\sigma(m)<\sigma(l)} (\alpha_{i_{\sigma(m)}},\alpha_{i_{\sigma(l)}})}\cdot 
  F_{i_{\sigma(l)}} \cdot \prod_{m<l} K^{-\alpha_{i_{\sigma(m)}}} + (\cdots) \,.
\end{multline*}
Here, $\underset{1\leq l\leq n}{\overset{\to}{\bigotimes}}$ denotes the ordered tensor product (that is, 
$\underset{1\leq l\leq n}{\overset{\to}{\bigotimes}} x_i= x_1\otimes x_2\otimes \cdots \otimes x_n$), while 
$(\cdots)$ denotes all other terms with at least one tensorand being of degree zero, which thus have a trivial 
$(\ ,\ )$-pairing with $E_{j_1}\otimes \cdots \otimes E_{j_n}$ by~\eqref{eq:degree-zero property}. 
As $(F_iK^{\mu},E_j)=(F_i,E_j)$, we thus get:
\begin{equation}\label{eq:pair-comp-2}
  (F_{i_1}\cdots F_{i_n},E_{j_n}\cdots E_{j_1}) =
  \sum_{\sigma\in S_n} v^{\sum_{1\leq m<l\leq n}^{\sigma(m)<\sigma(l)} (\alpha_{i_{\sigma(m)}},\alpha_{i_{\sigma(l)}})} \cdot
  (F_{i_{\sigma(1)}},E_{j_1}) \cdots (F_{i_{\sigma(n)}},E_{j_n}) \,.    
\end{equation}

Let us now similarly compute the twisted Hopf pairing $(\varpi^{\leqslant}(y),\varpi^{\geqslant}(x))'$ for $x$ and 
$y$ as above. Rewriting~\eqref{eq:neg-half-match} as 
  $\varpi^{\leqslant}(F_{i_1}\cdots F_{i_n})=v^{\sum_{1\leq m<l\leq n} \phi_{i_l i_m}}\tF_{i_1}\cdots \tF_{i_n}$, 
we obtain:
\begin{multline*}
  (\varpi^{\leqslant}(y),\varpi^{\geqslant}(x))' = \\ 
  v^{\sum_{1\leq m<l\leq n} \phi_{i_l i_m}}
  (\Delta'^{(n)}(\tF_{i_1})\cdots \Delta'^{(n)}(\tF_{i_n}) , 
    K^{-\nu^<_{j_1}-\ldots-\nu^<_{j_n}}\otimes E_{j_1}\otimes \cdots \otimes E_{j_n})' \,.    
\end{multline*}
Here, the $n$-fold coproduct $\Delta'^{(n)}(\tF_i)\in {\bU^{\leqslant}}^{\otimes(n+1)}$ is explicitly given by:
\begin{align*}
  \Delta'^{(n)}(\tF_i) & =
  \sum_{a=1}^{n+1} \underbrace{1\otimes \cdots \otimes 1}_{a-1 \text{ times}} \otimes \tF_i \otimes 
  \underbrace{K^{\zmu_i}\otimes \cdots \otimes K^{\zmu_i}}_{n+1-a \text{ times}} \\ 
  & = \sum_{a=1}^{n+1} \underbrace{1\otimes \cdots \otimes 1}_{a-1 \text{ times}} \otimes \tF_i\otimes 
  \underbrace{K^{-\kappa(\alpha_i)}\otimes \cdots \otimes K^{-\kappa(\alpha_i)}}_{n+1-a \text{ times}} \,.
\end{align*}
Therefore, we have: 
\begin{align*}
  \Delta'^{(n)}(\tF_{i_1})\cdots & \Delta'^{(n)}(\tF_{i_n}) \\
  & = 1\otimes \sum_{\sigma\in S_n} 
    \bigotimes^{\to}_{1\leq l\leq n} \prod_{m<l}^{\sigma(m)<\sigma(l)} K^{-\kappa(\alpha_{i_{\sigma(m)}})} 
    \cdot \tF_{i_{\sigma(l)}} \cdot \prod_{m<l}^{\sigma(m)>\sigma(l)} K^{-\kappa(\alpha_{i_{\sigma(m)}})} + (\cdots) \\
  & = 1\otimes \sum_{\sigma\in S_n} \bigotimes^{\to}_{1\leq l\leq n} 
    v^{\sum_{m<l}^{\sigma(m)<\sigma(l)} (\kappa(\alpha_{i_{\sigma(m)}}),\alpha_{i_{\sigma(l)}})} \cdot 
    \tF_{i_{\sigma(l)}} \cdot \prod_{m<l} K^{-\kappa(\alpha_{i_{\sigma(m)}})} + (\cdots) \,,
\end{align*}
where $(\cdots)$ denotes all other terms which, for degree reasons~\eqref{eq:degree-zero property}, have a zero 
pairing with $K^{-\nu^<_{j_1}-\ldots-\nu^<_{j_n}}\otimes E_{j_1}\otimes \cdots \otimes E_{j_n}$ via $(\ ,\ )'$.  
Due to the second equality of~\eqref{pairing condition 1}, we also~have:
\begin{equation*}
  (\tF_i K^{-\kappa(\mu)},E_j)' = 
  (\tF_i\otimes K^{-\kappa(\mu)},E_j\otimes K^{\sum_{p=1}^{r} \phi_{jp} \omega_p^\vee})' =
  v^{\sum_{p=1}^{r} \phi_{jp}(\mu,\omega_p^\vee)}(\tF_i,E_j)' \,.
\end{equation*} 
Combining all the above, we thus obtain:
\begin{multline*}
  (\varpi^{\leqslant}(F_{i_1}\cdots F_{i_n}),\varpi^{\geqslant}(E_{j_n}\cdots E_{j_1}))'=\\
  \sum_{\sigma\in S_n} \left(v^{\sum_{1\leq m<l\leq n}^{\sigma(m)<\sigma(l)} (\alpha_{i_{\sigma(m)}},\alpha_{i_{\sigma(l)}})} \cdot
  (\tF_{i_{\sigma(1)}},E_{j_1})' \cdots (\tF_{i_{\sigma(n)}},E_{j_n})' \right. \times \\ 
  \left. v^{\sum_{m<l} \phi_{i_l i_m} + \sum_{m<l} \phi_{i_{\sigma(l)} i_{\sigma(m)}} + 
     \sum_{m<l}^{\sigma(m)<\sigma(l)} 2\phi_{i_{\sigma(m)} i_{\sigma(l)}}} \right) \,.    
\end{multline*}
But the exponent of $v$ in the last line above vanishes due to $\phi_{ab}=-\phi_{ba}$, so that:
\begin{equation}\label{eq:pair-comp-3'}
\begin{split}
  & (\varpi^{\leqslant}(F_{i_1}\cdots F_{i_n}),\varpi^{\geqslant}(E_{j_n}\cdots E_{j_1}))' = \\
  & \qquad \qquad \qquad 
    \sum_{\sigma\in S_n} v^{\sum_{1\leq m<l\leq n}^{\sigma(m)<\sigma(l)} (\alpha_{i_{\sigma(m)}},\alpha_{i_{\sigma(l)}})} \cdot 
    (\tF_{i_{\sigma(1)}},E_{j_1})' \cdots (\tF_{i_{\sigma(n)}},E_{j_n})'  \,.    
\end{split}
\end{equation}
Comparing~(\ref{eq:pair-comp-2},~\ref{eq:pair-comp-3'}), evoking~\eqref{eq:degree-zero property} and
$(F_i,E_j)=-\frac{\delta_{i,j}}{v_i-v_i^{-1}}=(\tF_i,E_j)'$, we obtain: 
\begin{equation*}
  (y,x)=(\varpi^{\leqslant}(y),\varpi^{\geqslant}(x))' \qquad 
  \mathrm{for\ any}\quad y\in \bU^<_{P/2},\,  x\in \bU^>_{P/2} \,.
\end{equation*}

It remains only to incorporate the Cartan part. For any homogeneous $y\in \bU^<_{P/2}, x \in \bU^>_{P/2}$:  
\begin{equation*}
  (y K^\mu , x K^\nu) = v^{-(\mu,\nu)} \cdot (y,x) \,,
\end{equation*}
in accordance with~\eqref{eq:old-Cartan-reduction}. We also note that 
$\hat{y}:=\varpi^{\leqslant}(y)=\varpi^{\leqslant}(yK^\mu)K^{-\kappa(\mu)}$ lies in $\bU^<_\Phi$, 
while $\hat{x}:=\varpi^{\geqslant}(x)K^{-\gamma(\deg(x))}=\varpi^{\geqslant}(xK^\nu)K^{-\gamma(\deg(x))-\nu}$ 
lies in $\bU^>_\Phi$, so that by~\eqref{eq:new-Cartan-reduction} we get: 
\begin{equation*}
  (\varpi^{\leqslant}(yK^\mu),\varpi^{\geqslant}(xK^\nu))' =
  v^{-(\mu,\nu)} \cdot (\hat{y},\hat{x})' = 
  v^{-(\mu,\nu)} \cdot (\hat{y},\hat{x}K^{\gamma(\deg(x))})' = 
  v^{-(\mu,\nu)} \cdot (\varpi^{\leqslant}(y),\varpi^{\geqslant}(x))' \,. 
\end{equation*}
Combining the above two equalities with the earlier part of the proof, we obtain 
\begin{equation*}
  (\varpi^{\leqslant}(yK^\mu),\varpi^{\geqslant}(xK^\nu))' = 
  v^{-(\mu,\nu)} \cdot (\varpi^{\leqslant}(y),\varpi^{\geqslant}(x))' = \\
  v^{-(\mu,\nu)} \cdot (y,x) = (yK^\mu,xK^\nu) \,.
\end{equation*}

This completes our proof of Proposition~\ref{lem:pairing comparison}.
\end{proof}

Recall the elements $\{\tE_{\b_k}, \tF_{\b_k}\}_{k=1}^{N}$ of \eqref{eq:twisted-root-vector} 
and the constants $\{b^>_{\beta_k}, b^<_{\beta_k}\}_{k=1}^N$ of~\eqref{eq: normalizer b}.

\begin{Cor}\label{cor: pairing of PBW}
For any $\vec{k}=(k_1,\ldots,k_N), \vec{r}=(r_1,\ldots,r_N)\in \BZ^N_{\geq 0}$, we have
\begin{equation*}
  (\tF^{\cev{k}} , \tE^{\cev{r}})' = 
  \delta_{\vec{k}, \vec{r}}\  v^{A_{\vec{k}}}\prod_{p=1}^N v_{i_p}^{\frac{k_p(k_p-1)}{2}} 
    \frac{[k_p]_{v_{i_p}}!}{(v_{i_p}^{-1}-v_{i_p})^{k_p}} \,,
\end{equation*}
with $A_{\vec{k}}$ explicitly given by 
\begin{align*}
  A_{\vec{k}} &= \sum_{i=1}^N k_i(b^>_{\beta_i}+b^<_{\beta_i}) + \sum_{i=1}^N \frac{k_i(k_i-1)}{2} \big((\nu^>_{\b_i},\b_i) - (\nu^<_{\b_i},\b_i)\big) + 
    \sum_{i<j} k_i k_j \big((\nu^>_{\b_j}, \b_i) - (\nu^<_{\b_i}, \b_j) \big) \,.
\end{align*}
\end{Cor}

\begin{proof} 
First, we note that $A_{\vec{k}}\in \BZ$, due to Lemma \ref{lem: some values are integer}. Recall that 
$\tF_{\b_k}=v^{b^<_{\beta_k}}K^{-\nu^<_{\b_k}}F_{\b_k}$ and $ \tE_{\b_k}=v^{b^>_{\beta_k}}E_{\b_k}K^{\nu^>_{\b_k}}$. Hence 
\begin{equation}\label{eq:tF-vs-F-products}
\begin{split}
  \tF^{\cev{k}} &=
  \Lprod_{1\leq i \leq N}(v^{b^<_{\beta_i}}K^{-\nu^<_{\b_i}} F_{\b_i})^{k_i} = 
  \Lprod_{1\leq i \leq N} \Big( v^{k_ib^<_{\beta_i}} v^{-(\nu^<_{\b_i}, \b_i)\frac{k_i(k_i-1)}{2}}K^{-k_i\nu^<_{\b_i}}F^{k_i}_{\b_i} \Big) \\
  &= v^{\sum_i k_i b^<_{\beta_i}} v^{-\sum_i (\nu^<_{\b_i}, \b_i)\frac{k_i(k_i-1)}{2}} v^{-\sum_{i<j} k_ik_j(\nu^<_{\b_i},\b_j)} 
    \cdot  K^{-\sum_i k_i \nu^<_{\b_i}} F^{\cev{k}} \\
  &=v^{\sum_i k_i b^<_{\beta_i}} v^{-\sum_i (\nu^<_{\b_i}, \b_i)\frac{k_i(k_i-1)}{2}} v^{-\sum_{i<j} k_ik_j(\nu^<_{\b_i},\b_j)} 
    \cdot \varpi^{\leqslant}(F^{\cev{k}})
\end{split}
\end{equation}
and 
\begin{equation}\label{eq:tE-vs-E-products}
\begin{split}
  \tE^{\cev{r}}
  &= \Lprod_{1\leq i \leq N}(v^{b^>_{\beta_i}}E_{\b_i}K^{\nu^>_{\b_i}})^{r_i} = 
     \Lprod_{1\leq i \leq N} \Big( v^{r_ib^>_{\beta_i}}v^{(\nu^>_{\b_i}, \b_i)\frac{r_i(r_i-1)}{2}}E_{\b_i}^{r_i}K^{r_i\nu^>_{\b_i}} \Big) \\
  &=v^{\sum_i r_ib^>_{\beta_i}}v^{\sum_i(\nu^>_{\b_i}, \b_i)\frac{r_i(r_i-1)}{2}} 
    v^{\sum_{i<j}r_ir_j(\nu^>_{\b_j}, \b_i)} \cdot E^{\cev{r}} K^{\sum_i r_i \nu^>_{\b_i}} \\
  &=v^{\sum_i r_ib^>_{\beta_i}}v^{\sum_i(\nu^>_{\b_i}, \b_i)\frac{r_i(r_i-1)}{2}} v^{\sum_{i<j}r_ir_j(\nu^>_{\b_j}, \b_i)} 
    \cdot \varpi^{\geqslant}(E^{\cev{r}}) K^{\sum_i r_i(\nu^>_{\b_i}+\nu^<_{\b_i})} \,.
\end{split}
\end{equation}
Since $\varpi^{\leqslant}(F^{\cev{k}}) \in \bU^<_{\Phi}$ by~\eqref{eq:tF-vs-F-products}, we get 
\begin{align*}
  (\tF^{\cev{k}},\tE^{\cev{r}})' 
  = \delta_{\vec{k},\vec{r}}  v^{A_{\vec{k}}} 
    \left(\varpi^{\leqslant}(F^{\cev{k}}), \varpi^{\geqslant}(E^{\cev{k}})K^{\sum_i k_i(\nu^>_{\b_i}+\nu^<_{\b_i})} \right)' 
  =\delta_{\vec{k},\vec{r}} v^{A_{\vec{k}}} \left(\varpi^{\leqslant}(F^{\cev{k}}), \varpi^{\geqslant}(E^{\cev{k}})\right)' \,, 
\end{align*}
due to~(\ref{eq:tF-vs-F-products},~\ref{eq:tE-vs-E-products}) and \eqref{eq:new-Cartan-reduction}. 
By Proposition \ref{lem:pairing comparison}, we thus obtain: 
\begin{equation*}
  (\tF^{\cev{k}} , \tE^{\cev{r}})' = \delta_{\vec{k},\vec{r}} v^{A_{\vec{k}}} ( F^{\cev{k}} , E^{\cev{k}} ) \,.
\end{equation*}
But according to \cite[\S8.29--8.30]{j}, we have 
\begin{equation}\label{eq:PBW-pairing}
  (F^{\cev{k}},E^{\cev{k}}) = 
  \prod_{p=1}^N v_{i_p}^{\frac{k_p(k_p-1)}{2}} \frac{[k_p]_{v_{i_p}}!}{(v^{-1}_{i_p}-v_{i_p})^{k_p}} \,.
\end{equation}
This implies the equality of the corollary. 
\end{proof}

Let us now investigate the behavior of the pairing~\eqref{pairing of DJ halves} with respect to $\dU_\CA(\g)$ and $U^{ev}_\CA(\g)$:

\begin{Lem}\label{Lem:pairing_quantum_Borel}
The restriction of~\eqref{pairing of DJ halves} gives rise to $\CA$-valued pairings:
\begin{equation}\label{pairing of DCK and Lus halves 1}
  (\ ,\ )'\colon \dU_\CA^{\leqslant} \times U^{ev\,\geqslant}_\CA \longrightarrow \CA \,,
\end{equation}
\begin{equation}\label{pairing of DCK and Lus halves 2}
  (\ ,\ )'\colon U_\CA^{ev\,\leqslant} \times \dU_\CA^{\geqslant} \longrightarrow \CA \,.
\end{equation}
Both~(\ref{pairing of DCK and Lus halves 1},~\ref{pairing of DCK and Lus halves 2}) are uniquely determined 
by the properties~(\ref{pairing condition 1},~\ref{pairing condition 2}).
\end{Lem}


\begin{proof}
To show that $(y,x)'\in \CA$ for any $y\in \dU_\CA^{\leqslant}$ and $x\in U^{ev\,\geqslant}_\CA$, we first apply 
the formulas~(\ref{pairing condition 1}) to reduce to the case $x\in \{\tE_i, K^{\nu}\}_{1\leq i\leq r}^{\nu\in 2P}$ 
and $y\in \{\tF^{(s)}_j, K^{\mu}, \binom{K_j; a}{n}\}_{1\leq j\leq r,\mu\in 2P}^{a\in \BZ,n\in \BN}$. Hence, it remains 
to apply the following explicit formulas (easily derived from~\eqref{pairing condition 1} and~\eqref{pairing condition 2}):
\begin{equation}\label{explicit formulas 1}
\begin{split}
  & (\tF^{(s)}_j,\tE_i)' = -\frac{\delta_{i,j}\delta_{s,1}}{v_i-v_i^{-1}} \,, \\
  & (\tF^{(s)}_j,K^\nu)'=0 \,, \qquad
    (K^\mu,\tE_i)'=\left(\binom{K_j; a}{n},\tE_i\right)'=0 \,, \\
  & (K^\mu,K^\nu)'=v^{-(\kappa^{-1}(\mu),\nu)} \,, \qquad
    \left(\binom{K_j; a}{n},K^\nu\right)'=\binom{a-(\kappa^{-1}(\alpha^\vee_j),\nu)}{n}_{v_j}.
\end{split}    
\end{equation}
We note that $(\kappa^{-1}(\mu),\nu)\in \BZ$ and $(\kappa^{-1}(\a^\vee_j), \nu)=(\a^\vee_j, \gamma^{-1}(\nu))\in \BZ$ for 
any $\mu,\nu\in 2P$ by~\eqref{adjoint} and Lemma~\ref{Q-to-2P}. So, indeed $(y,x)'\in \CA$.

The verification of 
$(y,x)'\in \CA$ for any $y\in U^{ev\,\leqslant}$ and $x\in \dU_\CA^{\geqslant}$ is likewise based on:
\begin{equation}\label{explicit formulas 2}
\begin{split}
  & (\tF_j,\tE^{(s)}_i)'=-\frac{\delta_{i,j}\delta_{s,1}}{v_i-v_i^{-1}} \,, \\
  & (\tF_j,K^\mu)'=\left(\tF_j,\binom{K_i;a}{n}\right)'=0 \,, \qquad
    (K^\nu,\tE^{(s)}_i)'=0 \,, \\ 
  & (K^\nu,K^\mu)'=v^{-(\kappa^{-1}(\nu),\mu)} \,, \qquad
    \left(K^\nu, \binom{K_i; a}{n}\right)'=\binom{a-(\kappa^{-1}(\nu),\alpha^\vee_i)}{n}_{v_i} \,.
\end{split}    
\end{equation}

Finally, the uniqueness part of Lemma~\ref{Lem:pairing_quantum_Borel} is obvious.
\end{proof}

We can now complete the proof of Lemma \ref{lem:twisted-triangular-PBW}:

\begin{proof}[Proof of (b1), (b2), (c1), (c2) in Lemma \ref{lem:twisted-triangular-PBW}]
First, let us show that the second set in (b1) forms an $\CA$-basis of $U^{ev\,>}_\CA$. 
According to Lemmas~\ref{lem: modified elements}--\ref{lem:PBW-twisted-DJ}, we have an inclusion 
  $U^{ev\,>}_\CA \supseteq \bigoplus_{\vec{k}\in \BZ^N_{\geq 0}} \CA \cdot \tE^{\cev{k}}$. 
For any $x\in U^{ev\,>}_\CA$, we have $(\tF^{(\cev{k})},x)' \in \CA$ by Lemma \ref{Lem:pairing_quantum_Borel}. 
On the other hand, writing $x=\sum_{\vec{k}\in \BZ^N_{\geq 0}} c_{\vec{k}}\tE^{\cev{k}}$ with $c_{\vec{k}}\in \BQ(v^{1/2})$ 
via Lemma~\ref{lem:PBW-twisted-DJ} and using Corollary \ref{cor: pairing of PBW}, we get $(\tF^{(\cev{k})},x)'=f_{\vec{k}}(v)c_{\vec{k}}$. 
Here, $f_{\vec{k}}(v)$ is an invertible element of $\CA$, and therefore $c_{\vec{k}}\in \CA$. We thus conclude that 
$U^{ev\, >}_\CA=\bigoplus_{\vec{k}\in \BZ_{\geq 0}^N}\CA \cdot \tE^{\cev{k}}$. The proofs that the second sets 
in (b2), (c1), (c2) form $\CA$-bases for the corresponding algebras are completely analogous.

Let us now prove that the first set in (b1) is an $\CA$-basis of $U^{ev\,>}_\CA$. 
Recall the map $\tau$ of~\eqref{eq: extended tau map}. Since $\tau(F_{\b_k})=E_{\b_k}$ and $\tau(E_{\b_k})=F_{\b_k}$,  we have: 
\[ 
  \tau(\tF_i)=\tE_i K^{\a_i} \,, \qquad  \tau(\tF_{\b_k})=v^{-b^<_{\beta_k}-b^>_{\beta_k}}\tE_{\b_k}K^{\b_k} 
  \qquad \forall\, 1\leq i \leq r \,, 1\leq k \leq N \,,
\]
in which $-b^<_{\beta_k}-b^>_{\beta_k} \in \BZ$ by definition \eqref{eq: normalizer b} and Lemma~\ref{lem: some values are integer}, 
and we used $\nu^<_{\beta_k}-\nu^>_{\beta_k}=\beta_k$. Let $\mathscr{U}^{ev\,>}_\CA$ be the $\CA$-subalgebra of $\bU(\g,P/2)$ 
generated by $\tE_iK^{\a_i}$ for $1\leq i \leq r$. Then the map $\tau\colon U^{ev\,<}_\CA\rightarrow \mathscr{U}^{ev\,>}_\CA$ 
is a $\BZ$-algebra anti-isomorphism. Therefore, $\mathscr{U}^{ev\,>}_\CA$ has an $\CA$-basis consisting of elements 
$\Rprod_{1\leq j \leq N} (\tE_{\b_j}K^{\b_j})^{k_j}$. On the other hand, we have $Q_+$-gradings: 
\[
  U^{ev\,>}_\CA=\bigoplus_{\mu \in Q_+} U^{ev\,>}_\mu \,, \qquad  
  \mathscr{U}^{ev\, >}_\CA=\bigoplus_{\mu \in Q_+}\mathscr{U}^{ev\,>}_\mu \,,
\]
so that $U^{ev\,>}_\mu=\mathscr{U}^{ev\,>}_\mu K^{-\mu}$ for all $\mu \in Q_+$. Therefore, the elements $\tE^{\vec{k}}$ form 
an $\CA$-basis of $U^{ev\,>}_\CA$. The proofs that the first sets in (b2), (c1), (c2) form $\CA$-bases of the corresponding 
algebras are analogous. 
\end{proof}

After the base change with respect to $\sigma\colon \CA\to R$ of~(\ref{sigma homom}), we obtain:

\begin{Cor}\label{Lem:pairing_quantum_Borel_R}
There exist unique $R$-valued Hopf pairings
\begin{equation}\label{pairing of DCK and Lus halves 1R}
  (\ ,\ )'\colon \dU_q^{\leqslant} \times U^{ev\,\geqslant}_q \longrightarrow R \,,
\end{equation}
\begin{equation}\label{pairing of DCK and Lus halves 2R}
  (\ ,\ )'\colon U_q^{ev\,\leqslant} \times \dU_q^{\geqslant} \longrightarrow R \,,
\end{equation}
satisfying~(\ref{explicit formulas 1}) and~(\ref{explicit formulas 2}), respectively, where $v=q$.
\end{Cor}

For $\lambda \in P$, let us define the character $\hchi_\lambda: \dU^0_q\rightarrow R$ as follows:
\begin{equation}\label{eq: defi of chi-lambda} 
\hchi_\lambda: \qquad K^\mu \mapsto q^{(\mu, \lambda)}, \qquad \binom{K_j;a}{n} \mapsto  \binom{a+(\a^\vee_j, \lambda)}{n}_{q_j} \qquad \forall~ \mu \in 2P, n\in \BN.
\end{equation}
\begin{Lem}\label{lem: independent of characters}
(a) Characters $\{ \hchi_\lambda\}_{\lambda \in P}$ are pair-wise distinct.

\noindent
(b) Suppose $\sum a_\lambda \hchi_\lambda =0$ for finitely many $a_\lambda \neq 0 \in R$ then $a_\lambda =0$ for all $\lambda$.
\end{Lem}
We need the following lemma: 

\begin{Lem}\label{lem: simple spectrum}  
Let $q \in R$ be invertible. If $m\in \BZ$ is such that $q^{2m} =1$ and $\binom{m}{n}_q=0$ for all $n \in \BN$, then 
$m=0$.
\end{Lem}

\begin{proof}If $q$ is not a root of unity then $q^{2m}=1$ implies that $m=0$. So we can assume $q$ is a root of unity in $R$. Let $\m$ be a maximal ideal of $R$ and $\BF=R/\m$. Let $\bar{q}$ be the image of $q$ in $\BF^\times$. Let $\ell$ be the order of $\bar{q}^2$ in $\BF$. In the field $\BF$,  we have $\bar{q}^{2m}=1$ and $\binom{m}{n}_{\bar{q}}=0$ for all $n \in \BN$. 
Since $\bar{q}^{2m}=1$, we have $m=\ell a$ for some $a \in \BZ$. We consider two cases:
\begin{enumerate}

\item[$\bullet$]
if $\text{char} (\BF)=0$, then $a=\binom{m}{\ell}_q=0$ in $\BF$ by Lemma \ref{lem:aux-at-roots}(a), which implies that $a=0$ and so $m=0$;

\item[$\bullet$]
if $\text{char} (\BF)=p$, then $\binom{a}{p^n}=\binom{m}{p^n \ell}_q=0$ in $\BF$ for all $n \in \BN$, see Lemma \ref{lem:aux-at-roots}(a) below. Since $p \nmid \binom{a}{p^n}$ when either $p^n \leq a<p^{n+1}$ or $p^n-p^{n+1}\leq a<0$,  
it follows that $a=0$ and so $m=0$ in this case.

\end{enumerate}
This completes the proof.
\end{proof}

\begin{proof}[Proof of Lemma \ref{lem: independent of characters}]
(a) Now assume that $\hchi_\lambda= \hchi_\mu$ for $\lambda, \mu \in P$. We want to show that $\lambda =\mu$. By $\hchi_\lambda(K^{2\a_i})=\hchi_\mu(K^{2\a_i})$, we ave $q_i^{2(\lambda-\mu, \a_i^\vee)}=1$. By $\hchi_\lambda(x)=\hchi_\mu(x)$ with $x=\binom{K_i; (\mu, \a_i^\vee)}{n}$, we have $\binom{(\lambda-\mu, \a_i^\vee)}{n}_{q_i}=0$ for all $n \in \BN$. Hence $(\lambda-\mu, \a^\vee_i)=0$ for all $1 \leq i \leq r$, due to Lemma \ref{lem: simple spectrum}. Therefore, $\lambda -\mu=0$.

\noindent
(b) Assume we have  $\sum_{i=1}^k a_i \hchi_i =0$ with all $a_i \neq 0$ and $k>0$. We have 
\[
  0=\sum_{i=1}^ka_i \hat{\chi}_i(y_1y)-\hat{\chi}_1(y_1)\sum_{i=1}^k a_i \hat{\chi}_i(y)
  =\sum_{i=2}^ka_i(\hat{\chi}_i(y_1)-\hat{\chi}_1(y_1))\hat{\chi}_i(y)
\]
for all $y,y_1\in \dU^0_q$. Repeating this process, we get $a_k \prod_{i=1}^{k-1}(\hat{\chi}_k(y_i)-\hat{\chi}_i(y_i))=0$ for all $y_1, \dots, y_{k-1}\in \dU^0_q$. Let $I$ be the annihilator of $a_k$ in $R$, then $\prod_{i=1}^{k-1}(\hat{\chi}_k(y_i)-\hat{\chi}_i(y_i)) \in I$ for all $y_1, \dots, y_{k-1}\in R$.

Since $a_k \neq 0$, there is a maximal ideal $\m$ of $R$ containing $I$. Let $\BF=R/\m$. Consider the induced character $\uchi_i\colon  \dU^0_\BF \rightarrow \BF$ of $\hat{\chi}_i$. Since the characters $\{\uchi_i\}_{i=1}^{k-1}$ are pairwise distinct by part (a), it follows that there are $\bar{y}_1, \dots, \bar{y}_{k-1} \in \dU^0_\BF$ such that $\prod_{i=1}^{k-1}(\uchi_k(\bar{y}_i)-\uchi_i(\bar{y}_i))\neq 0$. Thus, there are $y_1, \dots, y_{k-1}\in \dU^0_q$ such that $\prod_{i=1}^{k-1}(\hat{\chi}_k(y_i)-\hat{\chi}_i(y_i))\not \in \m$, hence $\prod_{i=1}^{k-1}(\hat{\chi}_k(y_i)-\hat{\chi}_i(y_i)) \not \in I$, a contradiction.
\end{proof}

We conclude this subsection with the following key observation:

\begin{Thm}\label{one argument nondegeneracy}
(a) The pairing $(\ ,\ )'$ of~\eqref{pairing of DCK and Lus halves 1R} has the zero kernel in the second argument.

\noindent
(b) The pairing $(\ ,\ )'$ of~\eqref{pairing of DCK and Lus halves 2R} has the zero kernel in the first argument.
\end{Thm}

\begin{proof}
(a) By Lemma \ref{lem:twisted-triangular-PBW} and Corollary \ref{cor: pairing of PBW}, we see that the pairing 
$(\ ,\ )'\colon \dU^<_q \x U^{ev\,>}_q \rightarrow R$ is non-degenerate in each argument. Both algebras 
$\dU^<_q$ and $U^{ev\,>}_q$ are $Q$-graded via~\eqref{eq: Q-grading for twisted one}, with each graded 
component being a free $R$-module of finite rank, and the pairing $(\ ,\ )'$ is of degree zero~\eqref{eq:degree-zero property}. So it is enough to prove the restriction $(\ , \ )': \dU^0_q \x U_q^{ev 0} \rightarrow R$ is non-degenerated in the second argument.

Since $\{ K^\nu\}_{\nu \in 2P} \in U_q^{ev 0}$ are group-like, we obtain characters $\chi^+_\nu: \dU^0_q \rightarrow R$ for $\nu \in 2P$ defined via $\chi^+_\nu(x)=(x, K^\nu)'$. It is easy to see that $\chi^+_\nu=\hchi_{-\gamma^{-1}(\nu)}$, here $\gamma^{-1}(\nu)\in Q$ since $\nu \in 2P$. The non-degeneracy of the second argument in the pairing $(\ ,\ )': \dU^0_q \x U_q^{ev 0} \rightarrow 0$  is equivalent to the statement that if $\sum_\nu a_\nu \chi^+_\nu=0$ for finitely many $a_\nu \neq 0 \in R$ then $a_\nu =0$ for all $\nu$. But the latter statement follows by Lemma \ref{lem: independent of characters}.b).

\noindent 
(b) The proof is identical.
\end{proof}

\subsection{Twisted invariant pairing}\label{ssec twisted invariant pairing}
\

Consider a function $G$ on $(\h^*)^4=\h^*\times \h^*\times \h^*\times \h^*$ defined via:
\begin{multline*}
  G(\lambda_1,\lambda_2,\mu_1,\mu_2):=\\
  \frac{(\lambda_1,\kappa(\mu_1)+\gamma(\mu_2))+(\lambda_2,\kappa(\mu_2)+\gamma(\mu_1))
        +(\gamma(\mu_1-\mu_2),\gamma(\mu_1-\mu_2))-(\lambda_1,\lambda_2)}{2} + (2\rho,\mu_2) \,,
\end{multline*}
with the endomorphisms $\kappa,\gamma\in \End(\h^*)$ defined in~(\ref{kappa and gamma}).


Evoking the Hopf pairing $(\ ,\ )'$ of~(\ref{pairing of DJ halves}), the triangular 
decomposition~\eqref{eq:new-triangular-DJ} of $\bU(\g,P/2)$, and the $Q$-grading~\eqref{eq: Q-grading for twisted one} 
of $\bU^>_\Phi$ and $\bU^<_\Phi$ (cf.\ paragraph preceding Lemma~\ref{lem: triangular-U(g,P/2)}), 
we define a $\BQ(v^{1/2})$-bilinear pairing:
\begin{equation}\label{twisted pairing on whole DJ}
  \langle \ ,\ \rangle'\colon \bU(\g,P/2)\times \bU(\g,P/2)\longrightarrow \BQ(v^{1/\sN})
\end{equation}
via 
\begin{equation}\label{twisted pairing construction}
  \langle y_1K^{\lambda_1} x_1, y_2 K^{\lambda_2}x_2 \rangle' := 
  (y_1,x_2)' \cdot (y_2,x_1)' \cdot v^{G(\lambda_1,\lambda_2,\mu_1,\mu_2)} \,,
\end{equation}
for all 
  $x_1\in \bU^{>}_{\Phi, \mu_1}, x_2\in \bU^{>}_{\Phi,\mu_2}, y_1\in \bU^{<}_{\Phi,-\nu_1}, y_2\in \bU^{<}_{\Phi,-\nu_2}$
with $\mu_1,\mu_2,\nu_1,\nu_2\in Q_+$, $\lambda_1,\lambda_2\in P/2$.

\begin{Rem}
To simplify the exponent of $v$ in~(\ref{twisted pairing construction}),
we can rewrite the pairing~(\ref{twisted pairing on whole DJ}) as:
\begin{equation}\label{twisted pairing rewritten}
\begin{split}
  & \Big\langle (y_1K^{\kappa(\nu_1)})K^{\lambda_1} (x_1K^{\gamma(\mu_1)}), 
          (y_2K^{\kappa(\nu_2)}) K^{\lambda_2} (x_2 K^{\gamma(\mu_2)}) \Big\rangle' = \\
  & \qquad \qquad \qquad \qquad \qquad \qquad \qquad \qquad \qquad 
  (y_1,x_2)'\cdot (y_2,x_1)'\cdot v^{-\frac{(\lambda_1,\lambda_2)}{2}+(2\rho,\nu_1)}
\end{split}
\end{equation}
for all
  $x_1\in \bU^{>}_{\Phi, \mu_1}, x_2\in \bU^{>}_{\Phi, \mu_2}, 
   y_1\in \bU^{<}_{\Phi, -\nu_1}, y_2\in \bU^{<}_{\Phi, -\nu_2}$
with $\mu_1,\mu_2,\nu_1,\nu_2\in Q^+$, $\lambda_1,\lambda_2\in P/2$. 
Here, we use $\gamma+\kappa=2\Id_{\h^*}$, so that $\gamma\kappa=\kappa\gamma$ 
and $(\gamma(\mu),\gamma(\mu'))=(\kappa(\mu),\kappa(\mu'))$ due to~\eqref{adjoint}. 
Formula~\eqref{twisted pairing rewritten} also clarifies our first condition 
in the choice of $\sN$ in the beginning of Section~\ref{ssec twisted Hopf pairing}.
\end{Rem}

We note that: 
\begin{equation}\label{eq:orthogonal}
  \langle \bU^{<}_{\Phi,-\nu_1} \bU^{0}_\Phi \bU^{>}_{\Phi,\mu_1},
          \bU^{<}_{\Phi,-\nu_2} \bU^{0}_\Phi \bU^{>}_{\Phi, \mu_2} \rangle' = 0 
  \quad \mathrm{unless} \quad \nu_1=\mu_2 \,,\, \nu_2=\mu_1 \,.
\end{equation}
The following result is completely analogous to~\cite[Proposition 6.20]{j}:

\begin{Prop}\label{twisted pairing DJ}
The above pairing $\langle \ ,\ \rangle'$ of~\eqref{twisted pairing on whole DJ} on $\bU(\g,P/2)$ satisfies:
\begin{equation}\label{ad-invariance def 1}
  \langle \ad'(x)y,z \rangle' = \langle y,\ad'(S'(x))z \rangle' \qquad \mathrm{for\ any} \quad x,y,z\in \bU(\g,P/2) \,.
\end{equation}
\end{Prop}

\begin{proof}
Follows by applying a line-to-line reasoning as in~\cite[\S6.14--6.20]{j}. 
\end{proof}

Let us now investigate the behavior of the pairing~\eqref{twisted pairing on whole DJ} with respect to 
$\dU_\CA(\g)$ and $U^{ev}_\CA(\g)$.

\begin{Prop}\label{Prop:pairing_whole}
The restriction of~(\ref{twisted pairing on whole DJ}) gives rise to: 
\begin{equation}\label{pairing of DCK and Lus whole}
  \langle \ ,\ \rangle'\colon U^{ev}_{\CA}(\g) \times \dU_{\CA}(\g) \longrightarrow \CA[v^{\pm {1/\sN}}] \,.
\end{equation}
Evoking the adjoint action $\ad'\colon \dU_{\CA}(\g)\curvearrowright U^{ev}_{\CA}(\g)$ of 
Proposition~\ref{action of Lus on DCK}, the pairing~(\ref{pairing of DCK and Lus whole}) satisfies: 
\begin{equation}\label{invariance DCK+Lus}
  \langle \ad'(x)y, z \rangle' = \langle y, \ad'(S'(x))z \rangle' \qquad
  \mathrm{for\ any} \quad x,z\in \dU_{\CA}(\g) \,,\, y\in U^{ev}_{\CA}(\g) \,.
\end{equation}
\end{Prop}

\begin{proof}
The restriction of the pairing~(\ref{twisted pairing on whole DJ}) to $U^{ev}_{\CA}(\g) \times \dU_{\CA}(\g)$ 
satisfies the condition~(\ref{invariance DCK+Lus}) due to~\eqref{ad-invariance def 1}, and takes values in $\CA[v^{\pm {1/\sN}}]$ 
due to~\eqref{twisted pairing rewritten} and Lemma~\ref{Lem:pairing_quantum_Borel}. 
\end{proof}

Evoking the pairings $(\ ,\ )'$ of Lemma~\ref{Lem:pairing_quantum_Borel} 
(uniquely determined by~(\ref{pairing condition 1},~\ref{explicit formulas 1}) 
and~(\ref{pairing condition 1},~\ref{explicit formulas 2})), the triangular decompositions of
$\dU_{\CA}(\g)$ and $U^{ev}_{\CA}(\g)$ of Lemma~\ref{lem:twisted-triangular-PBW}, 
the $Q$-grading~\eqref{eq: Q-grading for twisted one} of both algebras, Lemma~\ref{Q-to-2P}, and formula~(\ref{twisted pairing rewritten}), we note that 
the pairing~(\ref{pairing of DCK and Lus whole}) is given by:
\begin{equation}\label{pairing of DCK and Lus rewritten}
\begin{split}
  & \left\langle (yK^{\kappa(\nu)}) K^{\lambda} (x K^{\gamma(\mu)}), 
               (\dy K^{\kappa(\dnu)}) K^{\dlambda} \prod_{i=1}^{r} \binom{K_i;0}{s_i} (\dx K^{\gamma(\dmu)}) \right\rangle' \, = \\
  & \qquad \qquad \qquad \qquad \qquad \qquad 
    (\dy,x)'\cdot (y,\dx)'\cdot v^{-\frac{(\lambda,\dlambda)}{2}+(2\rho,\nu)} \cdot
    \prod_{i=1}^{r} \binom{-\frac{(\alpha^\vee_i,\lambda)}{2}}{s_i}_{v_i}
\end{split}
\end{equation}
for any 
  $\dx\in \dU^{>}_{\CA,\dmu}, x\in U^{ev\,>}_{\CA,\mu}, 
   \dy\in \dU^{<}_{\CA,-\dnu}, y\in U^{ev\,<}_{\CA,-\nu}$, $\dmu,\mu,\dnu,\nu\in Q_+$, 
$\dlambda, \lambda\in 2P$, and $s_i\geq 0$. In particular, we note that $-(\alpha^\vee_i,\lambda)/2\in \BZ$, and so 
$\binom{-\frac{(\alpha^\vee_i,\lambda)}{2}}{s_i}_{v_i}$ is well-defined.

Given $q\in R$ as in the paragraph preceding~(\ref{sigma homom}), suppose $q$ has an $\sN$-th root in $R$, and fix such a root $q^{1/\sN}$.  
Lift $\sigma\colon \CA\to R$ of~(\ref{sigma homom}) to $\CA[v^{1/\sN}]\rightarrow R$ via $v^{1/\sN}\mapsto q^{1/\sN}$. We thus obtain:

\begin{Prop}\label{Prop:pairing_whole_R}
Define an $R$-bilinear pairing
\begin{equation}\label{pairing of DCK and Lus whole R}
  \langle \ ,\ \rangle'\colon U^{ev}_{q}(\g) \times \dU_{q}(\g) \longrightarrow R
\end{equation}
via~(\ref{pairing of DCK and Lus rewritten}), where $v^{1/\sN}=q^{1/\sN}$ and the pairings $(\dy,x)',(y,\dx)'$ refer 
to~(\ref{pairing of DCK and Lus halves 1R},~\ref{pairing of DCK and Lus halves 2R}). Then:

\medskip
\noindent
(a) The pairing~(\ref{pairing of DCK and Lus whole R}) is $\ad'$-invariant:
\begin{equation*}
  \langle \ad'(x)y, z \rangle' = \langle y, \ad'(S'(x))z \rangle' 
  \qquad \mathrm{for\ any} \quad x,z\in \dU_q(\g) \,,\, y\in U^{ev}_q(\g) \,.
\end{equation*}

\noindent
(b) The pairing~(\ref{pairing of DCK and Lus whole R}) has the zero kernel in the first argument.
\end{Prop}

\begin{proof}
Part (a) is obvious. Let us prove part (b). Let $x\in U_q^{ev}(\g)$ then, by using the PBW-bases of Lemma \ref{lem:twisted-triangular-PBW},  
\[x=\sum_{\vec{k}, \vec{r} \in \BZ^N_{\geq 0}} (\tF^{\cev{k}} K^{-\kappa(\weight(\tF^{\cev{k}}))})a_{\vec{k}, \vec{r}}(\tE^{\cev{r}} K^{\gamma(\weight(\tE^{\cev{r}}))}),\]
for finitely many nonzero $a_{\vec{k}, \vec{r}}\in U_q^{ev0}$. Suppose $\<x, u\>'=0$ for all $u \in \dU_q(\g)$. Apply this to  $u=(\tF^{(\cev{r})} K^{-\kappa(\weight(\tF^{(\cev{r})}))})u_0(\tE^{(\cev{k})}K^{\gamma(\weight(\tE^{(\cev{k})}))})$ for any $u_0\in \dU^0_q$ and using the computations in Corollary \ref{cor: pairing of PBW}, we obtain $\<a_{\vec{k}, \vec{r}}, u_0\>'=0$ for all $u_0\in \dU^0_q$. Note that the map $\dU^0_q \rightarrow R$ defined by $u_0 \mapsto \< K^{2\lambda}, u_0\>'$ is equal to $\hchi_{-\lambda}$ for all $\lambda \in P$. Therefore, by Lemma \ref{lem: independent of characters}, $a_{\vec{k}, \vec{r}}=0$, hence $x=0$.
\end{proof}

\section{Twisted quantum Frobenius homomorphism}\label{Quantum Frobenius}



We impose restrictions on $\ell$ as in Section \ref{ssec:qFr} below. Let $\CA'$ be the quotient of $\CA$ by the ideal generated by $\ell$-cyclotomic polynomial $f_\ell\in\CA$. 
In this section, we assume that the given ring homomorphism $\sigma\colon \CA\rightarrow R$ factors through a ring homomorphism 
$\CA'\rightarrow R$. 
To define the coproduct and the braiding for certain 
Hopf algebras below, we shall further assume that there is $\e^{1/\sN}\in R$, where $\sN$ satisfies 
$\tfrac{1}{2}(P/2,P/2)\subset \frac{1}{\sN}\BZ$, 
and we fix such $\sN,\e^{1/\sN}$. We set $\e_i=\e^{\sd_i}$, where $\sd_i=(\a_i,\a_i)/2 \in \{1,2,3\}$. Define a positive integer  $\ell_i$ by 
\begin{equation}\label{eq:li_equation}
  \ell_i:=\ell/\operatorname{GCD}(2d_i,\ell).
\end{equation}
More generally, for any $\alpha\in \Delta_+$, let 
\begin{equation}\label{eq: l_a} 
  \ell_\alpha:=\ell/\operatorname{GCD}((\alpha,\alpha),\ell).
\end{equation}
In this section, we recall the quantum Frobenius homomorphism of \cite{l-book} with small modifications.

Henceforth, we shall often use the following result:

\begin{Lem}\label{lem:aux-at-roots}
(a) For any $a\in \BZ$ and $b\in \BN$, we have 
$\ds \binom{a}{b}_{\e_i} =\, \binom{a_0}{b_0}_{\e_i} \cdot\, \binom{a_1}{b_1}$, where $a=\ell_i a_1+a_0$ and $b=\ell_i b_1+b_0$ 
with $a_1\in\BZ$, $b_1\in \BN$, and $0\leq a_0,b_0\leq \ell_i-1$.

\noindent
(b) The $v=\epsilon$ specialization of the $v$-multibinomial coefficient 
$\displaystyle \frac{(n\ell_i)_{v_i}!}{((\ell_i)_{v_i}!)^n}\in \BZ[v,v^{-1}]$ equals $n!$.

\noindent
(c) $(-1)^{\ell_i}\e_i^{\ell_i(\ell_i+1)}=-1$.
\end{Lem}
\begin{proof}
(a) Follows from~\cite[Lemma 34.1.2(c)]{l-book} upon using $\ds \bmat{a\\ b}_{\e_i}=\, \e_i^{b(a-b)}\binom{a}{b}_{\e_i}$ and 
$\e_i^{2\ell_i}=1$.

(b) Follows by induction on $n$ by applying part (a) to $a=n\ell_i,b=\ell_i$.

(c) We have $(-1)^{\ell_i}\epsilon_i^{\ell_i(\ell_i+1)}=(-1)^{\ell_i}\epsilon^{\sd_i \ell_i(\ell_i+1)}$. 
If $\ell_i$ is odd, then $(-1)^{\ell_i}=-1$, $\epsilon^{\sd_i\ell_i(\ell_i+1)}=1$. If $\ell_i$ is even, then 
we note that $\epsilon^{\sd_i \ell_i}=-1$ and also $(-1)^{\ell_i}=1$. The result follows. 
\end{proof}


\subsection{The domain of the quantum Frobenius homomorphism} 
\

Let us recall the Lusztig form $\dU_\e(\g)=\dU_\CA(\g)\otimes_{\CA} R$ from \eqref{eq: Lusztig form over R}, and let  
$\dU^>_\e$, $\dU^<_\e$ be its $R$-subalgebras generated by $\{\tE^{(n)}_i\}^{n\geq 1}_{1\leq i \leq r}$, 
$\{\tF^{(n)}_i\}^{n\geq 1}_{1\leq i \leq r}$. 
We recall the $q$-Serre relations, cf.~\eqref{eq:gen-rel-twistedDJ-2}: 
\begin{align*}
  & \sum_{m=0}^{1-a_{ij}}(-1)^m \e_i^{ma_{ij}(\epsilon_{ij}-1)-m(m-1)}\tE_i^{(1-a_{ij}-m)}\tE_j \tE_i^{(m)}=0 \qquad  
    (i \neq j) \,, \\
  & \sum_{m=0}^{1-a_{ij}}(-1)^m \e_i^{m a_{ij}(\e_{ij}-1)-m(m-1)} \tF_i^{(1-a_{ij}-m)} \tF_j \tF_i^{(m)}=0 \qquad 
    (i \neq j) \,.
\end{align*}
The generators $\tE_i^{(n)}$ and $\tF_j^{(n)}$ also satisfy the following relations, cf.~\eqref{eq:ef-twisted-swap-in-v}: 
\begin{equation}\label{eq:ef-twisted-swap} 
\begin{split}
  \tE_i^{(p)}\tF_j^{(s)}&=\epsilon^{ps(\a_i, \kappa(\a_j))} \tF_j^{(s)}\tE_i^{(p)} \qquad \mathrm{for}\ i\ne j \,, \\
  \tE_i^{(p)}\tF_i^{(s)}&=\sum_{c=0}^{\min(p,s)} \epsilon_i^{2ps-c^2} \tF_i^{(s-c)}\binom{K_i; 2c-p-s}{c} \tE_i^{(p-c)} \,. 
\end{split}
\end{equation}
   
\subsubsection{The idempotented Lusztig form $\hU_\e(\g, X)$}\label{sssec:idempotented-Lus}
\
%

For any lattice $X$ with $Q\subseteq X\subseteq P$, let $\hU_\e(\g, X)$ be the \textbf{idempotented Lusztig form} 
defined similarly to \cite[Chapter~$23$]{l-book} with generators 
\[
  \Big\{ \tE_i^{(n)}1_\lambda, \tF_i^{(n)}1_\lambda \,\Big|\, 1\leq i\leq r, n\geq 0, \lambda \in X \Big\} \,.
\]

We record the topological coproduct in $\hU_\e(\g, X)$: 
\begin{equation}\label{EF coproduct}
\begin{split}
  & \Delta(\tE^{(r)}_i 1_\lambda)=
    \sum_{c=0}^r \prod_{\lambda'+\lambda''=\lambda}  \e^{-(r-c)(\zeta^>_i, \lambda'')} 
    \tE_i^{(r-c)}1_{\lambda'}\otimes \tE_i^{(c)}1_{\lambda''} \,,\\
  & \Delta(\tF_i^{(r)}1_\lambda)=
    \sum_{c=0}^r \prod_{\lambda'+\lambda''=\lambda} \e_i^{2c(r-c)}\e^{c(\zeta^<_i, \lambda'')} 
    \tF_i^{(c)}1_{\lambda'}\otimes \tF_i^{(r-c)}1_{\lambda''} \,,
\end{split}
\end{equation}
cf.~\eqref{eq:twisted Hopf on divided powers}, as well as some relations in $\hU_\e(\g, X)$:
\begin{equation}\label{EF relations}
\begin{split}
  & \tE_i^{(p)}1_\lambda \tF^{(s)}_j = 
    \e^{ps(\a_i,\kappa(\a_j))}\tF_j^{(s)}1_{\lambda+s\a_j+p\a_i}\tE_i^{(p)} \qquad \mathrm{for} \  \ i \neq j \,, \\
  & \tE_i^{(p)}1_\lambda \tF_i^{(s)} = 
    \sum_{c\geq0} \e_i^{2ps-c^2} \binom{(\lambda, \a_i^\vee)+s+p}{c}_{\e_i}\tF_i^{(s-c)}1_{\lambda+(p+s-c)\a_i}\tE_i^{(p-c)} \,,
\end{split}
\end{equation}
cf.~\eqref{eq:ef-twisted-swap}, with $\lambda, \lambda', \lambda''\in X$ in the formulas above. We note that in \eqref{EF coproduct}, 
the element $\e^{1/\sN}$ is used in the definition of non-integer powers of $\e$, e.g.\  
$\e^{c(\zeta^<_i, \lambda'')}:=(\e^{1/\sN})^{\sN c(\zeta^<_i, \lambda'')}$.


\subsection{The codomain of the quantum Frobenius homomorphism}\label{target of Fr} 

\subsubsection{Root data}\label{SSS:root_data}
Let us consider the following data:
\begin{itemize}

\item The lattices $P^*=\bigoplus_{i=1}^r \BZ \w^*_i$ and $ Q^*=\bigoplus_{i=1}^r \BZ \a^*_i$ with 
$\w^*_i:=\ell_i \w_i$ and $\a^*_i:=\ell_i \a_i$. We also set $\w^{*\vee}_i:=\w^\vee_i/\ell_i$ and $\a^{*\vee}_i:=\a^\vee_i/\ell_i$.

\item The new Cartan matrix with $(i,j)$-entry 
\begin{equation}\label{New Cartan entry} 
  a_{ij}^*=2(\a^*_i,\a^*_j)/(\a^*_i, \a^*_i)=2\ell_j(\a_i, \a_j)/\ell_i(\a_i, \a_i) \,.
\end{equation}

\item The bilinear form on $P^*$ induced from the bilinear form on $P$ via the inclusion $P^*\subset P$.

\end{itemize}
The fact that $(a^*_{ij})_{i,j=1}^{r}$ is a Cartan matrix of a semisimple Lie algebra follows from~\cite[\S 2.2.4]{l-book};  
note that our $\ell$ is $\ell'$ in \cite[Chapter 35]{l-book}).
In the case when $\ell$ is divisible by $2\sd_i$ for all $i$, we have $a_{ij}^*=a_{ji}$ for all $i,j$, so that $(a^*_{ij})$ is 
the Cartan matrix of the Langlands dual $\g^\vee$ of $\g$.

\begin{Rem}\label{rem:lattice-star}
We note that $\kappa(Q^*)=\gamma(Q^*)=2P^*$ similarly to Lemma \ref{Q-to-2P}.
\end{Rem}


\subsubsection{The $\BQ(v^{1/2})$-Hopf algebra $\bU^*(\g, P^*/2)$}
\

All constructions in Sections \ref{sec Setting}--\ref{S_inv_pair} can be carried out with the above datum. Let 
$\sd^*_i=\frac{(\a^*_i,\a^*_i)}{2}=\sd_i \ell_i^2$, $v^*_i=v^{\sd^*_i}$, and respectively $\e_i^*=\e^{\sd^*_i}\in \{\pm 1\}$. Let 
$\g^d$ be the semisimple Lie algebra with the Cartan matrix $(a^*_{ij})$, the weight lattice $P^*$ and the root lattice $Q^*$. 
The graph $\text{Dyn}(\g^d)$, see definition in Section \ref{ssec Sevostyanov phi}, is the same as $\text{Dyn}(\g)$. Let us 
consider the associated matrix $(\epsilon_{ij})$ as in \eqref{epsilon matrix} and then the following twist: 
\[
  F^*:= 
  \prod_{\lambda,\mu \in P^*}v^{\sum_{i,j} \phi^*_{ij}(\w^{*\vee}_i, \lambda)(\w^{*\vee}_j, \mu)}1_\lambda \otimes 1_\mu = 
  \prod_{\lambda, \mu \in P^*} v^{\sum_{i,j} \phi_{ij}(\w^\vee_i, \lambda)(\w_j^\vee, \mu)}1_\lambda \otimes 1_\mu \,,
\]
where $\phi^*_{ij}=\epsilon_{ij}(\a^*_i, \a^*_j)/2=\ell_i\ell_j\phi_{ij}$, cf.~\eqref{formal Cartan twist} and~\eqref{eq:condition on Phi}.

With the above data, we define the $\BQ(v^{1/2})$-Hopf algebra $\bU^*(\g, P^*/2)$ generated by the set 
$\{\he_i, \hf_i, K^\mu\}_{1\leq i \leq r}^{\mu \in P^*/2}$ as in Section \ref{ssec: quantum DJ}. We also have the Lusztig's 
braid group action on $\bU^*(\g, P^*/2)$ defined by:
\begin{equation}\label{eq: *braid action}
\begin{split}
  & T^*_i(K^\mu)=K^{s^*_i\mu} \,, \qquad T^*_i(\he_i)=-\hf_iK^{\a^*_i} \,, \qquad T^*_i(\hf_i)=-K^{-\a^*_i}\he_i \,,\\
  & T^*_i(\he_j)=\sum_{k=0}^{-a^*_{ij}}(-1)^k\frac{(v^*_i)^{-k}}{[-a^*_{ij}-k]_{v^*_i}![k]_{v^*_i}!}\he_i^{-a^*_{ij}-k}\he_j\he_i^k \qquad (i \neq j) \,,\\
  & T^*_i(\hf_j)=\sum_{k=0}^{-a^*_{ij}}(-1)^k\frac{(v^*_i)^k}{[-a^*_{ij}-k]_{v^*_i}![k]_{v^*_i}!}\hf_i^k\hf_j\hf_i^{-a^*_{ij}-k} \qquad (i \neq j) \,,
\end{split}
\end{equation}
in which $s^*_i=s_{\a^*_i}$. The Weyl group of $\g^d$ is the same as the Weyl group of $\g$ via identifying $s^*_i$ with $s_i$, 
so we also denote the Weyl group of $\g^d$ by $W$. Fix the same reduced decomposition of the longest element $w_0$ in $W$ as in 
Section \ref{ssec: quantum DJ}. Then $\b^*_k=s^*_{i_1}\dots s^*_{i_{k-1}}\a^*_{i_k}=\ell_{i_k}\b_k\; (1\leq k\leq N)$ provides 
a labeling of all positive roots $\Delta^d_+$ of $\g^d$. We then define root vectors $\{\he_{\b^*_k}, \hf_{\b^*_k}\}_{1\leq k\leq N}$ 
in a standard way via:
\begin{equation*}
  \he_{\b^*_k}=T^*_{i_1}\dots T^*_{i_{k-1}}\he_{i_k} \,, \qquad \hf_{\b^*_k}=T^*_{i_1}\dots T^*_{i_{k-1}}\hf_{i_k} \,.
\end{equation*}
Following (\ref{tilda elements},~\ref{Sev twist}), for $1\leq i\leq r$, let 
\begin{equation*}
  \nu^{*>}_i:=-\a^*_i+\sum_{1\leq j\leq r} \phi^*_{ij}\w^{*\vee}_j=\ell_i \nu^>_i \,, \qquad 
  \nu^{*<}_i:=\sum_{1\leq j\leq r} \phi^*_{ij}\w^{*\vee}_j=\ell_i \nu^<_i \,,
\end{equation*}
and set 
\begin{equation*}
  \te_i:=\he_i K^{\nu^{*>}_i} \,, \qquad \tf_i:=K^{-\nu^{*<}_i} \hf_i \,.
\end{equation*}
Then the algebra $\bU^*(\g,P^*/2)$ is also generated over $\BQ(v^{1/2})$ by $\{\te_i, \tf_i, K^\mu\}_{1\leq i \leq r}^{\mu\in P^*/2}$ 
subject to the following relations:
\begin{equation}\label{eq: relations for *-version}
\begin{split}
  & K^{\mu}K^{\mu'}=K^{\mu+\mu'} \,, \qquad K^0=1 \,, \\
  & K^{\mu}\te_i K^{-\mu}=v^{(\a^*_i, \mu)}\te_i \,, \qquad K^\mu\tf_i K^{-\mu}=v^{-(\a^*_i, \mu)}\tf_i \,, \\
  & \te_i\tf_j=v^{(\a^*_i, -\zeta^{*<}_j)}\tf_j\te_i \;\;\;(i \neq j) \,, \qquad 
    \te_i\tf_i-(v^*_i)^2\tf_i\te_i = v^*_i \frac{1-(K^*_i)^{-2}}{1-(v^*_i)^{-2}} \,, \\
  & \sum_{m=0}^{1-a^*_{ij}}(-1)^m(v^*_i)^{ma^*_{ij}(\epsilon_{ij}-1)-m(m-1)}\te_i^{(1-a^*_{ij}-m)}\te_j\te_i^{(m)}=0 \qquad (i\neq j) \,, \\
  & \sum_{m=0}^{1-a^*_{ij}}(-1)^m(v^*_i)^{ma^*_{ij}(\epsilon_{ij}-1)-m(m-1)}\tf_i^{(1-a^*_{ij}-m)}\tf_j\tf_i^{(m)}=0\qquad (i\neq j) \,,
\end{split}
\end{equation}
where $K^*_i=K^{\a^*_i}$, $\te_i^{(m)}=\te^m/(m)_{v^*_i}$, $\tf_i^{(m)}=\tf_i^m/(m)_{v^*_i}$, and we set 
$\zeta^{*<}_j=\ell_j \zmu_j$, $\zeta^{*>}_j=\ell_j \zlambda_j$. Following the twisted construction in Section \ref{sec twisted coproduct}, 
we also have a twisted Hopf algebra structure on $\bU^*(\g, P^*/2)$ as follows: 
\begin{equation*}
\begin{split}
  & \Delta'(K^\mu)=K^\mu\otimes K^\mu \,, \qquad \Delta'(\te_i)=1\otimes \te_i +\te_i \otimes K^{-\zeta^{*>}_i} \,, \qquad 
    \Delta'(\tf_i)=1\otimes \tf_i+\tf_i\otimes K^{\zeta^{*<}_i} \,, \\
  & S'(K^\mu)=K^{-\mu} \,, \qquad S'(\te_i)=-\te_iK^{\zeta^{*>}_i} \,, \qquad S'(\tf_i)=-\tf_i K^{-\zeta^{*<}_i} \,.
\end{split}
\end{equation*}
The algebra $\bU^*(\g, P^*/2)$ is $Q^*$-graded via: 
\begin{equation}\label{eq: Q*-graded}
  \deg(\te_i)=\a^*_i \,, \qquad \deg(\tf_i)=-\a^*_i \,, \qquad \deg(K^\mu)=0 \,.
\end{equation}

\begin{Rem}\label{rem: *Lusztig form}
Similarly to Section \ref{sec Setting}, we define the (untwisted) Lusztig form $\dmU^*_\CA(\g)$ as the $\CA$-subalgebra of 
$\bU^*(\g, P^*/2)$ generated by elements $\{\he_i^{[n]}, \hf_i^{[n]}, K^\mu\}_{1\leq i \leq r, n \geq 1}^{\mu \in Q^*}$, in which
$\he_i^{[n]}:=\he^n_i/[n]_{v^*_i}!$, $\hf_i^{[n]}:=\hf_i^n/[n]_{v^*_i}!$. Then we also consider the specialization 
$\dmU^*_\e(\g):=\dmU^*_\CA(\g) \otimes_\CA R$.
\end{Rem}


\subsubsection{The $\CA$-Hopf algebra $\dU^*_\CA(\g)$ and its specalization $\dU^*_\e(\g)$}
\

We define the (twisted) Lusztig form $\dU^*_\CA(\g)$ as the $\CA$-subalgebra of $\bU^*(\g, P^*/2)$ generated by the elements 
$\{\te_i^{(n)}, \tf_i^{(n)}, K^\mu\}_{1\leq i \leq r}^{\mu \in 2P^*}$ as in Section \ref{ssec: new Lusztig form}. It is a Hopf algebra over $\CA$. We record the 
Hopf structure of $\dU^*_\CA(\g)$:
\begin{equation}\label{eq:*twisted Hopf on dived powers}
\begin{split}
    & \Delta'(\te_i^{(s)})=\sum_{c=0}^s \te_i^{(s-c)}\otimes \te_i^{(c)}K^{-(s-c)\zeta^{*>}_i} \,, \quad  
      S'(\te_i^{(s)})=(-1)^s (v^*_i)^{s(s-1)} \te_i^{(s)}K^{s\zeta^{*>}_i} \,, \\
    & \Delta'(\tf_i^{(s)})=\sum_{c=0}^s (v^*_i)^{2c(s-c)}\tf_i^{(c)}\otimes \tf_i^{(s-c)}K^{c\zeta^{*<}_i}\,, \quad 
      S'(\tf_i^{(s)})=(-1)^s(v^*_i)^{s(1-s)}\tf_i^{(s)}K^{-s\zeta^{*<}_i} \,, \\
    & \Delta'(K^\mu)=K^\mu\otimes K^\mu \,, \qquad S'(K^\mu)=K^{-\mu} \,, \\
    & \varepsilon'(\te_i^{(s)})=\varepsilon'(\tf_i^{(s)})=0 \,, \qquad \varepsilon'(K^\mu)=1 \,.
\end{split}
\end{equation}
The elements $\ds \binom{K^*_i;a}{n}:=\frac{\prod_{s=1}^n (1-(K_i^*)^{-2}(v^*_i)^{2(s-a-1)})}{\prod_{s=1}^n(1-(v^*_i)^{-2s})}$ 
belong to $\dU^*_\CA(\g)$. Let $\dU^{*<}_\CA, \dU^{*>}_\CA, \dU^{*0}_\CA$ denote the $\CA$-subalgebras of $\dU^*_\CA(\g)$ generated 
by $\ds \{\te_i^{(s)}\}, \{\tf_i^{(s)}\}, \Big\{K^\mu, \binom{K^*_i;a}{n} \Big\}$, respectively.

For any $\b^*=\sum_{i=1}^{r} a_i \a^*_i\in Q^*_+$, let  
\begin{equation}\label{eq: nu and normalizr b for *}
\begin{split}
  & \nu^{>}_{\b^*}:=\sum_i a_i \nu^{*>}_i \,, \qquad \nu^{<}_{\b^*}:=\sum_i a_i \nu^{*<}_i \,,\\
  & b^{>}_{\b^*}=\sum_{i<j}a_ia_j(\nu^{*>}_i , \a^*_j) \,, \qquad b^{<}_{\b^*}=-\sum_{i<j} a_ia_j(\nu^{*<}_j, \a^*_i) \,,
\end{split}
\end{equation}
cf.~\eqref{nu for alpha}--\eqref{eq: normalizer b}. Similarly to~\eqref{eq:twisted-root-vector} for any $\b^*_k\in \Delta^d_+$, define: 
\begin{equation}\label{eq:tilde-vs-hat}
  \te_{\b^*_k}:=v^{b^{>}_{\b^*_k}}\he_{\b^*_k}K^{\nu^{>}_{\b^*_k}} \,, \quad 
  \tf_{\b^*_k}:=v^{b^{<}_{\b^*_k}}K^{-\nu^<_{\b^*_k}}\tf_{\b^*_k} \,, \quad 
  \te_{\b^*_k}^{(n)}:=\frac{\te_{\b^*_k}^{n}}{(n)_{v^*_{i_k}}!} \,, \quad 
  \tf_{\b^*_k}^{(n)}:=\frac{\tf_{\b^*_k}^n}{(n)_{v^*_{i_k}}!} \,.
\end{equation}
For any $\vec{k}\in \BZ^N_{\geq 0}$, we also define the ordered monomials:
\begin{equation*}
\begin{split}
  & \te^{(\vec{k})}:=\te_{\b^*_1}^{(k_1)}\dots \te_{\b^*_N}^{(k_N)} \,, \qquad 
    \tf^{(\vec{k})}:=\tf_{\b^*_1}^{(k_1)}\dots \tf_{\b^*_N}^{(k_N)} \,, \\
  & \te^{(\cev{k})}:=\te_{\b^*_N}^{(k_N)}\dots \te_{\b^*_1}^{(k_1)} \,, \qquad 
    \tf^{(\cev{k})}:=\tf_{\b^*_N}^{(k_N)}\dots \tf_{\b^*_1}^{(k_1)} \,.
\end{split}
\end{equation*}
We have the following results analogous to Lemma \ref{lem:twisted-triangular-PBW}:

\begin{Lem}\label{lem:*twisted-trinagular-PBW}
(a) The subalgebras $\dU^{*<}_\CA, \dU^{*>}_\CA$ are $Q^*$-graded via \eqref{eq: Q*-graded}, and each of their degree components 
is a free $\CA$-module of finite rank.

\noindent
(b) The multiplication map:
\[ 
  \mm\colon \dU^{*<}_\CA\otimes_\CA\dU^{*0}_\CA\otimes_{\CA} \dU^{*>}_\CA\longrightarrow \dU^*_\CA(\g)
\]
is an isomorphism of free $\CA$-modules.

\noindent
(c) The sets $\{\te^{(\vec{k})}\}_{\vec{k}\in \BZ^N_{\geq 0}},$ $ \{\te^{(\cev{k})}\}_{\vec{k}\in \BZ^N_{\geq 0}}$ are 
$\CA$-bases of $\dU^{*>}_\CA$.

\noindent
(d) The sets $\{\tf^{(\vec{k})}\}_{\vec{k}\in \BZ^N_{\geq 0}},$ $\{\tf^{(\cev{k})}\}_{\vec{k}\in \BZ^N_{\geq 0}}$ are 
$\CA$-bases of $\dU^{*<}_\CA$.

\noindent
(e) The subalgebra $\dU^{*0}_\CA$ has the following $\CA$-basis:
\begin{equation}\label{eq: *basis-Cartan-twisted-Lus}
  \left\{ K^{2\varsigma^*_j}\cdot \prod_{i=1}^r \left((K^*_i)^{2\big\lfloor \tfrac{t_i}{2} \big\rfloor}\binom{K^*_i;0}{t_i}\right) 
  \,\Big|\, \;t_i\geq 0, 1\leq j\leq k \right\} \,,
\end{equation}
where $\varsigma^*_1, \dots, \varsigma^*_k \in P^*$ is a set of representatives of the left cosets $P^*/Q^*$.
\end{Lem}

Evoking the algebra homomorphism $\sigma\colon \CA \rightarrow R$ of~\eqref{sigma homom}, we define:
\[ 
  \dU^*_\e(\g):= \dU^*_\CA(\g)\otimes_{\CA} R \,.
\]
Let $\ds \te_i^{(n)},\tf_i^{(n)}, K^\mu, \binom{K^*_i;a}{n}$ denote the corresponding elements in $\dU^*_\e(\g)$, and let 
$\dU^{*>}_\e,$ $\dU^{*<}_\e,$ $\dU^{*0}_\e$ be the corresponding $R$-subalgebras of $\dU^*_\e(\g)$.

We also note the following equalities in $R$:   
\[
  \e^{(\a^*_i, \kappa(\a^*_j))} = 1 = (\e^*_i)^{ma^*_{ij}(\e_{ij}-1)-m(m-1)} \,.
\]
Therefore, in $\dU^*_\e(\g)$, the last two relations of \eqref{eq: relations for *-version} simplify as follows: 
\begin{equation}\label{eq:Serre-star}
\begin{split}
  & \sum_{m=0}^{1-a^*_{ij}} (-1)^m \te_i^{(1-a^*_{ij}-m)} \te_j \te_i^{(m)}=0 \qquad \mathrm{for} \quad i\neq j \,,\\
  & \sum_{m=0}^{1-a^*_{ij}}(-1)^m \tf_i^{(1-a^*_{ij}-m)}\tf_j \tf^{(m)}_i=0 \qquad \mathrm{for} \quad i \neq j \,.
\end{split}
\end{equation}
Moreover, the generators $\te_i^{(n)}$ and $\tf_j^{(n)}$ also satisfy the following relations, cf.~\eqref{eq:ef-twisted-swap}: 
\begin{equation} \label{eq: ef*-relation}
\begin{split}
  \te_i^{(p)}\tf_j^{(s)} &= \epsilon^{ps(\a^*_i, \kappa(\a^*_j))} \tf_j^{(s)}\te_i^{(p)} = \tf_j^{(s)}\te_i^{(p)} 
    \qquad \mathrm{for}\ i\ne j \,, \\
  \te_i^{(p)}\tf_i^{(s)}&=\sum_{c=0}^{\min(p,s)} (\epsilon_i^*)^{2ps-c^2} \tf_i^{(s-c)}\binom{K^*_i; 2c-p-s}{c} \te_i^{(p-c)} \,.
\end{split}
\end{equation}


\subsubsection{The idempotented Lusztig form $\hU^*_\e(\g, X^*)$}\

For any lattice $X^*$ with $Q^* \subseteq X^* \subseteq P^*$, we also form the \textbf{idempotented Lusztig form} 
$\hU^*_\e(\g, X^*)$ with the generators
\[ 
  \Big\{ \te_i^{(n)}1_\lambda, \tf_i^{(n)}1_\lambda \,\Big|\, 1\leq i\leq r, n \geq 0, \lambda \in X^* \Big\} \,.
\]
We record the topological coproduct in $\hU^*_\e(\g, X^*)$, cf.~(\ref{EF coproduct},~\ref{eq:*twisted Hopf on dived powers}):
\begin{equation}\label{ef coproduct}
\begin{split}
  \Delta(\te^{(r)}_i1_\lambda) 
  &= \sum_{c=0}^r \prod_{\lambda'+\lambda''=\lambda} \e^{-(r-c)(\zeta^{*>}_i, \lambda'')} 
     \te_i^{(r-c)}1_{\lambda'}\otimes \te_i^{(c)}1_{\lambda''} \,, \\
  \Delta(\tf_i^{(r)}1_\lambda) 
  &= 
  \sum_{c=0}^r \prod_{\lambda'+\lambda''=\lambda} \e^{c(\zeta^{*<}_i, \lambda'')} 
  \tf_i^{(c)}1_{\lambda'}\otimes \tf_i^{(r-c)}1_{\lambda''} \,,
\end{split}
\end{equation}
as well as some relations in $\hU^*_\e(\g, X^*)$, cf.~(\ref{EF relations},~\ref{eq: ef*-relation}):
\begin{equation}\label{ef relations}
\begin{split}
  & \te_i^{(p)}1_\lambda \tf_j^{(s)} = \tf_j^{(s)}1_{\lambda+s\a^*_j+p\a^*_i}\te_i^{(p)} \qquad 
      \mathrm{for} \ \ i\neq j \,, \\
  & \te_i^{(p)}1_\lambda \tf_i^{(s)} = 
    \sum_{c\geq0} (\e^*_i)^{2ps-c^2}\binom{(\lambda, \a_i^{* \vee})+s+p}{c}_{\e^*_i}
    \tf_i^{(s-c)}1_{\lambda+(p+s-c)\a^*_i}\te_i^{(p-c)} \,,
\end{split}
\end{equation}
with $\lambda, \lambda', \lambda''\in X^*$ in the formulas above.


\subsection{The Kostant form $\dU_R(\g^d)$}\label{ssec:Kostant Z-form}
\

It turns out that the algebra $\hU^*_\e(\g, X^*)$ is closely related to a classical object, the Kostant form of the universal enveloping algebra. 

Let us  recall the semisimple Lie algebra $\g^d$  with the Cartan matrix \eqref{New Cartan entry}, the weight lattice $P^*$, 
and the root lattice $Q^*$ from the previous subsections. Let $\{e_i, f_i, h_i\}_{1\leq i \leq r}$ be the Chevalley generators 
of $\g^d$. Following \cite{k}, the \textbf{Kostant form} $\dU_\BZ(\g^d)$ is the $\BZ$-subalgebra of the enveloping algebra $U_\BQ(\g^d)$ 
over $\BQ$ generated by the following elements:
\[
  e_i^{(n)}:=\frac{e_i^n}{n!} \,, \qquad f_i^{(n)}:=\frac{f_i^n}{n!} \,, \qquad 
  \binom{h_i;a}{n}:=\frac{(h_i+a)(h_i+a-1) \cdots (h_i+a-n+1)}{n!} \,.
\]
Let $\dU^>_\BZ(\g^d), \dU^<_\BZ(\g^d), \dU^0_\BZ(\g^d)$ be the $\BZ$-subalgebras of $\dU_\BQ(\g^d)$ generated by elements 
$\{e_i^{(n)}\}, \{f_i^{(n)}\}$, $\ds \Big\{\binom{h_i;a}{n}\Big\}$, respectively. We have the standard braid group action on 
$U_\BQ(\g^d)$ defined by
\begin{equation*}
\begin{split}
  & T^*_i(h_j)=s^*_i(h_j)  \,, \qquad T^*_i(e_i)=-f_i \,, \qquad T^*_i(f_i)=-e_i \,, \\
  & T^*_i(e_j)=\sum_{k=0}^{-a^*_{ij}}(-1)^k e_i^{(-a^*_{ij}-k)}e_je_i^{(k)} \,, \qquad 
    T^*_i(f_j)=\sum_{k=0}^{-a^*_{ij}}(-1)^k f_i^{(k)}f_j f^{(-a^*_{ij}-k)} \qquad (i \neq j) \,. 
\end{split}
\end{equation*}
Evoking the reduced decomposition of the longest element $w_0=s^*_{i_1}\dots s^*_{i_N}$ in the Weyl group of $\g^d$ and the resulting 
labeling $\b^*_k$ of all $\Delta^d_+$, one defines root vectors of $\g^d$ in a standard way:

\[ 
  e_{\b^*_k}:=T^*_{i_1}\dots T^*_{i_{k-1}}e_{i_k} \,, \qquad f_{\b^*_k}:= T^*_{i_1}\dots T^*_{i_{k-1}}f_{i_k} \,.
\]
For any $\vec{k}\in \BZ^N_{\geq 0}$, we define
\begin{equation*}
  e^{(\vec{k})}=e_{\b^*_1}^{(k_1)}\dots e_{\b^*_N}^{(k_N)} \,, \quad 
  e^{(\cev{k})}:=e_{\b^*_N}^{(k_N)}\dots e_{\b^*_1}^{(k_1)} \,, \quad 
  f^{(\vec{k})}:=f_{\b^*_1}^{(k_1)}\dots f_{\b^*_N}^{(k_N)} \,, \quad 
  f^{(\cev{k})}:=f_{\b^*_N}^{(k_N)}\dots f_{\b^*_1}^{(k_1)} \,.
\end{equation*}
This construction gives a Chevalley-type basis of $\g^d$ as used in~\cite{k} (we left the proof to interested readers). Then we can apply \cite[Theorem 1]{k} to get:

\begin{Lem}\label{lem: triangular-PBW-kostant form}
(a) The multiplication map:
\[ 
  \mm\colon \dU^<_\BZ(\g^d)\otimes_\BZ \dU^0_\BZ(\g^d) \otimes_\BZ \dU^>_\BZ(\g^d) \longrightarrow \dU_\BZ(\g^d)
\]
is an isomorphism of free $\BZ$-modules.

\noindent
(b) The sets $\{e^{(\vec{k})}\}_{\vec{k}\in \BZ^N_{\geq 0}}$, $\{e^{(\cev{k})}\}_{\vec{k}\in \BZ^N_{\geq 0}}$ are 
$\BZ$-bases of $\dU^>_\BZ(\g^d)$.

\noindent
(c) The sets $\{ f^{(\vec{k})}\}_{\vec{k}\in \BZ^N_{\geq 0}}$, $ \{f^{(\cev{k})}\}_{\vec{k}\in \BZ^N_{\geq 0}}$ are 
$\BZ$-bases of $\dU^<_\BZ(\g^d)$.

\noindent 
(d) The set $\ds \left\{ \prod_{i=1}^r \binom{h_i;0}{t_i} \,\Big|\, t_i \geq 0\right\}$ is a $\BZ$-basis of $\dU^0_\BZ(\g^d)$.
\end{Lem}

For any commutative ring $R$, we define the \textbf{Kostant form} $\dU_R(\g^d)$ via:  
\[
  \dU_R(\g^d):=\dU_{\BZ}(\g^d)\otimes_\BZ R \,.
\]
Let $\dU^>_R(\g^d), \dU_R^0(\g^d), \dU_R^<(\g^d)$ be the corresponding $R$-subalgebras of $\dU_R(\g^d)$. The next lemma relates 
the positive and negative subalgebras of $\dU^*_\e(\g)$ and $\dU_R(\g^d)$, compare to \cite[$\mathsection 33.2$]{l-book}.

\begin{Lem}\label{lem: dU*(g) iso Kostant form}
(a) There is a unique isomorphism of $R$-algebras $\CF^>\colon \dU^{*>}_\e \rightarrow \dU^>_R(\g^d)$ such that  
$\CF^>(\te^{(n)}_i)=e^{(n)}_i$ for all $n\geq 1, 1\leq i \leq r$.

\noindent
(b) There is a unique isomorphism of $R$-algebras $\CF^<\colon \dU^{*<}_\e\rightarrow \dU_R^<(\g^d)$ such that 
$\CF^<(\tf^{(n)}_i)=f_i^{(n)}$ for all $n \geq 1, 1\leq i \leq r$.
\end{Lem}



\begin{proof} We write $\dU^{*>}_\e$ instead of $\dU^{*>}_R$  to keep track of the base ring $R$. 

\noindent
(a) The uniqueness is obvious. To prove the existence, we note that the general case follows from the case $R=\CA'$ through the 
base change. Since $\CA'$ is an integral domain, it is embedded into its fraction field $\CK$. The algebra $\dU^{*>}_{\CA'}$ is 
thus embedded into $\dU^{*>}_\CK$ as the $\CA'$-subalgebra generated by $\{\te_i^{(n)}\}$ by Lemma~\ref{lem:*twisted-trinagular-PBW}. 
Likewise, the algebra $\dU^>_{\CA'}(\g^d)$ is embedded in $\dU^>_\CK(\g^d)$ as the $\CA'$-subalgebra generated by $\{e_i^{(n)}\}$ by 
Lemma \ref{lem: triangular-PBW-kostant form}. Therefore, the existence of the claimed isomorphism in the case $R=\CA'$ follows from 
its existence in the case $R=\CK$.

Since $\CK$ is a field of characteristics $0$, the algebra $\dU^>_\CK(\g^d)$ is generated by the elements $e_i$ subject to the usual 
Serre relations. Therefore, due to \eqref{eq:Serre-star} , there is a $\CK$-algebra homomorphism 
$(\CF^>)^{-1}\colon \dU^>_\CK(\g^d)\rightarrow \dU^{*>}_\CK$ defined by $(\CF^>)^{-1}(e_i)=\te_i$, so that 
$(\CF^>)^{-1}(e_i^{(n)})=\te_i^{(n)}$ for all $n \geq 1$, $1\leq i \leq r$. Therefore, $(\CF^>)^{-1}$ is surjective. 
Furthermore, $(\CF^>)^{-1}$ preserves the $Q^*$-grading on both algebras. By Lemmas \ref{lem:*twisted-trinagular-PBW} 
and~\ref{lem: triangular-PBW-kostant form}, the finite dimensional $\CK$-vector spaces $\dU^{*>}_{\CK, \mu}$ and $\dU^>_{\CK, \mu}(\g^d)$ 
have the same dimension for any $\mu \in Q^*_+$. Therefore, $(\CF^>)^{-1}$ is an isomorphism 
of $\CK$-algebras. The inverse is the claimed isomorphism $\CF^>$ for $R=\CK$.

\noindent
(b) The proof is the same.
\end{proof}

\begin{Cor}\label{cor: relation for pov-neg parts} 
If $R$ is a field of characteristics $0$, then the algebra $\dU^{*>}_\e$ is an algebra generated by $\{\te_i\}_{1\leq i \leq r}$ 
subject to $\sum_{m=0}^{1-a^*_{ij}} (-1)^m \te_i^{(1-a^*_{ij}-m)} \te_j \te_i^{(m)}=0$ for $i \neq j$, while 
the algebra $\dU^{*<}_\e$ is an algebra generated by $\{\tf_i\}_{1\leq i \leq r}$ subject to 
$\sum_{m=0}^{1-a^*_{ij}}(-1)^m \tf_i^{(1-a^*_{ij}-m)}\tf_j \tf^{(m)}_i=0$ for $i \neq j$.
\end{Cor}

\subsection{Completed forms $\tU_\e(\g, P), \tU^*_\e(\g, P^*), \tU^*_\e(\g, Q^*)$}\ 
\label{ssec: complete form}

Let us consider the following completion of $\hU_\e(\g, P)$
\[ \tU_\e(\g, P):=\bigoplus_{\mu, \lambda \in Q_+} \dU^<_{-\mu} \otimes_R \prod_{\nu \in P} R 1_\nu \otimes_R \dU^>_\lambda.\]
The  topology of $\tU_\e(\g, P)$ comes from the factor $\prod_{\nu \in P} R1_\nu$ as follows: small neighborhoods of zero in $\prod R 1_\nu$ contain elements $\prod a_\nu 1_\nu$ 
with $a_\nu =0$ for all $\nu$ in a suitably large ball $\mathcal{B}$ centered at the origin in the real vector space $\h^*_\BR$ 
containing $P$ . Note that the topological Hopf structure on $\hU_\e(\g, P)$ extends to a  topological Hopf structure on $\tU_\e(\g, P)$.

Similarly, we consider the topological Hopf algebras which are completions of $\hU^*(\g, P^*)$ and $ \hU^*(\g, Q^*)$, respectively:
\[ \tU^*_\e(\g, P^*):=\bigoplus_{\mu, \lambda \in Q^*_+} \dU^{*<}_{-\mu}\otimes_R \prod_{\nu \in P^*} R1_\nu \otimes_R \dU^{*>}_\lambda, \quad  \tU^*_\e(\g, Q^*):=\bigoplus_{\mu, \lambda \in Q^*_+} \dU^{*<}_{-\mu}\otimes_R \prod_{\nu \in Q^*} R1_\nu \otimes_R \dU^{*>}_\lambda.\]

 There are natural algebra homomorphisms:
 \begin{equation}\label{eq:relation_to_completion}
  \psi_1\colon \dU_\e(\g) \rightarrow \tU_\e(\g, P) \qquad \mathrm{and} \qquad \psi_2\colon \hU_\e(\g, P) \rightarrow \tU_\e(\g, P).
\end{equation}
The map $\psi_2$ is a natural embedding. The map $\psi$ is defined by $\psi_1(yu_0 x)=y(\prod_{\nu}\hchi_\nu(u_0) 1_\nu) x$ for $y\in \dU_\e^<, x\in \dU_\e^>$ and $u_0 \in \dU_\e^0$.

\begin{Lem}\label{lem: dense subalgebras} Both $\psi_1$ and $\psi_2$ are Hopf algebra homomorphisms. 
The images of $\psi_1$ and $\psi_2$ from~\eqref{eq:relation_to_completion} are dense subalgebras in the codomains. 
\end{Lem}


\begin{proof} 
The image of $\psi_2$ is clearly dense in $\tU_\e(\g, P)$. To show that the image of $\psi_1$ is dense in $\tU_\e(\g,P)$, it is enough to show that 
the image of the natural map $\psi_1^0\colon \dU^0_\e(\g) \rightarrow \prod_{\nu} R 1_\nu$ is dense. To this end, it suffices to show 
that for any ball $\mathcal{B}$ centered at the origin, the map $\dU^0_\e(\g)\rightarrow \prod_{\nu \in \mathcal{B}\cap P} R 1_\nu$ is 
surjective. Since the cardinality of $\mathcal{B}\cap P$ is finite, it is enough to show that for any finite set $S \subset P$, 
the map $\psi^S_1\colon \dU^0_\e(\g) \rightarrow \prod_{\nu \in S}R 1_\nu$ is surjective. The last statement follows from the case 
$R=\CA_N:=\CA[v^{\pm 1/\sN}]$ through the base change, which we shall deduce from the case when $R$ is a field.

\emph{Step 1: $R=\BF$ is a field.} 
Following~\eqref{eq: defi of chi-lambda}, let us consider the algebra homomorphisms $\hat{\chi}_\nu\colon \dU^0_\e(\g) \rightarrow \BF$ defined by 
\begin{equation}\label{eq:hat-chi-revised}
  \hat{\chi}_{\nu}\colon \quad K^\lambda \mapsto \epsilon^{(\lambda, \nu)} \,, \quad 
  \binom{K_i;0}{t} \mapsto \binom{(\a^\vee_i, \nu)}{t}_{\epsilon_i} \,.
\end{equation}
The characters $\{\hat{\chi}_\nu\}_{\nu \in S}$ of characters of $\dU^0_\e(\g)$ are linearly independent, cf.\ the proof of Proposition~\ref{Prop:pairing_whole_R}(b). 
The map $\dU^0_\e(\g)\rightarrow \prod_{\nu \in \mathcal{B}\cap P} R 1_\nu$ sends an element $a$ to $(\hat{\chi}_\nu(a))_{\nu\in S}$ finishing the proof.


\emph{Step 2: $R=\CA_\sN$.} 
For a maximal ideal $\mathfrak{p}$ of $\CA_\sN$, let $\CA_{\sN\fp}$ and $\BF_\fp$ be the localization and the residue field of $\CA_\sN$. 
The map $\psi^S_1$ is surjective if the maps $\psi^S_{1, \CA_{\sN\fp}}\colon \dU^0_{\CA_{\sN\fp}}(\g)\rightarrow \prod_{\nu \in S} \CA_{\sN\fp} 1_\nu$ 
are surjective for all maximal ideals $\fp$ of $\CA_\sN$ because the cokernel of $\psi^S_1$ is a finitely generated module annihilated by the localization at any maximal ideal and such a module must be $0$. Since 
$\prod_{\nu\in S} \CA_{\sN\fp} 1_\nu$ is a finitely generated $\CA_{\sN\fp}$-module, due to Nakayama's lemma, $\psi^S_{1, \CA_{\sN\fp}}$ is 
surjective iff the map $\psi^S_{1, \BF_\fp}\colon \dU^0_{\BF_\fp}(\g)\rightarrow \prod_{\nu \in S} \BF_{\fp}1_\nu$ is surjective.  
The latter follows by Step 1.
\end{proof}

 \begin{Rem}Similarly, we have Hopf algebra homomorphisms 
 \[ \psi^*_1: \dU^*_\e(\g)\rightarrow \tU^*_\e(\g, P^*), \qquad \psi^*_2: \hU^*_\e(\g, P^*)\rightarrow \tU^*_\e(\g, P^*),\]
 with dense images in the codomains. Moreover, $\psi^*_2$ is an inclusion.
 \end{Rem}

\begin{Lem}\label{lem: Kostant form and complete form} (a) We have the following algebra homomorphisms:  
\begin{equation}\label{eq: Kostant form vs complete} \phi_P: \dU_R(\g^d) \rightarrow \tU^*_\e(\g, P^*),  \qquad \phi_Q:  \dU_R(\g^d) \rightarrow \tU^*_\e(\g, Q^*),
\end{equation}
defined by
\[ \phi_X(e^{(n)}_i)= (\e^*_i)^n \te^{(n)}_i \prod_{\nu \in X^*} 1_\nu, \quad \phi_X(f^{(n)}_i)=\tf^{(n)}_i \prod_{\nu \in X^*} 1_\nu, \quad \phi_X\left(\binom{h_i;0}{t_i}\right)=\prod_{\nu \in X^*} \binom{(\nu, \a^{*\vee}_i)}{t_i}1_\nu,\]
with $1\leq i \leq r,$ $t_i \in \BZ_{\geq 0},$ and $X$ is either $P$ or $Q$.

\noindent
(b) $\phi_P, \phi_Q$ are inclusions with dense images.
\end{Lem}
\begin{proof}(a) The general case follows by the base change from the case when $R=\CA'$. Since $\dU_R(\g^d)$ and $\tU^*_\e(\g, X^*)$ are torsion free over the domain $R=\CA'$, the case when $R=\CA'$ follows by the case when $R=\CK$, the fraction field of $\CA'$. 

When $R=\CK$, the existence of algebra homomorphisms follows by using the Serre presentation of the enveloping algebra $\dU_\CK(\g^d)$ combining with \eqref{ef relations} and Lemma \ref{lem: dU*(g) iso Kostant form}. 

\noindent
(b) It is enough to show that $\phi^0_X: \dU^0_R(\g^d) \rightarrow \prod_{\nu \in X^*}R1_\nu$ is injective. Let $T^d_X$ be the torus with the lattice $X^*$, i.e., $R[T^d_X]=\oplus_{\nu \in X^*} R \chi_\nu$, hence, $\prod_{\nu \in X^*} R1_\nu =\Hom_R(R[T^d_X], R)$. On the other hand $\dU^0_R(\g^d)$ is  $\text{Dist}(T^d_X)$, the distribution algebra of $T^d_X$ at identity. Then one see that the map $\phi^0_X$ can be obtained from the pairing 
\[ \text{Dist}(T^d_X) \x R[T^d_X]\rightarrow R.\]
This pairing is nondegenerate on the first argument, hence $\phi^0_X$ is injective.  The proof of dense images are the same as the one of Lemma \ref{lem: dense subalgebras}.    
\end{proof}

\begin{Rem}While $\phi_Q$ is a Hopf algebra homomorphism, $\phi_P$ is not, e.g., by looking at the coproduct 
\[\tilde{\Delta}(\te_i \prod_{\nu \in P^*}1_\nu)= \te_i \prod_{\nu\in P^*} 1_\nu \otimes \prod_{\nu \in P^*} \e^{(-\zeta^{*>}_i, \nu)}1_\nu +\prod_{\nu \in P^*} 1_\nu \otimes \te_i \prod_{\nu \in P^*} 1_\nu.\]
\end{Rem}

\subsection{Quantum Frobenius homomorphism}\label{ssec:qFr} 
\

Following~\cite[Section~$35.1.2$]{l-book}, let us assume that for any $i \neq j$ such that $\ell_j \geq 2$, 
we have $\ell_i \geq 1-a_{ij}$. The excluded values of $\ell$ for each type of simple Lie algebras are recorded in the following table:
\begin{center}
\begin{tabular}{|c|c|c|c|c|c|c|}
\hline
A & B &C &D & E &F &G\\
\hline 
 $\emptyset$ & $4$& $4$ & $\emptyset$ & $\emptyset$ & $4$& $3,4,6$\\
\hline  
\end{tabular}
\end{center}
We have the following analogue of~\cite[Theorem~$35.1.9$]{l-book}:

\begin{Prop}\label{Fr in idem form} 
There is a unique $R$-algebra homomorphism (twisted quantum Frobenius)
\begin{equation}\label{eq: Fr for idemp} 
  \tFr\colon \hU_\e(\g, P) \longrightarrow \hU^*_\e(\g, P^*)
\end{equation}
such that
\begin{itemize}

\item 
$\tFr(1_\lambda)$ equals $1_\lambda$ if $\lambda \in P^*$ and is zero otherwise,

\item 
$\tFr(\tE_i^{(n)}1_\lambda)$ equals $\te_i^{(n/\ell_i)}1_\lambda$ if $\lambda \in P^*$ and $n$ is divisible by $\ell_i$, 
and is zero otherwise;

\item 
$\tFr(\tF_i^{(n)}1_\lambda)$ equals $\tf_i^{(n/\ell_i)}1_\lambda$ if $\lambda \in P^*$ and $n$ is divisible by $\ell_i$, 
and is zero otherwise.

\end{itemize}
Furthermore, this homomorphism is surjective and compatible with the comultiplications. 
\end{Prop}

The assumption that there is $\e^{1/\sN}$ in $R$ is only needed when we talk about the comultiplications. 
The proof of Proposition \ref{Fr in idem form} is similar to that of \cite[Theorem $35.1.9$]{l-book}. We start with the following result:


\begin{Prop}\label{prop: steps to construct Frobenious}
(a) There is a unique $R$-algebra homomorphism $\hFr^<\colon \dU^{*<}_\e\rightarrow \dU^{<}_\e$ such that 
$\hFr^<(\tf_i^{(n)})=\tF_i^{(n\ell_i)}$ for all $i$ and $n \geq 1$.Similarly, there is a unique $R$-algebra homomorphism $\hFr^>\colon \dU^{*>}_\e\rightarrow \dU^{>}_\e$ such that 
$\hFr^<(\te_i^{(n)})=\tE_i^{(n\ell_i)}$ for all $i$ and $n \geq 1$. 


\noindent
(b) There is a unique $R$-algebra homomorphism:
\begin{equation}\label{eq: Fr for positive part} 
  \tFr^>\colon \dU^>_\e \longrightarrow \dU^{*>}_\e
\end{equation}
such that $\tFr^>(\tE_i^{(n)})$ equals $\te_i^{(n/\ell_i)}$ if $n$ is divisible by $\ell_i$, and equals zero otherwise. 
Similarly, there is a unique $R$-algebra homomorphism
\begin{equation}\label{eq: Fr for negative part} 
  \tFr^<\colon \dU^<_\e \longrightarrow \dU^{*<}_\e
\end{equation}
such that $\tFr^<(\tF_i^{(n)})$ equals $\tf_i^{(n/\ell_i)}$ if $n$ is divisible by $\ell_i$, and equals zero otherwise.
\end{Prop}
\begin{proof}The proof is the same as in \cite{l-book}. One reduces to the case when  $R=\CK$, the fraction field of $\CA'$.

\noindent
(a) The uniqueness is clear. The existence follows by using the generators and relations of $\dU^{*<}_\e$ in Corrollary \ref{cor: relation for pov-neg parts} and the following identity
\[ \sum_{n+m=1-a^*_{ij}}(-1)^m \tF_i^{(\ell_i n)}\tF_j^{(\ell_j)} \tF_i^{(\ell_i m)}=0.\]
This identity is deduced from the formula \cite[$\mathsection 35.2.3(b)$]{l-book}.

\noindent
(b) Let us prove \eqref{eq: Fr for negative part} only. We use the following lemma which is a version of Theorem $35.4.2(b)$ in \cite{l-book} incorporating a Cartan twist
\begin{Lem} \label{lem: Spliting of Fr}
Let $\tilde{\mathfrak{f}}$ be the $\CK$-subalgebra of $\dU^<_\e$ generated by $\tF_i$ with $\ell_i \geq 2$. There is an 
isomorphism of $\CK$-vector spaces $\mathfrak{F}\colon \dU^{*<}_\e \otimes \tilde{\mathfrak{f}} \rightarrow \dU^<_\e$ 
defined by $x \otimes y \mapsto \hFr^<(x)y$.
\end{Lem}
With this lemma, one proceeds exactly the same as the proof of \cite[Theorem 35.1.7]{l-book} to finish the proof. 
\end{proof}

\begin{Rem}\label{partial KerFr}
First, we note that the compositions $\tFr^>\circ \hFr^>$  and $\tFr^< \circ \hFr^<$ are identities. Second, we claim that 
the kernel of the homomorphism $\tFr^>\colon \dU^>_\e\rightarrow \dU^{*>}_\e$ from~\eqref{eq: Fr for positive part} is 
a two-sided ideal generated by $\{\tE_i^{(n)} \,|\, \ell_i \geq 2, \ell_i \nmid n\}$, and likewise the kernel of the homomorphism 
$\tFr^<\colon \dU^<_\e\rightarrow \dU^{*<}_\e$ from~\eqref{eq: Fr for negative part} is a two-sided ideal generated by 
$\{\tF_i^{(n)} \,|\, \ell_i \geq 2, \ell_i \nmid n\}$. Let us establish the kernel of the map $\tFr^>$. It is enough to show that $\tFr^>(x)=0$ 
implies $x=0$ for the elements of the following form: $x=\sum a_{n_1, \dots n_m} \tE_{i_1}^{(n_1)}\dots \tE_{i_m}^{(n_m)}$ 
with $\ell_{i_j}\mid n_j$ and $ a_{n_1, \dots n_m}\in R$. For such $x$, consider 
$\tilde{x}=\sum a_{n_1, \dots n_m}\te_{i_1}^{(n_1/\ell_{i_1})}\dots \te_{i_m}^{(n_m/\ell_{i_m})}$, 
so that $x=\hFr^>(\tilde{x})$ and $\tilde{x}=\tFr^>(x)$. Therefore, if $\tFr(x)=0$ then $x=0$. 
\end{Rem}

\begin{proof}[Proof of Proposition \ref{Fr in idem form}]
In~\cite{l-book}, the proof of Theorem 35.1.9 is based on the presentation of the idempotented version by generators and relations, 
established in Section~31.1.3. The similar presentation also holds in our Cartan-twisted setup since the Cartan elements are all 
invertible. The only non-trivial relations we need to check are the ones between $\tE_i^{(n)}$ and $\tF_j^{(m)}$, but those directly 
follow by comparing the formulas \eqref{EF relations} and \eqref{ef relations}. It is straightforward to check that the homomorphism 
$\tFr$ is compatible with the coproducts in \eqref{EF coproduct} and \eqref{ef coproduct}.
\end{proof}

\begin{Rem}\label{rem: dU acts on U*(g,P*)}We have a natural adjoint action of $\dU_\e(\g)$ on $\hU_\e(\g, P)$ via $\ad'_l$. Via the Hopf morphism $\tFr: \dU_\e(\g, P) \rightarrow \dU^*_\e(\g, P^*)$, the Hopf algebra $\dU_\e(\g)$ naturally acts on $\dU^*_\e(\g, P^*)$ via $\ad'_l$ as follows: $ \ad'_l(x)b=\tFr(\ad'_l(x)b')$ for $b\in \dU^*_\e(\g, P^*)$ and any $b'\in \hU_\e(\g, P)$ such that $\tFr(b')=b$.
\end{Rem}

Let us recall the Kostant form $\dU_R(\g^d)$ from Subsection \ref{ssec:Kostant Z-form}. Then, we have the following twisted 
Frobenius homomorphism from the original Lusztig forms: 

\begin{Prop}\label{Fr in Lus form}
There is a unique $R$-Hopf algebra homomorphism (twisted quantum Frobenius)
\begin{equation*}
  \cFr\colon \dU_\e(\g)\longrightarrow \dU_R(\g^d)
\end{equation*}
such that
\[
  \tE_i^{(n)}\mapsto (\e_i^*)^{-n/\ell_i}e_i^{(n/\ell_i)} \,, \quad 
  \tF_i^{(n)}\mapsto f_i^{(n/\ell_i)} \,, \quad K^\lambda \mapsto 1 \,,
\]
where $\lambda \in 2P$ and we set $e_i^{(n/\ell_i)}=f_i^{(n/\ell_i)}=0$ if $\ell_i$ does not divide $n$. 
\end{Prop}

\begin{proof}Let us recall the topological Hopf algebras $\tU_\e(\g, P)$ and $\tU^*_\e(\g, Q^*)$ in Section \ref{ssec: complete form}. The Hopf algebra homomorphism in Proposition \ref{Fr in idem form} gives us a (topological) Hopf algebra homomorphism:
\begin{equation}\label{eq: Fr in complete form}
\cFr: \tU_\e(\g, P)\rightarrow \tU^*_\e(\g, Q^*).
\end{equation}
On the other hand, we have the following Hopf algebras morphisms:
\[ \psi_1: \dU_\e(\g)\rightarrow \tU_\e(\g,P), \qquad \phi_Q: \dU_R(\g^d) \rightarrow \tU^*_\e(\g, Q^*),\]
in which  $\phi_Q$ is an inclusion, see Section \ref{ssec: complete form}. Furthermore, the images of $\psi_1(\tE^{(n)}_i),$  $\psi_1(\tF^{(n)}_i),$ $ \psi_1(K^{2\lambda})$ under  \eqref{eq: Fr in complete form} are contained in $\phi_Q(\dU_R(\g^d))$\footnote{ We want to note that $\cFr(\psi_1(K^{2\lambda})=\prod_{\nu \in Q^*} 1_\nu$, only happen when $2\lambda \in 2P$ and $\nu$ runs over $Q^*$, another justification for the lattice restrictions in our arguments.}. Therefore, the Hopf algebra homomorphism \[ \dU_\e(\g) \xrightarrow[]{\psi_1} \tU(\g, P)\xrightarrow[]{\cFr} \tU(\g, Q^*)\]
gives a rise to the desired Hopf algebra homomorphism $\dU_\e(\g)\rightarrow \dU_R(\g^d)$.
\end{proof}

\begin{Lem}\label{KerFr}  
The kernel of  $\cFr\colon \dU_\e(\g) \rightarrow \dU_R(\g^d)$ is the two-sided ideal generated by 
\begin{equation}\label{eq: generators of Frokernel} \{K^\lambda-1 \,|\, \lambda \in 2P\}  \cup \{\tE^{(n)}_i, \tF^{(n)}_i \,|\, \ell_i \geq 2, \ell_i \nmid n\}.
\end{equation}
\end{Lem}
\begin{proof} We have 
\[ \cFr\circ \psi_1\left(\binom{K_i;0}{t_i}\right)= \prod_{\nu \in Q^*} \binom{(\nu, \a^\vee_i)}{t_i}_{\e_i} 1_\nu=\begin{cases} 0 \qquad &\text{if $\ell_i \nmid t_i$}\\
                   \prod_{\nu\in Q^*}\binom{(\nu, \a^{*\vee}_i)}{t_i/\ell_i} 1_\nu=\phi_Q\left(\binom{h_i;0}{t_i/\ell_i}\right) , &\text{if $\ell_i \mid t_i$}
    \end{cases}\]
therefore, 
\[ \phi\left(\binom{K_i;0}{t_i}\right)=\begin{cases} 0 &\text{if $\ell_i\nmid t_i$}\\
\binom{h_i;0}{t_i/\ell_i} &\text{if $\ell_i \mid t_i$}
\end{cases}.\]

\noindent
{\it Step 1:} We will show that $\Ker(\cFr)$ is the two-sided ideal generated by 
\[\{ K^\lambda-1|\lambda \in 2P\}\cup \left\{ \binom{K_i;0}{t_i}|\ell_i \nmid t_i\right\}\cup \{\tE^{(n)}_i, \tF^{(n)}_i \,|\, \ell_i \geq 2, \ell_i \nmid n\}.\]
The map $\cFr$ is compatible with triangular decomposition, that is $\cFr \cong \cFr^<\otimes \cFr^0 \otimes \cFr^>$ with 
\[ \cFr^>: \dU^>_\e \rightarrow \dU^>_R(\g^d), \qquad \cFr^0: \dU^0_\e\rightarrow \dU^0_R(\g^d), \qquad \cFr^<: \dU^<_\e \rightarrow \dU^<_R(\g^d).\]
By Remark \ref{partial KerFr}, $\Ker(\cFr^>)$ is the two side ideal generated by $\{ \tE_i^{(n)}|\ell_i\geq 2, \ell_i \nmid n\}$, meanwhile, $\ker(\cFr^<)$ is the two sided ideal generated by $\{\tF_i^{(n)}|\ell_i \geq 2, \ell_i \nmid n\}$. 
On the other hand, using the PBW-basis in Lemma \ref{lem:twisted-triangular-PBW}(c3), one can see that $\Ker(\cFr^0)$ is the two-sided ideal generated by 
\[\{ K^\lambda-1|\lambda \in 2P\}\cup \left\{ \binom{K_i;0}{t_i}|\ell_i \nmid t_i\right\}.\]

All of $\cFr^>, \cFr^<, \cFr^0$ are surjective. All of $\dU^>_R(\g^d), \dU^0_R(\g^d), \dU^>_R(\g^d)$ are free $R$-modules. Hence $\cFr^<, \cFr^0, \cFr^>$ are split morphisms of $R$-modules. Therefore, the compatibility of $\cFr$ with triangular decompositions implies that 
\[ \Ker(\cFr)= \Ker(\cFr^<)\otimes \cFr^0 \otimes \cFr^> +\cFr^< \otimes \Ker(\cFr^0) \otimes \cFr^> +\cFr^< \otimes \cFr^0 \otimes \Ker(\cFr^>).\]
This finishes the first step .

\noindent
{\it Step 2:} The lemma follows if we can show that $\ds \binom{K_i;-c}{t_i}, \ell_i \nmid t_i, \ell_i \mid  c\geq 0$ belongs to the two-sided ideal $\mathcal{J}$ generated by \eqref{eq: generators of Frokernel}. We have 
\begin{equation}
\label{eq: eq2} \binom{K_i; -c}{t} = \sum_{0\leq k \leq t} (-1)^k 
    \epsilon_i^{2t(c+k)-k(k+1)} \binom{c+k-1}{k}_i \binom{K_i;0}{t-k} \;\; (t\geq 0, c\geq 1)
\end{equation}

Since $\binom{c+k-1}{k}_i=0$ for $\ell_i \mid c$ and $\ell_i \nmid k$, we have 
\[ \binom{K_i;-c}{t}=\sum_{0\leq k \leq t, \ell_i \nmid k} (-1)^k \e_i^{2t(c+k)-k(k+1)}\binom{c+k-1}{k}_i \binom{K_i;0}{t-k}, ~~(t\geq 0, c\geq 1, \ell_i \mid c).\]
This implies that $\binom{K_i;-c}{t}$ with $c\geq 1, \ell_i \mid c, \ell_i \nmid t$ is a linear combination of $\binom{K_i;0}{t'}$ with $\ell_i \nmid t'\leq t$. 
Now the second step is proved by induction on $t_i$ and the following equation
\begin{equation}
\label{eq: eq1} \tE_i^{(p)}\tF^{(p)}_i=\sum_{c=0}^{\min(p,s)}\e^{2ps-c^2}_i \tF_i^{(p-c)}\binom{K_i;2c-2p}{c} \tE_i^{(p-c)}.
\end{equation}
By using  \eqref{eq: eq1} for $0<p<\ell_i$, we have $\binom{K_i;0}{t_i} \in \mathcal{J}$ for $0<t_i <\ell_i$, hence $\binom{K_i;-c}{t_i}$ with $c \geq 0, \ell_i \mid c, 0<t_i <\ell_i$ belongs to $\mathcal{J}$. By induction and \eqref{eq: eq1}, $\binom{K_i;0}{t}\in \mathcal{J}$ for $\ell_i\nmid t$ hence $\binom{K_i;-c}{t_i} \in \mathcal{J}$ for all $ c\geq 0, \ell_i \mid c, \ell_i \nmid t_i \leq t$.
\end{proof}

\begin{Prop}\label{support lem} 
(a) There is an $R$-Hopf algebra homomorphism $\phi\colon \dU^*_\epsilon(\g) \rightarrow \dU_R(\g^d)$ defined by 
\begin{equation}\label{eq: image of phi} 
  \te_i^{(n)} \mapsto (\e^*_i)^{-n} e_i^{(n)} \,, \quad \tf_i^{(n)} \mapsto f_i^{(n)} \,, \quad
  K^\lambda \mapsto 1 \ (\lambda\in 2P^*), \quad \binom{K^*_i;0}{t}\mapsto \binom{h_i;0}{t} \,.
\end{equation}

\noindent
(b) The kernel of $\phi$ is the Hopf ideal generated (as an ideal) by $\{K^\lambda-1\}_{\lambda \in 2P^*}$. 
\end{Prop}


\begin{proof} 
(a) We have Hopf algebra homomorphisms:
\[ \dU^*_\e(\g) \xrightarrow[]{\psi^*_1} \tU_\e(\g, P^*) \rightarrow \tU_\e(\g, Q^*) \xhookleftarrow[]{\phi_Q} \dU_R(\g^d).\]
The images of $\te^{(n)}_i, \tf^{(n)}_i , K^\lambda$ in $\tU_\e(\g, Q^*)$ are contained in $\phi_Q(\dU_R(\g^d))$. Hence we obtain the desired $R$-Hopf algebra homomorphism $\phi: \dU^*_\e(\g) \rightarrow \dU_R(\g^d)$.

\noindent
(b) The map $\phi$ is compatible with triangular decompositions, that is $\phi \cong \phi^< \otimes \phi^0 \otimes \phi^>$ with 
\[ \phi^<: \dU^{*<}_\e\rightarrow \dU^<_R(\g^d), \qquad \phi^0: \dU^{*0}_\e\rightarrow \dU^0_R(\g^d), \qquad \phi^>: \dU^{*>}_\e \rightarrow \dU^>_R(\g^d).\]
We note that $\phi^<, \phi^>$ are isomorphisms. By \eqref{eq: image of phi} and the PBW-basis of $\dU^*_\e(\g)$ in Lemma \ref{lem:*twisted-trinagular-PBW}, one see that $\Ker(\phi^0)$ is the two-sided ideal generated by $\{ K^\lambda-1|\lambda \in 2P^*\}$. On the other hand, $\phi^0$ is a split morphism of $R$-modules since $\dU^0_R(\g^d)$ is free over $R$, therefore, 
\[\Ker(\phi)=\dU^{*<}_\e\otimes \Ker(\phi^0)\otimes \dU^{*>}_\e.\]
The lemma follows.
\end{proof}

\begin{Rem}
The left adjoint action of $\dU^*_\epsilon(\g)$ on itself  and the adjoint action of $\dU^*_\e(\g)$ on $\tU_\e(\g,X^*)$ factor through 
$\phi\colon \dU^*_\epsilon(\g)\rightarrow \dU_R(\g^d)$ of Proposition~\ref{support lem}. This is because the adjoint action of the elements $K^\lambda, \lambda \in 2P^*,$ is trivial.
Therefore, $\phi: \dU^*_\e(\g) \rightarrow \dU_R(\g^d)$ and $\phi_Q: \dU_R(\g^d) \hookrightarrow \tU^*_\e(\g, X^*)$, where $X$ is either $P$ or $Q$,  are morphisms of $\dU_R(\g^d)$-modules.
\end{Rem}

We now state some technical results that will be needed in Section $5$. Let $A_{Q^*}$ be the $R$-subalgebra of $\dU^*_\epsilon(\g)$ generated by $ \te_i^{(n)}K^{n\gamma(\a^*_i)}$, 
$\tf_i^{(n)}K^{n\kappa(\a^*_i)}$, $\ds \binom{K^*_i;0}{t_i}$ for $1\leq i \leq r; $ $t_i, n \geq 1$.

\begin{Lem}\label{Lem about A_Q and tU^*}
(a) $A_{Q^*}$ is 
a $\dU_R(\g^d)$-submodule of $\dU^*_\epsilon(\g)$ and the restriction map $\phi\colon A_{Q^*}\rightarrow \dU_R(\g^d)$ is an $R$-algebra isomorphism.

\noindent
(b) The restriction map $\psi^*_1: A_{Q^*} \rightarrow \tU^*_\e(\g, P^*)$ is an inclusion of $\dU_R(\g^d)$-modules with dense image in the codomain.
\end{Lem}

\begin{proof} 
(a) By Lemma \ref{lem:*twisted-trinagular-PBW}, we can choose $Q^*$-weight $R$-bases $\{u^<\}$ and $\{u^>\}$ of $\dU^{*<}_\epsilon$ 
and $\dU^{*>}_\epsilon$, respectively, so that $A_{Q^*}$ has an $R$-basis of 
the form
\[ 
  u^< K^{\kappa(-\deg(u^<))} \cdot \prod_{i}\binom{K^*_i;0}{t_i} \cdot u^> K^{\gamma(\deg(u^>))} \,.
\]
Combining this with Lemmas \ref{lem: triangular-PBW-kostant form} and \ref{lem: dU*(g) iso Kostant form}, we see that 
$\phi\colon A_{Q^*}\rightarrow \dU_R(\g^d)$ is an isomorphism. It is straightforward to see that $A_{Q^*}$ is stable 
under the left adjoint action of $\dU^*_\epsilon(\g)$ by checking the action  on the set of generators 
$\{\te_i^{(n)}, \tf_i^{(n)}, K^\lambda \,|\, \lambda \in 2P^*\}$, hence $A_{Q^*}$ is a $\dU_R(\g^d)$-submodule of $\dU^*_\e(\g)$.  

\noindent
(b) The proof is the same as that of Lemma \ref{lem: Kostant form and complete form}(b).
\end{proof}
\begin{Rem}We have two inclusions of $\dU_R(\g^d)$-modules with dense image in the codomains: 
\[ \dU_R(\g^d) \overset{\phi^{-1}}{\cong}A_{Q^*} \xhookrightarrow[]{\psi^*_1} \tU^*_\e(\g, P^*), \qquad \dU_R(\g^d) \xhookrightarrow[]{\phi_P}\tU^*_\e(\g, P^*).\]
which coincide after composing with the map $\tU^*_\e(\g, P) \rightarrow \tU^*_\e(\g, Q^*)$. The first inclusion and the algebra $A_{Q^*}$ are used in the construction of the pairing $Z_{Fr}\x \dU_R(\g^d) \rightarrow R$ in Section \ref{sec: Frobenius kernel}. 
\end{Rem}


\begin{Rem}\label{rem:A< and A>} 
Let $A^>( \text{resp. } A^<)$ be the $R$-subalgebra of $\dU_\e(\g)$ generated by $\dU^>_\e(\text{resp. } \dU^<_\e)$ and $K^\lambda$ with all $\lambda \in 2P$. Then $A^>$ and $A^<$  are $R$-Hopf subalgebras of $\dU_\e(\g)$.
\end{Rem}

\begin{Lem}\label{lem: ad(x)(y) in Ker(Fr)}
Let us recall the algebra homomorphism $\tFr\colon \hU_\e(\g, P)\rightarrow \hU^*_\e(\g, P^*)$ of~\eqref{Fr in idem form}. 
For any $x\in \{\tE_i^{(n)}, \tF_i^{(n)}, K^\mu-1 \,|\, \ell_i \nmid n, \mu \in 2P\}$ and $y\in \hU_\e(\g, P)$, 
we have $\ad'(x)(y) \in \Ker(\tFr)$.
\end{Lem}


\begin{proof} 
First, we note that for $y=\tF_{i_1}^{(k_1)}\dots \tF_{i_m}^{(k_m)}1_\lambda \tE_{j_1}^{(r_1)}\dots \tE_{j_n}^{(r_n)}$ 
such that $\ell_{i_t}\mid k_t$ for all $1\leq t\leq m$ and $\ell_{j_t}\mid r_t$ for all $1\leq t \leq n$, we clearly have 
$\ad'(K^\mu-1)(y)=0$. Hence, for any $y\in \hU_\e(\g, P)$, the image $\ad'(K^\mu-1)(y)$ is an $R$-linear combination of 
elements of the form  $\tF_{i_1}^{(k_1)}\dots \tF_{i_m}^{(k_m)}1_\lambda \tE_{j_1}^{(r_1)}\dots \tE_{j_n}^{(r_n)}$ such that 
there is either $t\leq m$ such that $\ell_{i_t}\geq 2, \ell_{i_t}\nmid k_t$ or $t\leq n$ such that 
$\ell_{j_t}\geq 2, \ell_{j_t}\nmid r_t$. This implies that $\ad'(K^\mu-1)(y)\in \Ker(\tFr)$ for any $y\in \hU_\e(\g, P)$.

Any element of $\hU_\e(\g,P)$ is an $R$-linear combination of elements $\{y 1_\lambda \,|\, y \in \dU_\e(\g, P), \lambda \in P\}$. 
According to~\eqref{eq:twisted Hopf on divided powers}, we have:
\begin{align*}
  \ad'(\tE_i^{(n)})(y1_\lambda)
  &= \sum_{c=0}^n \tE_i^{(n-c)}y1_\lambda(-1)^c K^{(n-c)\zlambda_i} \e_i^{c(c-1)}\tE_i^{(c)}K^{c\zlambda_i} \\
  &= \sum_{c=0}^n(-1)^c \e^{n(\lambda, \zlambda_i)}\e_i^{-c(c+1)} \tE_i^{(n-c)}y1_\lambda\tE_i^{(c)} \,.
\end{align*}
If $\ell_i \nmid  n$, then either $n-c$ or $c$ is not divisible by $\ell_i$, and so 
$\tE_i^{(n-c)}y1_\lambda\tE_i^{(c)} \in \Ker(\tFr)$. This proves $\ad'(\tE_i^{(n)})(y) \in \Ker(\tFr)$ for any 
$y \in \hU_\e(\g, P)$ if $\ell_i \nmid n$. The proof for $\tF_i^{(n)}$ is similar.
\end{proof}


\subsection{Comparison with Lusztig's quantum Frobenius map}\label{comparison with untwisted Fr}
\

Let $\hmU_\epsilon(\g,P)$, $\hmU^*_\epsilon(\g, P^*)$ be the corresponding idempotented Lusztig forms of \cite{l-book}. 
Let $\dmU^>_\epsilon,$ $ \dmU^<_\epsilon$ be the $R$-subalgebras of $\dmU_\e(\g)$ generated by 
$\{E_i^{[n]}\}_{1\leq i \leq r}^{n \geq 1}$, $\{F_i^{[n]}\}_{1\leq i \leq r}^{ n \geq 1}$, respectively. 
Let $\dmU^{*>}_\e, \dmU^{*<}_\e$ be the $R$-subalgebras of $\dmU^*_\e(\g)$ generated by 
$\{\he_i^{[n]}\}^{n \geq 1}_{1\leq i \leq r}, \{\hf_i^{[n]}\}^{n \geq 1}_{ 1\leq i \leq r}$, respectively. 
According to~\cite{l-book}, we have the quantum Frobenius homomorphisms:
\begin{equation}\label{eq:usual-qFrob}
  \Fr\colon \hmU_\e(\g,P) \longrightarrow \hmU^*_\e(\g, P^*) \,, \qquad 
  \Fr^>\colon \dmU^{>}_\e \longrightarrow \dmU^{*>}_\e \,, \qquad 
  \Fr^<\colon \dmU^<_\e \longrightarrow \dmU^{*<}_\e \,.
\end{equation}

The twists $\sF$ and $\sF^*$ equip $\hmU_\e(\g, P)$ and $\hmU^*_\e(\g, P^*)$ new (topological) Hopf algebra structures so that $\Fr$ is still a morphism of Hopf algebras. 

Assume $R$ contains an element $\e^{1/2}$. Then we can identify $\hU_\e(\g,P)$ with (twisted Hopf structure) $\hmU_\e(\g, P)$  via \[ \tE_i^{(n)}1_\lambda = \e_i^{\frac{n(1-n)}{2}} \e^{(n \nu^>_i,\lambda)}E_i^{[n]}1_\lambda,\quad \quad F_i^{(n)}1_\lambda= \e_i^{\frac{n(n-1)}{2}}\e^{-(n \nu^<_i, \lambda)} F_i^{[n]} 1_\lambda.\]
Similarly, we can identify $\hU^*_\e(\g, P^*)$ with (twisted Hopf structure via the twist $\sF^*$) $\hmU^*_\e(\g, P^*)$ via
\[ \te_i^{(n)}1_\lambda =(\e^*_i)^{\frac{n(1-n)}{2}}\e^{(n \nu^{*>}_1, \lambda)}\he_i^{[n]}1_\lambda, \quad \quad \tf_i^{(n)}1_\lambda=(\e^*_i)^{\frac{n(n-1)}{2}}\e^{-(n\nu^{*<}_i, \lambda)}\hf^{[n]}_i1_\lambda.\]

Let us define the  map $\Phi: \hmU^*_\e(\g,P^*)\rightarrow \hmU^*_\e(\g,P^*)$ by 
\[ e_i^{[n]}1_\lambda \mapsto \Big(\e_i^{\frac{\ell_i^2-\ell_i}{2}}\Big)^n e_i^{[n]}1_\lambda, \quad \quad f_i^{[n]}1_\lambda \mapsto \Big(\e_i^{\frac{\ell_i-\ell_i^2}{2}}\Big)^n f_i^{[n]}1_\lambda,\]
The map $\Phi$ is an automorphism of Hopf algberas with either the usual or twisted Hopf structure (via the twist $\sF^*$) on $\hmU^*_\e(\g, P^*)$.
\begin{Lem}\label{lem: comparison of Frs} Recall the map $\tFr: \hU_\e(\g, P)\rightarrow \hU^*_\e(\g,P^*)$. Under the above identifications of Hopf algebras,   $\tFr=\Phi \circ \Fr$.   
\end{Lem}
\begin{proof} This follows by directly checking the equality on the generators $\tE_i^{(n)}1_\lambda, \tF_i^{(n)}1_\lambda$.  
\end{proof}

We now state some technical results that will be needed in Section $5$. 

\begin{Lem}\label{lem: PBW basis of Ker of Fr} 
Assume that $\ell_i \geq \textnormal{max}\{2, 1-a_{ij}\}_{1\leq j \leq r}$ for all $i$. Let $\mathscr{K}$ be a subset 
of $\BZ^N_{\geq 0}$ consisting of all $\vec{k}=(k_1,\ldots, k_N)$ with some $k_i$ not divisible by $\ell_{\b_i}$. 

\noindent
(a) The sets $\{F^{[\vec{k}]}\}_{\vec{k}\in \mathscr{K}}$, $\{F^{[\cev{k}]}\}_{\vec{k}\in \mathscr{K}}$ are $R$-bases of 
$\Ker(\Fr^<)$. Similarly, the sets $\{E^{[\vec{k}]}\}_{\vec{k}\in \mathscr{K}}$, $\{E^{[\cev{k}]}\}_{\vec{k}\in \mathscr{K}}$ 
are $R$-bases of $\Ker(\Fr^>)$.

\noindent
(b) The sets $\{\tF^{(\vec{k})}\}_{\vec{k}\in \mathscr{K}}$, $\{\tF^{(\cev{k})}\}_{\vec{k}\in \mathscr{K}}$ are $R$-bases of 
$\Ker(\tFr^<)$. Similarly, the sets $\{\tE^{(\vec{k})}\}_{\vec{k}\in \mathscr{K}}$, $\{\tE^{(\cev{k})}\}_{\vec{k}\in \mathscr{K}}$ 
are $R$-bases of $\Ker(\tFr^>)$.
\end{Lem}

\begin{proof}
(a) Let us fix  $\lambda \in P^*$. We have that  $\dmU^<_\e 1_\lambda$ is a free $\dmU^<_\e$-module and $\dmU^{*<}_\e 1_\lambda$ is a free $\dmU^{*<}_\e$-module. Since the homomorphism $\Fr\colon \hmU_\e(\g,P)\rightarrow \hmU^*_\e(\g, P^*)$ is compatible with the braid group action 
by~\cite[Remark 41.1.9]{l-book}, we see that $\Fr^<(F^{[\vec{k}]}1_\lambda)=\Fr^<(F^{[\vec{k}]})1_\lambda=0$ iff $\vec{k}\in \mathscr{K}$. Furthermore, the images 
$\Fr^<(F^{[\vec{k}]}1_\lambda)$ with $\vec{k}\in \BZ^N_{\geq 0}\setminus \mathscr{K}$ form an $R$-basis for $\dmU^{*<}_\e 1_\lambda$. Thus, the set 
$\{F^{[\vec{k}]}\}_{\vec{k}\in \mathscr{K}}$ is an $R$-basis of $\Ker(\Fr^<)$. The proofs for the other statements are analogous.

\noindent
(b) Follows from part (a) and Lemma \ref{lem: comparison of Frs}.
\end{proof}

\begin{Cor}\label{cor: kernel of Fr} (a) The kernel of $\tFr$ has an $R$-basis consisting of  elements
\begin{itemize}
    \item $\tF^{(\vec{k})} 1_\lambda \tE^{(\vec{r})}$ with $\lambda \in P\backslash P^*$ and $\vec{k}, \vec{r}\in \BZ^N_{\geq 0}$.
    \item $\tF^{(\vec{k})}1_\lambda \tE^{(\vec{r})}$ with $\lambda \in P^*$ and either $\vec{k}$ or $\vec{r}$ contained in $\mathscr{K}$.
\end{itemize}

\noindent
(b) The kernel of $\tFr$ is the two-sided ideal generated by $\{ 1_\mu, \tE^{(n)}_i1_\lambda, \tF^{(n)}_i1_\lambda\}$ with $1\leq i \leq r, \lambda \in P, \mu \not \in P^*, \ell_i \nmid n$ for $\ell_i \geq 2$.
\end{Cor}
\begin{proof}
(a) Follows immediately from Lemma \ref{lem: PBW basis of Ker of Fr}. Then (b) follows by (a) and Remark \ref{partial KerFr}.
\end{proof}

We shall now relate the root generators in the PBW-bases of $\dU_\e(\g), \dU^*_\e(\g), \dU_R(\g^d)$ under the morphisms 
$\cFr$ in Proposition \ref{Fr in Lus form} and $\phi$ in Proposition \ref{support lem}.

Let us introduce a notation: in an $R$-module $M$, we write $m_1 \dsim m_2$ if $m_1=am_2$ for some $a\in R^\x$. We start with the following technical lemma, see \eqref{eq: nu and normalizr b for *} for definitions of some terms:

\begin{Lem}\label{lem: tedious lemma}
Let $\mu \in Q^*_+$ and suppose $s^*_i(\mu)\in Q^*_+$. 

\noindent
(a) Let $x\in \dU^{*>}_{\e, \mu}$ and $\hat{x}:=\e^{b^>_{s^*_i(\mu)}}T^*_i(\e^{-b^>_\mu}xK^{-\nu^>_\mu})K^{\nu^>_{s^*_i(\mu)}}$. 
Then $\phi(\hat{x})\dsim T^*_i(\phi(x))$. 

\noindent
(b) Suppose $y\in \dU^{*<}_{\e,-\mu}$ and $\hat{y}:=\e^{b^<_{s^*_i(\mu)}}K^{-\nu^<_{s^*_i(\mu)}} T^*_i(\e^{-b^<_\mu}K^{\nu^<_\mu} y)$. 
Then $\phi(\hat{y})\dsim T^*_i(\phi(y))$ 
\end{Lem}

\begin{proof}
Let us prove (a) only. Let $\mu=\sum_j u_j\a^*_j$. Arguing as in Lemma \ref{lem: dU*(g) iso Kostant form}, we can assume that 
$R=\CK$ the field of characteristics $0$ which contains an element $\e^{1/2}$. Let 
\[
  x=\sum_{\a^*_{i_1}+\ldots+\a^*_{i_m}=\mu} p_{i_1, \dots, i_m}\te_{i_1}\dots \te_{i_m}
\]
with $p_{i_1, \dots, i_m}\in \CK$. Then
\begin{align*}
  \e^{-b^>_\mu}xK^{-\nu^>_\mu}=\sum_{\a^*_{i_1}+\ldots+\a^*_{i_m}=\mu}
  p_{i_1, \dots, i_m}\e^{-b^>_\mu+\sum_{1\leq j<l\leq m}(\nu^>_{i_j}, \a^*_{i_l})}\he_{i_1}\dots \he_{i_m} \,.
\end{align*}
Hence
\begin{align*}
  \hat{x}
  &= \sum_{\a^*_{i_1}+\ldots+\a^*_{i_m}=\mu} p_{i_1, \dots, i_m} \e^{b^>_{s^*_i(\mu)}-b^>_\mu+\sum_{j<l}(\nu^{*>}_{i_j}, \a^*_{i_l})} 
     \Big(\Rprod_{1\leq j\leq m} T^*_i(\he_{i_j})\Big)K^{\nu^>_{s^*_i(\mu)}}\\
  &= \sum_{\a^*_{i_1}+\ldots+\a^*_{i_m}=\mu} p_{i_1, \dots, i_m} \e^{b^>_{s^*_i(\mu)}-b^>_\mu+B_{i_1,\dots, i_m}}
     \Rprod_{1\leq j\leq m} \Big(T^*_i(\he_{i_j})K^{\nu^{*>}_{i_j}-a^*_{i,i_j}\nu^{*>}_i}\Big) \,,
\end{align*}
in which 
\[ 
  B_{i_1\dots, i_m}=\sum_{j<l}(\nu^{*>}_{i_j}, \a^*_{i_l})-\sum_{j<l}(\nu^{*>}_{i_j}-a^*_{i,i_j}\nu^{*>}_i, s^*_i(\a^*_{i_l})) \,.
\]

\noindent
\emph{Step 1:} We claim that for different $(i'_1, \dots, i'_m)$ we have $B_{i_1,\dots i_m}-B_{i'_1, \dots, i'_m}\in \BZ$ and 
\[ 
  \e^{B_{i_1, \dots, i_m}-B_{i'_1, \dots, i'_m}}=1 \,.
\]
It is enough to prove it in the case when $(i'_1, \dots, i'_m)$ is obtained from $(i_1, \dots, i_m)$ by permuting two 
consecutive indices $i_{j}$ and $i_{j+1}$. Then $B_{i_1, \dots, i_m}-B_{i'_1, \dots, i'_m}$ equals 
\begin{align*}
  & (\nu^{*>}_{i_j},\a^*_{i_{j+1}})-(\nu^{*>}_{i_j}-a^*_{i,i_j}\nu^{*>}_i, s^*_i(\a^*_{i_{j+1}}))-
    (\nu^{*>}_{i_{j+1}}, \a^*_{i_j})+(\nu^{* >}_{i_{j+1}}-a^*_{i,i_{j+1}}\nu^{*>}_i, s^*_i(\a^*_{i_j}))\\
  &= \sd^*_i a^*_{i, i_{j+1}}a^*_{i,i_{j}}(\e_{i,i_{j+1}}-\e_{i,i_j}) \in 2\sd^*_i \BZ \,.
\end{align*}

\noindent
\emph{Step 2:} We compute $T^*_i(\he_{i_j})K^{\nu^{*>}_{i_j}-a^*_{i,i_j}\nu^{*>}_i}$. 

\noindent
If $i_j=i$ then 
\[
  T^*_i(\he_{i_j})K^{\nu^{*>}_{i_j}-a^*_{i,i_j}\nu^{*>}_i} = 
  -\hf_{i}K^{\a^*_i}K^{-\nu^{*>}_i}=-\tf_i K^{2\a^*_i} \,.
\]
If $i_j\neq i$ then 
\[
  T^*_i(\he_{i_j})K^{\nu^{*>}_{i_j}-a^*_{i,i_j}\nu^{*>}_i} = 
  (\e^*_i)^{a^*_{i,i_j}(a^*_{i,i_j}+1)/2}\e^{a^*_{i,i_j}(\nu^{*>}_i, \a^*_{i_j})}
  \sum_{k=0}^{-a^*_{i,i_j}} (-1)^k \te_i^{(-a^*_{i,i_j}-k)}\te_{i_j}\te_i^{(k)} \,,
\]
where we use $(\e_i^*)^2=1$. 



\noindent
\emph{Step 3:} Therefore, we have
\begin{align*}
  \phi(\hat{x}) = 
  \sum_{\a^*_{i_1}+\ldots +\a^*_{i_m}=\mu} p_{i_1, \dots, i_m} \e^{b^>_{s^*_i(\mu)}-b^>_\mu+B_{i_1,\dots, i_m}+C_\mu} 
  \Rprod_{1\leq j\leq m} T^*_i(e_{i_j}) \,,
\end{align*}
in which 
\[
  C_\mu = \sum_{j\neq i}u_j\left(-\sd^*_j+\frac{\sd^*_ia^*_{ij}(1-a^*_{ij})}{2}+ a^*_{ij}(\nu^{*>}_i, \a^*_j)\right) \,.
\]
Since $\e^{b^>_{s^*_i(\mu)}-b^>_\mu+B_{i_1,\dots, i_m}+C_\mu}$ does not depend on the choice of $(i_1, \dots, i_m)$ by Step 1, 
let us fix such $(i_1, \dots, i_m)$. Then we have:
\begin{align*}
  \phi(\hat{x})&=
  \e^{b^>_{s^*_i(\mu)}-b^>_\mu+B_{i_1,\dots, i_m}+C_\mu} \sum_{\a^*_{i_1}+\ldots +\a^*_{i_m}} p_{i_1,\dots,i_m}\prod_j T^*_i(e_{i_j})\\
  &= \e^{b^>_{s^*_i(\mu)}-b^>_\mu+B_{i_1,\dots, i_m}+C_\mu}T^*_i(\phi(x)) \,.
\end{align*}
This finishes the proof since $\e^{b^>_{s^*_i(\mu)}-b^>_\mu+B_{i_1,\dots, i_m}+C_\mu} \in R^\x$. 
\end{proof}

\begin{Lem}\label{lem: computations of cFr}
(a) We have  $\phi(\te_{\b^*_k}^{(n)})\dsim e_{\b^*_k}^{(n)}$ and  $\phi(\tf_{\b^*_k}^{(n)})\dsim f_{\b^*_k}^{(n)}$ .

\noindent
(b) Let us recall the homomorphism $\cFr\colon \dU_\e(\g)\rightarrow \dU_R(\g^d)$ from Proposition \ref{Fr in Lus form}. 
Then 
\[ \cFr(\tE_{\b_k}^{(n)})\dsim\begin{cases}0 \qquad &\text{ if $\ell_{\b_k}\nmid n$}\\
 e_{\b^*_k}^{(n/\ell_{\b_k})}&\text{ 
if $\ell_{\b_k}\mid n$}
\end{cases} , \qquad   \cFr(\tF_{\b_k}^{(n)})\dsim \begin{cases} 0\qquad & \text{if $\ell_{\b_k}\nmid n$}\\
  f_{\b^*_k}^{(n/\ell_{\b_k})}& \text{ if $\ell_{\b_k}\mid n$.}
\end{cases}\]
\end{Lem}

\begin{proof}
(a) Let us prove the first statement only. According to~\eqref{eq:tilde-vs-hat}, we have 
\[
  \te_{\b^*_k}^{(n)} = \e^{nb^>_{\b^*_k}}\e^{(\nu^>_{\b^*_k}, \b^*_k)n(n-1)/2}
  T^*_{i_1}\dots T^*_{i_{k-1}}(\he^{(n)}_{i_k})K^{n\nu^>_{\b^*_k}} \,.
\]
Let $\b'_j=s^*_{i_j}\dots s^*_{i_{k-1}}\a_{i_k}$. Then $\b'_j \in Q^*_+$ and $s^*_{j-1}(\b'_{j})=\b'_{j-1}\in Q^*_+$. 
For $1\leq i \leq k$, let 
\[ 
  x_{j}= \e^{nb^>_{\b'_j}}T^*_{i_j}\dots T^*_{i_{k-1}}(\he^{(n)}_{i_k})K^{n\nu^>_{\b'_j}}\in \dU^>_\e \,, \qquad 
  e_{\b'_j}^{(n)}=T^*_{i_j}\dots T^*_{i_{k-1}}(e^{(n)}_{i_k})\in \dU^>_R(\g^d) \,.
\]
Let us prove by a decreasing induction on $1\leq j \leq k$ that $\phi(x_j)\dsim  e^{(n)}_{\b'_j}$. 
The base case $j=k$ is obvious as $x_k=\te^{(n)}_{i_k}$. 
Let us now prove the induction step. We have $x_j \in \dU^{*>}_{\e, \b^{'}_j}$, so that 
$\phi(x_{j-1})\dsim  T^*_{i_{j-1}}(\phi(x_j))$,  due to Lemma \ref{lem: tedious lemma}. 
By the induction hypothesis, we have $\phi(x_j)\dsim e_{\b'_j}^{(n)}$. Therefore 
$\phi(x_{j-1})\dsim T^*_{i_{j-1}}(e^{(n)}_{\b'_j})= e^{(n)}_{\b'_{j-1}}$, completing the step of induction. In particular, for $j=1$ we get 
$\phi(\te^{(n)}_{\b^*_k})\dsim e^{(n)}_{\b^*_k}$.  The claim follows.

\noindent
(b) Let us show the first statement only. By Lemma \ref{lem: comparison of Frs}, we have:
\begin{equation*}
    \begin{split}
        \tFr^>(\tE_{\b_k}^{(n)})\dsim
        \begin{cases}
              0   \quad &\text{if $\ell_{\b_k} \nmid n$}\\
             \te_{\b^*_k}^{(n/\ell_{\b_k})} \;\;\;&\text{if $\ell_{\b_k}\mid n$}
        \end{cases}.
   \end{split}
\end{equation*}
Note that the restriction of $\cFr\colon \dU_\e(\g)\rightarrow \dU_R(\g^d)$ on $\dU^>_\e(\g)$ is a composition of $\tFr^>$ and 
the restriction of $\phi$ on $\dU^{*>}_\e(\g)$. Therefore, combining the above formula with the part (a), 
we get the first statement of part (b). 
\end{proof}



\subsection{$R$-matrix}\label{sec10}
\

Following \cite[\S8.30]{j} and using the divided powers $\{F_{\b_i}^{[n]},E_{\b_i}^{[n]}\}_{1\leq i\leq N}^{n\geq 1}$, 
we have the following $R$-matrix of $\hmU_\epsilon(\g,P)$: 
\begin{equation} \label{R matrix for original U}
  \CR = \sum_{\lambda, \mu \in P} \epsilon^{-(\lambda, \mu)}1_\lambda \otimes 1_\mu + 
  \left(\sum_{\vec{k}\in \BZ_{\geq 0}^N \backslash (0,\dots,0)}c_{\vec{k}} F^{[\cev{k}]}\otimes E^{[\cev{k}]} \right)
  \left(\sum_{\lambda, \mu \in P} \epsilon^{-(\lambda, \mu)} 1_\lambda \otimes 1_\mu\right) \,,
\end{equation}
where $c_{\vec{k}}=\prod_{t=1}^N \left( \epsilon_{i_t}^{k_t(1-k_t)/2}(\epsilon_{i_t}^{-1}-\epsilon_{i_t})^{k_t} [k_t]_{i_t}! \right)$. 


Evoking~\eqref{R-matrix} and $F$ of~\eqref{formal Cartan twist}, we get the following $R$-matrix for $\hU_\e(\g,P)$:  
\begin{multline}\label{R matrix for U}
  \CR^\sF=\sF^{-1} \CR \sF^{-1}=\\
  \sum_{\lambda, \mu \in P}\e^{-(\lambda, \mu)-\sum_{i,j} 2\phi_{ij}(\w_i^\vee, \lambda)(\w^\vee_j, \mu)}1_\lambda\otimes 1_\mu + 
  \sF^{-1}\left(\sum_{\vec{k} \in \BZ^N_{\geq 0} \backslash (0, \dots, 0)} \sum_{\lambda, \mu \in P} d_{\lambda,\mu,\vec{k}} 
  \tF^{(\cev{k})} 1_\lambda \otimes \tE^{(\cev{k})}1_\mu\right)\sF^{-1}
\end{multline}
with 
  $d_{\lambda, \mu, \vec{k}} = \epsilon^{a_{\lambda, \mu, \vec{k}}} \cdot \prod_{t=1}^N(\epsilon^{-1}_{i_t}-\epsilon_{i_t})^{k_t}(k_t)_{i_t}!$ 
for some $a_{\lambda, \mu, \vec{k}}\in \BZ[\frac{1}{2}]$.

\begin{Lem} 
Assume that $\ell_i \geq \textnormal{max}\{2, 1-a_{ij}\}_{1\leq j \leq r}$ for all $i$. Then the image of  $\CR^{\sF}$ 
under the homomorphism (under suitable completions) $\tFr^{\otimes 2}\colon \hU_\e(\g, P)^{\otimes 2} \rightarrow \hU^*_\e(\g, P^*)^{\otimes 2}$ is
\begin{align}\label{R matrix for U*} 
  \CR^*=\sum_{\lambda, \mu \in P^*}\e^{-(\lambda, \mu)-\sum_{i,j} 2\phi_{ij}(\w_i^\vee, \lambda)(\w^\vee_j, \mu)}1_\lambda\otimes 1_\mu \,.
\end{align}
\end{Lem}

\begin{proof} 
If $\vec{k} \in \BZ^N_{\geq 0} \backslash (0,\dots,0)$ is such that $k_t\geq \ell_{i_t}$ for some $1\leq t \leq N$, then $(k_t)_{i_t}!=0$ 
and so $d_{\lambda, \mu, \vec{k}}=0$ for all $\lambda, \mu \in P$. If $\vec{k} \in \BZ^N_{\geq 0}\backslash (0,\dots,0)$ is such that 
$0\leq k_t <\ell_{i_t}$ for all $1\leq t\leq N$, then there is $t$ such that $0<k_t < \ell_{i_t}$ and so 
$\tF^{(\cev{k})} \in \Ker(\tFr^<), \tE^{(\cev{k})}\in \Ker(\tFr^>)$ by Lemma~\ref{lem: PBW basis of Ker of Fr}. 
Therefore, for any $\vec{k} \in \BZ^N_{\geq 0} \backslash (0, \dots, 0)$ and $\lambda, \mu \in P$, we have 
\[
  \tFr^{\otimes 2}(d_{\lambda, \mu, \vec{k}} \tF^{(\cev{k})}1_\lambda \otimes \tE^{(\cev{k})}1_\mu)=0.
\]
This implies the lemma.
\end{proof}
\begin{Rem}For $\lambda, \mu \in Q^*$, the coefficient of $1_\lambda \otimes 1_\mu$ in $\CR^*$ is $1$. Coefficients of other $1_\lambda \otimes 1_\mu$ may be equal to $1$ depending on the root of unity $\e$.
\end{Rem}

\begin{Rem} 
We will take this $\CR^*$ as an $R$-matrix of $\hU^*_\epsilon(\g, P^*)$ henceforth.
\end{Rem}


\section{The Frobenius center $Z_{Fr}$}\label{sec: Frobenius kernel}

In this section, we study the so-called Frobenius center of $U^{ev}_\e(\g)$ . 
We impose the following condition on $\ell$: 
$\ell_i \geq \textnormal{max}\{2, 1-a_{ij}\}_{1\leq j \leq r}$. Under this condition, the results of Section \ref{Quantum Frobenius}
hold, and furthermore in Remarks \ref{partial KerFr} and \ref{KerFr}, the index $i$ runs over all $\{1, \dots, r\}$. We assume that 
$R$  has an element $\e^{1/\sN}$. Let 
  $\CA_\sN:=\BZ[v^{\pm 1/\sN}]\Big[\Big\{ \tfrac{1}{v^{2k}-1} \Big\}_{1\leq k\leq \max\{\sd_i\}} \Big]$. 
Then with the choice of $\e^{1/\sN}$, the ring homomorphism $\sigma\colon \CA \rightarrow R$ factors through $\CA_\sN \rightarrow R$ 
which maps $v^{\pm 1/\sN}$ to $\e^{\pm 1/\sN}$.


\subsection{The pairing for the idempotented version}\label{pairing with idemp version}
\

Let us consider the  $R$-linear pairing $\<\cdot, \cdot\>'\colon U^{ev}_\e(\g) \x \hU_\e(\g, P) \rightarrow R$ defined by
\begin{equation}\label{idem pair} 
  \Big\< (yK^{\kappa(\nu)}) K^{\lambda}( xK^{\gamma(\mu)}), \dy 1_{\dlambda}\dx\Big\>' = 
  \delta_{-\lambda/2, \dlambda}(y,\dx)'(\dy,x)'\e^{(2\rho, \nu)} \e^{(\dmu, \gamma(\dmu))-(\dlambda, \kappa(\dnu)+\gamma(\dmu))} 
\end{equation}
for any $\lambda\in 2P, \dlambda\in P, y\in U^{ev<}_{-\nu}, x\in U^{ev>}_{\mu}, \dy\in \dU^{<}_{-\dnu}, \dx \in \dU^{>}_{\dmu}$ 
and $\mu,\nu,\dmu,\dnu\in Q_+$. 

\begin{Lem}The pairing \eqref{idem pair} is adjoint $\dU_\e(\g)$-invariant and non-degenerate in the first argument.
\end{Lem}

\begin{proof}
The pairing \eqref{idem pair} can be clearly extended to a pairing $U^{ev}_\e(\g) \x \tU_\e(\g, P) \rightarrow R$:
\[ 
  \Big\< (yK^{\kappa(\nu)}) K^{\lambda}( xK^{\gamma(\mu)}), \dy\prod a_{\dlambda}1_{\dlambda}\dx \Big\>' = 
  a_{-\lambda/2}(y,\dx)'(\dy,x)'\epsilon^{(2\rho, \nu)}\epsilon^{(\dmu, \gamma(\dmu))-(\dlambda, \kappa(\dnu)+\gamma(\dmu))} \,.
\]
This pairing is continuous, whereas $R$ and $U^{ev}_\e(\g)$ are equipped with the discrete topology and topology of $\tU_\e(\g, P)$ is defined 
in Section \ref{ssec: complete form}. 
The left adjoint $\dU_\e(\g)$-action on $\tU_\e(\g, P)$ is continuous. Furthermore, the pairing $U^{ev}_\e(\g) \x \dU_\e(\g)\rightarrow R$ 
of Proposition~\ref{Prop:pairing_whole_R} factors through the above pairing $U^{ev}_\e(\g) \x \tU_\e(\g,P) \rightarrow R$ via the homomorphism 
$\psi_1\colon \dU_\e(\g)\rightarrow \tU_\e(\g, P)$. Since $U^{ev}_\e(\g) \x \dU_\e(\g)\rightarrow R$ is non-degenerate in the first argument, 
so is $U^{ev}_\e(\g)\x \tU_\e(\g,P) \rightarrow R$. Moreover, since the pairing $U^{ev}_\e(\g) \x \dU_\e(\g)\rightarrow R$ is adjoint $\dU_\e(\g)$-invariant 
and the map $\psi_1\colon \dU_\e(\g) \rightarrow \tU_\e(\g,P)$ is a $\dU_\e(\g)$-module homomorphism with a dense image, the pairing 
$U^{ev}_\e(\g) \x \tU_\e(\g, P)\rightarrow R$ must be adjoint $\dU_\e(\g)$-invariant.

Since $\psi_2\colon \hU_\e(\g, P) \rightarrow \tU_\e(\g, P)$ is a $\dU_\e(\g)$-module embedding with a dense image, the pairing \eqref{idem pair}
 must be adjoint  $\dU_\e(\g)$-invariant and non-degenerate in the first argument.
\end{proof}

\subsection{The Frobenius center $Z_{Fr}$.}\label{ssec:e-center}
\

Recall the  following homomorphisms from Section~\ref{Quantum Frobenius}:
\begin{equation*}
  \tFr\colon \hU_\e(\g, P)\rightarrow \hU^*_\e(\g, P^*) \,,\quad 
  \tFr^>\colon \dU^>_\e\rightarrow \dU^{*>}_\e \,,\quad 
  \tFr^<\colon \dU^<_\e \rightarrow \dU^{*<}_\e \,.
\end{equation*}
By Corollary \ref{cor: kernel of Fr}, the kernel of the homomorphism $\tFr$ is spanned over $R$ by  the following elements: 
\begin{itemize} 

\item 
$\dy1_{\dlambda}\dx$, with $\dy\in \dU^<_\e, \dx \in \dU^>_\e$ and $\dlambda\in P\setminus P^*$;

\item 
$\dy1_{\dlambda}\dx$, with $\dlambda\in P^*$ and either $\dx\in \Ker (\tFr^>) $ or $\dy \in \Ker(\tFr^<)$. 
  
\end{itemize}

\begin{defi} 
The \emph{Frobenius center}, denoted by $Z_{Fr}$, is the orthogonal complement in $U^{ev}_\e$ of $\Ker(\tFr)$ under the adjoint $\dU_\e(\g)$-invariant pairing~(\ref{idem pair}).

\end{defi}

Let $Z_{Fr}^>:=Z_{Fr} \cap U^{ev\, >}_\e, Z_{Fr}^0:=Z_{Fr}\cap U^{ev\, 0}_\e$, $Z^<_{Fr}:=Z_{Fr}\cap U^{ev\, <}_\e$.

\begin{Lem}\label{basis properties of Ze}
(a) $Z_{Fr}$ is central in $U^{ev}_\e$. Moreover, $\ad'(x)(z)=0$ for any $z\in Z_{Fr}$ and 
$x\in \{\tE_i^{(n)}, \tF_i^{(n)}, K^\lambda-1 \,|\, \ell_i \nmid n, \lambda \in 2P\}$. 

\medskip
\noindent
(b) $Z^0_{Fr}=\oplus_{\lambda \in 2P^*} R\cdot K^\lambda$, $Z^>_{Fr}=R[\tE_{\a}^{\ell_\a}]_{\a\in \Delta_+}$, 
and $Z^<_{Fr}=R[\tF_\a^{\ell_\a}]_{\a\in \Delta_+}$. In particular, $Z^<_{Fr}, Z^0_{Fr}, Z^>_{Fr}$ are $R$-subalgebras of $U^{ev}_\e$.

\medskip
\noindent
(c) $Z_{Fr}=Z_{Fr}^<\otimes_R Z_{Fr}^0\otimes_R Z^>_{Fr}$.
\end{Lem}

\begin{proof}
(a) Let $z\in Z_{Fr}$. By Lemma \ref{lem: ad(x)(y) in Ker(Fr)}, we have $\ad'(S'(x))(y)\in \Ker(\tFr)$ for any $y \in \hU_\e(\g, P)$ and 
$x\in  \{\tE_i^{(n)}, \tF_i^{(n)}, K^\lambda-1 \,|\, \ell_i \nmid n, \lambda \in 2P\}$. Therefore, with these $x$ and $y$, we get:
\[ 
  \< \ad'(x) z,y\>'=\<z, \ad'(S'(x))y\>'=0 \,.
\]
Since the pairing \eqref{idem pair} is non-degenerate in the first argument, it follows that $\ad'(x)(z)=0$ for all 
$x\in \{\tE_i^{(n)}, \tF_i^{(n)}, K^\lambda-1 \,|\, \ell_i \nmid n, \lambda \in 2P\}$. 
In particular, we can take $n=1$. 
This implies that $z$ is central in $U^{ev}_\e$, due to \eqref{eq:explicit adjoint action}.

(b) From equation \eqref{idem pair}, it is easy to see that $Z^0_{Fr}=\bigoplus_{\lambda \in 2P^*} R \cdot K^{\lambda}$. 
For $x \in U^{ev\, >}_\e$, we have 
\[ 
  \< x, \dy1_{\dlambda}\dx\>' = \delta_{\gamma(\deg(x))/2, \dlambda}
  (\dy, x)'(1,\dx)' \e^{(\dmu, \gamma(\dmu))-(\dlambda, \kappa(\dnu)+\gamma(\dmu))}
  \e^{(\gamma(\deg(x)),\deg(x))}  \,,
\] 
which implies that 
\[Z^>_{Fr}=\big\{ x\in U^{ev\, >}_\e \,\big|\, (\dy,x)'=0 \;\text{for all $\dy \in \Ker(\tFr^<)$} \big\}.\]  
By Corollary \ref{cor: pairing of PBW} and Lemma \ref{lem: PBW basis of Ker of Fr}, $Z^>_{Fr}$ is a free $R$-module with a basis 
$\{\tE^{\cev{k}} \,|\, \vec{k}\in \BZ^N_{\geq 0}\setminus \mathscr{K}\}$, see Lemma \ref{lem: PBW basis of Ker of Fr} for the 
definition of $\mathscr{K}$. This implies that $Z^>_{Fr}=R[\tE_\a^{\ell_\a}]_{\a\in \Delta_+}$. The proof for $Z^<_{Fr}$ is the same.


(c) Due to Lemma \ref{lem:twisted-triangular-PBW}, any $u\in U^{ev}_\e$ can be written as 
\[u=\sum_{\vec{k}, \vec{r}\in \BZ^N_{\geq 0}}^{\lambda \in P} a_{\vec{k}, \vec{r}, \lambda} \tF^{\cev{k}}K^{2\lambda} \tE^{\cev{r}}\]
for finitely many nonzero $a_{\vec{k}, \vec{r}, \lambda}\in R$. By Corollary \ref{cor: pairing of PBW}, we have 
\[\<u, \tF^{(\cev{r})}1_{-\lambda} \tE^{(\cev{k})}\>'=a_{\vec{k}, \vec{r}, \lambda} b_{\vec{k}, \vec{r}, \lambda}\] 
for some invertible element $b_{\vec{k}, \vec{r}, \lambda}\in R$. Therefore, if $u \in Z_{Fr}$ then we must have 
$a_{\vec{k}, \vec{r}, \lambda}=0$ if either $\lambda \in P\setminus P^*$, or $\vec{k}\in \mathscr{K}$, or $\vec{r}\in \mathscr{K}$. 
Thus, if $u \in Z_{Fr}$ then 
  \[u=\sum_{\vec{k}, \vec{r}\in \BZ^N_{\geq 0}\setminus \mathscr{K}}^{\lambda \in P^*} 
   a_{\vec{k}, \vec{r}, \lambda} \tF^{\cev{k}}K^{2\lambda}\tE^{\cev{r}},\]so that  
$u \in Z^<_{Fr}\otimes_R Z^0_{Fr} \otimes_R Z^>_{Fr}$. 
\end{proof}

\begin{Cor}\label{U^ev is free over Z_e} 
$U^{ev}_\epsilon$ is a free module over $Z_{Fr}$ of rank 
$d=\underset{\a\in \Delta_+}{\prod} \ell^2_\a \cdot \underset{1\leq i\leq r}{\prod} \ell_i$ with $\ell_\a$  in \eqref{eq: l_a}.
\end{Cor}

\begin{Rem}\label{actions of U(gd)}  Since $\Ker(\tFr)$ is a two-side ideal in $\hU_\e(\g, P)$ hence closed under the adjoint action of $\dU_\e(\g)$, the subalgebra $Z_{Fr}$ is also closed under the adjoint action of $\dU_\e(\g)$. The Hopf algebra $\dU_\e(\g)$ also naturally acts on $\hU^*_\e(\g, P^*)$ as in Remark \ref{rem: dU acts on U*(g,P*)}. Furthermore, by the description of $\Ker(\cFr)$ in Lemma \ref{KerFr} and computations in Lemma \ref{lem: ad(x)(y) in Ker(Fr)} and Lemma \ref{basis properties of Ze}(a), it follows that both $\dU_\e(\g)$-actions on $Z_{Fr}$ and $\hU^*_\e(\g, P^*)$ factor through the morphism $\cFr: \dU_\e(\g) \rightarrow \dU_R(\g^d)$.
\end{Rem}
By definition of $Z_{Fr}$, we have an induced pairing 
\begin{equation}\label{center pair}
  \<\cdot, \cdot\>'\colon Z_{Fr}\x \hU^*_\e(\g, P^*) \longrightarrow R \,,
\end{equation}
which is adjoint $\dU_R(\g^d)$-invariant, see Remark \ref{actions of U(gd)}, and is non-degenerate in the first argument.

Let $G^d$ be the simply connected algebraic group with the Lie algebra $\g^d$. Note that $\g^d$ has the root lattice $Q^*$ and 
the weight lattice $P^*$. Let $G^d_0$ be the open Bruhat cell in $G^d$, that is $G^d_0=U^d_- \x T^d \x U^d_+$, where $T^d$ is 
the maximal torus and $U^d_\pm$ are opposite maximal unipotent subgroups (so that $U^d_+$ corresponds to $\Delta_+$). Everything is defined over $\BZ$, hence makes sense over $R$ by base change.  We are now 
going to identify $Z_{Fr}$ with $R[G_0^d]$, see Proposition \ref{Ze and Bruhat cell}.\footnote{Let us comment on the analog of $Z_{Fr}$ for the usual De Concini-Kac forms: when $q$ is an odd order root of unity, the analog is the central subalgebra $Z_0$ studied in \cite[$\mathsection 5.4, 6$]{dckp}; meanwhile, when $q$ is an even order root of unity, the analog is studied in \cite[$\mathsection 4$]{t2}.} To do so, we start with some technical results.

The perfect  pairings (the "perfect" follows by PBW bases in Lemma \ref{lem:twisted-triangular-PBW} and Corollary \ref{cor: pairing of PBW})
\[ 
  (\cdot, \cdot)'\colon U^{ev\, <}_\e \x \dU^>_\e \longrightarrow R \,, \qquad 
  (\cdot,\cdot)'\colon \dU^<_\e \x U^{ev\, >}_\e \longrightarrow R
\]
give us the pairings
\begin{equation}\label{<> center pair} 
  (\cdot, \cdot)'\colon Z^<_{Fr} \x \dU^>_R(\g^d) \longrightarrow R \,, \qquad 
  (\cdot, \cdot)'\colon \dU^<_R(\g^d) \x Z^>_{Fr} \longrightarrow R \,,
\end{equation}
through restricting to a subalgebra $Z^<_{Fr}\subset U^{ev\, <}_\e$ (resp.\ $Z^>_{Fr}\subset U^{ev\, >}_\e$) and passing through the 
quotient map $\cFr^>: \dU^>_\e\rightarrow  \dU^>_R(\g^d)$ (resp.\  $ \cFr^<: \dU^<_\e\to \dU^<_R(\g^d)$), see Lemma \ref{Fr in Lus form}. 

\begin{Lem}\label{lem: perfect pairings} Two pairings in \eqref{<> center pair} are perfect. In particular,
\begin{equation}\label{eq: perfect pairings}
\begin{split}
    (\tF^{\cev{k}}, e^{(\cev{r}^*)})'&=\begin{cases} \in R^\x &\text{if $\vec{k}=\vec{r}\in \BZ^N_{\geq 0}/\mathscr{K}$}\\
    0  &\text{if $\vec{k},  \vec{r}\in \BZ^N_{\geq 0} /\mathscr{K}$ and $\vec{k}\neq \vec{r}$},
    \end{cases} \\
    (f^{(\cev{r}^*)}, \tE^{\cev{k}})'&=\begin{cases}
        \in R^\x &\text{ if $\vec{k}=\vec{r}\in \BZ^N_{\geq 0} /\mathscr{K}$}\\
        0  &\text{if $\vec{k}, \vec{r}\in \BZ^N_{\geq 0}/\mathscr{K}$ and $\vec{k}\neq \vec{r}$}, 
    \end{cases}
\end{split}
\end{equation}
here $\vec{r}^*=(r_1/\ell_{\b_1}, \dots, r_N/\ell_{\b_N})$ for $\vec{r}\in \mathscr{K}$. The subset $\mathscr{K}$ is defined in Lemma \ref{lem: PBW basis of Ker of Fr}.
\end{Lem}
\begin{proof} By Lemma \ref{basis properties of Ze}, $Z^<_{Fr}$ has a $R$-basis $\{ \tF^{\cev{k}}| \vec{k} \in \BZ^N_{\geq 0}\backslash \mathscr{K}\}.$
By Lemma \ref{lem: PBW basis of Ker of Fr}, the elements 
$\{\cFr^<(\tE^{(\cev{k})})| \vec{k} \in \BZ^N_{\geq 0} \backslash \mathscr{K}\}$
form a $R$-basis of $\dU_R^>(\g^d)$. Moreover, by Lemma \ref{lem: computations of cFr}
\[\cFr^<(\tE^{(\cev{k})})= a_{\vec{k}}e^{(\cev{k}^*)} \qquad \text{for some $a_{\vec{k}}\in R^\x$.}\]
 Combining these results  with computations in  Corollary \ref{cor: pairing of PBW}, we get the first computation in \eqref{eq: perfect pairings}, hence the first pairing in \eqref{<> center pair} is perfect. The proof for the second pairing is the same.
\end{proof}


For any $\lambda \in P^*$, we have the corresponding homomorphism $\chi_{\lambda}\colon \dU^0_R(\g^d) \rightarrow R$. 
In other words, we have a pairing
\[
  R[T^d]\x \dU^0_R(\g^d) \longrightarrow R \,.
\]
This is just the Hopf algebra pairing between the ring of regular functions $R[T^d]$ and the distribution algebra 
$\text{Dist}_R(T^d) \simeq \dU^0_R(\g^d)$. 
In particular, we have
\[ 
  \chi_{\lambda_1+\lambda_2}(u_0)=\chi_{\lambda_1}(u_{0(1)})\chi_{\lambda_2}(u_{0(2)}) 
  \qquad \forall\, u_0 \in \dU^0_R(\g^d) \,.
\]
\begin{Lem}\label{lem center pair}
The pairing (\ref{center pair}) gives rise to a $\dU_R(\g^d)$-adjoint invariant pairing
\[ 
  \< \cdot, \cdot\>'\colon Z_{Fr} \x \dU_R(\g^d) \longrightarrow R \,,
\] 
which is non-degenerate in the first argument. The pairing is given by 
\begin{equation}\label{eq: center pair 3}
  \Big\<(yK^{\kappa(\nu)})K^\lambda (xK^{\gamma(\mu)}), \dy u_0\dx\Big\>'=\chi_{-\lambda/2}(u_0)(y,\dx)'(\dy,x)' \,,
\end{equation}
where  $\lambda \in 2P^*, u_0\in \dU^0_R(\g^d), y \in Z^<_{Fr, -\nu}, x\in Z^>_{Fr, \mu}, \dy \in \dU^<_R(\g^d)_{-\dnu}, \dx \in \dU^>_R(\g^d)_{\dmu}$
with $\nu, \mu, \dnu, \dmu\in Q^*_+$. The pairings $(y, \dx)', (\dy, x)'$ come from \eqref{<> center pair}.
\end{Lem}
The construction is given in the proof.

\begin{proof} 
We have the following inclusions of $\dU_R(\g^d)$-modules:
\begin{equation}\label{eq: morphisms}\dU_R(\g^d) \overset{\phi^{-1}}{\cong}A_{Q^*}\xhookrightarrow[]{\psi^*_1} \tU^*_\e(\g, P^*) \hookleftarrow \hU^*_\e(\g, P^*).
\end{equation}
The pairing $Z_{Fr} \x \hU^*_\e(\g, P^*) \rightarrow R$ naturally extends to a pairing $Z_{\Fr} \x \tU^*_\e(\g, P^*)\rightarrow R$. Since $\hU^*_\e(\g, P^*)$ is a dense  $\dU_R(\g^d)$-submodule of $\tU^*_\e(\g, P^*)$, the pairing $Z_{Fr} \x \tU^*_\e(\g, P^*) \rightarrow R$ is adjoint $\dU_R(\g^d)$-invariant and non-degenerate in the first argument. Since the map $\psi^*_1: A_{Q^*} \hookrightarrow \tU^*_\e(\g, P^*)$  is an injective morphism of $\dU_R(\g^d)$-modules with a dense image by Lemma \ref{Lem about A_Q and tU^*}, the pairing $Z_{\Fr} \x A_{Q^*} \rightarrow R$ obtained by restriction is adjoint $\dU_R(\g^d)$-invariant and non-degenerate in the first argument. The pairing $Z_{Fr} \x A_{Q^*} \rightarrow R$ is explicitly given by (cf.~\eqref{idem pair}):
\begin{equation*}
  \Big\<(yK^{\kappa(\nu)})K^\lambda (xK^{\gamma(\mu)}), (\dy K^{\kappa(\dnu)})u_0(\dx K^{\gamma(\dmu)})\Big\>' = 
  (y,\dx)'(\dy,x)' \hat{\chi}_{-\lambda/2}(u_0)\epsilon^{(2\rho, \nu)}
\end{equation*}
with $\hat{\chi}_{-\lambda/2}$ of~\eqref{eq: defi of chi-lambda}, for any $\nu, \mu, \dnu, \dmu \in Q^*_+, \lambda \in 2P^*$, 
$u_0 \in A^0_{Q^*}$, $y \in Z^<_{Fr, -\nu}, x\in Z^>_{Fr, \mu}, \dy \in \dU^{*<}_\e(\g)_{-\dnu}, \dx \in \dU^{*>}_\e(\g)_{\dmu}$. Note that the coefficient $\e^{(\dmu, \gamma(\dmu))-(\dlambda, \kappa(\dnu)+\gamma(\dmu))}$ in \eqref{idem pair} does not show up because we go back to the non-idempotent version, compare to \eqref{pairing of DCK and Lus rewritten}. 
Moreover, $\epsilon^{(2\rho, \nu)}=1$ for $\nu \in Q^*_+$, hence
\begin{equation}\label{center pair 2} 
  \Big\<(yK^{\kappa(\nu)})K^\lambda (xK^{\gamma(\mu)}), (\dy K^{\kappa(\dnu)})u_0(\dx K^{\gamma(\dmu)})\Big\>' = 
  (y,\dx)'(\dy,x)' \hat{\chi}_{-\lambda/2}(u_0),
\end{equation}
with indices as above.

 Now using the isomorphism of $\dU_R(\g^d)$-modules $A_{Q^*} \xrightarrow[]{\phi} \dU_R(\g^d)$ in Lemma  \ref{Lem about A_Q and tU^*},
we obtain the desired $\dU_R(\g^d)$-invariant pairing $Z_{Fr}\x \dU_R(\g^d)\rightarrow R$. The equation \eqref{eq: center pair 3} follows from \eqref{center pair 2} by the following observations:

$\bullet$ $\hat{\chi}_{-\lambda/2}: \dU_R^0(\g^d)\cong A^0_{Q^*} \rightarrow R$ coincides with $\chi_{-\lambda/2}: \dU_R^0(\g^d)\rightarrow R$.

$\bullet$ The pairing $(y, \dx)': Z^<_{Fr}\x \dU^{*>}_\e \rightarrow R$ in \eqref{center pair 2} is obtained from the pairing $Z^<_{Fr}\x \dU^>_\e \rightarrow R$ by passing through the quotient $\tFr^>: \dU^>_\e \rightarrow \dU^{*>}_\e$. On the other hand, the morphism $\cFr^>: \dU^>_\e \rightarrow \dU^>_R(\g^d)$ is equal to the composition $ \dU^>_\e\xrightarrow[]{\tFr^>} \dU^{*>}_\e \xrightarrow[]{\phi} \dU^>_R(\g^d)$.
\end{proof}

\begin{Lem}\label{Z< and U>} 
(a) For any $x_1,x_2\in Z^>_{Fr}$ and $\dy \in \dU^<_R(\g^d)$, we have
\begin{equation}\label{eq19}
  (\dy, x_1x_2)'=(\dy_{(1)}, x_1)'(\dy_{(2)}, x_2)' \,.
\end{equation}

\noindent
(b) For any $y_1, y_2 \in Z^<_{Fr}$ and $\dx \in \dU^>_R(\g^d)$, we have
\begin{equation}\label{eq18}
  (y_1y_2, \dx)'=(y_1, \dx_{(1)})'(y_2, \dx_{(2)})' \,.
\end{equation}
\end{Lem}

\begin{proof}
Let us recall the $R$-Hopf subalgebras $A^>$ and $A^<$ from Remark~\ref{rem:A< and A>}. The homomorphism 
of Hopf algebras $\cFr\colon \dU_\e(\g)\rightarrow \dU_R(\g^d)$ restricts to  homomorphisms of Hopf algebras:
\[ 
  \cFr^{A,>}\colon A^> \longrightarrow \dU^>_R(\g^d) \,, \qquad \cFr^{A,<}\colon A^< \longrightarrow \dU^<_R(\g^d) \,.
\]

\noindent
(a) The Hopf pairing $(\cdot,\cdot)'\colon \dU^{\leqslant}_\e \x U^{ev\, \geqslant}_\e \rightarrow R$ restricts to the pairing 
$(\cdot,\cdot)'\colon A^< \x Z^>_{Fr} \rightarrow R$ which satisfies
\begin{equation}\label{eq: A< x Z>} (\dy, x_1x_2)'=(\dy_{(2)}, x_1)'(\dy_{(1)}, x_2)'
\end{equation}
for $\dy \in A^<$ and $x_1, x_2\in Z^>_{Fr}$.
The kernel of the map $\cFr^{A,<}$ is spanned over $R$ by $\ker(\cFr^<)$ and $\dU^<_\e(K^\lambda-1)$ with $\lambda \in 2P$. Moreover, $Z^>_{Fr}$ is orthogonal to $\dU^<_\e(K^\lambda-1)$, indeed, for any $x\in Z^>_{Fr, \mu}$ for any $\mu \in Q^*_+$ and $\dy \in \dU^<_\e$, we have 
\[ 
  (\dy(K^\lambda-1), x)'=(\dy, x)'(\e^{(\lambda, \mu)}-1)=0 \,,
\]
since $\mu \in Q^*_+$ and $\lambda \in 2P$. Therefore, $Z^>_{Fr}$ is orthogonal to $\Ker(\cFr^{A,<})$. Hence the pairing $A^<\x Z^>_{Fr} \rightarrow R$ gives a rise to a pairing $\dU^<_R(\g^d) \x Z^>_{Fr}\rightarrow R$ by descending to the quotient, which coincides with the pairing $\dU^<_R(\g^d) \x Z^>_{Fr} \rightarrow R$ in \eqref{<> center pair}. By \eqref{eq: A< x Z>}, it follows that the pairing $\dU^<_R(\g^d) \x Z^>_{Fr} \rightarrow R$ satisfies:
\begin{equation*} (\dy, x_1x_2)'=(\dy_{(2)}, x_1)'(\dy_{(1)}, x_2)'
\end{equation*}
for $\dy\in \dU^<_R(\g^d)$ and $x_1, x_2 \in Z^>_{Fr}$. But $\dU^<_R(\g^d)$ is cocommutative, hence part $(a)$ follows.


\noindent
(b) The proof is similar.
\end{proof}

Let us consider the embedding $Z_{Fr} \hookrightarrow \Hom_R(\dU_R(\g^d), R)$ arising from Lemma \ref{lem center pair}. Note that the algebra structure on $\Hom_R(\dU_R(\g^d), R)$ comes from the coalgebra structure on $\dU_R(\g^d)$.

\begin{Lem}\label{lem1}
The above embedding $Z_{Fr} \hookrightarrow \Hom_R(\dU_R(\g^d), R)$ is an $R$-algebra homomorphism.
\end{Lem}

\begin{proof} 
 Let $a_i=(y_iK^{\kappa(\nu_i)})K^{\lambda_i}(x_iK^{\gamma(\mu_i)})$ with  
$y_i \in Z^<_{Fr, -\nu_i}, x_i\in Z^>_{Fr, \mu_i}$, $\lambda_i\in 2P^*$, $\nu_i, \mu_i \in Q^*_+$ for $i=1,2$. 
Then, for any $u=\dy u_0\dx\in \dU_R(\g^d)$, we need to verify the following equality: 
\[ 
  \< a_1a_2, u\>'= \<a_1, u_{(1)}\>'\<a_2, u_{(2)}\>' \,.
\]
Indeed, using the Sweedler's notation~\eqref{Sweedler}, we have:
\begin{align*}
  \<a_1, u_{(1)}\>' \<a_2, u_{(2)}\>' &= 
  \chi_{-\lambda_1/2}(u_{0(1)})(y_1, \dx_{(1)})'(\dy_{(1)}, x_1)' \,\cdot\, \chi_{-\lambda_2/2}(u_{0(2)})(y_2, \dx_{(2)})'(\dy_{(2)}, x_2)'\\
  &= \chi_{-\lambda_1/2}(u_{0(1)})\chi_{-\lambda_2/2}(u_{0(2)}) \,\cdot\, (y_1, \dx_{(1)})'(y_2, \dx_{(2)})' \,\cdot\, 
     (\dy_{(1)}, x_1)'(\dy_{(2)}, x_2)'\\
  &= \chi_{(-\lambda_1-\lambda_2)/2}(u_0)(y_1y_2, \dx)'(\dy, x_1x_2)' = \<a_1a_2, u\>' \,,
\end{align*}
where we used that $K^{\gamma(\mu_i)},K^{\kappa(\nu_i)}$ are central, due to Remark~\ref{rem:lattice-star}.
This completes the proof.
\end{proof}

We are now ready to identify $Z_{Fr}$ with $R[G^d_0]$. There are direct sum decompositions into weight components: 
$Z^>_{Fr}=\bigoplus_{\mu \in Q^*_+} Z^>_{Fr, \mu}$ and $Z^<_{Fr}=\bigoplus_{\nu \in Q^*_+} Z^<_{Fr, -\nu}$. Let 
\[ 
  \tZ^<_{Fr} =\bigoplus_{\nu\in Q^*_+} \tZ^<_{Fr, -\nu}\,, \quad 
  \tZ^>_{Fr}=\bigoplus_{\mu\in Q^*_+} \tZ^>_{Fr, \mu} \quad \mathrm{with} \quad
  \tZ^<_{Fr, -\nu}=Z^<_{Fr, -\nu}K^{\kappa(\nu)}\,, \quad 
  \tZ^>_{Fr, \mu}=Z^>_{Fr, \mu} K^{\gamma(\mu)}\,.
\]
Then we still have a decomposition $Z_{Fr}=\tZ^<_{Fr}\otimes_R Z^0_{Fr}\otimes_R \tZ^>_{Fr}$, cf.~Lemma \ref{basis properties of Ze}(c).

\begin{Prop}\label{Ze and Bruhat cell} 
There is a $\dU_R(\g^d)$-linear algebra isomorphism 
  $$ \varphi\colon Z_{Fr} \,\iso\, R[G^d_0] \simeq R[U^d_-]\otimes_R R[T^d]\otimes_R R[U^d_+] \,. $$ 
Furthermore, under this isomorphism, we have: $\tZ_{Fr}^< \simeq R[U^d_+]$, $\tZ_{Fr}^> \simeq R[U^d_-]$, $Z^0_{Fr} \simeq R[T^d]$. 
\end{Prop}

\begin{proof} 
We have two adjoint $\dU_R(\g^d)$-invariant pairings which are non-degenerate in the first arguments:  
\begin{equation*}
  Z_{Fr} \x \dU_R(\g^d) \longrightarrow R \qquad \mathrm{and} \qquad R[G^d_0] \x \dU_R(\g^d) \longrightarrow R \,.
\end{equation*}
They give us $R$-algebra embeddings 
\[ 
  Z_{Fr} \hookrightarrow \text{Hom}_R(\dU_R(\g^d), R) \hookleftarrow R[G^d_0] \,,
\]
which intertwine the $\dU_R(\g^d)$-actions. So it is enough to show that the images of $\tZ^<_{Fr}, \tZ^>_{Fr}, Z^0_{Fr}$ 
coincide with the images of $R[U^d_+], R[U^d_-], R[T^d]$ as $R$-submodules, respectively.

Due to Lemma~\ref{lem center pair}, we see that  $K^\lambda$ is identified with 
$\chi_{-\lambda/2} \in R[T^d]$, hence $Z^0_{Fr}$ is identified with $R[T^d]$ viewed as $R$-subalgebras of $ \text{Hom}_R(\dU_R(\g^d), R)$.

Let $\dU^0_R(\g^d)_+, \dU^>_R(\g^d)_+$ be the kernels of the counit maps restricted to $\dU^0_R(\g^d), \dU^>_R(\g^d)$, respectively. 
For any $z\in \tZ^>_{Fr, \mu}$, we have
\[ 
  \< z, \dy u_0 \dx\>'=0 \,,
\]
if either $u_0 \in \dU^0_R(\g^d)_+$ or $\dx\in \dU^>_R(\g^d)_+$. Moreover, the pairing 
$\tZ^>_{Fr, \mu} \x \dU^<_R(\g^d)_{ -\mu} \rightarrow R$ is a perfect pairing by Lemma \ref{lem: perfect pairings}
. Therefore, $\tZ^>_{Fr, \mu} $ must be identified with $R[U^d_-]_{-\mu}$, and so $\tZ^>_{Fr} \simeq R[U^d_-]$.

The proof of the coincidence of the images $\tZ^<_{Fr} \simeq R[U^d_+]$ is similar.
\end{proof}

\subsection{The locally finite part of $Z_{Fr}$ when $R=\BF$ algebraically closed field}\

 Let us define
\begin{equation}\label{eq: def of Zfin}
Z^{fin}_{Fr}:=\{ z\in Z_{Fr}| \dim_\BF(\dU_\BF(\g^d)z)< \infty\}
\end{equation}
This is an  subalgebra of $Z_{Fr}$  since $Z_{Fr}$ is a $\dU_{\BF}(\g^d)$-module algebra.

\begin{Lem}\label{lem: Zfin in field F}
(a) $Z^{fin}_{Fr} \cong \BF[G^d]$ under the isomorphism $\varphi$ in Proposition \ref{Ze and Bruhat cell}.

\noindent
(b) Let $\lambda_0=\sum_i \w^*_i=\sum_i \ell_i \w_i$, where $\w_i$ are the fundamental weights and $\ell_i$ are as in Section \ref{Quantum Frobenius}. Then $Z^{fin}_{Fr}[K^{2\lambda_0}]=Z_{Fr}$, where the left hand side is the localization of $Z^{fin}_{Fr}$ by $K^{2\lambda_0}$.

\noindent
(c) Suppose $\BF$ has characteristic $0$ then $Z^{fin}_{Fr}=\bigoplus_{\lambda \in P^*_+} \dU_\BF(\g^d)K^{-2\lambda}$.
\end{Lem}
\begin{proof}
(a) We only need to show that the $\dU_\BF(\g^d)$-locally finite part of $\BF[G^d_0]$ is $\BF[G^d]$. First, all conjugacy classes of $G^d(\BF)$ intersect $G^d_0(\BF)$ because $G^d_0(\BF)$ contains a Borel subgroup $B^d(\BF)$ and any element of $G^d(\BF)$ can be conjugated to some element in $B^d(\BF)$.

Let $V^d(\lambda_0)$ be the Weyl representation of $\dU_\BF(\g^d)$. Let $v^d_{\lambda_0}$ be the highest weight vector of $V^d(\lambda_0)$ and $(v^d_{\lambda_0})^*$ be the dual weight basis of $v^d_{\lambda_0}$ in $(V^d(\lambda_0))^*$. By using the Bruhat cell decomposition $G^d(\BF)=\bigsqcup_{w\in W} U^d_-(\BF) \dot{w} T^d(\BF)U^d_+(\BF)$, one can show that the vanishing locus of  $f= c_{(v^d_{\lambda_0})^*, v^d_{\lambda_0}}$ is $G^d(\BF)\backslash G^d_0(\BF)$, the complement of the open Bruhat cell $G^d_0(\BF)$ in $G^d(\BF)$.

Let $V$ be a finite dimensional $\dU_\BF(\g^d)$-submodule of $\BF[G^d_0]$. Let $C$ be a subalgebra of $\BF[G^d_0]$ generated by $\BF[G^d]$ and $V$. Then $C$ is a $\dU_\BF(\g^d)$-locally finite and  finitely generated as an algebra. Let $X=\mathrm{Spec}(C)$. We have an inclusion of reduced $G^d(\BF)$-algebras $\BF[G^d] \hookrightarrow C$, inducing an isomorphism of localizations $\BF[G^d]_f \cong C_f \cong \BF[G^d_0]$. Hence a map of $G^d(\BF)$-varieties $\psi: X\rightarrow G^d(\BF)$ such that $X_f \rightarrow G^d_0[\BF]$ is an isomorphism, where $X_f =\mathrm{Spec}(C_f)$. Let $G^d X_f$ be the minimal $G^d$-subvariety of $X$ containing $X_f$. Since $G^d_0(\BF)$ intersects all $G^d(\BF)$-orbits, we have that $\psi: G^d X_f \rightarrow G^d$ is an isomorphism. On the other hand, since $C$ is an integral domain, it follows that $X$ is an irreducible variety and $G^d X_f$ is an open dense subset of $X$. Therefore, a composition of the inclusion map $\BF[G^d_0] \hookrightarrow C \rightarrow \BF[G^d X_f]$ is an isomorphism, so that $\BF[G^d] =C$. This implies that any finite dimensional $\dU_\BF(\g^d)$-submodule $V$ of $\BF[G^d_0]$ is contained to $\BF[G^d]$.

\noindent
(b)Using the equation \eqref{eq: center pair 3}, it is easy to show that $\varphi(K^{-2\lambda_0})=c_{(v^d_{\lambda_0})^*, v_{\lambda_0}}$. Then part (b) follows by part (a).

\noindent
(c) For any $\lambda \in P^*_+$, let $V^d(\lambda)$ be the Weyl module of $\dU_\BF(\g^d)$. Let $v^d_\lambda$ be the nonzero highest weight vector of $V^d(\lambda)$ and $(v^d_\lambda)^*$ be the dual weight vector. When $\BF$ has characteristics 0, $\BF[G^d]=\bigoplus \dU_\BF(\g^d) c_{(v^d_\lambda)^*, v^d_\lambda}$ by the Peter-Weyl theorem. Using the equation \eqref{eq: center pair 3} again, $\varphi(K^{-2\lambda})=c_{v^d_\lambda, v_\lambda}$ for $\lambda \in P^*_+$. Hence part (c) follows.
\end{proof}

\section{Rational representations of $\dmU_q(\g)$ and $\dU_q(\g)$}\label{sec rational reps}


Consider the ring $\uCA:=\BZ[v,v^{-1}]$ so that $\CA$ is a finite localization of $\uCA$. 
Let $R$ be an $\uCA$-algebra and $q$ denote the image of $v$ in $R$. Then we can form the Lusztig form 
$\dmU_q(\g)$ over $R$ as above. 

The goal of this section is to review several known results and constructions related 
to the rational representations of $\dmU_q(\g)$. We then carry  these results to rational representations of $\dU_q(\g)$. 


\begin{defi}\label{defi:rat_rep}
Let $M$ be a $\dmU_q(\g)$-module. We say that it is {\it rational of type 1} (in what follows, just rational) 
if the following conditions hold:
\begin{itemize}
    \item[(i)] $M$ is a  weight module, i.e., there is a decomposition $M=\bigoplus_{\lambda\in P}M_\lambda$, 
    where $K_i$ acts on $M_\lambda$ by $q^{(\lambda, \a_i)}$ and $\bmat{K_i; 0 \\ t}$ acts by $\bmat{(\lambda, \a^\vee_i)\\ t}_{q_i}$ for 
    all $i=1,\ldots,r$ and all $t>0$.
    \item[(ii)] For any $m\in M$ there is $k>0$ such that $E_i^{(s)}m=0$ for all $s>k$ and all $i=1,\ldots,r$.
    \item[(iii)] For any $m\in M$ there is $k>0$ such that $F_i^{(s)}m=0$ for all $s>k$ and all $i=1,\ldots,r$.
\end{itemize}
\end{defi}
\begin{defi} Let $\dmU_q(\g)\Lmod$ be the category of left $\dmU_q(\g)$-modules. Let $\Rep(\dmU_q(\g))$ be the category of rational representations of $\dmU_q(\g)$. The morphisms in $\Rep(\dmU_q(\g))$ are morphisms of left $\dmU_q(\g)$-modules which preserve the weight decompositions. 
\end{defi}


\begin{Prop}\label{prop: Rep(Uq(g)) is full subcat} The natural functor $\mathscr{I}: \Rep(\dmU_q(\g))\rightarrow \dmU_q(\g)\Lmod$ is fully faithful. Moreover, the image is closed under taking subquotients.
\end{Prop}
Before going to the proof, we need the following result. For any $\lambda \in P$, let $\chi_\lambda: \dmU^0_q(\g)\rightarrow R$ be the $R$-algebra homomorphism defined by 
\begin{equation}\label{eq: defi of character for usual form}
\chi_\lambda: \qquad \qquad K_i \mapsto q^{(\lambda, \a_i)},  \qquad \bmat{K_i;0\\t}_{q_i} \mapsto \bmat{(\lambda, \a^\vee_i)\\t}_{q_i}.
\end{equation}
For any $N\in \dmU_q(\g)\Lmod$, let $N_{(\lambda)}:=\{ n \in N|xn=\chi_\lambda(x)n \; \forall x\in \dmU^0_q(\g)\}$.
\begin{Lem}\label{lem: weight direct sum} The sum $\sum_{\lambda\in P} N_{(\lambda)} $ in $N$ is a direct sum $\bigoplus_\lambda N_{(\lambda)}$. As a result, for any $M\in\Rep(\dmU_q(\g))$, the weight space $M_\lambda$ is uniquely characterized by $M_\lambda=\{m \in M|xm=\chi_\lambda(x)m\; \forall x\in \dmU^0_q(\g)\}$.
\end{Lem}
\begin{proof}Suppose $\sum_{i=1}^m n_{\lambda_i}=0$ with $n_{\lambda_i}\in N_{(\lambda_i)}$. Then   $\chi_{\lambda_1}(x)n_{\lambda_1}+\dots +\chi_{\lambda_m}(x) n_{\lambda_m}=0$ for all $x\in \dmU^0_q(\g)$. Arguing as in Lemma \ref{lem: independent of characters}.b, we have $n_{\lambda_i}=0$ for all $i$.
\end{proof}

\begin{proof}[Proof of Proposition \ref{prop: Rep(Uq(g)) is full subcat}]Let $M_1, M_2\in \Rep(\dmU_q(\g))$. To show the fully faithfulness it is enough to show that  any morphism of $\dmU_q(\g)$-modules $M_1\rightarrow M_2$ preserves the weight decompositions.  This follows by Lemma \ref{lem: weight direct sum}. Let us prove the image is closed under taking subquotients. Let $M\in \Rep(\dmU_q(\g))$. Suppose we have a short exact sequence in $\dmU_q(\g)\Lmod$: $0\rightarrow  M_1 \xrightarrow[]{\iota} M \xrightarrow[]{\pi} M_2\rightarrow 0$.  We have $M_2=\sum_\lambda \pi(M_\lambda)$ and $\pi(M_\lambda) \subset M_{2,(\lambda)}$. By Lemma \ref{lem: weight direct sum}, it follows that $M_2=\bigoplus_\lambda \pi(M_\lambda)$, hence, $M_2\in \Rep(\dmU_q(\g))$ and then $M_1 \in \Rep(\dmU_q(\g))$.   
\end{proof}
\begin{Rem}\label{rem: max-rat-subrep} Let $N \in \dmU_q(\g)\Lmod$. By Proposition \ref{prop: Rep(Uq(g)) is full subcat} and Lemma \ref{lem: weight direct sum} , we can talk about the maximal rational subrepresentation $N^{ral}$ of $N$. Moreover, for any $M\in \Rep(\dmU_q(\g))$,
\[\Hom_{\dmU_q(\g)}(M, N^{ral})=\Hom_{\dmU_q(\g)}(M,N).\]
\end{Rem}
We will need the following result.

\begin{Prop}\label{prop:rational_finiteness}
Let $M\in \Rep(\dmU_q(\g))$ and  $m\in M$. Then $\dmU_q(\g)m$ is finitely generated over $R$
\end{Prop}

 \begin{proof}
Pick any reduced decomposition $w_0=s_{i_1}s_{i_2}\cdots s_{i_N}$ of the longest element $w_0$ of the Weyl group $W$.
The desired result follows immediately from the following lemma:
\end{proof}
\begin{Lem}
(a) Let $\dmU^>_i$ denote the $R$-subalgebra generated by $\{E_i^{(s)}\}_{s\in \BN}$. Then:
  $$\dmU^{>}=\dmU^{>}_{i_1}\dmU^{>}_{i_2}\cdots \dmU^{>}_{i_N}.$$

\noindent
(b) Let $\dmU^<_i$ denote the $R$-subalgebra generated by $\{F_i^{(s)}\}_{s\in \BN}$. Then:
  $$\dmU^{<}=\dmU^{<}_{i_1}\dmU^{<}_{i_2}\cdots \dmU^{<}_{i_N}.$$
\end{Lem}
\begin{proof}
For the case of $R=k[v,v^{-1}]$ (where $k$ is any field), this was proved in~\cite[\S5.8-5.11]{jl}. 

The case of general $R$ is handled below in several cases.

{\it Case 1:} $R$ is a field. This follows from the case of $R=k[v^{\pm 1}]$ by base change. 

{\it Case 2:} $R$ is a local ring with residue field $k$. Consider the situation of (a), (b) is similar. Pick $\mu\in Q^+$. Then the $\mu$-weight components in $\dmU^{>}$ and
its submodule $\dmU^{>}_{i_1}\dmU^{>}_{i_2}\cdots \dmU^{>}_{i_N}$ are finitely generated $R$-module.
We note that the following diagram is commutative:
\[
\begin{tikzcd} 
  (\dmU^{>}_{R,i_1}\dmU^{>}_{R,i_2}\cdots \dmU^{>}_{R,i_N})_\mu \arrow[r] \arrow[d]& \dmU^{>}_{R,\mu} \arrow[d]\\
  (\dmU^{>}_{k,i_1}\dmU^{>}_{k,i_2}\cdots \dmU^{>}_{k,i_N})_\mu \arrow[r]& \dmU^{>}_{k,\mu} 
\end{tikzcd} \,.
\]
The vertical arrows are base changes from $R$ to $k$. By Case 1, the bottom horizontal arrow is surjective. We apply the Nakayma lemma and deduce that the top horizontal arrow is surjective. 

{\it Case 3}: $R$ is general. By Case 2, the homomorphism $(\dmU^{>}_{R,i_1}\dmU^{>}_{R,i_2}\cdots \dmU^{>}_{R,i_N})_\mu\rightarrow \dmU^{>0}_{R,\mu}$ becomes surjective  after localizing at every maximal ideal. Hence it is surjective. 
\end{proof}
%


\subsection{Weyl and dual Weyl modules}\label{SS_Weyl}
\

Let $\Rep(\dmU^\geqslant_q)$ (resp. $\Rep(\dmU^\leqslant_q)$) denote the category of rational representations of $\dmU^\geqslant_q$( resp. $\dmU^\leqslant_q$), those satisfying the condition (i) and (ii) (resp. (iii)) of Definition \ref{defi:rat_rep}. Let $\Rep^{fd}(\dmU_q(\g))$, $ \Rep^{fd}(\dmU^\geqslant_q)$, $\Rep^{fd}(\dmU^{\leqslant}_q)$ denote the subcategories of the corresponding categories consisting of all objects which are finitely generated over $R$.

Following \cite[$\mathsection 1.1.14$]{APW1}, we define Joseph's induction functor $\fJ: \Rep^{fd}(\dmU^\geqslant_q)\rightarrow \Rep^{fd}(\dmU_q(\g))$ as follows: for any $N \in \Rep(\dmU^\geqslant_q)$, let $\fJ(N)$ be the maximal rational quotient of $\dmU_q(\g)\otimes_{\dmU^\geqslant_q} N$. The existence  of such maximal rational quotient is in \cite[Proposition 1.14]{APW1}. 
\begin{Rem} \label{rem: 1rem about APW} The algebras considered in \cite{APW1} are defined over the localization $\uCA_{\m_0}$ at the maximal ideal $\m_0=(v-1,p)$ for some odd prime $p\in \BZ$. However, at the moment, the results about Joseph's induction functor in \cite{APW1} up to Section $1.20$ hold for general $R$. We note that Lemma $1.13$ in \cite{APW1} follows by using Lusztig's braid group action in  arbitrary rational $\dmU_q(\g)$-modules constructed in \cite[$\mathsection 41.2$]{l-book}, so that we do not use the proof in \cite{APW1}  which does not apply to general ring $R$. The proof of Lemma $1.11$ in \cite{APW1} also holds with care. 
\end{Rem}
The following result is in \cite[Proposition 1.17]{APW1}
\begin{Prop}\label{prop: left adjoint to F}The functor $\fJ$ is left adjoint  to the restriction functor $\fF: \Rep^{fd}(\dmU_q(\g))\rightarrow \Rep^{fd}(\dmU^\geqslant_q)$.
\end{Prop}
Any $\lambda \in P$ gives a rise to the  homomorphism $\dmU^\geqslant_q\rightarrow R$ given by  the character $\chi_\lambda: \dmU^0_q(\g) \rightarrow R$ inflated to $\dmU_q^\geqslant(\g)$. We write $R_\lambda$ for the rank $1$ module over $\dmU^\geqslant_q$ with the action corresponding to this homomorphism. 
\begin{defi} The $\dmU_q(\g)$-module $\Delta_q(\lambda):=\dmU_q(\g)\otimes_{\dmU^\geqslant_q} R_\lambda$is called {\it the Verma module}. For $\lambda \in P_+$, the module $W_q(\lambda):=\fJ(R_\lambda)$ is called  {\it the Weyl module}.
\end{defi}
 Proposition \ref{prop: left adjoint to F}  yields natural isomorphism:
\begin{equation}\label{eq:Weyl_universal}
  \Hom_{\dmU_q(\g)}(W_q(\lambda),?)\xrightarrow{\sim} \Hom_{\dmU_q^{\geqslant}}(R_\lambda,?).  
\end{equation}
Follow \cite[Proposition 1.20]{APW1}, we have
\begin{Prop}\label{prop: Description of Weyl modules} $W_q(\lambda)$ is the quotient of $\Delta_q(\lambda)$ by the relations $F^{[k]}_i v_\lambda =0$ for all $k>(\lambda, \a^\vee_i)$ and all $i=1, \dots,r$. Here, $v_\lambda$ denotes the element in $R_\lambda$ corresponding to $1$.
\end{Prop}
Proposition \ref{prop: Description of Weyl modules} implies that $W_q(\lambda)=W_{\uCA}(\lambda)\otimes_{\uCA} R$. Then by Equation $2.6.2$ in \cite[$\mathsection 2.6$]{rh}, we have an inportant property of the Weyl modules: 

\begin{Lem}\label{Lem:Weyl_char_formula}
The module $W_{q}(\lambda)$ is free over $R$, admits a $R$-basis consisting of weight vectors and its character is given by the Weyl character formula.
\end{Lem}

\begin{Rem}\label{Rem:APW1_results_generalize}
Thus, the results of \cite[$\mathsection$1-4]{APW1} over the ring $\uCA_{\m_0}$ also hold over any $\uCA$-algebra $R$ where one should replace the condition of freeness of finitely generated $\CA_{\m_0}$-modules with the condition of projectivity of finitely generated $R$-modules (except some places).

\end{Rem}

Furthermore,  we  have the usual induction functor $H^0\colon \Rep(\dmU_q^{\leqslant})\rightarrow \Rep(\dmU_q(\g))$ defined as follows. For any $M\in \Rep(\dmU^{\leqslant}_q)$, the space $\Hom_{\dmU^{\leqslant}_q}(\dmU_q(\g), M)$ carries a left $\dmU_q(\g)$-module structure: $(xf)(u)=f(ux)$ for any $x,u \in \dmU_q(\g),f\in \Hom_{\dmU^{\leqslant}_q}(\dmU_q(\g), M)$. Then $H^0(M)$  is defined as the maximal rational $\dmU_q(\g)$-subrepresentation of $\Hom_{\dmU^{\leqslant}_q}(\dmU_q(\g), M)$. The following result is in \cite[Proposition 2.12]{APW1}

\begin{Prop}\label{prop: right adjoint to F}The functor $H^0: \Rep(\dmU^\leqslant_q)\rightarrow \Rep(\dmU_q(\g)$ is right adjoint to the  restriction functor $\fF: \Rep(\dmU_q(\g))\rightarrow \Rep(\dmU^\leqslant_q)$.
\end{Prop}
\begin{defi}For any dominant $\lambda$ , the module $H^0_q(\lambda):=H^0(R_\lambda)$ is called {\it the dual Weyl module}.
\end{defi}
By Proposition \ref{prop: right adjoint to F}, we have
\begin{equation}\label{eq:dual_Weyl_universal}
    \Hom_{\dmU_q(\g)}(?,H^0_q(\lambda)) \iso \Hom_{\dmU_q^{\leqslant}}(?,R_\lambda).
\end{equation}
\begin{Rem} We follow the definition of the functor $H^0$ in \cite{APW1} in contrast to \cite{rh}, see \cite[$\mathsection 2.9$]{rh}. This is just a matter of convention so we can apply the results in \cite{rh} to $H^0$.
\end{Rem}


\subsection{Some homological properties}\label{ssec:homological properties}\

Replacing $\dmU_q^{\leqslant}$ with $R$ and $\dmU^0_q$ we have the following induction functors
\begin{equation*}
    \begin{split}
H^0(\dmU/R, -)&: R\Lmod \rightarrow \Rep(\dmU_q(\g)),\\
H^0(\dmU/\dmU^0,-)&: \Rep(\dmU^0_q)\rightarrow \Rep(\dmU_q(\g)).
\end{split}
\end{equation*}
\begin{Lem}\label{rem: H(dmU/R,-)}(a) The functor $H^0(\dmU/R,-)$ is exact, sends injective objects to injective objects and $H^0(\dmU/R, M)\cong H^0(\dmU/R, R) \otimes_R M$ for any $R$-module $M$. In particular, $H^0(\dmU/R,-)$ commutes with direct limits. 

\noindent
(b) The functor $H^0(\dmU/\dmU^0,-)$ is exact and sends injective objects to injective objects.

\noindent
(c) $H^0(\dmU/\uCA, \uCA)\otimes_{\uCA} R\cong H^0(\dmU/R, R)$ for any $\uCA$-algebra $R$. As a result, for any $R$-algebra $R'$, we have $H^0(\dmU/R, R)\otimes_R R'\cong H^0(\dmU/R', R')$.
\end{Lem}
\begin{proof}Part (a) follows by   \cite[$\mathsection 1.31$]{APW1} and Remark \ref{Rem:APW1_results_generalize}. Part (b) follows by \cite[$\mathsection 2.11,2.13$]{APW1} and Remark \ref{Rem:APW1_results_generalize}. Part(c) is in \cite[$\mathsection 4.2$]{AW}(thanks to the results explained in Sections \ref{SS_Weyl} and \ref{SS_Kempf_vanish} we can work with the ring $\uCA$ instead of the ring $\CA_1$ from \cite[$\mathsection 4.1$]{AW}).
\end{proof}
Let us construct an  injective resolution  for $N\in \Rep(\dmU_q(\g))$. First, we embed $N$ into an  injective $R$-module $I^0$. Then $N$ is a $\dmU_q(\g)$-submodule of $N^0:=H^0(\dmU/R, I^0)$. Repeat the process with $N^0/N$. We then obtain an exact sequence $\dots 0\rightarrow N\rightarrow N^0 \rightarrow N^1 \rightarrow \dots$  such that $N^i=H^0(\dmU/R, I^i)$ for some injective $R$-module $I^i$, and then $N^0 \rightarrow N^1 \dots $ is an injective resolution of $N$.

There is a \textbf{big standard resolution} of $N$  constructed as follows: $N \hookrightarrow H^0(\dmU/R, N)$ makes $N$ into a $\dmU_q(\g)$-submodule and a direct $R$-module summand  of $H^0(\dmU/R, N)$. Let $Q^0=H^0(\dmU/R,N)$. Repeat the process with $Q^0/N$ then we obtain a resolution $N \rightarrow Q^0 \rightarrow Q^1 \rightarrow \dots $ in which $Q^i =H^0(\dmU/R, Q^{i-1}/Q^{i-2})$.


Moreover, there is a \textbf{standard resolution} of $N$ constructed as follows: $N \hookrightarrow H^0(\dmU/\dmU^0, N)$ makes $N$ into a $\dmU_q(\g)$-submodule and a direct $\dmU^0_q$-module summand  of $H^0(\dmU/\dmU^0,M)$. Let $Q^0:=H^0(\dmU/\dmU^0,N)$. Repeat the process with $Q^0/N$ then we obtain a resolution $N \rightarrow Q^0 \rightarrow Q^1 \rightarrow \dots $ in which $Q^i=H^0(\dmU/\dmU^0, Q^{i-1}/Q^{i-2})$. We abuse the notation $Q^i$ here. This standard resolution is the one considered in \cite[$\mathsection 2.17$]{APW1}.

We now establish some homological properties of $\Rep(\dmU_q(\g))$, which are not automatic since $\dmU_q(\g)$ is neither Noetherian nor finitely generated as an algebra. The point is that the induction functor $H^0(\dmU/R,-)$ allows us to translate those homological statements to the ones in $R\Lmod$. To emphasize the base ring of $\dmU_q(\g)$, we will denote $\dmU_q(\g)$ by $\dmU_R(\g)$. 

\begin{Lem}\label{lem: homological properties}Let $R$ be  a Noetherian ring. Let $M\in \Rep(\dmU_R(\g))$ such that $M$ is a finitely generated over $R$. Then the functor $\Ext^i_{\Rep(\dmU_R(\g))}(M, -)$ commutes with filtered direct limits. 
\end{Lem}

\begin{proof} First, we recall a functorial way to embed $R$-modules into injective $R$-modules. Let $M$ be an $R$-module. Set $M^\vee:=\Hom_{\BZ}(M,\BQ/\BZ)$. We have an embedding of $R$-modules $M \hookrightarrow (M^\vee)^\vee$. Let $F_{M^\vee}:=\oplus_{m\in M^\vee} R$ be a free $R$-module where the index runs over all elements $m \in M^\vee$ then $F^\vee_{M^\vee}=\Hom_{\BZ}(F_{M^\vee}, \BQ/\BZ)$ is an injective $R$-module. There is a natural surjective morphism of $R$-modules $F_{M^\vee}\twoheadrightarrow M^\vee$, which gives us an injective morphism of $R$-modules $(M^\vee)^\vee \hookrightarrow F^\vee_{M^\vee}$. Then the embedding $M\hookrightarrow F^\vee_{M^\vee}$ is functorial in $M$. In particular, for any direct system of $R$-modules $\{M_i\}_I$, we can construct a direct system of embeding $\{M_i \hookrightarrow F^\vee_{M^\vee_i}\}_I$.

Let $\{N_i\}$ be a direct system in $\Rep(\dmU_R(\g))$ and $\varinjlim N_i=N \in \Rep(\dmU_R(\g))$. By the above paragraph, we can construct a direct system of embeddings of $R$-modules $\{N_i \hookrightarrow I^0_i\}_I$, here $I^0_i$ are injective $R$-modules. This gives us a direct system of embeddings $\{N_i \hookrightarrow N^0_i\}_I$, where $N^0_i=H^0(\dmU/R, I^0_i)$. Repeating this process, we can construct a direct system of injective resolution $\{N \rightarrow N^0_i \rightarrow N^1_i\dots\}_I$ in $\Rep(\dmU_R(\g))$, in which $N^j_i=H^0(\dmU/R, I^j_i)$ for some injective $R$-module $I^j_i$.


Filtered direct limits preserve exact sequences in the category of modules and the direct limit of injective $R$-modules is again an injective $R$-module when $R$ is Noetherian. Therefore, the direct limit of $\{N_i\rightarrow N_i^0\rightarrow N_i^1 \dots\}$ gives us an injective resolution $N \rightarrow N^0 \rightarrow N^1 \rightarrow \dots $ in which $N^j=\varinjlim H^0(\dmU/R, I_i^j)=H^0(\dmU/R, \varinjlim I_i^j)$ ($\varinjlim I^j_i$ is an injective $R$-module), the last equality holds since $H^0(\dmU/R, -)$ commutes with direct limits by Lemma \ref{rem: H(dmU/R,-)}. Now the lemma follows from
\[ \varinjlim\Hom_{\dmU_R(\g)}(M, N^j_i)=\varinjlim\Hom_R(M, I_i^j) =\Hom_R(M, \varinjlim I_i^j)=\Hom_{\dmU_R(\g)}(M, N^j),\]
the second equality holds since $R$ is Noetherian and $M$ is a finitely generated $R$-module.
\end{proof}

\begin{Lem}\label{lem: homological properties 2}
 Let $R$ be a Noetherian ring.  Let $M\in \Rep(\dmU_R(\g))$ such that $M$ is a finitely generated $R$-module. Then for any $N \in \Rep(\dmU_R(\g))$, we have an isomorphism
\begin{equation}\label{eq: Ext and loc}\Ext^i_{\Rep(\dmU_R(\g))}(M, N) \otimes_R R_\p \cong \Ext^i_{\Rep(\dmU_{R_\p}(\g))}(M_\p, N_\p),
\end{equation}
for $i \geq 0$, here $\dmU_{R_\p}(\g):=\dmU_R(\g)\otimes_R R_\p$, $M_\p:=M\otimes_R R_\p$ and $ N_\p:=N\otimes_R R_\p$.
\end{Lem}
\begin{proof}
{\it Step 1.} We will prove that for any $N \in \Rep(\dmU_R(\g))$, the natural map $N^{\dmU}\otimes_R R_\p \rightarrow (N_{\p})^{\dmU}$ is an isomorphism. By Proposition  \ref{prop:rational_finiteness}, $N$ is a union of modules in $\Rep(\dmU_R)$ which are finitely generated over $R$. So by Lemma \ref{lem: homological properties}, we can assume $N$ is finitely generated over $R$. Then there is $m_0$ such that $E^{(m)}_in=F^{(m)}_i n=0$ for  all $m\geq m_0, n \in N$ and $1\leq i \leq r$. Hence,
\begin{align*}N^{\dmU}&=\{ n \in N_0| E^{(m)}_in=F^{(m)}_i n=0 \;\; \forall 1\leq m <m_0, 1\leq i \leq r\}\\
N^{\dmU}_\p&=\{ n \in (N_\p)_0| E^{(m)}_i n=F^{(m)}_i n=0\;\; \forall 1\leq m <m_0, 1 \leq i \leq r\}
\end{align*}
Since localization commutes with finite intersections, it follows that $N^{\dmU}\otimes_R R_\p \cong (N_\p)^{\dmU}$.

{\it Step 2.} For any $N \in \Rep(\dmU_R(\g))$, the natural map 
\begin{equation}\label{eq: Hom and loc}\Hom_{\dmU_R}(M,N) \otimes_R R_\p \rightarrow \Hom_{\dmU_{R_\p}}(M_\p, N_\p),
\end{equation}
is an isomorphism. Indeed, 
\[\Hom_{\dmU_R}(M,N)\otimes_R R_\p =\Hom_R(M,N)^{\dmU}\otimes_R R_\p, \quad \Hom_{\dmU_{R_{\p}}}(M_\p, N_\p)=\Hom_{R_\p}(M_\p, N_\p)^{\dmU}.\]
Since  $M$ is a finitely generated $R$-module, it follows that $\Hom_R(M,N) \in \Rep(\dmU_R(\g))$ and $\Hom_{R_\p}(M_\p, N_\p)\in \Rep(\dmU_{R_\p}(\g))$ and $\Hom_{R_{\p}}(M_\p, N_\p)=\Hom_R(M,N)\otimes_R R_\p$. Therefore, \eqref{eq: Hom and loc} holds by Step 1.

{\it Step 3.} Let $N \rightarrow N^0 \rightarrow N^1 \rightarrow \dots $ be an injective resolution of $N$, in which $N^i=H^0(\dmU/R, I^i) \cong H^0(\dmU/R, R)\otimes_R I^i$ for some injective $R$-module $I^i$, the isomorphism follows by Lemma \ref{rem: H(dmU/R,-)}. We have $N^i\otimes_R R_\p=H^0(\dmU/R_\p, R_\p)\otimes_{R_{\p}} I^i_{\p}$ and $I^i_\p$ is an injective $R_\p$-module because $R$ is Noetherian. Therefore, applying $-\otimes_R R_\p$ to the injective resolution of $N$, we get an injective resolution of $N_\p$. By using Step 2 and computing Ext groups with  the injective resolutions of $N$ and $N_\p$, we obtain \eqref{eq: Ext and loc}.
\end{proof}

Let $I$ be an ideal of $R$ and $R_I:=R/I$. Let $N\in \Rep(\dmU_{R_I}(\g))$ then $N$ can be viewed as an object in $\Rep(\dmU_R(\g))$. Let $M\in \Rep(\dmU_R(\g))$. Then $M_{R_I}:=M_R\otimes_R R_I \in \Rep(\dmU_{R_I}(\g))$.
\begin{Lem}\label{lem: homological properties 3}Let $R$ be a Noetherian ring. Assume $M$ is a finitely generated projective module over $R$ then for all $i \geq 0$, 
\begin{equation}\label{eq: Ext2} \Ext^i_{\Rep(\dmU_R(\g))}(M_R, N_{R_I})\cong \Ext^i_{\Rep(\dmU_{R_I}(\g))}(M_{R_I}, N_{R_I}).
\end{equation}

\end{Lem}
\begin{proof} Since $M_R$ is a finitely generated projective $R$-module, $M_{R_I}$ is a finitely generated projective $R_I$-module. Therefore, we can use the big  standard resolutions of $N_{R_I}$ in $\Rep(\dmU_R(\g))$ and $\Rep(\dmU_{R_I}(\g))$ to compute the Ext groups in \eqref{eq: Ext2}, see the argument in last paragraph in the proof of Lemma \ref{lem: homological properties 2}. These two big standard resolutions of $N_{R_I}$ are as follows:
\begin{equation*}
\begin{split}
    N_{R_I}\rightarrow Q^0_R\rightarrow Q^1_R\rightarrow \dots, \qquad N_{R_I}\rightarrow Q^0_{R_I}\rightarrow Q^1_{R_I} \rightarrow \dots
\end{split}
\end{equation*}
in which $Q^i_R=H^0(\dmU/R, Q^{i-1}_R/Q^{i-2}_R)$ and $Q^i_{R_I}=H^0(\dmU/{R_I}, Q^{i-1}_{R_I}/Q^{i-2}_{R_I})$.

By Lemma \ref{rem: H(dmU/R,-)}, 
\[Q^0_R=H^0(\dmU/R, N_{R_I})\cong H^0(\dmU/R, R)\otimes_R N_{R_I}\cong H^0(\dmU/R_I, R_I)\otimes_{R_I} N_{R_I}\cong H^0(\dmU/R_{I}, N_{R_I})=Q^0_{R_I}\]
By induction, $Q^i_R \cong Q^i_{R_I}$. On the other hand,
\[ \Hom_{\Rep(\dmU_R(\g))}(M_R, Q^i_{R_I}) \cong \Hom_{\Rep(\dmU_{R_I}(\g))}(M_{R_I}, Q^i_{R_I}).\]
Therefore, \eqref{eq: Ext2} holds.    
\end{proof}

\subsection{Kempf vanishing}\label{SS_Kempf_vanish}\

The category $\Rep(\dmU_q^{\leqslant})$ has enough injectives, see \cite[$\mathsection$2.13]{APW1}. So we can define the higher derived functors $H^i$ of the induction functor $H^0$. 

%

We write $H^i_q(\lambda)$ for $H^i(R_\lambda)$. Let $\lambda^* :=-w_0 \lambda$.

\begin{Prop}\label{Prop:Kempf_vanishing}
Suppose $R$ is Noetherian and $\lambda$ is dominant. The  following claims are true:
\begin{itemize}
    \item[(i)] We have $H^0_q(\lambda) \simeq \Hom_R(W_q(\lambda^*),R)$.
    \item[(ii)] We have $H^i_q(\lambda)=\{0\}$ for all $i>0$.
\end{itemize}
\end{Prop}
The second part of the proposition is known as the Kempf vanishing theorem. In \cite{APW1}, specalizing $\uCA_{\m_0}$ to its residue field $k$ with $q=1 \in k $ allowed the authors to use some results about modular representations in the proof of the Kempf vanishing theorem. Let us now consider the case when the base ring is $\uCA$. For any maximal ideal $\m \in \uCA$, the field $\uCA/\m$ is a finite field and hence $q\in \uCA/\m$ is a root of unity. Some results in \cite{rh}  in  the case when the field $k$ of positive characteristics and $q\in k$ is a root of unity can replace the references to modular representations in the proof of the  Kempf vanishing theorem in \cite{APW1}. This allows us to prove the Kempf vanishing theorem over $\uCA$ and then over the $\uCA$-algebra $R$ via base change.

\begin{proof}
(i) Following  Remark \ref{Rem:APW1_results_generalize}, this part is proved in \cite[Proposition 3.3]{APW1}.

\noindent
(ii)
{\it Step 1:} First, we consider the case when the base ring is $\uCA$. This is shown by generalizing results from \cite[$\mathsection 5.5$]{APW1} as follows.  For any maximal ideal $\m \in \uCA$, let $\uCA_\m$ be the localization  and $k:=\uCA_\m /\m \uCA_m \cong \uCA/\m$ be the residue field of $\uCA_\m$. Then $k$ is a finite field of positive characteristics. Let $\bar{q}$ denote the image of $q$ in $k$ then $\bar{q}$ is a root of unity in $k$.  
\begin{itemize}
    \item \cite[Theorem $5.1$]{APW1} holds over $\uCA$. To show that the natural homomorphism $H^0_q(\lambda)\rightarrow H^0_q(s,\lambda)$ (the target is introduced before the theorem) is surjective, it is enough to show the surjectivity over the localization $\uCA_\m$  at any maximal ideal $\m \subset \uCA$. The later is done as in \emph{loc.cit.} In the proof of (ii), to show the surjectivity  of the natural homomorphism $H^0_{\bar{q}}(\lambda) \rightarrow H^0_{\bar{q}}(s,\lambda)$  one replaces the references there with the reference to \cite[Lemma $5.3$]{rh}: that lemma  combined the isomorphism $\mathcal{D}_{w_0}(\lambda) \cong \mathcal{D}(\lambda)$ in \cite[Theorem $5.4$]{rh} imply that the dual of $H^0_{\bar{q}}(\lambda)  \rightarrow H^0_{\bar{q}}(s, \lambda)$ is injective. 
    \item \cite[Lemma $5.3$]{APW1} holds over $\uCA$. As in the proof in \emph{loc. cit.}, inside $H^0(\dmU^\leqslant/\dmU^0, \mu)$, we produce a union of submodules isomorphic to $H^0_{\uCA}(m \rho)\otimes \uCA_{m \rho+\mu}$ over $m \geq 0$. But then this ascending chain of submodules is equal to $H^0(\dmU^\leqslant/\dmU^0, \mu)$ after localizing at any maximal ideal $\m \subset \uCA$ by the same proof in \emph{loc. cit.}. Therefore, \cite[Lemma 5.3]{APW1} holds over $\uCA$. 
    \item Using this we see that \cite[Theorem 5.4]{APW1} holds over $\uCA$.
    \item \cite[Proposition $5.6$]{APW1} holds over $\uCA$. Part (i) of that proposition holds over localization $\uCA_\m$ at any  maximal ideal $\m$: in the proof of $H^0_{\bar{q}}(\lambda) \iso H^0_{\bar{q}}(w_0, \lambda)$ one replaces the reference there with \cite[Theorem 5.4]{rh}. Therefore part(i) holds over $\uCA$ and then parts (ii), (iii) follow.
\end{itemize}
Then the claims that $H^i_{\uCA}(\lambda)=0$ for $i \geq 1$ follow by combining \cite[Theorem 5.4(i)]{APW1} and \cite[Proposition 5.6 (iii)]{APW1}. 

{\it Step 2:} For a general $\uCA$-algebra $R$, we now can use the spectral sequence from \cite[$\mathsection 3.4-3.5$]{APW1}  (with $\uCA_\m$ is replaced by $\uCA$ and  $\Gamma=R$), see also \cite[Remark 3.5]{APW1},  to see that $H^i_q(\lambda)=\{0\}$ for all $i>0$, while $$H^0_q(\lambda)=R\otimes_{\uCA}H^0_{\uCA}(\lambda).$$
\end{proof}

We remark that for two dominant weights $\lambda,\lambda'$, we have 
\begin{equation}\label{eq:Hom_Weyl_dual_Weyl} \Hom_{\dmU_q(\g)}(W_q(\lambda), H^0_q(\lambda'))\simeq R^{\oplus \delta_{\lambda,\lambda'}}.
\end{equation}

When $R$ is a field, the image of a nonzero homomorphism $W_q(\lambda)\rightarrow H^0_q(\lambda)$ is simple. We denote it by $L_q(\lambda)$. The assignment $\lambda\mapsto L_q(\lambda)$ is a bijection between the dominant weights and the simple modules in $\operatorname{Rep}(\dmU_q(\g))$.

\subsection{Good filtrations}\label{ssec: good filtration}

\begin{defi}\label{defi:good_filtration}
Let $M\in \operatorname{Rep}(\dmU_q(\g))$. 
An exhaustive $\dmU_q(\g)$-module filtration $\{0\}=M_0\subsetneq M_1\subsetneq M_2\subsetneq\ldots$ is called {\it good} if, for each $i$ we have, $M_i/M_{i-1} \simeq H^0_q(\lambda_i)\otimes_R P_i$ for some dominant weight $\lambda_i$ and some finitely generated projective $R$-module $P_i$.
\end{defi}

\begin{Lem}\label{Lem:good_filtration_equivalent}
Let $R$ be a field.
Suppose that $M\in \operatorname{Rep}(\dmU_q(\g))$ satisfies the following condition
\begin{equation}\label{eq:condition_finite}
    \dim_R\Hom_{\dmU_q(\g)}(W_q(\lambda), M)<\infty , \;\;\forall \lambda\in P_+,
\end{equation}
where $P_+$ denotes the set of dominant weights.
Then the following two conditions are equivalent.
\begin{enumerate}
    \item $M$ admits a good filtration. 
    \item  $\operatorname{Ext}^i_{\operatorname{Rep}(\dmU_q(\g))}(W_q(\lambda), M)=0, 
$    
    for all $i>0$ and dominant weights $\lambda$.
\end{enumerate}
\end{Lem}
\begin{proof}The proof is in several steps.

{\it Step 1}. Let us show $(1) \Rightarrow (2)$. The submodule $M_j$ has a good filtration and is finite dimensional. Hence $\Ext^i_{\Rep(\dmU_q(\g))}(W_q(\lambda), M_j)=0$ for all $i>0$ and dominant weights $\lambda$ by Theorem $3.1$ in \cite{p}. By Lemma \ref{lem: homological properties}, 
\[\Ext^i_{\Rep(\dmU_q(\g))}(W_q(\lambda),M)=\varinjlim \Ext^i_{\Rep(\dmU_q(\g))}(W_q(\lambda), M_j)=0,\]
for all $i>0$ and dominant weights $\lambda$. 

{\it Step 2}. It remains to show $(2) \Rightarrow (1)$. We refine the order of $P_+$ to a total order: $\lambda_1 <\lambda_2< \dots$. We will construct inductively a filtration  $\{0\}=M_0\subset M_1\subset \dots $ on $M$ that is exhaustive and satisfies the following conditions
\begin{itemize}
\item[(3)] $M_i/M_{i-1}$ is the direct sum of several (finitely many) copies of $H^0_q(\lambda_i)$,
\item[(4)] $\Ext^n_{\Rep(\dmU_q(\g))}(N, M/M_i)=0$ for all $n\geq 0$ and finite dimensional module $N$ contained in the Serre span of $W_q(\lambda_j)$ with $j \leq i$. 
\end{itemize}

{\it Step 2.1. The base case $i=0$.} $M_0=\{0\}$ and $(3)$, $(4)$ are vacuously true.

{\it Step 2.2. The induction step}. Assume we have constructed $M_0 \subset M_1 \subset \dots M_i$ satisfying (3) and (4). Then $M/M_i$ admits no nonzero homomorphism from objects filtered by $L_q(\lambda_j)$ with $j\leq i$. Therefore, the following map is injective 
\begin{equation}\label{eq: Hom to M/Mi}\Hom_{\Rep(\dmU_q(\g))}(H^0_q(\lambda_{i+1}), M/M_i)\otimes H^0_q(\lambda_{i+1}) \rightarrow M/M_i.
\end{equation}
Let $M_{i+1}$ be the preimage of the image of \eqref{eq: Hom to M/Mi}  under the projection $M \rightarrow M/M_i$. Since the kernel and cokernel of the map $W_q(\lambda_{i+1})\rightarrow H^0_q(\lambda_{i+1})$ are contained in the Serre span of $W_q(\lambda_j)$ with $j \leq i$, we have by (4) that 
\[ \Hom_{\Rep(\dmU_q(\g))}(H^0_q(\lambda_{i+1}), M/M_i) \cong \Hom_{\Rep(\dmU_q(\g))}(W_q(\lambda_{i+1}), M/M_i).\]
Then $\Ext^n_{\Rep(\dmU_q(\g))}(W_q(\lambda_j), M/M_{i+1})=0$ for all $n \geq 0$ and $j \leq i+1$. This implies that $M_{i+1}$ satisfies (4). 


{\it Step 2.3}. It is left to show that the filtration we have constructed is exhastive. By $(4)$, for any $N\in \Rep(\dmU_q(\g))$ where all weights are $\leq \lambda_i$, we have 
\[ 0\rightarrow \Hom_{\dmU_q(\g)}(N,M_i)\rightarrow \Hom_{\dmU_q(\g)}(N,M) \rightarrow \Hom_{\dmU_q(\g)}(N,M/M_i)=0,\]
This implies that $M_i$ is the maximal submodule of $M$ where all weights are $\leq \lambda_i$. Thanks to Proposition  \ref{prop:rational_finiteness}, $\dmU_q(\g)m$ is finite dimensional for all $m\in M$. Suppose that all weights in this submodule are less than or equal to $\lambda_i$. Then $m \in M_i$. 
\end{proof}

The following is the main result of \cite{p}. 
\begin{Prop}\label{Prop:good_filt_tensor}
Let $R$ be a field. Then the tensor product of two finite dimensional modules in $\Rep(\dU_q(\g))$ with good filtrations also has a good filtration.
\end{Prop}
 
We will need analogs of Lemma  \ref{Lem:good_filtration_equivalent} 
and Proposition \ref{Prop:good_filt_tensor}
over regular Noetherian domains $R$ (we will be interested in the completions of localizations of $\uCA$). A technical result we are going to use is the following claim  from 
\cite[$\mathsection 5.13$]{APW1}, see also \cite[$\mathsection 4.2$]{p}. 

\begin{Lem}\label{Lem:good_filtr_specialization}
Let $M$ be an object in $\Rep(\dmU_q(\g))$ that is a finitely generated projective $R$-module. Then $M$ admits a good filtration if and only if $k\otimes_{R}M\in 
\Rep(\dmU_k(\g))$ admits a good filtration for every epimorphism $R\twoheadrightarrow k$ onto a field.
\end{Lem}
\begin{proof}The proof follows \cite[$\mathsection 4.2$]{p}. We only write down needed modifications . Let $\lambda\in P_+$ be maximal among the weights of $M$. By  the assumption on $M$, it follows that $M_\lambda$ is a finitely generated projective $R$-module. The $\dmU^\leqslant_q$-modules homomorphism $M \twoheadrightarrow M_\lambda$ gives rise to a $\dmU_q(\g)$-homomorphism $M \rightarrow H^0_q(M_\lambda)$. We note that $H^0_q(M_\lambda)\cong H^0_q(\lambda)\otimes_R M_\lambda$. This is because $H^0_q(-)$ commutes with direct sums, $H^0_q(R_\lambda^{\oplus n})\cong H^0_q(\lambda)^{\oplus n}$ and $M_\lambda$ is a finitely generated projective $R$-module. So we have a $\dmU_q(\g)$-modules homomorphism $M \rightarrow H^0_q(\lambda)\otimes_R M_\lambda$. The rest of the proof follows \cite[$\mathsection 4.2$]{p}.
\end{proof}

The direct analog of Proposition \ref{Prop:good_filt_tensor} follows:
\begin{Cor}\label{Cor:good-filt-tensor-over-R} Tensor product of two finitely generated $R$-modules in  $\Rep(\dmU_q(\g))$ with good filtrations also has a good filtration.
\end{Cor}

We will now get to an analog of Lemma \ref{Lem:good_filtration_equivalent}. Let us start with  the following lemma.
\begin{Lem}\label{lem:Ext(Weyl, dualWeyl)} Let $R$ be a Noetherian ring. Then $\Ext^i_{\Rep(\dmU_q(\g))}(W_q(\lambda), H^0_q(\lambda'))=0$ for all $i>0$ and dominant weights $\lambda, \lambda'$.
\end{Lem}
\begin{proof}Since $W_q(\lambda)$ is free over $R$ and $\Hom_R(W_q(\lambda), R)\cong H^0_q(\lambda^*)$, we have
\[ \Ext^i_{\Rep(\dmU_q(\g))}(W_q(\lambda), H^0_q(\lambda'))\cong \Ext^i_{\Rep(\dmU_q(\g))}(R_0, H^0_q(\lambda^*)\otimes_R H^0_q(\lambda')),\]
here $R_0$ is the $\dmU_q(\g)$-module via the counit $\varepsilon: \dmU_q(\g) \rightarrow R$. By Corollary \ref{Cor:good-filt-tensor-over-R}, $H^0_q(\lambda^*)\otimes_R H^0_q(\lambda')$ has a good filtration.  Therefore, it is enough to show that 
    \begin{equation}\label{eq: Ext=0} \Ext^i_{\Rep(\dmU_q(\g))}(R_0, H^0_q(\lambda))=0,
    \end{equation}
for $i>0$ and dominant weights $\lambda$. Since $H^i_q(\lambda)=0$ for $i>0$ and dominant $\lambda$, it follows that
\[ \Ext^i_{\Rep(\dmU_q(\g))}(R_0, H^0_q(\lambda)) \cong \Ext^i_{\Rep(\dmU^\leqslant_q)}(R_0, R_\lambda).\]
Let consider the induction functor $H^0(\dmU^\leqslant/\dmU^0,-): \Rep(\dmU^0_q)\rightarrow \Rep(\dmU^\leqslant_q)$. We then form a standard resolution for $R_\lambda$ in $\Rep(\dmU^\leqslant_q)$ as in Section \ref{ssec:homological properties}:
\begin{equation}\label{eq: standard rel} R_\lambda \rightarrow Q^0\rightarrow Q^1 \rightarrow \dots,
\end{equation}
where $Q^0=H^0(\dmU^\leqslant/\dmU^0, R)$,  $Q^i=H^0(\dmU^\leqslant/\dmU^0,Q^{i-1}/Q^{i-2})$ for $i>0$; here we set $Q^{-1}:=R_\lambda$. Then the space $\Ext^i_{\Rep(\dmU^\leqslant_q)}(R_0,R_\lambda)$ can be computed by this standard resolution. We note the following: all weights of $Q^i$ are contained in $\lambda+Q_+$. Hence, $Q^i$ does not have weights smaller than $0$. Therefore, applying $\Hom_{\Rep(\dmU^\leqslant_q)}(R_0, -)$ to \eqref{eq: standard rel}, we get $(R_\lambda)_0 \rightarrow Q^0_0 \rightarrow Q^1_0\rightarrow \dots$ which is exact. It follows that $\Ext^i_{\Rep(\dmU^\leqslant_q)}(R_0,R_\lambda)=0$ for $i >0$, hence \eqref{eq: Ext=0} holds.
\end{proof}

\begin{Lem}\label{Lem:good_filtration_equivalent_deformed}
Let $R$ be a regular Noetherian domain.  Suppose that $M\in \operatorname{Rep}(\dmU_q(\g))$ satisfies the following condition
\begin{equation}\label{eq:condition_finite 2}
    \Hom_{\dmU_q(\g)}(W_q(\lambda), M)\text{ is a finitely generated }R\text{-module}, \forall \lambda\in P_+,
\end{equation}
Then the following two conditions are equivalent.
\begin{enumerate}
    \item $M$ admits a good filtration. 
    \item  The $R$-module $M$ is isomorphic to the direct sum of finitely generated projective $R$-modules,  and $$\Ext^i_{\Rep(\dmU_q(\g))}(W_q(\lambda), M)=0, 
$$    
    for all $i>0$ and dominant weights $\lambda$.
\end{enumerate}
\end{Lem}

\begin{proof}The proof is in several steps. We fix some notations: For any $R$-algebra $R'$ and $R$-module $M$, we write $W_{R'}(\lambda):=W_q(\lambda)\otimes_R R'$, $H^0_{R'}(\lambda):=H^0_q(\lambda)\otimes_R R'$ and $M_{R'}:=M\otimes_R R'$.

{\it Step 1.} Let us show $(1) \Rightarrow (2)$. By Lemma \ref{lem:Ext(Weyl, dualWeyl)}, $\Ext^i_{\Rep(\dmU_q(\g))}(W_q(\lambda), M_j)=0$ for $i>0$ and all $j$. By Lemma \ref{lem: homological properties}, we have
\[ \Ext^i_{\Rep(\dmU_q(\g))}(W_q(\lambda),M)=\varinjlim \Ext^i_{\Rep(\dmU_q(\g))}(W_q(\lambda), M_j)=0.\]

{\it Step 2.} Assume $M$ satisfies $(2)$ and \eqref{eq:condition_finite 2}.  We will show that
$M$ satifies the following:
\begin{enumerate}
  \item[(a)] $\Hom_{\Rep(\dmU_q(\g))}(W_q(\lambda), M)$ is a finitely generated projective $R$-module,
  \item[(b)] $\Hom_{\Rep(\dmU_k(\g))}(W_k(\lambda), M_k)$ is finite dimensional over $k$, and $\Ext^i_{\Rep(\dmU_k(\g))}(W_k(\lambda), M_k)=0$ for all $i>0$, all dominant $\lambda$, and all epimorphisms $R\twoheadrightarrow k$ onto a field. 
\end{enumerate}
Lemma \ref{lem: homological properties 2} allows us to reduce to the local case. So we can assume $R$ is a local regular Noetherian domain. Let $\m=(x_1, \dots, x_n)$ be the maximal ideal of $R$, where $x_1, \dots, x_n$ is a maximal regular sequence. Let $R_i=R/(x_1, \dots, x_i)$ and $M_{R_i}=M/(x_1M +\dots +x_i M)$. We see that $M_{R_i}\in \Rep(\dmU_{R_i}(\g))$.  Since $M$ is projective over $R$, we have a short exact sequence
\[ 0\rightarrow M \xrightarrow[]{ \cdot x_1} M \rightarrow M_{R_1}\rightarrow 0.\]
Applying $\Hom_{\Rep(\dmU_q(\g))}(W_q(\lambda), -)$ to get long exact sequence of Ext's group, we obtain
\begin{equation*}
    \begin{split}
        &\Hom_{\Rep(\dmU_q(\g))}(W_q(\lambda), M)\xhookrightarrow[]{\cdot x_1}\Hom_{\Rep(\dmU_q(\g))}(W_q(\lambda), W)\twoheadrightarrow \Hom_{\Rep(\dmU_q(\g))}(W_q(\lambda), M_{R_1}),\\
        &\Ext^i_{\Rep(\dmU_q(\g))}(W_q(\lambda), M_{R_1})=0 \;\; \forall i>0.
    \end{split}
\end{equation*}
Combining these with  Lemma \ref{lem: homological properties 3}, we have 
\begin{equation*}
    \begin{split}
        &\Hom_{\Rep(\dmU_q(\g))}(W_q(\lambda), M)\xhookrightarrow[]{\cdot x_1} \Hom_{\Rep(\dmU_q(\g))}(W_q(\lambda), W)\twoheadrightarrow \Hom_{\Rep(\dmU_{R_1}(\g))}(W_{R_1}(\lambda), M_{R_1}),\\
        &\Ext^i_{\Rep(\dmU_{R_1}(\g))}(W_{R_1}(\lambda), M_{R_1})=0 \;\; \forall i>0.
    \end{split}
\end{equation*}
Since $M$ is a direct sum of finitely generated projective $R$-modules, $M_{R_1}$ is a direct sum of finitely generated projective $R_1$-modules. So we can proceed inductively to get (b) and also the claim that $x_1, \dots, x_n$ is a regular sequence of $\Hom_{\Rep(\dmU_q(\g))}(W_q(\lambda), M)$. So $\Hom_{\Rep(\dmU_q(\g))}(W_q(\lambda), M)$ is a finitely generated Cohen-Macaulay module over a local regular Noetherian domain $R$. Hence it is free over $R$, equivalently, (a) holds. 


{\it Step 3.} Let us show $(2) \Rightarrow (1)$. We will construct inductively a filtration $\{0\}=M_0 \subset M_1\subset M_2\subset \dots $ on $M$ that is exhaustive and satisfies
\begin{enumerate}
    \item[(3)] $M_i/M_{i-1} \cong H^0_q(\lambda_i)\otimes_R P_i$ for some finitely generated projective $R$-module $P_i$.
    \item[(4)] $M_i$ is a direct summand of $M$ as an $R$-module.
    \item[(5)] For any epimorphism $R\twoheadrightarrow k$ onto a field, $M_{k,i}:=M_i\otimes_R k$ is the $i$-th component of $M_k:=M\otimes_R k$ constructed in Step 2 of the proof of Lemma \ref{Lem:good_filtration_equivalent}.
    \item[(6)] $\Hom_{\Rep(\dmU_q(\g))}(W_q(\lambda), M/M_{i-1})$ is a finitely generated $R$-module for all dominant $\lambda$.
    \item[(7)] $\Ext^n_{\Rep(\dmU_q(\g))}(N,M/M_i)=0$ for all $n \geq 0$ and finitely generated $R$-module $N$ contained in Serre span of $W_q(\lambda_j)$ with $j \leq i$.
\end{enumerate}

{\it Step 3.1. The base case $i=0$.} $M_0=\{0\}$ and $(3)$-$(7)$ are  vacuously true.

{\it Step 3.2. The induction step.} Suppose we have already constructed $M_0\subset M_1 \subset \dots \subset M_{i-1}$. Since the kernel and cokernel of the natural homomorphism $W_q(\lambda_i)\rightarrow H^0_q(\lambda_i)$ have weights strictly less than $\lambda_i$, it follows that they lie in the Serre span of the objects $W_q(\lambda_j)$ with $j<i$. By $(7)$, we have
\begin{equation}\label{eq: Hom is proj} \Hom_{\Rep(\dmU_q(\g))}(H^0_q(\lambda_i), M/M_{i-1})\cong \Hom_{\Rep(\dmU_q(\g))}(W_q(\lambda_i), M/M_{i-1}).
\end{equation}
By (6)-(7) and Step 2, \eqref{eq: Hom is proj} is a  finitely generated projective $R$-module. Let us consider the homomorphism
\begin{equation}\label{eq: inj map}
\Hom_{\Rep(\dmU_q(\g))}(H^0_q(\lambda_i), M/M_{i-1})\otimes_R H^0_q(\lambda_i) \rightarrow M/M_{i-1}.
\end{equation}
In Step 2, we saw that 
\[ \Hom_{\Rep(\dmU_q(\g))}(W_q(\lambda_i), M/M_{i-1})\otimes_R k \cong \Hom_{\Rep(\dmU_k(\g))}(W_k(\lambda_i), (M/M_{i-1})_k).\]
Combining this with \eqref{eq: Hom is proj} and its version over $k$, we have
\[ \Hom_{\Rep(\dmU_q(\g))}(H^0_q(\lambda_i), M/M_{i-1})\otimes_R k \cong \Hom_{\Rep(\dmU_k(\g))}(H^0_k(\lambda_i), (M/M_{i-1})_k).\]
Therefore, the functor $\bullet \otimes_R k$ sends \eqref{eq: inj map} to the similarly defined homomorphism over $k$. The latter is injective. Combining this with the fact that $M/M_{i-1}$ is a direct sum of finitely generated projective $R$-modules, we conclude that \eqref{eq: inj map} is injective and the cokernel is a direct sum of finitely generated projective $R$-modules. 

We now let $M_i$ be the preimage of the left hand side of \eqref{eq: inj map} under the projection $M\rightarrow M/M_{i-1}$. Then $M_i$ satisfies (3)-(6). The cokernel of \eqref{eq: inj map} is $M/M_i$.  Applying the functor $\Hom_{\Rep(\dmU_q(\g))}(W_q(\lambda_j),-)$ with $j \leq i$ to \eqref{eq: inj map}, one can show that $\Ext^n_{\Rep(\dmU_q(\g))}(W_q(\lambda_j),M/M_i)=0$ for all $n \geq 0$ and $j \leq i$. This implies that $M_i$ satisfies (7).

{\it Step 3.3.} By using (7) we see that $M_i$ is the maximal submodule where all weights are $\leq \lambda_i$.  Thanks to Proposition  \ref{prop:rational_finiteness}, $\dmU_q(\g)m$ is finitely generated over $R$ for any $m \in M$. Suppose that all weights in this submodule are less than or equal to $\lambda_i$ then $\dmU_q(\g) m \subset M_i$. So the filtration we constructed is exhaustive.
\end{proof}

\subsection{Quantized coordinate algebra}\label{ssec: quantize coordinate algebra}\

Let $R$ be a Noetherian ring. 
Following \cite[Section 1]{APW1} we define the {\it quantized coordinate algebra} $R[\dmU_q(\g)]$ as the maximal rational subrepresentation in $\Hom_{R}(\dmU_q(\g),R)$ (with respect to either left or right action of $\dmU_q(\g)$, the result does not depend on which action we choose by 
\cite[Section 1.30]{APW1}). By the definition, for $M\in \operatorname{Rep}(\dmU_q(\g))$ we have a natural isomorphism
\begin{equation}\label{eq:dual_module}
\Hom_{\dmU_q(\g)}(M,R[\dmU_q(\g)])=
\Hom_{\dmU_q(\g)}(M,\Hom_R(\dmU_q(\g),R)) \iso \Hom_R(M,R).    
\end{equation}
So,  $R[\dmU_q(\g)]=H^0(\dmU/R,R)$ in Section \ref{ssec:homological properties}.  If $M$ is projective over $R$, we have 
\begin{equation}\label{eq: R[U] relative R-inj}
\Ext^i_{\Rep(\dmU_q(\g))}(M, R[\dmU_q(\g)])=0, \quad \forall i>0.
\end{equation}

\begin{Rem}On $R[\dmU_q(\g)]$, we have two left $\dmU_q(\g)$-actions $\gamma, \delta$ as follows: 
\[(\delta(x)\cdot f)(u)=f(ux), \qquad (\gamma(x)\cdot f)(u)=f(S(x)u),\] for $x,u \in \dmU_q(\g)$ and $f\in R[\dmU_q(\g)]$. So $R[\dmU_q(\g)]$ is a $\dmU_q(\g)\otimes_R \dmU_q(\g)$-module (the first action is $\delta$ and the second action is $\gamma$). The left $\dmU_q$-action on $R[\dmU_q(\g)]$ defined via $\delta$ is what we considered in the induction functors in the previous sections as well as the one used  in \eqref{eq:dual_module}-\eqref{eq: R[U] relative R-inj}.  Moreover,  \eqref{eq:dual_module} is an isomorphism of $\dmU_q(\g)$-modules with $\dmU_q(\g)$-module structures on both sides as follows: the left $\dmU_q(\g)$-action on $R[\dmU_q(\g)]$ defined via $\gamma$ gives the left $\dmU_q(\g)$-module structure on $\Hom_{\dmU_q(\g)}(M, R[\dmU_q(\g)])$, meanwhile, the left $\dmU_q(\g)$-module structure on $\Hom_R(M,R)$ is defined as usual. 
\end{Rem}

\begin{Lem}\label{Lem:coord_alg_freeness}
The following claims are true:
\begin{enumerate}
    \item $\uCA[\dmU_{\uCA}(\g)]$ is a free over $\uCA$.
    \item $R[\dmU_q(\g)] \simeq R\otimes_{\uCA}\uCA[\dmU_{\uCA}(\g)].$
\end{enumerate}
\end{Lem}
\begin{proof}
(1) is essentially in \cite[Section 4.4]{AW} , while (2) is in \cite[Section 4.2]{AW} (thanks to the results explained in Sections \ref{SS_Weyl} and \ref{SS_Kempf_vanish} we can work with the ring $\uCA$ instead of the ring $\CA_1$ from \cite[Section 4.1]{AW}).  
\end{proof}


\begin{Rem} We have a natural morphism
\begin{equation}\label{eq:coord_alg_embed}R[\dmU_q(\g)]\otimes_R R[\dmU_q(\g)]\rightarrow R[\dmU_q(\g\times \g)]
\end{equation}
arising from 
$\Hom_R(\dmU_q(\g),R)\otimes_R \Hom_R(\dmU_q(\g),R)
\rightarrow \Hom_R(\dmU_q(\g)\otimes_R\dmU_q(\g),R)$. This morphism (\ref{eq:coord_alg_embed}) is an isomorphism. 
Indeed, thanks to Proposition  \ref{prop:rational_finiteness}, it is enough to show that 
$$\Hom_{\dmU_q(\g)\otimes \dmU_q(\g)}(M, R[\dmU_q(\g)]\otimes_R R[\dmU_q(\g)]) \iso
\Hom_{\dmU_q(\g)\otimes \dmU_q(\g)}(M, R[\dmU_q(\g\times\g)])$$
for all rational  $\dmU_q(\g)\otimes_R\; \dmU_q(\g)$-modules $M$ that are finitely generated over $R$.
This is an easy consequence of (\ref{eq:dual_module}) and its analog for $\g\times\g$. 
\end{Rem}

Here is our main result about the $\dmU_q(\g)\otimes_R \dmU_q(\g)$-module structure of $R[\dmU_q(\g)]$. 

\begin{Prop}\label{Prop:coord_good_filtration}
Equip $P_+$ with a total order refining the usual partial order and let $\lambda_1<\lambda_2<\ldots$ be the elements. There is an exhaustive filtration $\{0\}=M_0\subset M_1\subset\ldots$ on $R[\dmU_q(\g)]$ such that $M_i/M_{i-1}\simeq H^0_q(\lambda_i)\otimes_R H^0_q(\lambda_i^*)$.  
\end{Prop}

\begin{proof} 

Thanks to (2) of Lemma \ref{Lem:coord_alg_freeness}, it is enough to prove the claim when $\uCA\rightarrow R$ is an isomorphism. From (\ref{eq:dual_module})
combined with Proposition \ref{Prop:Kempf_vanishing} we conclude that 
\begin{equation}\label{eq:Hom_Weyl_coordinate}
    \Hom_{\dmU_q(\g)}(W_q(\lambda),R[\dmU_q(\g)]) \simeq H^0_q(\lambda^*), \;\;\forall \lambda\in P_+,
\end{equation}
here the $\dmU_q(\g)$-module structure on $R[\dmU_q(\g)]$ is defined via $\delta$. 

 We are now going to check condition (2) of Lemma \ref{Lem:good_filtration_equivalent_deformed} for $\g$ replaced with $\g\times \g$ -- note that $\dmU_q(\g)\otimes_R \dmU_q(\g) \simeq \dmU_q(\g\times\g)$ -- and $M=R[\dmU_q(\g)]$. The lemma is applicable when $R\cong \uCA$.
The claim that the $R$-module $R[\dmU_q(\g)]$ is isomorphic to the direct sum of finitely generated projective $R$-modules follows from (1) of 
Lemma \ref{Lem:coord_alg_freeness}. So we need to show that 
\begin{equation}\label{eq:Ext_vanish_coord}
    \Ext^i_{\Rep(\dmU_q(\g\times\g))}(W_q(\lambda)\otimes_R W_q(\mu), R[\dmU_q(\g)])=0,
\end{equation}
for all $i>0$ and dominant weights $\lambda, \mu$.

 We are going to check \eqref{eq:Ext_vanish_coord}. By Lemma \ref{lem: homological properties}, any object of the form $R[\dmU_q(\g \times \g)]\otimes_R Q$, where $Q$ is an injective $R$-module, is injective in $\Rep(\dmU_q(\g \x \g))$.

Every object in $\Rep(\dmU_q(\g \times \g))$ admits a resolution by injective objects of this form. By \eqref{eq:coord_alg_embed}, in $\Rep(\dmU_q(\g \x \g))$, 
\[R[\dmU_q(\g \times \g)]\otimes Q=R[\dmU_q(\g)]\otimes_R R[\dmU_q(\g)]\otimes_R Q.\]
Thanks to the isomorphism $\dmU_q(\g \x \g) \cong \dmU_q(\g)\otimes_R \dmU_q(\g)$, any $M\in \Rep(\dmU_q(\g \x \g))$ admits two left $\dmU(\g)$-actions. We have the functor 
\[ \Hom_{\Rep(\delta(\dmU_q))}(W_q(\mu), \bullet): \Rep(\dmU_q(\g \times \g)) \rightarrow \Rep(\dmU_q(\g)),\]
where we consider the first $\dmU_q(\g)$-action on modules in $\Rep(\dmU_q(\g \x \g)$. We use $\delta(\dmU_q)$ in the supscript since the first $\dmU_q(\g)$-action on $R[\dmU_q(\g)]$ is defined via $\delta$. This functor sends $R[\dmU_q(\g \times \g)]\otimes_R Q$ to 
\[ \Hom_{\delta(\dmU_q)}(W_q(\mu), R[\dmU_q(\g)]\otimes_R R[\dmU_q(\g)]\otimes_R Q) = R[\dmU_q(\g)]\otimes_R \Hom_R(W_q(\lambda), Q).\]
The $R$-module $W_q(\mu)$ is  free of finite rank hence $\Hom_R(W_q(\mu),Q)$ is injective. Therefore, $R[\dmU_q(\g)]\otimes_R \Hom_R(W_q(\mu), Q)$ is an injective object in $\Rep(\dmU_q(\g))$ (where $\Hom_R(W_q(\mu), Q)$  has  the trivial action).


Hence we have a spectral sequence with second page equal to 
\begin{equation}\label{eq:2nd_page}
\Ext^i_{\Rep(\dmU_q(\g))}(W_q(\lambda),
\Ext^j_{\Rep(\delta(\dmU_q))}(W_q(\mu), R[\dmU_q(\g)]))
\end{equation}
converging to $\Ext^{i+j}_{\Rep(\dmU_q(\g\times\g))}(W_q(\lambda)\otimes_R W_q(\mu), R[\dmU_q(\g)])$. 
Thanks to (\ref{eq:dual_module})-\eqref{eq: R[U] relative R-inj}, we have 
$\Ext^j_{\Rep(\delta(\dmU_q))}(W_q(\mu), R[\dmU_q(\g)]))=H^0_q(\mu^*)$ for $j=0$ and $\{0\}$ for $j>0$. So (\ref{eq:2nd_page}) becomes  
$\Ext^i_{\Rep(\dmU_q(\g))}(W_q(\lambda), H^0_q(\mu^*))$ for $j=0$ and zero else. This is $R^{\oplus \delta_{\lambda,\mu^*}}$ if $i=0$ and zero else. 
In particular, (\ref{eq:Ext_vanish_coord}) follows
and $R[\dmU_q(\g)]\in \Rep(\dmU_q(\g\times \g))$ admits a good filtration. 

The previous paragraph implies also that the only dual Weyl modules that appear in a good filtration of $R[\dmU_q(\g)]$ are of the form $H^0_q(\lambda)\otimes_R H^0_q(\lambda^*)$. Furthermore,
\[ \Hom_{\dmU_q(\g)\times \dmU_q(\g)}(W_q(\lambda)\otimes_R W_q(\lambda^*), R[\dmU_q(\g)])\cong R. \]
Therefore, our claim on the filtration holds.
\end{proof}
\begin{Rem}\label{rem: filtration on Oq[G]} Let us define a total order on $P_+ \times P_+$ as follows. Firstly, we equip the first component $P_+$ with the total order used in Proposition \ref{Prop:coord_good_filtration}. Secondly,  applying $-w_0$, where $w_0$ is the longest element in $W$, to the total order used in Proposition \ref{Prop:coord_good_filtration}, we obtain another total order on $P_+$. We then equip the second component $P_+$ with this new total order. Finally, we consider $P_+ \times P_+$ with the lexicographic order $(\lambda, 0) >(0, \mu)$. For this order on $P_+\x P_+$, the module $M_i$ in Proposition \ref{Prop:coord_good_filtration} will be  the maximal $\dmU_q(\g)\otimes_R\dmU_q(\g)$-subrepresentation of $R[\dmU_q(\g)]$ whose weights are bounded by $(\lambda_i, \lambda^*_i)$.
\end{Rem}

\subsection{Rational representations of $\dmU_q(\g)$ at roots of unity}\   \label{ssec: rat rep at roots}

In this section, we will work over a field $\BF$. Let $q:=\epsilon\in \BF$ be  a root of unity  of order $\ell$. We assume that $\ell_i \geq \max\{2, 1-a_{ij}\}_{1\leq j \leq r}$. We will review some results about the rational representations of $\dmU_\e(\g)$ at roots of unity.  Many arguments follow \cite{N23} and \cite{APW1}. To simplify the notations, tensor products without subscript in this section are over $\BF$.

We will need a technical lemma. Let $\hmU^\geqslant_\e(\g,P), \hmU^{*\geqslant}_\e(\g, P^*)$ be the idempotented versions of $\dmU^\geqslant_\e, \dmU^{*\geqslant}_\e$, respectively, defined as in  Section \ref{comparison with untwisted Fr}. Let $\hmU^{\leqslant}_\e(\g, P), \hmU^{*\leqslant}_\e(\g, P^*)$ be the idempotented versions of $\dmU^{\leqslant}_\e, \dmU^{*\leqslant}_\e$, respectively.
\begin{Lem}\label{lem: normality of u}(a) The kernel of $\Fr: \hmU_\e(\g, P)\rightarrow \hmU^*_\e(\g, P^*)$ is equal to the left ideal generated by $\{ 1_\mu, E_i 1_\lambda, F_i 1_\lambda| \mu \not  \in P^*, \lambda \in P, 1\leq i \leq r\}$ and also equal to the right ideal generated by the same set of elements.

(b) The kernel of $\Fr^{\geqslant}: \hmU^\geqslant_\e(\g, P)\rightarrow \hmU^{*\geqslant}_\e(\g, P^*)$ is equal to the left ideal generated by $\{ 1_\mu, E_i1_\lambda|\mu \not \in P^*, \lambda \in P, 1\leq i \leq r\}$ and also equal to the right ideal generated by the same set of elements. There is the similar statement for $\Fr^\leqslant: \hmU^\leqslant_\e(\g, P)\rightarrow \hmU^{*\leqslant}_\e(\g, P^*)$.
\end{Lem}
\begin{proof}(a) {\it Step 1:} We have a  result similar to Corollary \ref{cor: kernel of Fr}(b) for $\Fr$ so that $\Ker(\Fr)$ is a two-sided ideal generated by 
\[ \{ 1_\mu, E^{(n)}_i 1_\lambda, F^{(n)}_i 1_\lambda| \mu \not \in P^*, \lambda \in P, \ell_i \nmid n\}.\]
For $\ell_i \nmid n$, let $n=\ell_i n_1+n_0$ with $0<n_0<\ell_i$ then 
\[ E_i^{(n)}=\frac{1}{(n_0)_{\e_i}}E_i^{(\ell_i n_1)}E_i^{n_0}, \qquad F_i^{(n)}=\frac{1}{(n_0)_{\e_i}} F_i^{(\ell_i n_1)}F_i^{n_0}.\]
Therefore, $\Ker(\Fr)$ is the two-sided ideal generated by 
\[ S= \{ 1_\mu, E_i 1_\lambda, F_i 1_\lambda| \mu \not \in P^*, \lambda \in P, 1\leq i \leq r\}.\]

\noindent
{\it Step 2:} We will prove that $\Ker(\Fr)$ is the left ideal generated by $S$. The proof that $\Ker(\Fr)$ is the right ideal generated by $S$ is similar. We write $\hmU$ for $\hmU_\e(\g, P).$
Let $\hmU S$ denote the left ideal  of $\hmU$ generated by $S$. Let $1_\mu \hmU, E_i 1_\lambda \hmU , F_i 1_\lambda \hmU $ be the right ideals generated by the corresponding elements in $S$. Then it is enough  to prove that these right ideals are contained in $\hmU S$.

Let  $\mu \not \in P^*$. We have: 
\[ 1_\mu E_j^{(n)}1_\nu =\delta_{\nu, \mu -n\a_j} E_j^{(n)}1_{\mu -n\a_j}, \qquad 1_\mu F_j^{(n)}1_\nu =\delta_{\nu, \mu +n \a_j} F_j^{(n)}1_{\mu+n\a_j}.\]
If  $\ell_j \nmid n$ then $1_\mu E_j^{(n)}1_\nu , 1_\mu F_j^{(n)}1_\nu \in \hmU S$. On the other hand, if $\ell_j \mid n$ then $\mu \pm n \a_j \not \in P^*$, hence, $E_j^{(n)}1_{\mu -n\a_j}, F_j^{(n)}1_{\mu +n \a_j} \in \hmU S$. Therefore, $1_\mu \hmU \subset \hmU S$ for $\mu \not \in P^*$.

Let us consider $E_i 1_\lambda$. We have
\begin{equation}\label{eq: commutator with E1-lambda}
\begin{split}
    E_i 1_\lambda F_j^{(n)}1_\nu&= \delta_{\lambda+n\a_j}F_j^{(n)}E_i 1_\nu \qquad (i \neq j), \qquad E_i1_\lambda E_i^{(n)}1_\nu =\delta_{\lambda-n \a_i, \nu} E_i^{(n)} E_i 1_\nu\\
    E_i 1_\lambda F_i^{(n)}1_\nu &=\delta_{\lambda+n \a_i, \nu}\left(  F_i^{(n)}E_i 1_\nu+b_n F_i^{(n-1)}1_\nu\right)\\
    E_i1_\lambda E_j^{(n)}1_\nu &= \delta_{\lambda -n \a_j, \nu}\sum_{0 \leq s \leq -a_{ij}<\ell_i} a_s E_i^{(n-s)}E_j E^{(s)}_i1_\nu \qquad (i\neq j;~ n\geq 1-a_{ij})
\end{split}
\end{equation}
The last equality follows from {\it higher order quantum Serre relations} in \cite[$\mathsection 7.1.6$]{l-book} for some $b_n, a_s \in \BF$.

If $\ell_i \nmid n$ then $E_i 1_\lambda F_i^{(n)}1_\nu \in \hmU S$. On the other hand, if $\ell_i \mid n$ then $F_i^{(n)}E_i1_\nu +b_nF_i^{(n-1)}1_\nu \in \hmU S$. So we always have $E_i 1_\lambda F_i^{(n)}1_\nu \in \hmU S$. So \eqref{eq: commutator with E1-lambda} implies that $E_i 1_\lambda \hmU \subset \hmU S$. Similarly, we have $F_i 1_\lambda \hmU \subset \hmU S$. This finishes the proof of part (a). 

\noindent
(b) The proof is similar to part (a).   
\end{proof}
\begin{Rem}\label{rem: ideal EF in Lus form} By the same analysis in Lemma \ref{lem: normality of u}, one can show that the left ideal of $\dmU_\e(\g)$ generated by $\{ E_i, F_i\}_{1\leq i \leq r}$ is equal to the right ideal of $\dmU_\e(\g)$ generated by $\{E_i, F_i\}_{1\leq i\leq r}$. Therefore, the weight space $\dmU_\e(\g)_\nu$ with $\nu \not \in Q^*$ is contained in this left ideal of $\dmU_\e(\g)$.
\end{Rem}

\begin{defi}Let $\fu\,\dmU^0$ be the Hopf subalgebra of $\dmU_\e(\g)$ generated by $E_i, F_i, \dmU^0_\e$. Let $\fu^>$ (resp. $\fu^<$) be the subalgebras of $\dmU_\e(\g)$ generated by $E_i$ (resp. $F_i$). Let $\fu$ be the Hopf subalgebra of $\dmU_\e(\g)$ generated by $E_i, F_i, K^{\a_i}$. Here $1\leq i\leq r$.
\end{defi}
\begin{Lem}\label{lem: basis of uU}

(a) We have a triangular decomposition via the multiplication map:
\[ \fu^< \otimes \dmU^0_\e\otimes \fu^> \xrightarrow[]{\m} \fu \; \dmU^0\]

\noindent
(b) The algebra  $\fu^<$ has a $\BF$-basis $\{ F^{\vec{k}}:=F_{\b_1}^{k_1}\dots F_{\b_N}^{k_N}| 1\leq k_i \leq \ell_{\b_i}-1\}$. The algebra $\fu^>$ has a $\BF$-basis $\{E^{\vec{k}}:=E_{\b_1}^{k_1}\dots E_{\b_N}^{k_N}| 0\leq k_i \leq \ell_{\b_i}-1\}$. As a result, $\dim_\BF \fu^<=\dim_\BF \fu^>=\prod_{\a\in \Delta_+} \ell_\a$.
\end{Lem}
\begin{proof}By using commutations between $E_i, F_i, \dmU^0_\e$,  it is easy to show that the image of $\m$ is a subalgebra  containing all generators of $\fu \dmU^0$  hence $\m$ is surjective. Using the triangular decomposition of $\dmU_\e(\g)$ in Lemma \ref{triangular_DCK}, we see that $\m$ is injective. Hence part (a) follows. Let us prove part (b). The algebra $\fu^<$ is the image of the algebra morphism $\iota: \mU^<_\e \rightarrow \dmU^<_\e$. Recall the PBW-bases of $\dmU^<_\e$ and $\mU^<_\e$ in Lemma \ref{triangular_Lus} and Lemma \ref{triangular_DCK}. We have $\iota(F^{\vec{k}})= \prod_{j=1}^N [k_j]_{\e_{i_j}}! F^{[\vec{k}]}$. On the other hand, $[k_j]_{\e_{i_j}}=0$ if $k_j\geq \ell_{\b_j}$ and $\neq 0$ if $0 \leq k_j \leq \ell_{\b_j}-1$. This implies the statement of part (b) for $\fu^<$. The proof for $\fu^>$ is similar.
\end{proof}

\begin{defi} A $\fu\, \dmU^0$-module $M$ is called a rational representation of  type $1$ if there is a decomposition $M=\bigoplus_{\lambda \in P} M_\lambda$, where  $\dmU^0_\e$ acts on $M_\lambda$ via the character $\chi_\lambda$.
\end{defi}
Let us consider the idempotented version of $\fu\, \dmU^0$ to denoted by $\hat{\fu}$. The $R$-matrix $\CR$ in \eqref{R matrix for original U} is an element in $\hat{\fu}$: if there is  any entry $k_t$ of $\vec{k}$ such that $k_t \geq \ell_t $ then $c_{\vec{k}}=0$, meanwhile if all $k_t <\ell_t$ then $F^{[\cev{k}]}$ and $E^{[\cev{k}]}$ are elements in $\fu \dmU^0$. 

We can define the following induction functors:
\begin{align*}
H^0(\fu\,\dmU^0/\BF, -) &: \BF \text{-mod}\rightarrow \Rep(\fu\,\dmU^0),\\
H^0(\fu\,\dmU^0/\dmU^0, -)&: \Rep(\dmU^0) \rightarrow \Rep(\fu\,\dmU^0),\\
H^0(\dmU/\fu\,\dmU^0, -)&: \Rep(\fu\,\dmU^0) \rightarrow \Rep(\dmU).
\end{align*}
The following lemma is proved as \cite[Proposition 2.16]{APW1}
\begin{Lem}\label{lem: tensor and induction}(a) Let $V\in \Rep(\dmU_\e(\g))$ and $M\in \Rep(\fu\,\dmU^0)$. Then there is a natural isomorphism of $\dmU_\e(\g)$-modules:  $H^0(\dmU/\fu\,\dmU^0, M\otimes V)\cong H^0(\dmU/\fu\,\dmU^0, M)\otimes V$.

\noindent
(b) Let $V\in \Rep(\fu\,\dmU^0)$ and $N\in \Rep(\dmU^0_\e)$. Then there is a natural isomorphism of $\fu\,\dmU^0$-modules:  $H^0(\fu\,\dmU^0/\dmU^0, N\otimes V)\cong H^0(\fu\,\dmU^0/\dmU^0, N)\otimes V$. 
\end{Lem}
For any $\lambda \in P$, let us recall the character $\chi_\lambda: \dmU^0_\e\rightarrow\BF$ defined in \eqref{eq: defi of character for usual form}. Then we have a $\fu^> \dmU^0$-module $\BF_\lambda$ via the homomorphism $\fu^> \dmU^0\rightarrow \dmU^0 \xrightarrow[]{\chi_\lambda}\BF$. Set 
\[ \hW_\e(\lambda):=\fu\,\dmU^0\otimes_{\fu^> \dmU^0} \BF_\lambda.\]
Let $\tau$ be the anti-involution on $\fu\,\dmU^0$ defined by $\tau(E_i)=F_i, \tau(F_i)=E_i, \tau(u_0)=u_0$ for all $u_0\in \dmU^0$, $1\leq i \leq r$. For any finite dimensional module $M$ in $ \Rep(\fu\,\dmU^0)$, let $M^\tau:=\Hom_\BF(M, \BF)$ as vector spaces, where $\fu\,\dmU^0$ acts via $\tau$. The following lemma is standard:
\begin{Lem}\label{lem: simple uU-mod}
(a) $\hW_\e(\lambda)$ has a unique simple quotient $\hL_\e(\lambda)$. The assignment $\lambda \mapsto \hL_\e(\lambda)$ is one-to-one correspondence between $P$ and simple modules in $\Rep(\fu\,\dmU^0)$.

\noindent
(b) $\hL_\e(\lambda)^\tau \cong \hL_\e(\lambda)$. For any finite dimensional modules $M,N$ in $\Rep(\fu\,\dmU^0)$,
\[ \Ext^1_{\Rep(\fu\,\dmU^0)}(M,N) \cong \Ext^1_{\Rep(\fu\,\dmU^0)}(N^\tau, M^\tau).\]
\end{Lem}

Let $\lambda_\St:=\sum_i (\ell_i-1)\w_i$ and $\St:=L_\e(\lambda_\St)$, the simple $\dmU_\e(\g)$-module with highest weight $\lambda_\St$. We are now going to study properties of $\St$ as a $\fu\,\dmU^0$-module and show that  the functor $H^0(\dmU/\fu\,\dmU^0,-)$ is exact. 

The proof of the next two lemmas follow \cite[$\mathsection 8.2$]{N23}, where similar results are proved over complex number $\BC$.
\begin{Lem}\label{lem: restricted simple mod}For any $\lambda \in P_\ell:=\{ \sum_i a_i \w_i| 0\leq a_i \leq \ell_i-1, \;\forall 1\leq i \leq r\}$, the $\dmU_\e(\g)$-module $L_\e(\lambda)$ is still a simple $\fu\,\dmU^0$-module. 
\end{Lem}

\begin{proof}
{\it Step 1:} We will prove that $(L_\e(\lambda))^{\fu^>}$ is one dimensional.
The proof follows  \cite[Lemma 8.4]{N23}. 
Since $\lambda$ is a highest weight of $L_\e(\lambda)$, we have $\dim\Hom_{\fu^>\dmU^0}(\BF_\lambda, L_\e(\lambda))=1$.
It remains to show that $\Hom_{\fu^>\dmU^0}(\BF_\mu, L_\e(\lambda))=0$ for all $\mu \neq \lambda$, equivalently, $(L_\e(\lambda) \otimes \BF_{-\mu})^{\fu^> \dmU^0}=0$ for $\mu \neq \lambda$. Note that $\BF_{-\mu}$ is naturally a $\dmU^\geqslant_\e$-module. 

 For any $M\in \Rep(\dmU^\geqslant_\e(\g)) =\Rep(\hmU^\geqslant_\e(\g, P))$, let 
\[M^\bullet:=\{ m \in M| 1_\mu m=E_i1_\lambda m =0 \; \forall  1\leq i \leq r, \lambda \in P, \mu \not \in P^* \}.\]
By Lemma \ref{lem: normality of u}(b), $M^\bullet$ is a $\hmU^\geqslant_\e(\g,P)$-submodule and then $M^\bullet$ is a module over $\hmU^{*\geqslant}_\e(\g, P^*)$. 

Let us assume $(L_\e(\lambda)\otimes \BF_{-\mu})^{\fu^> \dmU^0}\neq 0$ for $\mu \neq \lambda$. Since this space is contained in $(L_\e(\lambda)\otimes \BF_{-\mu})^\bullet$, the latter is nonzero. Since $L_\e(\lambda)$ is a simple $\dmU_\e(\g)$-module, both $L_\e(\lambda)$ and $L_\e(\lambda)\otimes \BF_{-\mu}$ have one dimensional highest weight spaces which are $L_\e(\lambda)_\lambda$ and $L_\e(\lambda)_\lambda\otimes \BF_{-\mu}$, respectively. Therefore, $L_\e(\lambda)_\lambda\otimes \BF_{-\mu}$ is the highest weight vector space of $(L_\e(\lambda)\otimes \BF_{-\mu})^\bullet$. 

Since $\lambda \neq \mu$, there is a simple root $\a_i$ such that $(L_\e(\lambda)\otimes \BF_{-\mu})^\bullet$ has a nonzero vector  of a weight $\lambda -\mu -m \ell_i \a_i$ for some $m>0$ (if $\text{char}~ \BF=0$ the we can choose $m=1$). Then the weight space of $L_\e(\lambda)$ contains the $W$-orbit of $\lambda -m \ell_i \a_i$. However, $s_i(\lambda- m \ell_i \a_i)=\lambda+(m \ell_i-a_i) \a_i >\lambda$, contradiction. Therefore, if $\lambda \neq \mu$ then $(L_\e(\lambda) \otimes \BF_{-\mu})^{\fu^> \dmU^0} =0$.



\noindent
{\it Step 2:} If $L_\e(\lambda)$ is not simple as a $\fu\,\dmU^0$-module, then $L_\e(\lambda)$ must contain a non-zero $\fu\,\dmU^0$-submodule $N \subset \oplus_{\mu<\lambda} L_\e(\lambda)_\mu$. Note that $N^{\fu^>} \neq 0$, hence, $L_\e(\lambda)^{\fu^>}$ must be at least two-dimensional, contradiction. Hence $L_\e(\lambda)$ is simple over $\fu\,\dmU^0$.
\end{proof}

\begin{Lem}\label{lem: solce of uU-mod} The socle of $\hW_\e(\lambda)$ in the category of $\Rep(\fu^<\dmU^0)$ is one-dimensional and is generated by the lowest weight space.
\end{Lem}

\begin{proof}The proof follows \cite[Lemma 9.5]{N23}.
We have $\hW_\e(\lambda) \cong \fu^<$ as $\fu^<$-modules. It is enough to show that the space of invariants $(\fu^<)^{\fu^<}:=\{ u\in \fu^<| xu=0 \; \forall x\in \fu^<\}$ is one dimensional. This space of invariants is at least one-dimensional because it contains the one-dimensional lowest weight space of $\fu^<$.  Note that $\fu^<$ is $Q$-graded. So, any invariant vector decomposed  into the sum of  homogeneous invariant vectors. 

Let $u$ be an invariant vector of weight $\lambda$ in $\fu^<$, then $u(\sum_{\mu} \e^{(\lambda, \mu)}K^\lambda)$ is a homogeneous invariant vector in $\fu^{\leqslant}$. This gives us an embedding of vector space $(\fu^<)^{\fu^<}\hookrightarrow (\fu^\leqslant)^{\fu^\leqslant}$. But $\fu^{\leqslant}$ is a finite dimensional Hopf algebra and  the invariants in any finite dimensional Hopf algebra are 1-dimensional \cite[Theorem 2.1.3]{SM}. Therefore the dimension of $(\fu^<)^{\fu^<}$ must be one.   
\end{proof}

\begin{Lem}\label{lem: St-mod} The natural morphism $\hW_\e(\lambda_\St) \rightarrow \St$ is an isomorphim of $\fu\,\dmU^0$-modules. As a result, $\St\cong W_\e(\lambda_\St)$ as $\dmU_\e(\g)$-modules.
\end{Lem}

\begin{proof} Since $\St$ is a simple $\fu\,\dmU^0$-module by Lemma \ref{lem: restricted simple mod}, the natural morphism $\hW_\e(\lambda_\St)\rightarrow \St$ is surjective. The lowest weight of $\St$ is $w_0(\sum_i (\ell_i-1)\w_i) =\sum_i (\ell_i-1) w_0(\w_i)$ and the dimension of this lowest weight subspace is $1$. On the other hand,  $\hW_\e(\lambda_\St)$ has a unique one-dimensional lowest weight subspace of the weight 
\[ \sum_i (\ell_i-1) \w_i - \sum_{\a \in \Delta_+}(\ell_\a-1) \a =\sum_i(\ell_i-1) w_0(\w_i),\]
where the equality follows by properties of the root system in Lemma \ref{lem: properties of root systems}.d' below and the equality $\ell_\a =\ell_i$ if $\a =w(\a_i)$ for some $w\in W$. Hence the lowest weight space of $\hW_\e(\lambda)$ maps into $\St$. By Lemma \ref{lem: solce of uU-mod}, the morphism $\hW_\e(\lambda_\St)\rightarrow \St$ is injective.

Hence $\dim_\BF(\St)=\dim_\BF(\hW_\e(\lambda_\St))=\prod_{\a\in \Delta_+}\ell_\a$. On the other hand,  by the Weyl character formula, $\dim_\BF(W_\e(\lambda_\St))=\prod_{\a\in \Delta_+} \ell_\a$. Therefore, the surjective morphism of $\dmU_\e(\g)$-modules $W_\e(\lambda_\St)\rightarrow \St$ must be an isomorphism.
\end{proof}
\begin{Lem} $\St$ is projective  in $\Rep(\fu\,\dmU^0)$.
\end{Lem}
\begin{proof}With Lemma \ref{lem: simple uU-mod}, \ref{lem: restricted simple mod} and \ref{lem: St-mod}, the proof is the same as in \cite[Proposition 10.2]{j2} 
\end{proof}
 \begin{Lem}
The functor $H^0(\dmU/\fu\,\dmU^0,-)$ is exact.
 \end{Lem}
\begin{proof} Let $0\rightarrow M\rightarrow N \rightarrow P\rightarrow 0$ be an exact sequence in $\Rep(\fu\,\dmU^0)$. Let $\St^*$ be the dual of $\St$ in $\Rep(\dmU_\e(\g))$. Note that 
\begin{align*}
H^0(\dmU/\fu\,\dmU^0, M\otimes \St^*)&\cong (M\otimes \St^*\otimes \BF[\dmU_\e(\g)])^{\fu\,\dmU^0}\\
&\cong \Hom_{\fu\,\dmU^0}(\BF, M\otimes \St^*\otimes \BF[\dmU_\e(\g)])\\
&\cong \Hom_{\fu\,\dmU^0}(\BF, M\otimes \BF[\dmU_\e(\g)]\otimes \St^*)\\
&\cong \Hom_{\fu\,\dmU^0}(\St, M\otimes \BF[\dmU_\e(\g)])
\end{align*}
Since $\St$ is projective in $\Rep(\fu\,\dmU^0)$, it follows that we have an exact sequence
\[ 0\rightarrow H^0(\dmU/\fu\,\dmU^0, M\otimes \St^*)\rightarrow H^0(\dmU/\fu\,\dmU^0, N\otimes \St^*)\rightarrow H^0(\dmU/\fu\,\dmU^0,P\otimes \St^*)\rightarrow 0.\]
Since $\St^*\in \Rep(\dmU_\e(\g))$, by Lemma \ref{lem: tensor and induction}, $H^0(\dmU/\fu\,\dmU^0, M\otimes \St^*)\cong H^0(\dmU/\fu\,\dmU^0, M)\otimes \St^*.$ Therefore, the sequence
\[ 0\rightarrow H^0(\dmU/\fu\,\dmU^0, M)\rightarrow H^0(\dmU/\fu\,\dmU^0, N) \rightarrow H^0(\dmU/\fu\,\dmU^0, P)\rightarrow 0\]
is exact. 
\end{proof}


Let $\BF[\fu\,\dmU^0]:=H^0(\fu\,\dmU^0/\BF, \BF)$. By restriction, we have a natural morphism of $\fu\,\dmU^0$-modules: $\BF[\dmU_\e(\g)]\rightarrow \BF[\fu\,\dmU^0]$. 
 \begin{Lem}\label{lem: res coordinate rings} The morphism $\BF[\dmU_\e(\g)]\rightarrow \BF[\fu\,\dmU^0]$ is surjective.
 \end{Lem}
\begin{proof} Recall the characters $\chi_\nu: \dmU^0_\e\rightarrow \BF$ for $\nu \in P$. Set $D(\nu):=\fu\,\dmU^0/\fu\,\dmU^0 \text{Ker}(\chi_\nu)$, then $\BF[\fu\,\dmU^0]_\nu \cong \Hom_\BF(D(\nu), \BF)$ as vector spaces. The natural map $\fu^<\fu^> \rightarrow D(\nu)$ is an isomorphism of vector spaces, here $\fu^<\fu^>$ is the linear span of elements $u_1u_2$ with $u_1\in \fu^<, u_2\in \fu^>$.

For any dominant weight $\lambda \in P_+$, let $W'_\e(-\lambda)$ be the maximal rational quotient of $\dmU_\e(\g)\otimes_{\dmU^\leqslant_\e} \BF_{-\lambda}$. Then $W'_\e(-\lambda)$ is the lowest weight analog of the Weyl module $W_\e(\lambda)$. The following results in this paragraph are in \cite[$\mathsection 1$]{APW1}. Let $\lambda, \mu $ be dominant weights such that $\lambda-\mu=\nu$. Let $x_\lambda$ be the highest weight vector of $W_\e(\lambda)$ and $x_{-\mu}$ be the lowest weight vector of $W'_\e(-\mu)$. Then $x_{-\mu}\otimes x_\lambda$ generates $W'_\e(-\mu)\otimes W_\e(\lambda)$ as a $\dmU_\e(\g)$-module. Let $J(\mu, \lambda)$ be the left ideal of $\dmU_\e(\g)$ generated by $\text{Ker}(\chi_\mu), \text{Ann}_{\dmU^<_\e}x_\lambda$ and $\text{Ann}_{\dmU^>_\e}(x_{-\mu})$. Set $D(\mu, \lambda):=\dmU_\e(\g)/J(\mu, \lambda)$. Then the map $\dmU_\e(\g) \rightarrow W'_\e(-\mu)\otimes W_\e(\lambda)$  defined by $u \mapsto u(x_{-\mu}\otimes x_\lambda)$ factors through an isomorphism of vector spaces $D(\mu, \lambda) \cong W'_\e(-\mu)\otimes W_\e(\lambda)$. Furthermore, $\Hom_\BF(D(\mu, \lambda), \BF) \subset \BF[\dmU_\e(\g)]_\nu$.

Let us consider the following maps:
\[ \fu^< \rightarrow W_\e(\lambda), \quad u \mapsto u x_\lambda; \quad \quad \fu^> \rightarrow W'_\e(-\mu), \quad u \mapsto u x_{-\mu}.\]
One can show that for sufficiently large $\lambda, \mu$, the above two maps will be injective. From this, for sufficiently large $\lambda, \mu$ with $\lambda -\mu=\nu$, the  map $\fu^<\fu^> \rightarrow W'_\e(-\mu)\otimes W_\e(\lambda)$ defined by $u \mapsto u(x_{-\mu}\otimes x_\lambda)$ will be injective. On the other hand, the map $\fu\,\dmU^0 \rightarrow W'_\e(-\mu)\otimes W_\e(\lambda)$ defined by $u \mapsto u(x_{-\mu}\otimes x_\lambda)$ factors through the map $D(\nu)\rightarrow W'_\e(-\mu)\otimes W_\e(\lambda)$.  Therefore, for sufficiently large $\lambda, \mu$ with $\lambda-\mu=\nu$, the natural map $\fu\,\dmU^0 \rightarrow W'_\e(-\mu)\otimes W_\e(\lambda)$ factors through an injection $D(\nu)\hookrightarrow W'_\e(-\mu)\otimes W_\e(\lambda)$.

For any $\lambda, \mu $ such that $\lambda-\mu=\nu$, the inclusion $\fu\,\dmU^0\hookrightarrow \dmU_\e(\g)$ induces the natural map $D(\nu)\rightarrow D(\mu, \lambda)$. The previous two paragraphs imply that for sufficiently large $\lambda, \mu$ with $\lambda-\mu=\nu$, the map $D(\nu)\rightarrow D(\mu, \lambda)$ becomes injective, then $\Hom_\BF(D(\mu, \lambda), \BF) \rightarrow \Hom_\BF(D(\nu), \BF)$ is surjective. This implies that the map $\BF[\dmU_\e(\g)]_\nu \rightarrow \BF[\fu\,\dmU^0]_\nu$ is surjective. Hence $\BF[\dmU_\e(\g)]\rightarrow \BF[\fu\,\dmU^0]$ is surjective. 
\end{proof}
 Let us recall  the idempotented version of $\fu\,\dmU^0$ to be denoted by $\hat{\fu}$. By restriction, the Frobenius homomorphism $\Fr: \hmU_\e(\g, P)\rightarrow \hmU^*_\e(\g, P^*)$ gives rise to a homomorphism $\fr: \hat{\fu}\rightarrow \oplus_{\lambda \in P^*} \BF 1_\lambda$. The latter gives us the pullback functor $\fr^*: \Rep(\dmU^{*0}_\e)\rightarrow \Rep(\fu\,\dmU^0)$.
\begin{Lem}\label{lem: Ind and Fr} Let $M\in \Rep(\dmU^{*0}_\e)$. Then $H^0(\dmU/\fu\,\dmU^0, \fr^*(M))$ is contained in the image of the functor $\Fr^*: \Rep(\dmU^*_\e(\g))\rightarrow \Rep(\dmU_\e(\g))$.
\end{Lem}
\begin{proof} Recall $\hmU:=\hmU_\e(\g, P)=\bigoplus_\lambda \dmU_\e(\g) 1_\lambda$. Observe that $\hmU_\e(\g, P)$ is a bimodule over $\dmU_\e(\g)$. 

Let $\hom_{\fu\,\dmU^0}(\hmU, \fr^*(M))$ $\subset \Hom_{\fu\,\dmU^0}(\hmU, \fr^*(M))$ denote the subspace consisting of all homomorphisms $f$ that vanish on all except finitely many summands $\dmU_\e(\g) 1_\lambda$. We see that $\hom_{\fu\,\dmU^0}(\hmU, \fr^*(M))$ is a $\dmU_\e(\g)$-submodule of $\Hom_{\fu\,\dmU^0}(\hmU, \fr^*(M))$. Furthermore, 
\[\hom_{\fu\,\dmU^0}(\hmU, \fr^*(M))=\bigoplus_\lambda \Hom_{\fu\,\dmU^0}(\dmU_\e(\g)1_\lambda, \fr^*(M)).\]
So $\hom_{\fu\,\dmU^0}(\hmU, \fr^*(M))$ is also a module over $\hmU_\e(\g, P)$.

 By Lemma \ref{lem: normality of u}, the kernel of $\Fr: \hmU_\e(\g, P) \rightarrow \hmU^*_\e(\g, P^*)$ is equal to the left ideal generated by $\{E_i 1_\lambda, F_i 1_\lambda, 1_\mu| \lambda \in P, \mu \in P/P^*\}$ as well as the right ideal generated by the same set of elements. Therefore, the action of $\hmU_\e(\g, P)$ on $\hom_{\fu\,\dmU^0}(\hmU, \fr^*(M))$ factors through Frobenius morphism $\Fr: \hmU_\e(\g, P)\rightarrow \hmU^*_\e(\g, P^*)$. On the other hand, we can interpret the module $H^0(\dmU/\fu\,\dmU^0, \fr^*(M))$ as the maximal rational $\dmU_\e(\g)$-submodule of $\hom_{\fu\,\dmU^0}(\hmU, \fr^*(M))$.  Therefore, $H^0(\dmU/\fu \dmU, \fr^*(M))$ viewed as the maximal rational $\dmU_\e(\g)$-submodule of $\hom_{\fu \dmU^0}(\hmU, \fr^*(M))$ will be contained in the image of the functor $\Fr^*: \Rep(\dmU^*_\e(\g)) \rightarrow \Rep(\dmU_\e(\g))$. 
\end{proof}

\begin{Lem}\label{lem: mat coefficient} There is a finite dimensional module $V \in \Rep(\dmU_\e(\g))$ with the following property. For any finite dimensional module $M$ in $\Rep(\dmU_\e(\g))$ there is a finite dimensional module $N$ in $\Rep(\dmU^*_\e(\g))$
such that $M$ is isomorphic to a subquotient of  $\Fr^*(N) \otimes V$.
\end{Lem}
\begin{proof} The proof is in several steps.

{\it Step 1.} We have an identification $\BF[\dmU^{*0}_\e]\cong \bigoplus_{\lambda \in P^*} \BF_\lambda$. Thanks to this identification, $\BF[\dmU^{*0}_\e]$ becomes a $\fu\,\dmU^0$-module. There is a finite dimensional $\dmU^0_\e$-module $V_1$ such that $\BF[\dmU^0_\e]\cong V_1\otimes \BF[\dmU^{*0}]$ as $\dmU^0_\e$-modules. By Lemma \ref{lem: tensor and induction}, 
\[\BF[\fu\,\dmU^0]\cong H^0(\fu\,\dmU^0/\dmU^0, V_1)\otimes \BF[\dmU^{*0}_\e].\]
Note that $H^0(\fu\,\dmU^0/\dmU^0, V_1)$ is finite dimensional.

{\it Step 2}. We will show that there is a finite dimensional $\dmU_\e(\g)$-module $V$ in $\Rep(\dmU_\e(\g))$ such that   $\BF[\dmU^{*0}_\e]\otimes V \twoheadrightarrow \BF[\fu\,\dmU^0]$ in $\Rep(\fu\,\dmU^0)$. Indeed, since the trivial representation is a direct summand of $\BF[\dmU^{*0}]$, we have a surjective morphism of $\fu\,\dmU^0$-modules $\BF[\dmU^{*0}_\e]\rightarrow \BF$. This gives us a surjective morphism of $\fu\,\dmU^0$-modules 
\[ \BF[\fu\,\dmU^0] \cong H^0(\fu\,\dmU^0/\dmU^0, V_1)\otimes \BF[\dmU^{*0}_\e] \twoheadrightarrow H^0(\fu\,\dmU^0/\dmU^0, V_1).\]
Combining this with Lemma \ref{lem: res coordinate rings}, we have a surjective morphism of $\fu\,\dmU^0$-modules 
\[\BF[\dmU_\e(\g)]\twoheadrightarrow \BF[\fu\,\dmU^0] \twoheadrightarrow H^0(\fu\,\dmU^0/ \dmU^0, V_1).\]
Then any finite dimensional $\dmU_\e(\g)$-submodule $V$ of $\BF[\dmU_\e(\g)]$ that maps onto $H^0(\fu\,\dmU^0/\dmU^0, V_1)$ will work.

{\it Step 3}. The exactness of $H^0(\dmU/\fu\,\dmU^0,-)$ yields a surjective morphism 
\[H^0(\dmU/\fu\,\dmU^0, \BF[\dmU^{*0}_\e]) \otimes V\twoheadrightarrow H^0(\dmU/\fu\,\dmU^0, \BF[\fu\,\dmU^0])\cong \BF[\dmU_\e(\g)].\]
Note that $V$ is a $\dmU_\e(\g)$-module hence $H^0(\dmU/\fu\,\dmU^0, \BF[\dmU^{*0}_\e] \otimes V)\cong H^0(\dmU/\fu\,\dmU^0, \BF[\dmU^{*0}_\e])\otimes V$.

{\it Step 4}. Let us fix a module $V$ as in Step $2$.  Let $M$ be a finite dimensional module in $\Rep(\dmU_\e(\g))$. Then we have an embedding $M \hookrightarrow H^0(\dmU/\BF, M)\cong M_{triv}\otimes \BF[\dmU_\e(\g)]$, where  $M_{triv}$ is the same space as $M$ but with a trivial $\dmU_\e(\g)$-action. Consider the surjective homomorphism \[\pi: M_{triv} \otimes H^0(\dmU/\fu\,\dmU^0, \BF[\dmU^{*0}_\e])\otimes V\twoheadrightarrow M_{triv}\otimes \BF[\dmU_\e(\g)],\]
 We can find a finite dimensional  $\dmU_\e(\g)$-submodule $N$ of $M_{triv}\otimes H^0(\dmU/\fu\,\dmU^0, \BF[\dmU^{*0}_\e])$ such that the image of $N\otimes V$ under $\pi$ contains $M$. By Lemma \ref{lem: Ind and Fr}, $H^0(\dmU/\fu\,\dmU^0, \BF[\dmU^{*0}_\e])$ is contained in the image of the functor $\Fr^*$, hence so is $N$. Therefore, $M$ is a subquotient of $\Fr^*(N)\otimes V$. This completes the proof.    
\end{proof}

\subsection{Rational representations of $\hmU_q(\g,P)$}\
\label{ssec: rat of idem form}

Let us recall the idempotented versions $\hmU_q(\g,P)$, $\hmU^\geqslant_q(\g,P)$, $\hmU^\leqslant_q(\g, P)$.  
\begin{defi}A $\hmU_q(\g, P)$-module $M$ is called \emph{rational} if for any $m\in M$ we have 
\begin{enumerate}
    \item[(i)] There are only finitely many $\lambda \in P$ such that $1_\lambda m \neq 0$, and $\sum_\lambda 1_\lambda m =m$.
    \item[(ii)]There is $k >0$ such that $E_i^{[s]}1_\lambda m=0$ for all $s>k$ and all $i=1, \dots, r$.
    \item[(iii)] There is $k>0$ such that $F_i^{[s]}1_\lambda m=0$ for all $s>k$ and all $i=1,\dots, r$.
\end{enumerate}
\end{defi}
Then we define $\Rep(\hmU_q(\g,P))$ to be the category of all rational $\hmU_q(\g,P)$-modules. The categories $\Rep(\hmU_q^\geqslant(\g,P)),$ $\Rep(\hmU^\leqslant_q(\g,P)$ are defined similarly. The categories $\Rep^{fd}(\hmU_q(\g,P))$, $\Rep^{fd}(\hmU^\geqslant_q(\g,P)$ and $\Rep^{fd}(\hmU^\leqslant_q(\g,P)$ are the full subcategories of the corresponding categories consisting of all objects which are finitely generated over $R$.

There is a  natural equivalence of braided monoidal categories $\Rep(\dmU_q(\g))\rightarrow \Rep(\hmU_q(\g,P))$, here we equip $\Rep(\hmU_q(\g, P))$ with the braided structure via the $R$-matrix $\CR$ in \eqref{R matrix for original U}. Therefore, the following constructions and results in Sections \ref{SS_Weyl}-\ref{ssec: good filtration} carry over to $\Rep(\hmU_q(\g,P))$:
\begin{itemize}
    \item Joseph's induction  functor $\fJ: \Rep^{fd}(\hmU^\geqslant_q(\g,P))\rightarrow \Rep^{fd}(\hmU_q(\g,P))$ 
    \item The induction functor $H^0: \Rep(\hmU_q^\leqslant(\g,P))\rightarrow \Rep(\hmU_q(\g,P))$ and various invariants in Section \ref{ssec:homological properties}
    \item The Weyl module $W_q(\lambda):=\fJ(R_\lambda)$, the dual Weyl module $H^0_q(\lambda):=H^0(R_\lambda)$, the quantized coordinate algebra $R[\dmU_q(\g)]:= H^0(\dmU/R, R)$.
    \item The notion of good filtrations and the good filtration on the  $\dmU_q(\g) \x \dmU_q(\g)$-module $R[\dmU_q(\g)]$.
\end{itemize}

Let $\hmU^\sF_q(\g,P)$, $\hmU^{\geqslant,\sF}_q(\g,P)$, $ \hmU^{\leqslant, \sF}_q(\g,P)$ denote the new Hopf algebras obtained from $\hmU_q(\g,P)$, $\hmU^\geqslant_q(\g,P)$, $\hmU^\leqslant_q(\g,P)$ via the twist 
\[ \sF=\prod_{\lambda, \mu \in P} q^{\sum_{ij}\phi_{ij}(\w^\vee_i, \lambda)(\w^\vee_j, \mu)} 1_\lambda\otimes 1_\mu.\]
Then the category $\Rep(\hmU^\sF_q(\g,P)$ is just the abelian category $\Rep(\hmU_q(\g,P))$ with new  monoidal structure, similarly with $\Rep(\hmU^{\geqslant, \sF}_q(\g,P))$, $\Rep(\hmU^{\leqslant, \sF}_q(\g,P))$. Furtheremore, $\Rep(\hmU^\sF_q(\g,P)$ is braided with the $R$-matrix $\CR^\sF$ in \eqref{R matrix for U}.

\begin{Rem}The categories $\Rep(\hmU^\sF_q(\g,P))$ and $\Rep(\hmU_q(\g,P))$ are equivalent as braided monoidal categories.
\end{Rem}
\subsection{Rational representations of  $\dU_q(\g)$}
\begin{defi} A $\dU_q(\g)$-module $M$ is called \emph{rational} if the following conditions hold:
\begin{enumerate}
    \item[(i)] There is a weight decomposition $M=\bigoplus_{\lambda\in P} M_\lambda$, where $\dU^0_q$ acts on $M_\lambda$ via the character $\hat{\chi}_\lambda$.
    \item[(ii)] For any $m \in M$ there is $k>0$ such that $\tE_i^{(s)}m=0$ for all $s>k$ and all $i=1,\dots, r$.
    \item[(iii)] For any $m\in M$ there is $k>0$ such that $\tF_i^{(s)}m=0$ for all $s>k$ and all $i=1,\dots, r$.
\end{enumerate}
\end{defi}
We define $\Rep(\dU_q(\g))$ to be the category of all rational $\dU_q(\g)$-representations. Similarly, we can define the categories $\Rep(\dU^\geqslant_q)$, $\Rep(\dU^\leqslant_q)$, $\Rep(\dU^0_q)$. 

Then similarly to Sections \ref{SS_Weyl}-\ref{ssec: quantize coordinate algebra}, we can define the following objects
\begin{itemize}
    \item Joseph's induction functor $\fJ: \Rep^{fd}(\dU^\geqslant_q)\rightarrow \Rep^{fd}(\dU_q(\g))$. 
    \item The induction functor $H^0(\dU_q/\mathcal{B},-):  \Rep(\mathcal{B})\rightarrow \Rep(\dU_q(\g))$, here $\mathcal{B}$ is one of these algebras: $R, \dU^0_q, \dU^\geqslant_q, \dU^\leqslant_q$.
    \item The Weyl module $W_q(\lambda):=\fJ(R_\lambda)$ . The dual Weyl module $H^0_q(\lambda):=H^0(\dU_q/\dU^\leqslant_q, R_\lambda)$.
    \item The quantized coordinated algebra $R[\dU_q(\g)]$.
    \item The notion of good filtration.
\end{itemize}






The results in Sections \ref{SS_Weyl}-\ref{ssec: rat rep at roots} carry over to similar results in $\Rep(\dU_q(\g))$. The justification for this is as follows: we have an equivalence of braided  monoidal categories 
\[\Rep(\dU_q(\g))\rightarrow \Rep(\hmU^\sF_q(\g,P))\]
and equivalences of monoidal categories
\[ \Rep(\dU^\geqslant_q(\g))\rightarrow \Rep(\hmU^{\geqslant, \sF}_q(\g,P)), \qquad \Rep(\dU^\leqslant_q(\g))\rightarrow \Rep(\hmU^{\leqslant, \sF}_q(\g,P)),\]
combining with  the discussions in Section \ref{ssec: rat of idem form}.

We will need an analog of  Lemma \ref{lem: mat coefficient} for $\Rep(\dU_q(\g))$
\begin{Lem}\label{lem: mat coefficient of dU}Let $R=\BF$ and $q=\e\in \BF$ be a root of unity of order $\ell$ such that $\ell_i \geq \max\{ 2, 1-a_{ij}\}_{1\leq j\leq r}$.
There is a finite dimensional module $V\in \Rep(\dU_q(\g))$ with the following property.  For any finite dimensional module $M$ in $\Rep(\dU_\e(\g))$, there is a finite dimensional module $N$ in $\Rep(\dU_\e(\g))$ such that $M$  is isomorphic to a subquotient of $\tFr^*(V) \otimes V$.
\end{Lem}
The only difference here is that we use the morphism $\tFr^*$ instead of $\Fr^*$, but this nuisance does not cause issues by Lemma \ref{lem: comparison of Frs} that says $\tFr^*$ is $\Fr^*$ composed with a self-equivalence $\Phi^*$ of the category $\Rep(\dU^*_\e(\g))$.

We end this section with a result  about the submodule of $\dU_q(\g)$-invariants for the adjoint action on $R[\dU_q(\g)]$. Let us restate a result about the good filtration of $\dU_q(\g)\otimes_R \dU_q(\g)$-module $R[\dU_q(\g)]$.
\begin{Lem}\label{prop: good filtration on Oq[G]} Equip $P_+$ with a total order refining the usual partial order and let $\lambda_1<\lambda_2<\dots $ be the elements so that we have a total order on $P_+ \x P_+$ as in Remark \ref{rem: filtration on Oq[G]}. There is an exhaustive filtration $\{0\}=M_0\subset M_1\subset \dots $ on $R[\dU_q(\g)]$ such that $M_i/M_{i-1} \cong H^0_q(\lambda_i)\otimes H^0_q(\lambda^*_i)$. The module $M_i$ is the maximal $\dU_q(\g)\otimes_R\dU_q(\g)$-subrepresentation of $R[\dU_q(\g)]$ whose weights are bounded from the above by $(\lambda_i, \lambda^*_i)$.
\end{Lem}
The $\dU_q(\g)\otimes \dU_q(\g)$-module structure on $R[\dU_q(\g)]$ is as follows: \[(x\otimes y)f(u)=f(S(y)ux)\qquad \text{for all $x,y,u \in \dU_q(\g)$ and $f\in R[\dU_q(\g)]$}.\]
Let us make a small modification. Lemma \ref{prop: good filtration on Oq[G]} is the same if we consider the following $\dU_q(\g)\otimes \dU_q(\g)$-module structure on $R[\dU_q(\g)]$: \[(x\otimes y)f(u)=f(S(y)uS^2(x))\qquad \text{ for all $x,y,u \in \dU_q(\g)$ and $f\in R[\dU_q(\g)]$.}\]
Using this modified action and the coproduct map $\Delta': \dU_q(\g)\rightarrow\dU_q(\g)\otimes_R \dU_q(\g)$, we get the adjoint action of $\dU_q(\g)$ on $R[\dU_q(\g)]$ given by $(xf)(u)=f(\ad'_l(S(x))u)$ for all $x, u \in \dU_q(\g)$ and $f\in R[\dU_q(\g)]$.

For any $V\in \Rep(\dU_q(\g))$, we define an $R$-linear map 
\[V\otimes_R V^* \rightarrow R[\dU_q(\g)],\qquad  v\otimes f \mapsto c_{f, K^{-2\rho} v}, \forall\,\, v\in V, f\in V^*.\]
This is a homomorphism of $\dU_q(\g)\otimes_R \dU_q(\g)$-modules (note that we use the modified action of $\dU_q(\g)\otimes_R\dU_q(\g)$ on $R[\dU_q(\g)]$ here). 

The Weyl module $W_q(\lambda)$ is free over $R$ with a weight basis to be denoted by $\{v_i\}$. Hence $\End_R(W_q(\lambda))\cong W_q(\lambda)\otimes_R W_q(\lambda)^*$ as $\dU_q(\g)$-modules. The image of $\Id: W_q(\lambda)\rightarrow W_q(\lambda)$ in $W_q(\lambda)\otimes_R W_q(\lambda)^*$ is $\sum_{i} v_i\otimes v^*_i$, here $v^*_i$ is the dual weight basis of $W_q(\lambda)^*$. Hence $\sum_i v_i \otimes v_i^* \in (W_q(\lambda)\otimes_R W_q(\lambda)^*)^{\dU_q}$. 

Let $c_\lambda$ be the image of $\sum_i v_i\otimes v_i^*$ under the map $W_q(\lambda)\otimes_R W_q(\lambda)^* \rightarrow R[\dU_q(\g)]$. Then $c_\lambda \in R[\dU_q(\g)]^{\dU_q}$.  
\begin{Lem}\label{lem: basis of invariant of R[U]}Let us consider the adjoint action of $\dU_q(\g)$ on $R[\dU_q(\g)]$. The $\dU_q(\g)$-invariant part $R[\dU_q(\g)]^{\dU_q}$ is a free $R$-module with a basis $\{c_\lambda|\lambda \in P_+\}$.
\end{Lem}

\begin{proof}
{\it Step 1.} Since the dominant weights of the $\dU_q(\g)\otimes_R \dU_q(\g)$-module $W_q(\lambda_i)\otimes_R W_q(\lambda_i)^*$ are bounded by $(\lambda_i, \lambda^*_i)$, the image of $W_q(\lambda_i)\otimes_R W_q(\lambda_i)^*\rightarrow R[\dU_q(\g)]$ is contained in $M_i$.

{\it Step 2.}  Under the first $\dU_q(\g)$-module structure on $R[\dU_q(\g)]$ (that is defined via $\delta$), the  map 
\[ \Hom_{\dU_q(\g)}(W_q(\lambda_i), W_q(\lambda_i)\otimes_R W_q(\lambda_i)^*)\rightarrow \Hom_{\dU_q(\g)}(W_q(\lambda_i), R[\dU_q(\g)]),\]
is just the identity morphism $W_q(\lambda_i)^*\xrightarrow[]{\text{Id}} W_q(\lambda_i)^* \cong H^0_q(\lambda_i^*)$. Therefore, the composition
\[ W_q(\lambda_i)\otimes_R W_q(\lambda_i)^* \rightarrow M_i\twoheadrightarrow H_q(\lambda_i)\otimes_R H^0_q(\lambda^*_i)\]
just comes from the natural homomorphism $W_q(\lambda_i)\rightarrow H^0_q(\lambda_i)$.

Note that $(H^0_q(\lambda_i)\otimes_R H^0_q(\lambda^*_i))^{\dU_q}\cong \Hom_{\dU_q(\g)}(W_q(\lambda_i), H^0_q(\lambda_i)) \cong R$, here the first isomorphism uses the $R$-freeness of $W_q(\lambda_i)$ and the isomorphism $H^0_q(\lambda^*_i)\cong W_q(\lambda_i)^*$.
Hence we see that the image of $c_{\lambda_i}$ in $H^0_q(\lambda_i)\otimes_R H^0_q(\lambda^*_i)$ spans the free $R$-module of rank one $(H^0_q(\lambda_i) \otimes_R H^0_q(\lambda^*_i))^{\dU_q}$.

{\it Step 3.} We now proceed by induction to show that $M_i^{\dU_q}$ is a free $R$-module with basis $\{c_{\lambda_j}|0\leq j \leq i\}$. The base case $i=1$ holds by Step 2. Now we do the induction step. Assume that $M_{i-1}^{\dU_q}$ is a free $R$-module with the basis $\{ c_{\lambda_j}| 0 \leq j \leq i-1\}$. The module $M_{i-1}$ has a good filtration and the trivial $\dU_q(\g)$-module $R$ is just the Weyl module $W_q(0)$. Therefore, the short exact sequence 
\[0\rightarrow M_{i-1}\rightarrow M_i\rightarrow H^0_q(\lambda_i)\otimes_R H^0_q(\lambda^*_i)\rightarrow 0\]
gives us a short exact sequence
\[ 0\rightarrow (M_{i-1})^{\dU_q} \rightarrow (M_i)^{\dU_q} \rightarrow (H^0_q(\lambda_i)\otimes_R H^0_q(\lambda^*_i))^{\dU_q}\rightarrow 0.\]
Since $(H^0_q(\lambda_i)\otimes_R H^0_q(\lambda^*_i))^{\dU_q}$ is a free  $R$-module of rank one, by induction hypothesis, $(M_i)^{\dU_q}$ is a free $R$-module of finite rank. Furthermore, the image of $c_{\lambda_i}\in (M_i)^{\dU_q}$ spans $(H^0_q(\lambda_i)\otimes_R H^0_q(\lambda^*_i))^{\dU_q}$ by Step 2. Combining this with the induction hyhothesis again, it follows that  the elements $\c_{\lambda_j}$ for $0\leq j\leq i$ form an $R$-basis in $(M_i)^{\dU_q}$.

{\it Step 4.} Since the filtration $M_0\subset M_1\subset M_2\dots $ of $R[\dU_q(\g)]$ is exhaustive, the lemma follows by Step 3.
\end{proof}





\section{Reflection equation algebras}\label{sec: REA}

 Recall that an {\it $H$-module algebra} $A$ over a Hopf algebra $H$ is a  unital algebra $A$ with a left $H$-module structure such that 
 \begin{equation}
  h1_A=\varepsilon(h)1_A, \qquad h\cdot ab=\sum (h_{(1)}\cdot a)(h_{(2)}\cdot b) \qquad \forall\, a, b \in A,\, h\in H \,.
 \end{equation}
In this section, we show that one can twist an algebra structure on $R[\dU_q(\g)]$ by the $R$-matrix so that $R[\dU_q(\g)]$ becomes an $\dU_q(\g)$-module algebra with the $\dU_q(\g)$-module structure given by
\[ (xf)(y)=f(\ad'_l(S'(x))(y)), \qquad \forall x,y \in \dU_q(\g), f\in R[\dU_q(\g)].\]
The new algebra, denoted by $O_q[G]$ , is called the \emph{reflection equation algebra} in the literature. 
 
\subsection{Reflection equation algebra $O_q[G]$}\
\label{ssec: description of REA}

For a technical reason, see Remark \ref{rem: technical mod}, we will work with $\rU_q:=\dU_q^{\text{op,cop}}(\g)$, the Hopf algebra obtained from $\dU_q(\g)$ by taking the opposites of both product and coproduct structures. Let $\ad_l^\circ, \ad_r^\circ$ denote the left and right adjoint actions of $\rU_q$ on itself. We will use $\cdot^\op$ to denote the product on $\rU_q$ as well as the action of $\rU_q$ on its modules, to distinguish from the corresponding constructions for $\dU_q(\g)$. Let us begin by recording  Proposition-Definition 2.5 in \cite{KV}:
\begin{PropDef}\label{prop: Drinfeld twist}(a) Let $(H, 1_H, \mu, \varepsilon, \Delta)$ be a bialgebra  and let $F\in H\otimes H$ be an element such that 
\[ (\Delta \otimes \Id)(F) F_{12}=(\Id \otimes \Delta)(F)F_{23}, \qquad  (\varepsilon\otimes \Id)(F)=1_H=(\Id\otimes \varepsilon) (F).\]
For all $x\in H$ define 
\[ \Delta^F(x)=F^{-1}\Delta(x) F.\]
Then $H_F:=(H, 1_H, \mu, \varepsilon, \Delta_F)$ is a bialgebra. We say $F$ is a {\rm twist} for $H$.

\noindent
(b) Let $H$ be a bialgebra and $F\in H\otimes H$ a twist for $H$. Let $(A, 1_A, m)$ be  a unital left $H$-module algebra with multiplication $m: A\otimes A \rightarrow A$. Define a linear map 
\[ m_F: A\otimes A\rightarrow A, \qquad m_F(a\otimes b)=m(F(a\otimes b)).\]
Then $A_F:=(A, 1_A, m_F)$ is an algebra. The left $H$-module structure on $A$ turns $A_F$ into a left $H_F$-module algebra.

\noindent
(c) Let $(U, 1, \mu, \varepsilon, \Delta,\CR)$ be a braided  bialgebra with universal $R$-matrix $\CR$ and $H:= U^{\textnormal{cop}}\otimes U$ the product  bialgebra. Then the two following elements are twists for $H$: 
\[ F:=\CR_{13}\CR_{23} \in H\otimes H, \qquad \overline{F}:=\CR^{-1}_{24} \CR^{-1}_{14} \in H\otimes H.\]

\end{PropDef}

\begin{Lem}\label{lem: coproduct gives Hopf}With notations from Proposition \ref{prop: Drinfeld twist}.c), the coproduct of $U$ defines a bialgebra homomophism $\Delta: U\rightarrow (U^{\text{cop}}\otimes U)_F$ with $F=\CR_{13}\CR_{23}$.
\end{Lem}
\begin{proof}Since the algebra structure of $(U^{\text{cop}}\otimes U)_F$ is the same as of $U\otimes U$, we only need to check that the map is a coalgebra homomorphism. Let $\tilde{\Delta}, \mathring{\Delta}$ denote the coproducts of $(U^{\text{cop}}\otimes U)_F, U^{\text{cop}}\otimes U$, respectively. Then 
\begin{align*} \tilde{\Delta}(x\otimes y)&=\CR^{-1}_{23}\CR^{-1}_{13}\mathring{\Delta}(x\otimes y) \CR_{13}\CR_{23}\\
&=\CR^{-1}_{23}\CR^{-1}_{13}(x_{(2)}\otimes y_{(1)}\otimes x_{(1)}\otimes y_{(2)})\CR_{13}\CR_{23}\\
&=\CR^{-1}_{23}(x_{(1)}\otimes y_{(1)}\otimes x_{(2)}\otimes y_{(2)})\CR_{23}
\end{align*}
Therefore, 
\begin{multline*}
    \tilde{\Delta}(\Delta(x))=\tilde{\Delta}(x_{(1)}\otimes x_{(2)})=\CR^{-1}_{23}(x_{(1)}\otimes x_{(3)}\otimes x_{(2)}\otimes x_{(4)})\CR_{23}
    =x_{(1)}\otimes \CR^{-1}\Delta^{\op}(x_{(2)}) \CR\otimes x_{(3)}\\
    =x_{(1)}\otimes \Delta(x_{(2)})\otimes x_{(3)}=(\Delta\otimes \Delta)\Delta(x)
\end{multline*}
\end{proof}
There is a direct analog of the category of rational representations for the algebra $\Rep(\rU_q)$. We also have the quantized coordinate algebra $R[\rU_q]$ as in Section \ref{ssec: quantize coordinate algebra}. Moreover, by viewing $\rU_q$ and $\dU_q(\g)$ as the same $R$-module, we can identify the $R$-modules $R[\rU_q]$ and $R[\dU_q(\g)]$. 

The quantized coordinate algebra $R[\rU_q]$ naturally carries an algebra structure coming from the coproduct of $\rU_q$. Furthermore, $R[\rU_q]$ is a $\rU_q^{\text{cop}}\otimes \rU_q$-module algebra with the $\rU_q^{\text{cop}}\otimes \rU_q$-module structure defined by 
\[ (x\otimes y)f(u)=f(S'(x)\cdot^\op u \cdot^\op y),\]
for any $x,y,u \in \rU_q$ and $f\in R[\rU_q]$.

Let us  consider the following $R$-matrix for $\rU_q$
\begin{equation}\label{eq: R-matrix for rU} 
\oR:=(\CR^\sF_{21})^{-1}= \sum^{\nu, \mu \in P}_{\vec{k}\in \BZ_{\geq 0}^\BN} q^{(\nu, \mu)-(\nu+\mu, \kappa(\weight(\tE^{(\cev{k})}))}d_{\vec{k}} ( \tE^{(\cev{k})} \cdot^\op 1_\mu) \otimes (S'(\tF^{(\cev{k})}) \cdot^\op 1_\nu),
\end{equation}
with the coefficients $d_{\vec{k}} \in R$ in which $d_{\vec{0}}=1$

We have the following lemma by applying the construction in Proposition \ref{prop: Drinfeld twist} to 
\[H=\rU_q^{\textnormal{cop}}\otimes \rU_q,\qquad  A=R[\rU_q], \qquad F=\oR_{13} \oR_{23},\]
\begin{Lem} $R[\rU_q]_F$
is a left $(\rU_q^{\text{cop}}\otimes \rU_q)_F$-module algebra.
\end{Lem}

\begin{Rem}$\oR$ is not an element in $\rU_q\otimes \rU_q$ and $\rU_q$ is not quasi-triangular. However the algebra $R[\rU_q]_F$ is still well-defined. This is because any element of $R[\rU_q]$ is a matrix coefficient of some finitely generated $R$-module object in $\Rep(\rU_q)$.
Hence, in the formula of Proposition \ref{prop: Drinfeld twist}.b applying to $R[\rU_q]$, all but finitely many components of $F=\oR_{13}\oR_{23}$ act by zero. 
\end{Rem}

By Lemma \ref{lem: coproduct gives Hopf}, the coproduct $\Delta^\op$  of $\rU_q$ is a Hopf algebra homomorphism $\Delta^\op: \rU_q\rightarrow (\rU_q^{cop}\otimes \rU_q)_F$. So that we obtain the following
\begin{Lem}\label{lem: module algebra Oq[G]}$R[\rU_q]_F$ is a left $\rU_q$-module algebra with the $\rU_q$-module structure defined by 
\[ (x\cdot^\op f)(u)=f(\ad_r^\circ(x)(u)),\]
for $x, u \in \rU_q$ and $f\in R[\rU_q]_F$.
\end{Lem}
We arrive at the definition of $O_q[G]$.
\begin{defi}\label{def Oq[G]}The {\it reflection equation algebra} $O_q[G]$ is the $R$-module $R[\dU_q(\g)]$ with the algebra structure of $R[\rU_q]_F$. Here we identify the $R$-modules $R[\dU_q(\g)]$ and $R[\rU_q]_F$ as above. 
\end{defi}

\begin{Lem}\label{lem: Oq[G] is left module algebra} $O_q[G]$ is a left $\dU_q(\g)$-module algebra with the $\dU_q(\g)$-module structure given by 
\[ (xf)(u)=f(\ad'_l(S'(x)) (u)),\]
for $x, u \in \dU_q(\g)$ and $f\in O_q[G]$    
\end{Lem}
\begin{proof} We have that 
\begin{equation*}
(xf)(u)=f(\ad'_l(S'(x))(u))=f(\ad^\circ_r(S'(x))(u))=(S'(x)\cdot^\op f)(u),  
\end{equation*}
for $x,u \in \dU_q(\g)$ and $f\in O_q[G]$. In the expression $S'(x) \cdot^\op f$, we view $S'(x)$ as an element in $\oU_q$. Therefore, we have
\begin{equation*}
    x(fg)=S'(x) \cdot^\op (fg)=(S'(x_{(1)})\cdot^\op f)(S'(x_{(2)}) \cdot^\op g)=(x_{(1)} f)(x_{(2)} g),
\end{equation*}
for $x\in \dU_q(\g)$ and $f, g \in O_q[G]$.
\end{proof}
\begin{Rem} \label{rem: technical mod} One can apply the construction in Proposition \ref{prop: Drinfeld twist}.b) directly with $\dU_q(\g)$, however, the  algebra $O_q[G]$  obtained will be a left $\dU_q(\g)$-module structure with the $\dU_q(\g)$-module structure defined by $(xf)(u)=f(\ad'_r(x))(u)$ for $x, u \in \dU_q(\g)$ and $f\in O_q[G]$. Later, we want the $\dU_q(\g)$-module structure as in Lemma \ref{lem: Oq[G] is left module algebra}, which is why we need to work with $\rU_q$.
\end{Rem}

\begin{Rem}\label{rem: Oq[G] is functorial} The construction of $O_q[G]$ is functorial with respect to the base ring. Namely, for any $\CA_\sN$-algebra $R$, the algebra $O_{\CA_\sN}[G]\otimes_{\CA_{\sN}} R$ is the same as $O_R[G]$ (we already know the $R$-module isomomorphism $O_{\CA_\sN}[G]\otimes_{\CA_\sN} R \cong O_R[G]$ by Lemma \ref{Lem:coord_alg_freeness}).
\end{Rem}

\subsubsection{Categorical construction of $O_q[G]$} \label{ssec: cat def of REA}
We end  this subsection by interpreting the algebra structure of $O_q[G]$ using the categorical language. For any $X\in \Rep(\rU_q)$, set $X^*:=\Hom_R(X, R)\in \Rep(\rU_q)$. Then 
\begin{align}
  O_q[G] \simeq \bigsqcup X^*\otimes X/\sim,
\end{align}
where $X$ runs over  all modules in $\Rep(\rU_q)$ that are finitely generated over $R$. The equivalence relation $\sim$ is given by  
\[ 
  \phi^* y^*\otimes x \sim y^*\otimes \phi x \qquad  \forall \, \phi\colon X\rightarrow Y, x\in X, y^*\in Y^* \,.
\]
 Let $\sigma_{-,-}$ be the braiding on $\Rep(\rU_q)$. Then the multiplication $\mm$ on $O_q[G]$ can be described as follows: 
\begin{equation}\label{eq: cat rel of mult}\begin{tikzcd}[column sep=large] X^*\otimes X\otimes Y^*\otimes Y \arrow[r, "\sigma_{X^*\otimes X, Y^*}"]\arrow[d] & Y^*\otimes X^*\otimes X\otimes Y \arrow[r]& (X\otimes Y)^*\otimes (X\otimes Y) \arrow[d]\\
 O_q[G]\otimes O_q[G]\arrow[rr, "\mm"]&&O_q[G]
 \end{tikzcd}
\end{equation}
One routinely  recovers the  algebraic description of algebra structure on $O_q[G]$ constructed above from \eqref{eq: cat rel of mult}, c.f.  \cite[$\mathsection 2.3$]{jo1}

\subsection{The case when $R=\BF$ is a field}\

 Let $\Rep^{fd}(\rU_q)$ be the full subcategory of $\Rep(\rU_q)$ consisting of all finite dimensional representations. Recall that any element  of $\BF[\rU_q]$ can be represented  as a linear combination of  matrix coefficients of some object in $\Rep^{fd}(\rU_q)$. Let $X\in \Rep^{fd}(\rU_q)$ and $X^*=\Hom_\BF(X,\BF)$. Elements $v\in X$ and $f\in X^*$ give rise to the  matrix coefficient $\rc_{f,v}$. We use the notation $\rc_{-,-}$ to distinguish these matrix coefficients from the matrix coefficients  $c_{-,-}$ of $\dU_q(\g)$. 
 
 The Hopf algebra structure of $\BF[\rU_q]$ can be described as follows:
\begin{align*}
  \rc_{f,v}\cdot \rc_{g,w}=\rc_{f\otimes g, v\otimes w} \,, \quad \Delta(\rc_{f,v})=\sum \rc_{f,v_i}\otimes \rc_{v_i^*, v} \,, \\
  \varepsilon(\rc_{f,v})=f(v) \,, \qquad  [S^{\pm 1}(\rc_{f,v})](u)=f((S^{\pm 1}u)v)\,,
\end{align*}
where $v\in V, f\in V^*$ and $w\in W$ with $V, W\in \Rep^{fd}(\rU_q)$; $u \in \rU_q$ and  $\{v_i\},\{v^*_i\}$ are dual bases of $V, V^*$.

Let us define a bilinear form $\r: \BF[\rU_q]\otimes \BF[\rU_q]\rightarrow \BF$ by 
\[ \r(\rc_{f,v}, \rc_{g,w})= \<f\otimes g, \oR(v\otimes w)\>,\]
with $\oR$ from \eqref{eq: R-matrix for rU}. This bilinear form makes $\BF[\rU_q]$  a {\it coquasi-triangular Hopf algebra} according to the following definition, see \cite[$\mathsection 10.1.1$]{KS} or \cite[Definition 1.1]{KV}.

\begin{defi}
A coquasi-triangular bialgebra/Hopf algebra $(A,\r)$ over $\k$ is a pair consisting of a bialgebra/Hopf algebra $A$ over $\k$ 
with a convolution invertible linear map $\r\colon A\otimes A\rightarrow \k$ which satisfies the following relations:
\begin{align*}
  \r(a_{(1)}\otimes b_{(1)})a_{(2)}b_{(2)} &= b_{(1)}a_{(1)}\r(a_{(2)}\otimes b_{(2)}) \,,\\
  \r(ab\otimes c) &= \r(a\otimes c_{(1)})\r(b\otimes c_{(2)}) \,,\\
  \r(a\otimes bc) &= \r(a_{(1)}\otimes c) \r(a_{(2)}\otimes b) \,.
\end{align*}
\end{defi}
We can twist the algebra structure on $\BF[\rU_q]$ as follows, see \cite[Proposition-Definition 2.1]{KV}
\begin{equation}\label{eq:coquasi-twist} a \cdot_{\r} b:=\r(a_{(1)}, b_{(2)})\r(a_{(3)}, Sb_{(1)}) a_{(2)} b_{(3)}=\r(a_{(2)}, b_{(3)})\r(a_{(3)}, Sb_{(1)})b_{(2)} a_{(1)},
\end{equation}
with $a, b\in \BF[\rU_q]$.  
\begin{Lem}[Lemma 2.6, \cite{KV}] The algebra $O_q[G]$ can be obtained via the twist \eqref{eq:coquasi-twist}.
\end{Lem}



\subsection{The case when $\BF=\BQ(v^{1/\sN})$ and $q=v$}\

Set $\BF[\rU_q]^*:=\Hom_\BF(\BF[\rU_q], \BF)$. This is naturally an algebra. Let us define a bilinear map $\textbf{q}(\cdot,\cdot): \BF[\rU_q]\otimes \BF[\rU_q]\rightarrow \BF$ via:
\[ \textbf{q}(a\otimes b)= \r(b_{(1)}, a_{(1)})\r(a_{(2)}, b_{(2)}),\qquad \text{for $a, b\in \BF[\rU_q]$.}\]
 This gives rise to a linear map $\tl_\r: \BF[\rU_q] \rightarrow \BF[\rU_q]^*$ via 
\[ \tl_\r(a)=\textbf{q}(a, \cdot), \qquad \text{for $a\in \BF[\rU_q]$.}\]
When $\BF=\BQ(v^{1/\sN})$, the natural Hopf pairing $\BF[\rU_q] \x \rU_q\rightarrow \BF$ is non-degenerate in the second argument, see  \cite[Proposition 5.11]{j}. So we have an algebra embedding $\rU_q\hookrightarrow \BF[\rU_q]^*$.
\begin{Lem}\label{lem: tilde(l)r morphism}
(a) $\tl_\r(\BF[\rU_q]) \subset \rU_q$.

\noindent
(b) The map $\tl_\r: O_q[G]\rightarrow \rU_q$ is a homomorphism of left $\rU_q$-module algebras, where the $\rU_q$-action on $O_q[G]$ is in Lemma \ref{lem: module algebra Oq[G]} and the $\rU_q$-action on $\rU_q$ is the left adjoint action $\ad_l^\circ$.
\end{Lem}

\begin{proof}(a) Let $\oR=\sum_{i \in I} s_i \otimes t_i$. Let $V, W\in\Rep^{fd}(\rU_q)$. Let $v,w,f,g$ be weight vectors in $V,W, V^*, W^*$, respectively. Let $\{v_l\}, \{w_k\}$ be weight bases of $V, W$, respectively. Let $\{v^*_l\}, \{w^*_k\}$ be the corresponding dual bases of $V^*,W^*$. Then we have
\begin{equation}\label{eq: computation of q(.,.)}
\begin{split}
    \q(\rc_{f,v}, \rc_{g,w})&=\sum_{l,k}\Big\< g\otimes f, \oR(w_k\otimes v_l)\Big\>\Big\<v^*_l \otimes w^*_k, \oR(v\otimes w)\Big\>\\
                &=\sum_{l,k,i,j} g(s_i w_k)f(t_i v_l)v^*_l(s_j v)w^*_k(t_j w)\\
                &=g(\sum_{i,j} s_it_jf(t_i s_j v)w)
\end{split}
\end{equation}
Using \eqref{eq: computation of q(.,.)} and \eqref{eq: R-matrix for rU}, we get 
\begin{multline*}
\tl_\r(\rc_{f,v})=\sum_{\vec{k}, \vec{r}\in \BZ^N_{\geq 0}}  d_{\vec{k}}d_{\vec{r}}q^{-(\weight(v), \kappa(\weight(\tE^{(\cev{k})}))+\kappa(\weight(\tE^{(\vec{r})}))}f(S'(\tF^{(\cev{k})}) \cdot^\op \tE^{(\cev{r})}(v))\\
\tE^{(\cev{k})}\cdot^\op S'(\tF^{(\cev{r})}) \cdot^\op K^{2\weight(v)-\kappa(\weight(\tE^{(\cev{k})}))-\kappa(\weight(\tE^{(\vec{r})}))}.
\end{multline*}
Then we see that $\tl_\r(\rc_{f,v})\in \rU_q$ since $\kappa(\weight(\tE^{(\cev{k})}))\in 2P$ for all $\vec{k} \in \BZ^N_{\geq 0}$.


\noindent
(b) This follows by Proposition 11 (equations 10.1.3(28)-(30)) in \cite[$\mathsection 10.1.3$]{KS}, see also Proposition 1.7 and 2.8 in \cite{KV}. Note that in equation 10.1.3(28)  \cite{KS}, the term $g(a_{(1)}S(a_{(3)}))$ should be replaced by $g(S(a_{(1)})a_{(3)})$. This misprint has been noticed in \cite[Proposition 1.7]{KV}. 
\end{proof}

Let us consider the following composition:
\begin{equation}\label{eq: def of lr} l_\r: O_q[G]\xrightarrow[]{\tl_\r} \rU_q\xrightarrow[]{(S')^{-1}} \dU_q(\g).
\end{equation}
For each $\lambda \in P_+$, let $L_q(\lambda)$ be the simple module of highest weight $\lambda$ in $\Rep(\dU_q(\g))$. Let $v_\lambda$ be a nonzero highest weight vector in $L_q(\lambda)$ and $v^*_\lambda$ be the dual weight vector in $(L_q(\lambda))^*$. So we have the matrix coefficient $c_{v^*_\lambda, v_\lambda} \in O_q[G]$.

\begin{Lem}\label{lem: lr morphism}
(a) The map $l_\r$ is a morphism of left $\dU_q(\g)$-module algebras, where the $\dU_q(\g)$-action on $O_q[G]$ is  in  Lemma \ref{lem: Oq[G] is left module algebra} and the $\dU_q(\g)$-action on $\dU_q(\g)$ is the left adjoint action $\ad'_l$.

\noindent
(b) $l_\r(c_{v^*_\lambda, v_\lambda})=K^{-2\lambda}$.
\end{Lem}
\begin{proof}(a) Both $\tl_\r$ and $(S')^{-1}$ are algebra homomorphisms hence so is $l_\r$. It remains to check that $l_\r$ is $\dU_q(\g)$-linear:
\begin{multline*}
  l_\r(x f) = (S')^{-1}[\tl_\r(xf)] = 
  (S')^{-1}[\tl_\r(S'(x)\cdot^\op f)] = (S')^{-1}[\ad^\circ_l(S'(x))(\tl_\r(f))] = \\ 
  \ad'_l(x)((S')^{-1}[\tl_\r(f)]) = \ad'_l(x)(l_\r(f)) \,,
\end{multline*}
for $x\in \dU_q(\g)$ and $ f\in O_q[G]$. In the expression $S'(x) \cdot^\op f$, we view $S'(x)$ as an element in $\oU_q$.

\noindent
(b) Any left $\dU_q(\g)$-module $V$ acquires the left $\rU_q$-module structure as follows: $x\cdot^\op v=S'(x) v$ for all $x\in \rU_q$ and $v\in V$.  Let $\mathring{L}_q(\lambda)$ be the left $\rU_q$-module obtained from $L_q(\lambda)$. So we have a matrix coefficient $\rc_{v^*_\lambda, v_\lambda}$ of $\rU_q$. Note that $v_\lambda$ has the weight $-\lambda$ in $\mathring{L}_q(\lambda)$.

It turns out that $c_{v^*_\lambda, v_\lambda}=(S')^{-1}\rc_{v^*_\lambda, v_\lambda}$ in $O_q[G]$. Indeed,
\[ (S')^{-1}\rc_{v^*_\lambda, v_\lambda}(x)=v^*_\lambda((S')^{-1}x \cdot^\op v_\lambda)=v^*_\lambda(x v_\lambda)=c_{v^*_\lambda, v_\lambda}(x),\]
for all $x\in \rU_q$, where we identify $\rU_q$ with $\dU_q(\g)$  as $\BF$-vector spaces.

On the other hand,  notice that $v_\lambda$ has the weight $-\lambda$ in $\mathring{L}_q(\lambda)$. Then one argues as in the proof of part a) of Lemma \ref{lem: tilde(l)r morphism} to get
\[ \textbf{q}({S'}^{-1}\rc_{v^*_\lambda, v_\lambda}, \rc_{g,w})=g(q^{2(\lambda, \weight(w))} w)=g(K^{2\lambda} w),\]
where $w, g$ are weight vectors in $W, W^*$ for some $W\in \Rep^{fd}(\rU_q)$. It follows that
\begin{gather*}\tl_\r((S')^{-1}\rc_{v^*_\lambda, v_\lambda})=K^{2\lambda},\\
l_\r(c_{v^*_\lambda, v_\lambda})=(S')^{-1}(\tl_\r((S')^{-1} \rc_{v^*_\lambda, v_\lambda}))=(S')^{-1}(K^{2\lambda})=K^{-2\lambda}.
\end{gather*}
This finishes the proof.
\end{proof}

\begin{Lem}\label{lem: U-direct sum of Oq[G]}$O_q[G] \cong \bigoplus_{\lambda \in P_+} \dU_q(\g)c_{v^*_\lambda, v_\lambda}$.
\end{Lem}
\begin{proof}
For any $V\in \Rep^{fd}(\dU_q(\g))$, the map $V\otimes V^* \rightarrow O_q[G]$ defined by $v\otimes f\mapsto c_{f, K^{-2\rho}v}$ is  $\dU_q(\g)$-linear, where the $\dU_q(\g)$-action on $O_q[G]$ is from Lemma \ref{lem: Oq[G] is left module algebra}.
Then we have 
\[ O_q[G]\cong \bigsqcup_{X\in \Rep^{fd}(\dU_q(\g))} X\otimes X^*/ \sim,\]
where the equivalence relation $\sim$ is given by 
\[ x\otimes \phi^* y^* \sim \phi x\otimes y^* \qquad \forall \phi: X\rightarrow Y, x\in X, y^*\in Y^*.\]

Since we are working over $\BF=\BQ(v^{1/\sN})$, the category $\Rep^{fd}(\dU_q(\g))$ is semisimple with simple objects $\{ L_q(\lambda)|\lambda \in P_+\}$. Therefore, 
\[ O_q[G]\cong \bigoplus_{\lambda\in P_+} L_q(\lambda)\otimes L_q(\lambda)^*.\]
On the other hand, it is easy to show that $L_q(\lambda) \otimes L_q(\lambda)^*$ is generated by $(K^{2\rho}v_\lambda) \otimes v^*_\lambda=q^{(2\rho, \lambda)}v_\lambda \otimes v^*_\lambda$ over $\dU_q(\g)$. Moreover, $(K^{2\rho}v_\lambda)\otimes v^*_\lambda$ is mapped to $c_{v^*_\lambda, v_\lambda}$ under the morphism $L_q(\lambda)\otimes L_q(\lambda)^* \rightarrow O_q[G]$ defined in the beginning of the proof.  Therefore, the lemma follows.
\end{proof}


\section{Isomorphism between REA and the locally finite part}\label{sec:REA as locfin}
In this section, let $R$ be an $\CA_\sN$-algebra where $\CA_\sN=\CA[v^{\pm 1/\sN}]$. 

\subsection{Ad-invariant bilinear form and the induced embedding} 
\

Let us recall the pairing $\<.,.\>': U^{ev}_q(\g) \x \dU_q(\g) \rightarrow R$ given by 
\begin{equation}\label{eq16}
  \left\<(yK^{\kappa(\nu)})K^\lambda (xK^{\gamma(\mu)}), (\dy K^{\kappa(\dnu)})u_0(\dx K^{\gamma(\dmu)})\right \>' = 
  (y,\dx)' \cdot (\dy,x)' \cdot \hat{\chi}_{-\lambda/2}(u_0)\cdot q^{(2\rho, \nu)},
\end{equation}
 for $y\in U^{ev<}_{-\nu}, x\in U^{ev>}_{\mu}, \dy\in \dU^{<}_{-\dnu}, \dx \in \dU^{>}_{\dmu}$, 
 $\nu, \mu, \dnu, \dmu \in Q_+$, and $\dlambda, \lambda \in 2P$,
see 
Proposition~\ref{Prop:pairing_whole_R}. 


This pairing induces an embedding of $\dU_q(\g)$-modules:
\begin{equation}\label{eq:iota-map}
  \iota\colon  U^{ev}_q(\g) \hookrightarrow \Big(\dU_q(\g)\Big)^* :=\text{Hom}_R(\dU_q(\g), R) \,.
\end{equation}

\begin{Lem}\label{lem: matrix coeff in Uev}
(a) Let $\phi\colon \dU^{<}_{-\nu} \x \dU^{>}_{\mu} \rightarrow R$ be a bilinear homomorphism and $\lambda\in P$. 
Then there is a unique element $u \in (U^{ev<}_{-\mu}K^{\kappa(\mu)})K^{2\lambda}(U^{ev>}_{\nu}K^{\gamma(\nu)})$ such that 
\begin{equation}\label{eq:elt-via-pairing} 
  \<u, (\dy K^{\kappa(\nu)})u_0(\dx K^{\gamma(\mu)})\>'=\phi(\dy,\dx)\hat{\chi}_{-\lambda}(u_0) 
  \qquad \forall\, \dx \in \dU^{>}_\mu, \dy \in \dU^{<}_{-\nu}, u_0\in \dU^0_q \,.
\end{equation}

\noindent
(b) The image of $\iota$ from~\eqref{eq:iota-map} contains $R[\dU_q(\g)]$.
\end{Lem}

\begin{proof}
(a) Recall the perfect pairings $(\cdot,\cdot)'\colon U^{ev<}_{-\mu} \x \dU^{>}_\mu \rightarrow R$ and 
$(\cdot,\cdot)'\colon \dU^{<}_{-\nu} \x U^{ev>}_{\nu} \rightarrow R$. We choose a basis $u_1^\mu, \dots , u_{r(\mu)}^\mu$ for $\dU^{>}_\mu$ 
and the dual basis $v_1^{-\mu}, \dots,v_{r(\mu)}^{-\mu}$ of $U^{ev<}_{-\mu}$. Similarly, we choose a basis 
$u_1^{-\nu}, \dots, u^{-\nu}_{r(\nu)}$ for $\dU^{<}_{-\nu}$ and the dual basis $v^\nu_1, \dots, v^\nu_{r(\nu)}$ 
of $U^{ev>}_{\nu}$. Then
\[
  u=\sum \phi(u^{-\nu}_j, u^{\mu}_i)q^{-(2\rho, \nu)}(v_i^{-\mu}K^{\kappa(\mu)})K^{2\lambda}(v^{\nu}_jK^{\gamma(\nu)})
\]
satisfies~\eqref{eq:elt-via-pairing}. The uniqueness of $u$ is due to the non-degeneracy of $\< \cdot,\cdot \>'$ 
in the first argument  by Proposition \ref{Prop:pairing_whole_R} .

(b) Let $V \in \Rep(\dU_q(\g))$ such that $V$ is finitely generated over $R$. Let $m\in V$ be an element of 
weight $\lambda$ and $f\in V^*:=\text{Hom}_R(V,R)$ be an element of weight $\lambda'$. It suffices to show 
that $c_{f, m}$ is contained in the image of~$\iota$ since $R[\dU_q(\g)]$ is spanned by these matrix coefficients over $R$. We note that $c_{f,m}(\dU^{<}_{-\nu}\dU^0_q \dU^{>}_\mu) \neq 0$ 
for only finitely many pairs $(\nu, \mu)\in Q_+\x Q_+$. For any such $\nu,\mu$ and all 
$\dx\in \dU^>_{\mu}, \dy\in \dU^<_{-\nu},u_0\in U^0_q$, we have 
\[ 
  c_{f,m}((\dy K^{\kappa(\nu)})u_0(\dx K^{\gamma(\mu)}))=
  f((\dy K^{\kappa(\nu)})u_0(\dx K^{\gamma(\mu)})m)=
  \hat{\chi}_{\mu+\lambda}(u_0)f((\dy K^{\kappa(\nu)})(\dx K^{\gamma(\mu)})m) \,.
\]
By part (a), there is $u_{\mu\nu}\in U^{ev<}_{-\mu}U^{ev,0}U^{ev>}_{\nu}$ such that $\< u_{\mu\nu},v \>'=c_{f,m}(v)$ for 
all $v\in \dU^{<}_{-\nu}\dU^0_q\dU^{>}_\mu$. Hence, the element $u=\sum_{\mu, \nu} u_{\mu\nu}$ ( all but finitely many 
$u_{\mu\nu}$ are $0$) satisfies $\iota(u)=c_{f,m}$.
\end{proof}

For any $\lambda\in P_+$, we recall the Weyl module $W_q(\lambda)$ from Section \ref{sec rational reps}. By Lemma \ref{Lem:Weyl_char_formula}, $W_q(\lambda)$ is free of finite rank over $R$ and has a basis consisting of weight vectors. Let $v_\lambda$ be a non-zero vector spanning the highest weight space $W_q(\lambda)_\lambda$. Then we can find a weight vector $v^*_\lambda\in W_q(\lambda)^*$ such that $v^*_\lambda(v_\lambda)=1$ and $v^*_\lambda(v_i)=0$ for other weight vectors in the $R$-basis of $W_q(\lambda)$.
\begin{Lem}\label{lem: image of iota(K(-2lambda))} We have $\iota(K^{-2\lambda})=c_{v^*_\lambda, v_\lambda}$.
\end{Lem}
\begin{proof}
The pairing $\<K^\lambda, u\>'$ equals to $\hat{\chi}_{-\lambda/2}(u)$ for any $u \in \dU^0_q$ and vanishes for all elements $u \in \bigoplus_{(\nu,\mu)\in Q_+ \x Q_+ \backslash(0,0)} \dU^<_{-\nu} \dU^0_q \dU^>_\mu$. On ther other hand, $c_{v^*_\lambda, v_\lambda}(u)$ equals $\hat{\chi}_\lambda(u)$ for any $u \in \dU^0_q$ and vanishes for all  $u \in \bigoplus_{(\nu,\mu)\in Q_+ \x Q_+ \backslash(0,0)} \dU^<_{-\nu} \dU^0_q \dU^>_\mu$. This immediately implies the result.
\end{proof}
Thus, we obtain the embedding of left $\dU_q(\g)$-modules such that $\hiota(c_{v^*_\lambda, v_\lambda})=K^{-2\lambda}$,
\[ \hiota: R[\dU_q(\g)]\hookrightarrow U^{ev}_q(\g).\]
 Here we equip $R[\dU_q(\g)]$ with a left $\dU_q(\g)$-module structure via $(xf)(y)=f(\ad'_l(S'(x))y)$ for all $x, y\in \dU_q(\g)$  and $f\in R[\dU_q(\g)]$. Let us recall the algebra $O_q[G]$ in Section \ref{ssec: description of REA}.
\begin{Lem}\label{lem: Oq[G] --> even part} The embedding $\hiota: O_q[G]\hookrightarrow U^{ev}_q(\g)$ is a homomophism of $\dU_q(\g)$-module algebras.
\end{Lem}
\begin{proof}
    Since $\hiota$ is a homomorphism of $\dU_q(\g)$-modules, it is enough to show that $\hiota$ is an algebra homomorphism. In the rest of this proof, we shall replace the subscript $q$ by the base ring that $q$ is contained in. Let $\K=\BQ(v^{1/\sN})$. Consider the following diagram
\begin{equation*}
\begin{tikzcd} 
  O_{\K}[G]\arrow[r, "\hiota_\K"]& U^{ev}_{\K}(\g)\\
  O_{\CA_{\sN}}[G] \arrow[r, "\hiota_{\CA_\sN}"]\arrow[u, hook]\arrow[d]& U^{ev}_{\CA_\sN}(\g)\arrow[d]\arrow[u, hook]\\
  O_R[G]\arrow[r, "\hiota_R"]&U^{ev}_R(\g)
\end{tikzcd}
\end{equation*}
The other  horizontal maps are obtained from the middle map by base change. The two upper vertical arrows are inclusions since $O_{\CA_\sN}[G]$ and $U^{ev}_{\CA_\sN}(g)$ are free over $\CA_{\sN}$ by Lemma \ref{Lem:coord_alg_freeness} and Lemma \ref{lem:twisted-triangular-PBW}. Therefore, it is enough to show $\hiota_\K: O_\K[G]\rightarrow U^{ev}_\K(\g)$ is an algebra homomorphism.

Let us recall the homorphism of $\dU_\K(\g)$-module algebras $l_\r: O_\K[G]\rightarrow \dU_\K(\g)$ in \eqref{eq: def of lr}. On the other hand, we consider the morphism $\hiota_\K: O_\K[G]\rightarrow U^{ev}_\K(\g)\cong \dU_\K(\g)$.
Both $\hiota_\K$ and $l_\r$ are morphisms of $\dU_q(\g)$-modules such that $\hiota_\K(c_{v^*_\lambda, v_\lambda})=l_\r(c_{v^*_\lambda, v_\lambda})=K^{-2\lambda}$, see Lemma \ref{lem: lr morphism} and Lemma \ref{lem: image of iota(K(-2lambda))} . Moreover, $O_\K[G]$ is generated by $c_{v^*_\lambda, v_\lambda}$ over $\dU_\K(\g)$ by Lemma \ref{lem: U-direct sum of Oq[G]}. Therefore, $l_\r$ and $\hiota_\K$ coincide. This implies that $\hiota_\K$ is an algebra homomorphism. 
\end{proof}

\subsection{The locally finite part $U^{fin}_q$}\

Let us define
\begin{equation*}
    U^{fin}_q:=\Big\{ a\in U^{ev}_q(\g)\Big|\ad'(\dU_q(\g))(a)\; \text{is a finitely generated $R$-module}\Big\}
\end{equation*}
Since $U^{ev}_q(\g)$ is a weight module, $U^{fin}_q$ is the maximal rational subrepresentation of $U^{ev}_q(\g)$ and then $\hiota(O_q[G])\subset U^{fin}_q$. It is obvious that $U^{fin}_q$ is an $R$-subalgebra of $U^{ev}_q(\g)$ by \eqref{ad on products}. It is this subalgebra that we call the {\it locally finite part}. We will show that $\hiota(O_q[G]) \iso U_q^{fin}$ for any  Noetherian $\CA_\sN$-algebra $R$, and identify $O_q[G]$ with a subalgebra of $U_q^{ev}$ via $\hiota$.

The key result of this section is similar to \cite[Theorem 6.4]{jl}
\begin{Prop}\label{prop: basis of Ufin}
(a) $K^{2\b} \in U^{fin}_q$ iff $\b \in -P_+$.

\noindent
(b) $U^{ev}_q(\g)$ is obtained from $U^{fin}_q$ ( resp. $O_q[G]$) by localizing the elements $\{K^{-2\w_i}\}_{i=1}^r$. 
\end{Prop}

\begin{proof}
(a) We claim that $K^\b \in U^{fin}_q \cap U^{ev,0}_q(\g)$ if and only if $\frac{(\a_i, \b)}{2d_i} \in \BZ_{\leq 0}$ for all $i$. The if part follows by $\hiota(c_{v^*_\lambda, v_\lambda})=K^{-2\lambda}$ for all $\lambda \in P_+$ and $\hiota(O_q[G])\subset U^{fin}_q$. Let us prove the only if part.

Assuming the contradiction, pick $\b\in 2P$ such that $K^\b \in U^{fin}_q$ and there is $i \in \{1, \dots, r\}$ with $c:=\frac{(\a_i, \b)}{2d_i} \in \BZ_{>0}$. According to \eqref{dividedE-acts-K}, we have
\begin{equation}
    \ad'(\tE^{(m)}_i)(K^\beta) = 
  \frac{\prod_{s=0}^{m-1} (1-q_i^{2(c+s)})}{(m)_{q_i}!} \cdot \tE^{m}_iK^{\beta+m\zlambda_i}\,,
\end{equation}
where, as always, the coefficient $\frac{\prod_{s=0}^{m-1} (1-q_i^{2(c+s)})}{(m)_{q_i}!}$ denotes the image of $\frac{\prod_{s=0}^{m-1} (1-v_i^{2(c+s)})}{(m)_{v_i}!}\in \CA$ under the algebra homomorphism $\sigma: \CA \rightarrow R$ of \eqref{sigma homom}. We will show that
\begin{equation}\label{eq: nonzero coefficients in R}\text{there are infinitely many $m\geq 1$ such that $c_m:= \frac{\prod_{s=0}^{m-1} (1-q_i^{2(c+s)})}{(m)_{q_i}}\neq 0$ in $R$.}
\end{equation}This implies that $\ad'(\dU_q(\g))(K^\b)$ contains nonzero multiplies of infinitely many elements of $\{ \tE_i^m K^{\b+m \zlambda_i}\}_{m \geq 1}$. These elements are $R$-linearly independent by  PBW-bases in Lemma \ref{lem:twisted-triangular-PBW}. Hence, $\ad'(\dU_q(\g))(K^\b)$ is not finitely generated over $R$, which leads to contradiction.

Let $\m$ be any maximal ideal of $R$ and let $\BF=R /\m$. Let $\bar{q}$ be the image of $q$ in $\BF$ then it is enough to show \eqref{eq: nonzero coefficients in R} when we replace $(R,q)$ with $(\BF, \bar{q})$. But we have 
\[ c_m = (-1)^m(1-\bar{q}_i^{-2})^m \bar{q}_i^{(2c+m-1)m}\binom{c+m-1}{m}_{\bar{q}_i}.\]
From this one can see that there are infinitely many $m \geq 1$ such that $c_m \neq 0$ in $\BF$ regardless of characteristics of $\BF$ and whether $\bar{q}$ is a root of unity in $\BF$ or not. 

\noindent
(b) Since $K^{-2\w_i}\in U^{fin}_q$ and $K^{-2\w_i}\in O_q[G]$, we get 
\begin{align*}
    (1-q_i^{-2}) \tE_iK^{\zlambda_i} K^{-2\w_i}&=\ad'(\tE_i)(K^{-2\w_i}) \in U^{fin}_q, \; O_q[G],\\
    (1-q_i^{-2})\tF_iK^{-\zmu_i} K^{-2\w_i}& =\ad'(\tF_i)(K^{-2\w_i}) \in U^{fin}_q,\; O_q[G].
\end{align*}
As $1-q_i^{-2}\in R$ is assumed to be invertible, we then obtain
\[ \tE_i K^{\zlambda_i}, \tF_i K^{-\zmu_i} \in U^{fin}_q K^{2\w_i}, \; O_q[G] K^{2\w_i} \qquad \forall 1\leq i \leq r.\]
This completes our proof since $U^{ev}_q(\g)$ is generated by $\{ \tE_iK^{\zlambda_i}, \tF_i K^{-\zmu_i}, K^{\pm 2\w_i} \}_{i=1}^r$.
\end{proof}
\subsection{$O_\e[G]\cong U^{fin}_\e$ when $R=\BF$ a field and $q$ is not a root of unity}\
\label{ssec: REA-Ufin nonroot case}

The content of this section follows \cite{jl2}. For reader convenience, we sketch main ideas in \cite{jl2} along with some modifications accounting for the twists in this paper.

Let $R=\BF$ be a field and $q\in \BF$ be not a root of unity. 
In this case, we have $U^{ev}_q(\g) \cong \dU_q(\g)$.
The category $\Rep(\dU_q(\g)) \cong \Rep(\dmU_q(\g))$ is semisimple, see \cite[Corollary 7.7]{APW1}. Combining this observation with the twisted pairing \eqref{pairing of DCK and Lus rewritten} and arguments in \cite[Section 6]{j},  we obtain
\begin{Prop}\label{prop: center not root of 1}Let $Z$ be the center of $U^{ev}_\e(\g) \cong \dU_q(\g)$.  The Harish Chandra homomorphism (for the construction, see \cite[$\mathsection 6.2$]{j} or Section \ref{sec: HC center} below) gives an isomorphism $Z \cong \BF[K^{2\lambda}]_{\lambda \in P}^{W_\bullet}$.
\end{Prop}

Let us define the category $O$ for $\dU_q(\g)$. The objects in $O$ are $\dU_q(\g)$-modules $M$ with  weight decompositions $M=\oplus_\lambda M_\lambda$ such that  
\begin{itemize}
\item[(i)] The set $\{ \lambda| M_\lambda \neq \{0\}\}$ is bounded from above with respect to the usual order on the weight lattice.
\item[(ii)] $\dim_\BF M_\lambda <\infty$ for all $\lambda$. 
\end{itemize}

The morphisms in the category $O$ are homomorphisms of $\dU_q(\g)$-modules.  The {\it Verma module} $M_q(\lambda)=\dU_q(\g)\otimes_{\dU_q^{\geqslant }} \BF_\lambda$ is an object of $O$.

We want to introduce a duality $\tau$ on the category $O$. Let us consider the inclusion $\dU_q(\g) \subset \dU_q(\g, P/2)$. Then each object in $O$ is naturally a module over $\dU_q(\g, P/2)$. On $\dU_q(\g, P/2)$, we have the anti-automorphism $\tau$ defined by 
\[\tau(E_i)=F_i,\quad \tau(F_i)=E_i,\quad \tau(K^\lambda)=K^{\lambda}\qquad \text{ for $\lambda \in P/2$.}\]
For any $M\in O$, set $M^\tau:=\oplus_\lambda (M_\lambda)^* \subset \Hom_{\BF}(M, \BF)$, as vector spaces. The $\dU_q(\g,P/2)$-module structure on $M^\tau$ is defined by 
\[xf(m)=f(\tau(x)m)\qquad \text{for $ x \in \dU_q(\g, P/2),~ f\in M^\tau,~ m\in M$}.\]
Then $M^\tau$ becomes an object in $O$ by restricting the action to $\dU_q(\g)$. 

Once we have the result about the center of $\dU_q(\g)$ in Proposition \ref{prop: center not root of 1}, the following is proved by standard arguments \cite[Chapter 1]{h}.

\begin{Lem}\label{lem: cat O}
(a) Each object in $O$ has  finite length.

(b) The module $M_q(\lambda)$ has a unique simple quotient $L_q(\lambda)$. The module $M_q(\lambda)^\tau$ has a unique simple submodule $L_q(\lambda)$. The module $L_q(\lambda)$ is finite dimensional if and only if $\lambda \in P_+$.
\end{Lem}

We now define a filtration $\CF$ on $U^{ev}_q(\g)$ which is stable under the adjoint action of $\dU_q(\g)$ on $U^{ev}_q(\g)$. This filtration is  labeled by the weight lattice $P$ with the usual dominance order.  On the free associative algebra generated by $\tE_i, \tF_i, K^\lambda$, we put 
\[\deg(K^{2\lambda})=-\lambda,\qquad  \deg(\tE_i)=\zlambda_i/2,\qquad \deg(\tF_i)=-\zmu_i/2.\]
This gives a filtration labeled by the weight lattice $P$ with the dominant order. This induces the filtration $\CF$ on $U^{ev}_q(\g)$. Recall the adjoint action of $\dU_q(\g)$ on $U^{ev}_q(\g)$:
\[ \ad'(K^\mu)(x)=K^{\mu}xK^{-\mu}, \qquad \ad'(\tE_i)(x)=[\tE_i, x]K^{\zlambda_i}, \qquad \ad'(\tF_i)(x)=[\tF_i,x]K^{-\zmu_i}.\]
We see that the $\dU_q(\g)$-action preserves the filtration $\CF$ on $U^{ev}_q(\g)$.

Let $G^-$ and $G^+$ be the subalgebras of $\gr_\CF(U^{ev}_q(\g))$ generated by $\{\tF_i K^{-\zmu_i}\}$ and $\{\tE_i K^{\zlambda_i}\}$, respectively. Then 
\[ \gr_\CF(U^{ev}_q(\g))=\bigoplus_\lambda G(\lambda), \qquad \text{here $G(\lambda)=G^-\otimes_\BF \BF K^{-2\lambda} \otimes_\BF G^+$}.\]
We have $G(\lambda)G(\mu)\subset G(\lambda+\mu)$ and $G^-, G^+\subset G(0)$. The isomorphism $G^-\otimes_\BF \BF K^{-2\lambda} \otimes_\BF G^+ \rightarrow G(\lambda)$ is given by the multiplication map (we view $G^-,G^+$ as subsets of $G(0)$ and $K^{-2\lambda}$ as an element in $ G(\lambda)$). The action of $U_q(\g)$ on $\gr_\CF(U^{ev}_q(\g))$ satisfies 
\begin{equation}\label{ad on gr} \ad'(x)(ab)=\ad'(x_{(1)})(a)\ad'(x_{(2)})(b),
\end{equation}
for all $x\in \dU_q(\g)$ and $a, b \in \gr_\CF(U^{ev}_q(\g))$. Furthermore, the $Q$-grading on $U^{ev}_q(\g)$ induces $Q$-gradings on $G(\lambda), G^-, G^+$.

\begin{Lem}[c.f. \cite{jl2}] $G^{-}, G^{+}$ are $\dU_q(\g)$-submodules of $\gr_\CF(U^{ev}_q(\g))$.  
\end{Lem}
\begin{proof}Let us prove this in the case of $G^{-}$. The actions of the elements $K^{\lambda}, \tF_i\in \dU_q(\g)$ obviously preserve $G^{-}$. On the other hand, in $U^{ev}_q(\g)$ , we have
\[ \ad'(\tE_i)(\tF_jK^{-\zmu_i})=0, \quad \ad'(\tE_i)(\tF_iK^{-\zmu_i})= \frac{q_i}{1-q_i^{-2}}(K^{2\a_i}-1),\]
for $i \neq j$. Hence in $\gr_\CF(U^{ev}_q(\g))$, we have 
\[ \ad'(\tE_i)(\tF_jK^{-\zmu_i})=-\delta_{ij}\frac{q_i}{1-q_i^{-2}}.\]

Combining this with \eqref{ad on gr}, we can see that $G^{-}$ is a $\dU_q(\g)$-submodule of $\gr_\CF(U^{ev}_q(\g))$.
\end{proof}

Let $G^{fin}(\lambda):=\{ m \in G(\lambda)| \dim \ad'(\dU_q(\g)) m <\infty\}$. Here is  an important result from \cite{jl2}.
\begin{Lem} $G^{fin}(\lambda)=0$ unless $\lambda \in P^+$. When $\lambda \in P^+$ then $G^{fin}(\lambda)=\ad'(\dU_q(\g))(K^{-2\lambda})$.
\end{Lem}
\begin{proof}
{\it Step 1.} Here is the key observation:  $G^{-}\otimes_\BF \BF K^{-2\lambda}$ can be equipped with a $\dU_q(\g)$-module structure as follows: 
\begin{align*} \tE_i(u_-\otimes K^{-2\lambda})&=q^{-(\lambda, \zlambda_i)}\ad'(\tE_i)(u_-) \otimes K^{-2\lambda}\\
K^{\mu}(u_-\otimes K^{-2\lambda})&=q^{(\lambda, \mu)}\ad'( K^\mu)(u_-\otimes K^{-2\lambda})\\
\tF_i(u_-\otimes K^{-2\lambda})&= q^{(\lambda, \zmu_i)}\ad'(\tF_i)(u_-\otimes K^{-2\lambda}),
\end{align*}
for $u_-\in G^-$, where we use that $G^{-}\otimes_\BF \BF K^{-2\lambda}$ is a $\dU^<_q$-submodule and $G^{-}$ is a $\dU_q(\g)$-submodule of $\gr_\CF(U^{ev}_q(\g))$.  Then we see that 
\begin{equation}\label{eq: U-mod vs U<mod}
\{ m \in G^-\otimes_\BF \BF K^{-2\lambda}| \dim \ad'(\dU^<_q)(m)<\infty\}=\{ m \in G^-\otimes_\BF \BF K^{-2\lambda}| \dim \dU_q(\g)(m) < \infty\}
\end{equation}

{\it Step 2.} Arguing as in \cite[$\mathsection 4.8$]{jl2}, where we use Lemma \ref{lem: cat O}  , we have $G^-\otimes_\BF \BF K^{-2\lambda} \cong M_q(\lambda)^\tau$. Combining this with \eqref{eq: U-mod vs U<mod}, we have
\[ L(\lambda)^-:=\{ m\in G^-\otimes_\BF \BF K^{-2\lambda}| \dim \ad'(\dU^<)(m) <\infty\}=\begin{cases} 0 \qquad \text{if $\lambda$ is not dominant,}\\
                        \ad'(\dU^<)(K^{-2\lambda}) \quad  \text{if $\lambda$ is dominant.}
        \end{cases}
\]
If $\lambda$ is dominant, then $L(\lambda)^-=K(\lambda)^-\otimes \BF K^{-2\lambda}$ for some  $K(\lambda)^-\subset G^-$. By \eqref{eq: U-mod vs U<mod}, $L(\lambda)^-$ is a $\dU_q(\g)$-submodule of $G^-\otimes \BF K^{-2\lambda}$. By looking at the action of $\tE_i$  on $G^-\otimes \BF K^{-2\lambda}$ defined in Step $1$, it follows that $K(\lambda)^-$ is a $\dU^>_q$-submodule of $G^-$.

{\it Step 3.} Similarly, we have 
\[L(\lambda)^+:= \{ m \in G^+\otimes_\BF \BF K^{-2\lambda}| \dim \ad'(\dU^>)(m) <\infty\}=\begin{cases} 0 \qquad \text{if $\lambda$ is not dominant,}\\
                        \ad'(\dU^>)(K^{-2\lambda}) \quad \text{if $\lambda$ is dominant.}
        \end{cases}\]
When $\lambda$ is dominant, $L(\lambda)^+=K(\lambda)^+\otimes \BF K^{-2\lambda}$ for some $\dU^<_q$-submodule $K(\lambda)^+$ of $G^+$. 

{\it Step 4.} Using Steps 2 and 3, we proceed exactly the same as in \cite[$\mathsection 4.9$]{jl2} to finish the proof of the lemma.
\end{proof}



\begin{Cor}\label{span of Ufin} $U^{fin}_q$ is spanned by $\{K^{-2\lambda}\}_{\lambda \in P_+}$ as a $\dU_q(\g)$-module.
\end{Cor}

\begin{Prop} The inclusion $O_q[G]\hookrightarrow U^{fin}_q$ is an isomorphism of $\dU_q(\g)$-module algebras.
\end{Prop}
\begin{proof}It is enough to show that this is an epimorphism of $\dU_q(\g)$-modules. Lemma \ref{lem: image of iota(K(-2lambda))} implies that $K^{-2\lambda}, \lambda \in P_+$, is contained in the image. Therefore, by Corollary \ref{span of Ufin}, the inclusion $O_q[G]\hookrightarrow U_q^{fin}$ is surjective.
\end{proof}

\subsection{$O_q[G]\cong U^{fin}_q$ when $R=\BF$ a field and $q=\e$ is a root of unity}\
\label{ssec: REA-Ufin root case}

Let us first assume that $\BF$ is algebraically closed.

Recall the Hopf algebra $\rU_q$ from Section \ref{sec: REA}. Let $\rU^*_q$ be the Hopf algebra obtained from $\dU^*_q(\g)$ by taking the opposite of both product and coproduct structures. Then the  category of rational representations $\Rep(\rU^*_q)$ is defined in an obvious way. Let $\Rep^{fd}(\rU_q), \Rep^{fd}(\rU^*_q)$ be the full subcategories of $\Rep(\rU_q), \Rep(\rU^*_q)$ consisting of finite dimensional objects. We then can define the algebra $O_\e[G^d]$ for $\dU^*_\e(\g)$ as in Section \ref{sec: REA} with respect to the $R$-matrix  $\CR^*$ in \eqref{R matrix for U*}.

The Hopf algebra homomorphism $\tFr: \hU_\e(\g, P)\rightarrow \hU^*_\e(\g, P^*)$ from \eqref{eq: Fr for idemp} gives rise to braided monoidal fully faithful functor
\[ \tFr^*: \Rep(\rU^*_q)\rightarrow \Rep(\rU_q).\]
By using the categorical realization of the REA from Section \ref{ssec: cat def of REA}, we obtain an inclusion 
\[ O_\e[G^d]\hookrightarrow O_\e[G].\]
The categorical realization
\[O_\e[G^d]\cong \bigsqcup_{X\in \Rep^{fd}(\rU^*_q)} X^* \otimes X/ \sim\]
also tells us that $O_\e[G^d]$ is a $\dU_\e(\g)$-submodule of $O_\e[G]$. Furthermore, the action of $\dU_\e(\g)$ on $O_\e[G^d]$ factors through $\dU_\BF(\g^d)$: the elements $\tE_i, \tF_i$ act by zero; the weights of $O_\e[G^d]$ are contained in $Q^*$, hence $K^{\lambda}$ acts as the identity on $O_\e[G^d]$ for all $\lambda \in 2P$.

So we have the composition of embedding 
\[ \varphi: O_\e[G^d]\hookrightarrow O_\e[G] \xhookrightarrow[]{\hiota} U^{ev}_\e(\g),\]
where $\hiota$ comes from Lemma \ref{lem: Oq[G] --> even part}. Our next result describes the image  of $O_\e[G^d]$ under $\varphi$. To this end, we recall  $Z^{fin}_{Fr}$ in \eqref{eq: def of Zfin}:
\[ Z^{fin}_{Fr}=\{ z\in Z_{Fr}| \dim_\BF(\dU_\BF(\g^d)z) < \infty\}.\]

\begin{Lem}\label{lem: Oe[G^d] and Z^ev} Evoking the algebra homomorphism $Z_{Fr}\cong \BF[G^d_0]$ from Proposition \ref{Ze and Bruhat cell}, we have $\varphi(O_\e[G^d])=Z^{fin}_{Fr}\cong\BF[G^d]$.
\end{Lem}
\begin{proof} The $\mathbb{F}$-vector space
$O_\e[G^d]$ is spanned by the matrix coefficients of the $\dU_\e(\g,P)$-modules that are objects in $\Rep^{fd}(\dU^*_\e)$. These matrix coefficients vanish on the kernel of $\tFr: \dU_\e(\g, P)\rightarrow \dU^*_\e(\g, P^*)$. By definition of $Z_{Fr}$, we have $\varphi(O_\e[G^d])\subset Z_{Fr}$.

 Recall homomorphisms \eqref{eq: morphisms}:
\begin{equation}\label{eq: U(g^d) embeded into U*}\dU_\BF(\g^d) \overset{\phi^{-1}}{\cong} A_{Q^*} \hookrightarrow \tU^*_\e(\g, P^*) \hookleftarrow \dU^*_\e(\g, P^*).
\end{equation}
All representations in $\Rep(\dU^*_\e)$  and $\Rep(\dU_\BF(\g^d))$ can be naturally extended to representations of $\tU^*_\e(\g, P^*)$. By taking these extensions and then restricting the actions to the other algebras, we obtain the following functors:
\begin{equation}\label{eq: Rep(U*) vs Rep(U(gd))}\mathfrak{T}:  \Rep(\dU^*_\e(\g)) \rightarrow \Rep(\dU_\BF(\g^d)), \qquad \mathfrak{T}': \Rep(\dU_\BF(\g^d))\rightarrow \Rep(\dU^*_\e(\g)),
\end{equation}
These functors are mutually quasi-inverse. In particular, $\mathfrak{T}$ is an equivalence of categories.

By analyzing the construction of $Z_{Fr}\cong \BF[G^d_0]$, one sees that $\varphi(O_\e[G^d])\subset Z_{Fr}\cong \BF[G^d_0]$ is spanned by the matrix coefficients of the  representations in $\Rep^{fd}(\dU^*_\e(\g))$ viewed as modules over $\dU_\BF(\g^d)$ via the functor $\mathfrak{T}$. However, $\mathfrak{T}$ is an equivalence of categories. Therefore, $\varphi(O_\e[G^d])\subset \BF[G^d_0]$ contains all matrix coefficients of the representations in $\Rep^{fd}(\dU_\BF(\g^d))$. But these matrix coefficients span $\BF[G^d]$ and $\varphi(O_\e[G^d]) \subset \BF[G^d]$. Hence $\varphi(O_\e[G^d])=\BF[G^d]$.
\end{proof}
\begin{Rem}\label{rem: Oq[G^d] is central} By Lemma \ref{lem: Oe[G^d] and Z^ev}, $O_\e[G^d]$ is central in $O_\e[G]$.
\end{Rem}
The following technical result will be established in Section \ref{ssec: prop of Oe[G] is proj}:

\begin{Prop}\label{prop: Oe[G] is proj} $O_\e[G]$ is a finitely generated projective left and right $O_\e[G^d]$-module.
\end{Prop}

Recall that $\dU_\e(\g)$ acts on $O_\e[G^d]$ making $O_\e[G^d]$ into a $\dU_\e(\g)$-module algebra. Thus we can form the following categories: $O_\e[G^d]\Lmod^{\dU_\e},\; O_\e[G^d]\Rmod^{\dU_\e}, O_\e[G^d]\Bimod^{\dU_\e}$ of $\dU_\e(\g)$-equivariant left modules, right modules, and bimodules over $O_\e[G^d]$, respectively. We refer the reader to Appendix \ref{sec:equiv} for details.

\begin{defi}\label{def: module categories}Let $O_\e[G^d]\bimod^{G_\e}$ be the full subcategory of  $O_\e[G^d]\Bimod^{\dU_\e}$ consisting of all objects which are finitely generated as left and right $O_\e[G^d]$-modules, and rational as $\dU_\e(\g$)-representations. The categories $O_\e[G^d]\lmod^{G_\e}$ and $O_\e[G^d]\rmod^{G_\e}$ are defined similarly.
\end{defi}

Since $O_\e[G^d]\cong \BF[G^d]$ is Noetherian and the category $\Rep(\dU_\e(\g))$ is abelian, all three categories $O_\e[G^d]\bimod^{G_\e}, O_\e[G^d]\lmod^{G_\e}, O_\e[G^d]\rmod^{G_\e}$ are abelian. We note that $O_\e[G]$ belongs to $O_\e[G^d]\bimod^{G_\e}$, thanks to Proposition \ref{prop: Oe[G] is proj}. 

For any $\dU_\e(\g)$-representation $M$, let 
\begin{equation}\label{eq: Mfin}
M^{fin}:=\{ m\in M|\dim_\BF(\dU_\e(\g) m)<\infty\}.
\end{equation}

Since the action of $\dU_\e(\g)$ on $O_\e[G^d]$ is locally finite as $O_\e[G^d]$ is a rational $\dU_\e(\g)$-representation, for any $M \in O_\e[G^d]\Rmod^{\dU_\e}$, the locally finite part $M^{fin}$ is an object in $O_\e[G^d]\Rmod^{\dU_\e}$. Note that the action of $\dU_\e(\g)$ on $M^{fin}$ is not required to be rational.

\begin{Lem}\label{lem: taking finite part} Let $M\in O_\e[G^d]\Rmod^{\dU_\e}$ and $P\in O_\e[G^d]\lmod^{G_\e}$, where $P$ is a finitely generated projective left $O_\e[G^d]$-module. Then the natural map
\begin{equation}\label{eq: taking finite part}
    M^{fin}\otimes_{O_\e[G^d]} P \rightarrow M\otimes_{O_\e[G^d]}P
\end{equation}
is injective with image $(M\otimes_{O_\e[G^d]} P)^{fin}$.
\end{Lem}
\begin{proof}Since $P$ is a rational $\dU_\e(\g)$-representation, the action of $\dU_\e(\g)$ on $P$ is locally finite. Therefore, the action of $\dU_\e(\g)$ on $M^{fin}\otimes_{O_\e[G^d]} P$ is locally finite and the image of \eqref{eq: taking finite part} is contained in $(M\otimes_{O_\e[G^d]} P)^{fin}$.

It suffices to show that the induced map
\begin{equation}\label{eq: Hom(V,.)}
\Hom_{\dU_\e(\g)}(V, M^{fin}\otimes_{O_\e[G^d]} P)\rightarrow \Hom_{\dU_\e(\g)}(V, M\otimes_{O_\e[G^d]} P)   
\end{equation}
is bijective for any finite dimensional $\dU_\e(\g)$-module $V$.

Let $Q=\Hom_{O_\e[G^*]\Lmod}(P,O_\e[G^*])$. Since $P$ is a finitely generated projective left $O_\e[G^*]$-module, $Q$ is a finitely generated projective right $O_\e[G^*]$-module and we have an isomorphism $P \simeq \Hom_{O_\e[G^*]\Rmod}(Q, O_\e[G^*])$ in the category $O_\e[G^*]\lmod^{G_\e}$ by Lemma~\ref{tensor-hom-adj}. Using the tensor-hom adjunction of Lemma~\ref{tensor-hom-adj}, we have:
\begin{equation}\label{eq:(V,M xP)}
\begin{split}
  \Hom_{\dU_\e(\g)}(V, M\otimes_{O_\e[G^*]}P) &=\Hom_\BF(V, M\otimes_{O_\e[G^*]}P)^{\dU_\e(\g)}\\
  &=\Hom_{O_\e[G^*]\Rmod}(V\otimes_\BC Q, M)^{\dU_\e(\g)}\\
  &=\Hom_{O_\e[G^*]\Rmod^{\dU_\e(\g)}}(V\otimes_\BC Q,M) \,.
  \end{split}
\end{equation}
Since $Q\in O_\e[G^*]\Rmod^{G_\e}$, the action of $\dU_\e(\g)$ on $Q$ is locally finite, and thus we have:
\begin{align*}
  \Hom_{\dU_\e(\g)}(V, M\otimes_{O_\e[G^*]}P) &=\Hom_{O_\e[G^*]\Rmod^{\dU_\e(\g)}}(V\otimes_\BC Q, M^{fin})\\
  &=\Hom_{\dU_\e(\g)}(V, M^{fin}\otimes_{O_\e[G^*]} P) \,.
\end{align*}
 This completes the proof.
\end{proof}

We are now ready to prove the main result of this section. Here $\BF$ is not required to be algebraically closed.
\begin{Prop}\label{prop: Oe[G] iso Ufin } The inclusion $\hiota: O_\e[G]\hookrightarrow U^{fin}_\e$ is an isomorphism of $\dU_\e(\g)$-module algebras
\end{Prop}
\begin{proof}
First, we assume $\BF$ is algebraically closed. According to Lemma \ref{lem: Oe[G^d] and Z^ev} and Lemma \ref{lem: Zfin in field F}, we have $O_\e[G^d][K^{2\lambda_0}]=Z^{ev}_\e$. By Proposition \ref{prop: basis of Ufin}, $O_\e[G][K^{2\lambda_0}=U^{ev}_\e$. As $O_\e[G]$ is finitely generated over $O_\e[G^d]$ by Proposition \ref{prop: Oe[G] is proj},
\[ Z_{Fr} \otimes_{O_\e[G^d]} O_\e[G]\cong O_\e[G^d][K^{2\lambda_0}]\otimes_{O_\e[G^d]} O_\e[G] \cong O_\e[G][K^{2\lambda_0}]=U^{ev}_\e.\]
We have that $O_\e[G]\in O_\e[G^d]\lmod^{G_\e}$, $Z_{Fr}\in O_\e[G^d]\Rmod^{\dU_\e}$,  and $O_\e[G]$ is finitely generated projective over $O_\e[G^d]$. Hence thanks to Lemma \ref{lem: Oe[G^d] and Z^ev} and Lemma \ref{lem: taking finite part},   it follows that 
\[ (Z_{Fr}\otimes_{O_\e[G^d]} O_\e[G])^{fin}\cong Z^{fin}_{Fr}\otimes_{O_\e[G^d]} O_\e[G] \cong O_\e[G^d] \otimes_{O_\e[G^d]}O_\e[G] \cong O_\e[G],\]
 This implies the isomorphism $O_\e[G]\cong U^{fin}_\e$ when $\BF$ is algebraically closed.


In general, let $\overline{\BF}$ be an algebraically closed field of $\BF$. The following commutative diagram
\[ \begin{tikzcd}O_\BF[G]\otimes_{\BF}\overline{\BF}\arrow[d, "\cong"]\arrow[r, hook]& U^{fin}_\BF\otimes_\BF \overline{\BF}\arrow[d, hook]&\\
O_{\overline{\BF}}[G]\arrow[r, "\cong"]& U^{fin}_{\overline{\BF}}
\end{tikzcd}\]
implies the top horizontal map is bijective. Hence $O_\BF[G]\rightarrow U^{fin}_\BF$ is an isomorphism. 
\end{proof}

\subsection{Proof of Proposition \ref{prop: Oe[G] is proj}}\
\label{ssec: prop of Oe[G] is proj}

Here $\BF$ is algebraically closed. Recall the functor
\[ \tFr^*:\Rep(\rU^*_q)\rightarrow \Rep(\rU_q).\]

\begin{Lem}\label{lem: Oe[G] is fg}$O_\e[G]$ is finitely generated over $O_\e[G^d]$.
\end{Lem}
\begin{proof} By Lemma \ref{lem: mat coefficient}, there is a finite dimensional module $V\in \Rep(\rU_q)$ such that any finite dimensional module  $M\in \Rep(\rU_q)$ is a subquotient of $\tFr^*(N)\otimes V$ for some finite dimensional module $N\in \Rep(\rU^*_q)$ . Hence, $O_\e[G]$ is spanned by matrix coefficients of the modules of the form $\tFr^*(N)\otimes V$ with $N\in \Rep(\rU^*_q)$. So, $O_\e[G]$ is linearly spanned by the images of the maps
\[ (\tFr^*(N)\otimes V)^* \otimes (\tFr^*(N)\otimes V)\rightarrow O_\e[G].\]
Evoking the multiplication in \eqref{eq: cat rel of mult}, we have
\[
\begin{tikzcd}
    (\tFr^*(N))^*\otimes \tFr^*(N) \otimes V^*\otimes V \arrow[d] \arrow[r, "\cong"]& (\tFr^*(N)\otimes V)^*\otimes \tFr^*(N)\otimes V\arrow[d]&\\
    O_\e[G]\otimes O_\e[G]\arrow[r, "\mm"]& O_\e[G]
\end{tikzcd}
\]
On the other hand, the image of $(\tFr(N))^*\otimes \tFr^*(N) \rightarrow O_\e[G]$ is contained in $O_\e[G^d]$. From these observations, we see that $O_\e[G]$ is spanned by the image of $V^*\otimes V\rightarrow O_\e[G]$ over $O_\e[G^d]$. Since $V$ is finite dimensional, $O_\e[G]$ is finitely generated over $O_\e[G^d]$.
\end{proof}

To show that $O_\e[G]$ is projective over $O_\e[G^d]$, we need two technical lemmas. Recall the categories $O_\e[G^d]\rmod^{G_\e}$, $O_\e[G^d]\lmod^{G_\e}$, $O_\e[G^d]\bimod^{G_\e}$ from Definition \ref{def: module categories}. 

\begin{Lem}\label{lem: internal Ext}
For any $M\in O_\e[G^d]\rmod^{G_\e}$, all $\Ext$-groups $\Ext^i_{O_\e[G^d]\rmod}(M, O_\e[G^d])$ are objects of $O_\e[G^d]\lmod^{G_\e}$.
\end{Lem}
\begin{proof} Since $M$ is finitely generated over $O_\e[G^d]$, there is a surjective morphism $V_1\otimes O_\e[G^d]\twoheadrightarrow M$ in the category $O_\e[G^d]\rmod^{G_\e}$ for some finite dimensional rational $\dU_\e(\g)$-module $V_1$. Applying this process to the kernel of this surjective morphism and so on, we obtain a long exact sequence in $O_\e[G^d]\rmod^{G_\e}$, where $V_i$ are finite dimensional rational $\dU_\e(\g)$-modules: 
\[ \dots \rightarrow V_2\otimes O_\e[G^d]\rightarrow V_1\otimes O_\e[G^d]\rightarrow M\rightarrow 0,\]
 This is a free resolution of the right $O_\e[G^d]$-module $M$.  On the other hand,  by Lemma \ref{Hom in H-equiv A-Rmod},  $\Hom_{O_\e[G^d]\rmod}(N, O_\e[G^d])$ is an object in $O_\e[G^d]\lmod^{G_\e}$ for any $N \in O_\e[G^d]\rmod^{G_\e}$. Hence, $\Ext^i_{O_\e[G^d]\rmod} (M, O_\e[G^d])$  are indeed objects in $O_\e[G^d]\lmod^{G_\e}$.
\end{proof}

Recall that the vanishing locus of $f=K^{-2\lambda_0}\in O_\e[G^d]\cong \BF[G^d]$ in $G^d(\BF)$ is $G^d(\BF)\backslash G^d_0(\BF)$. For any $M\in O_\e[G^d]\lmod^{G_\e}$, let $M_f$ be the localization of $M$ by $f\in O_\e[G^d]$.


\begin{Lem}\label{lem: zero localization} If $M\in O_\e[G^d]\lmod^{G_\e}$ is such that $M_f=0$, then $M=0$.
\end{Lem}
This result is proved using its classical counterpart for equivaraint sheaves under an action of  an algebraic group. Let $\fu_\e$ be the Hopf subalgebra of $\dU_\e(\g)$ generated by $\{\tE_i, \tF_i, K^\lambda\}_{1\leq i \leq r}^{\lambda \in 2P}$. 
\begin{Lem}\label{lem: u-invariant part}(a) For any rational $\dU_\e(\g)$-module $M$, the $\fu_\e$-invariant part $M^{\fu_\e}$ is a rational $\dU_\BF(\g^d)$-module.

\noindent
(b) For any $M\in O_\e[G^d]\lmod^{G_\e}$, the $\fu_\e$-invariant part $M^{\fu_\e}$ belongs to $O_\e[G^d]\Mod^{G^d}$, the category of $G^d$-equivariant coherent sheaves on $\Spec(O_\e[G^d])=G^d$.    
\end{Lem}
\begin{proof}(a) This follows if we can show that $M^{\fu_\e}$ is a $\dU_\e(\g)$-submodule of $M$. Because by that, $M^{\fu_\e}$ is a rational $\dU_\e(\g)$-module and  the action of $\dU_\e(\g)$ on $M^{\fu_\e}$ will factor through the morphism $\cFr: \dU_\e(\g) \rightarrow \dU_\BF(\g^d)$ by Lemma \ref{KerFr}.

As in Remark \ref{rem: ideal EF in Lus form}, the left ideal of $\dU_\e(\g)$ generated by $\{ \tE_i, \tF_i\}_{1\leq i \leq r}$ is equal to the right ideal of $\dU_\e(\g)$ generated by $\{ \tE_i, \tF_i\}_{1\leq i \leq r}$; and the weight space $\dU_\e(\g)_\nu$ with $\nu \not \in Q^*$ is contained in this left ideal of $\dU_\e(\g)$.

Let $u$ be a weight element in $\dU_\e(\g)$ and $m \in M^{\fu_\e}$. We will show that $um \in M^{\fu_\e}$. If the weight of $u$ is not contained in $Q^*$, then by the second paragraph, $um=0$. Therefore, we can assume that the weight of $u$ is contained in $Q^*$. Then $K^{2\lambda}u=uK^{2\lambda}$ for all $\lambda \in P$, hence $K^{2\lambda}um=uK^{2\lambda}m=um$. Meanwhile, the weights of $\tE_iu, \tF_i u$ are not contained in $Q^*$, therefore, by the second paragraph, $\tE_iu$ and $\tF_iu$ are contained in the left ideal of $\dU_\e(\g)$ generated by $\{ \tE_i, \tF_i\}_{1\leq i \leq r}$. This implies that $\tE_ium=\tF_ium=0$.

\noindent
(b) Since $\fu_\e$ acts on $O_\e[G^d]$ via the counit and using  part (a), $M^{\fu_\e}$ is a $\dU_\BF(\g^d)$-equivariant $O_\e[G^d]$-module. Since $O_\e[G^d]$ is Noetherian and $M$ is finitely generated over $O_\e[G^d]$, it follows that $M^{\fu_\e}$ is finitely generated over $O_\e[G^d]$.
\end{proof}
\begin{proof}Assume that $M\neq 0$. Then there is a finite dimensional rational $\dU_\e(\g)$-module $V$ such that $\Hom_{\dU_\e(\g)}(V, M)\neq 0$. This implies that $(M\otimes V^*)^{\fu_\e}=\Hom_{\fu_\e}(V,M) \neq 0$. On the other hand, $M\otimes V^*\in O_\e[G^d]\Mod^{G^d}$ and $(M\otimes V^*)_f=0$, where $O_\e[G^d]$ acts only on the first factor in $M\otimes V^*$. Therefore, $(M\otimes V^*)^{\fu_\e}\neq 0$ belongs to $O_\e[G^d]\Mod^{G^d}$, and $(M\otimes V^*)^{\fu_\e}_f=0$. However, the latter is impossible. Indeed, the support of $(M\otimes V^*)^{\fu_\e}$ must be a $G^d(\BF)$-invariant closed subvariety of $G^d=\Spec(O_\e[G^d])$, while $(M\otimes V^*)^{\fu_\e}_f=0$ implies that this support is contained in $G^d(\BF)\backslash G^d_0(\BF)$. However, the open Bruhat cell intersects nontrivially any conjugacy classes in $G^d(\BF)$, see part (a) of Lemma \ref{lem: Zfin in field F}. Therefore, the support of $(M\otimes V^*)^{\fu_\e}$ must be empty. Thus, $(M\otimes V^*)^{\fu_\e}=0$, contradiction.
\end{proof}

We are now ready to finish the proof of Proposition \ref{prop: Oe[G] is proj}
\begin{Prop} $O_\e[G]$ is a projective left and right module over $O_\e[G^d]$.
\end{Prop}

\begin{proof}As $O_\e[G^d]$ is a central subalgebra of $O_\e[G]$ by Remark \ref{rem: Oq[G^d] is central} , it is enough to show that $O_\e[G]$ is projective as a right $O_\e[G^d]$-module. Since $(O_\e[G])_f \cong U^{fin}_\e[K^{2\lambda_0}]\cong U^{ev}_\e$ by Proposition \ref{prop: basis of Ufin} and $O_\e[G^d]_f\cong Z_{Fr}$, it follows by Corollary \ref{U^ev is free over Z_e} that $(O_\e[G])_f$ is free as a right $O_\e[G^d]_f$-module. By viewing $O_\e[G]$ as an object in $O_\e[G^d]\rmod^{G_\e}$, we know that all $\Ext^i_{O_\e[G^d]\rmod}(O_\e[G], O_\e[G^d])$ are objects in $O_\e[G^d]\lmod^{G_\e}$ by Lemma \ref{lem: internal Ext}. 

Furthermore, since $(O_\e[G])_f$ is a free right $(O_\e[G^d])_f$-module, we must have
\[ (\Ext^i_{O_\e[G^d]\rmod}(O_\e[G], O_\e[G^d]))_f=0, \qquad, \forall i >0\]
By Lemma \ref{lem: zero localization}, it follows that 
\begin{equation}\label{eq:O_Ext_vanishing}\Ext^i_{O_\e[G^d]\rmod}(O_\e[G], O_\e[G^d])=0, \forall i>0.\end{equation} Since the variety $G^d=\Spec(O_\e[G^d])$ is smooth, hence Cohen-Macaulay, and $O_\e[G]$ is finitely generated over $O_\e[G^d]$, (\ref{eq:O_Ext_vanishing}) implies that $O_\e[G]$ is a projective right $O_\e[G^d]$-module.
\end{proof}
\subsection{$O_q[G]\cong U^{fin}_q$ for Noetherian $\CA_\sN$-algebra $R$}\

\subsubsection{A technical result about filtrations on $O_\e[G]$ and $U^{ev}_\e(\g)$.}\

In this section, we assume that $R=\BF$ is a field. Let us refine $P$ into a linearly  ordered set, in particular,  $P_+$ has elements $ \lambda_1< \lambda_2< \dots$ We note that the results in this section depend on the refined order on $P$.

By Lemma \ref{prop: good filtration on Oq[G]} , there is an exhaustive good $\dU_\e(\g) \otimes_\BF \dU_\e(\g)$-module filtration on $O_\e[G]$ where $M_i$ is the maximal $\dU_\e(\g)\otimes_\BF \dU_\e(\g)$-submodule of $O_\e[G]$ whose weights are bounded by $(\lambda_i, \lambda^*_i)$. 
\[ M_1\subset M_2 \subset M_3 \subset \dots\]
The $\dU_\e(\g)\otimes \dU_\e(\g)$-action on $O_\e[G]$ is given by $(x\otimes y)f(u)=f(S(y)uS^2(x))$ for all $x,y,u \in \dU_\e(\g)$ and $f\in O_\e[G]$. Let $c_{f,v}$ be a matrix coefficient of $V\in \Rep^{fd}(\dU_\e(\g))$, then the action reads $(x\otimes y)c_{f,v}=c_{yf, S^2(x)v}$. From this description of the action  and the property that the weights of $M_i$ are bounded by $(\lambda_i,\lambda_i^*)$, it follows that 
\begin{Lem}\label{lem: mat-coeffs in Mi} $M_i$ is spanned by the matrix coeffecients of all $V\in \Rep^{fd}(\dU_\e(\g))$ such that the weights of $V$ are bounded by $\lambda_i$.
\end{Lem}
\begin{proof} Matrix coefficients of $V \in \Rep^{fd}(\dU_\e(\g))$ is contained in $M_i$ since $M_i$ is the maximal $\dU_e(\g) \otimes_{\BF} \dU_\e(\g)$-submodule of $O_\e[G]$ with the weight space bounded by $(\lambda_i, \lambda_i^*)$. Let us prove the other direction.

Let $\phi \in M_i$. Let $V:=\dU_\e(\g) \phi \subset M_i$ denote the $\dU_\e(\g)$-submodule under the first $\dU_\e(\g)$-action. It follows that the weights of $V$ are bounded by $\lambda_i$ and the weights  of $V^*$ are bounded by $\lambda^*_i$. Let $f\in V^*$ be defined by $f(v)=v(1)$ for all $v\in V \subset O_\e[G]$. 

Recall the $\dU_\e(\g) \otimes_\BF \dU_\e(\g)$-linear map $V\otimes V^* \rightarrow O_\e[G]$ defined by $v\otimes f \mapsto c_{f, K^{-2\rho} v}$. The image of $K^{2\rho} \phi\otimes f$ is $c_{f, \phi}$. This element coincides with $\phi$, indeed,
\[ c_{f, \phi}(u)=f(u\phi)=u\phi(1)=\phi(u), \qquad \text{for all $u \in \dU_\e(\g)$}.\]
This finishes the proof.
\end{proof}
As in Section \ref{ssec: REA-Ufin nonroot case}, $U^{ev}_\e(\g)$ has a $\dU_\e(\g)$-module filtration labeled by the weight lattice $P$ such that $ \deg(\tE_i K^{\zlambda_i})=0,\;\;\deg(\tF_i K^{-\zmu_i})=0,\;\;\deg(K^{2\lambda})=-\lambda\;$
for $1\leq i \leq r, \lambda \in P$. Set 
\[ \tU^{ev>}_\e=:\bigoplus_{\mu \in Q_+} U^{ev>}_{\e,\mu}K^{\gamma(\mu)}, \qquad \tU^{ev<}_\e:=\bigoplus_{\mu \in Q_+} U^{ev<}_{\e,-\mu}(\g)K^{\kappa(\mu)}.\]
Then 
\[ U^{ev}_{\e, \leq \lambda}:=\bigoplus_{\mu \leq \lambda} \tU^{ev<}_\e \otimes K^{-2\mu}\otimes\tU^{ev>}_\e.\]
\begin{Rem}\label{rem: a filtration of Uev}  $U^{ev}_\e(\g)=\bigcup_{\lambda_i \in P_+} U^{ev}_{\e, \leq \lambda_i}$.
\end{Rem}
\begin{Lem}\label{lem: filtrations on Oe[G] and Uev}Let $R=\BF$ a field. Then for all $\lambda_i\in P_+$ we have 
\[(U^{ev}_{\e, \leq \lambda_i})^{fin}=U^{ev}_{\e, \leq  \lambda_i} \cap O_\e[G]=M_i.\]
\end{Lem}
\begin{proof}The equality $(U^{ev}_{\e, \leq \lambda_i})^{fin}=U^{ev}_{\e, \leq \lambda_i}\cap O_\e[G]$ holds since $O_\e[G]\cong U^{fin}_\e$.

 Let $c_{f,v}$ be a matrix coefficient of $V\in \Rep^{fd}(\dU_\e(\g))$ such that $c_{f,v}\in U^{ev}_{\e, \leq \lambda_i}$. By Lemma  \ref{lem: mat-coeffs in Mi}, it is enough to show that $c_{f,v}$ can be regarded as a matrix coefficient of some module $V'\in \Rep^{fd}(\dU_\e(\g))$ such that the weights of $V'$ are bounded by $\lambda_i$.

Note that we can assume $f,v$ be weight vectors of $V^*, V$, respectively. Furthermore, we can assume $V=\dU_\e(\g) v$. Let $N$ be the $\dU_\e(\g)$-submodule of $\dU_\e(\g)v$ generated by all weight vectors $v_i$ of weight bigger than $\lambda_i$. Then we will show that $V'$ can be chosen to be $V/N$.

In Lemma \ref{lem: matrix coeff in Uev}, we constructed the element $c_{f,v}\in U^{ev}_\e(\g)$. That is $c_{f,v}=\sum_{\mu, \nu} u_{\mu\nu}$ for $u_{\mu, \nu}\in (U^{ev<}_{\e,-\mu}K^{\kappa(\mu)})K^{-2\weight(v)-2\mu} (U^{ev>}_{\e,\nu} K^{\gamma(\nu)})$ such that  
\[ \<u_{\mu \nu}, (\dy K^{\kappa(\nu)}) u_0 (\dx K^{\gamma(\mu)})\>'= \hat{\chi}_{\weight(v)+\mu}(u_0)f((\dy K^{\kappa(\nu)})(\dx K^{\gamma(\mu)})v),\]
for all $\dy \in \dU^<_{\e, -\nu}$, $\dx \in \dU^>_{\e, \mu}$ and $u_0\in \dU^0_\e$. In particular, $u_{\mu\nu}=0$ if and only if $f((\dy K^{\kappa(\nu)})(\dx K^{\gamma(\mu)})v)=0$ for $\dx, \dy, u_0$ as above. 

Since $c_{f,v}\in U^{ev}_{\e, \leq \lambda_i}$, it follows that $u_{\mu, \nu} =0$ for all $ \mu > \lambda_i-\weight(v)$. Hence $f((\dy K^{\kappa(\nu)})(\dx K^{\gamma(\mu)})v)=0$ for all above $\dy, \dx, u_0$ such that $\mu> \lambda_i-\weight(v)$. It follows that $f(N)=0$, hence $c_{f,v}$ can be regarded as a matrix coefficient of $V'=V/N$.   
\end{proof}


\subsubsection{$O_q[G]\cong U^{fin}_q$}\

We make a remark that for any weight $\dU_\e(\g)$-module $M=\oplus_\lambda M_\lambda$, the $\dU_\e(\g)$-locally finite part $M^{fin}$ is also the maximal rational $\dU_\e(\g)$-submodule of $M$. 
\begin{Lem}\label{lem: fin of even/Oq[G]} Let $R=\BF$, a field. Then the $\dU_\e(\g)$-locally finite part $(U^{ev}_\BF(\g)/O_\BF[G])^{fin}$ is zero.
\end{Lem}

\begin{proof}We have a short exact sequence
\begin{equation}\label{eq: sequence of O[G] and Uev} 0\rightarrow O_\BF[G]\rightarrow U^{ev}_\BF(\g)\xrightarrow[]{\pi} U^{ev}_\BF(\g)/O_\BF[G] \rightarrow 0.
\end{equation}
Assume that $(U^{ev}_\BF(\g)/O_\BF[G])^{fin} \neq 0$ then there is a finite dimensional rational $\dU_\e(\g)$-submodule $V \neq 0$ in $U^{ev}_\BF(\g)/O_\BF[G]$. By Remark \ref{rem: a filtration of Uev}, there $\lambda_i$ such that the image of $M:=U^{ev}_{\BF, \leq \lambda_i}$ under the quotient $\pi: U^{ev}_\BF(\g) \rightarrow U^{ev}_\BF/O_\BF[G]$ contains $V$. By Lemma \ref{lem: filtrations on Oe[G] and Uev}, $M^{fin}$ is a finite dimensional rational $\dU_\e(\g)$-module.

The short exact sequence \eqref{eq: sequence of O[G] and Uev} gives a rise to a short exact sequence
\[ 0\rightarrow M^{fin}\rightarrow M \rightarrow \pi(M)\rightarrow 0.\]
We note that  $M^{fin}$ and $V$ are finite dimensional rational $\dU_\e(\g)$-modules and $M$ is a weight $\dU_\e(\g)$-module. Therefore, the preimage $\pi^{-1}(V)\subset M$ under the map $\pi: M \rightarrow \pi(M)$ must be a rational $\dU_\e(\g)$-submodule properly containing $M^{fin}$. This leads to a contradiction since $M^{fin}$ is the maximal rational $\dU_\e(\g)$-submodule of $M$.  
\end{proof}

\begin{Prop}\label{prop: REA-Ufin domain case}  The inclusion $O_R[G]\hookrightarrow U^{fin}_R$ is an isomorphism when $R$ is a domain.
\end{Prop}
\begin{proof} For any nonzero $x\in R$, we have the following commutative diagram, where the rows are exact while all vertical morphisms are injective: 
\[ \begin{tikzcd} 0\arrow[r]& O_R[G] \arrow[d, hook]\arrow[r,"\cdot x"] & O_R[G]\arrow[d, hook] \arrow[r,]& O_{R/(x)}[G] \arrow[d, hook]\arrow[r] & 0\\
0\arrow[r]& U^{ev}_R(\g) \arrow[r, "\cdot x"] & U^{ev}_R(\g) \arrow[r]& U^{ev}_{R/(x)}(\g) \arrow[r]&0
\end{tikzcd}
\]
By the Snake Lemma, the morphism 
\[ U^{ev}_R(\g)/O_R[G] \xrightarrow[]{\cdot x} U^{ev}_R(\g) /O_R[G]\]
is injective for all nonzero $x\in R$. In  other words, $U^{ev}_R(\g) /O_R[G]$ is torsion free over $R$.  Let $\mathbb{K}$ be the fraction field of $R$. Then we have the inclusion 
\[ U^{ev}_R(\g)/O_R[G]\hookrightarrow (U^{ev}_R(\g)/O_R[G])\otimes_R \mathbb{K} \cong U^{ev}_\mathbb{K}(\g)/O_\mathbb{K}[G].\]
Since the $\dU_\mathbb{K}(\g)$-locally finite part of $U^{ev}_\mathbb{K}(\g)/O_\mathbb{K}[G]$ is zero by Lemma \ref{lem: fin of even/Oq[G]}, it follows that $(U^{ev}_R(\g)/O_R[G])^{fin}=0$. On the other hand, we have a left exact sequence
\[ 0\rightarrow O_R[G]\rightarrow U^{fin}_R\rightarrow (U^{ev}_R/O_R[G])^{fin}.\]
Therefore, $O_R[G]\cong U^{fin}_R$.
\end{proof}

The following is the main result of this section.

\begin{Thm}\label{prop: REA-Ufin general case}For any Noetherian $\CA_\sN$-algebra $R$, the inclusion $O_R[G]\hookrightarrow U^{fin}_R$ is an isomorphism.
\end{Thm}
\begin{proof}

Let $\mathfrak{p}$ be a prime ideal of $R$. Let $\bullet^{fin, \dU_{R/\mathfrak{p}}}$ denote the $\dU_{R/\mathfrak{p}}(\g)$-locally finite part. We reserve the notation $\bullet^{fin}$ for the $\dU_R(\g)$-locally finite part. 
Then one see that 
\[ (U^{ev}_{R/\mathfrak{p}}(\g))^{fin}\cong (U^{ev}_{R/\mathfrak{p}}(\g))^{fin, \dU_{R/\mathfrak{p}}}.\]
By Proposition \ref{prop: REA-Ufin domain case}, we have the isomorphism
\[ O_{R/\mathfrak{p}}[G]\rightarrow (U^{ev}_{R/\mathfrak{p}}(\g))^{fin, \dU_{R/\mathfrak{p}}}.\]
Note that $O_R[G]\otimes_R R/\mathfrak{p}=O_{R/\mathfrak{p}}[G]$ and $U^{ev}_R(\g) \otimes R/\mathfrak{p}=U^{ev}_{R/\mathfrak{p}}(\g)$. Therefore, the natural morphism 
\begin{equation}\label{eq: iso over R/p} O_R[G]\otimes_R R/\mathfrak{p}\rightarrow (U^{ev}_R(\g) \otimes_R R/\mathfrak{p})^{fin}
\end{equation}
is an isomorphism for any prime ideal $\mathfrak{p}$ of $R$.

Since $R$ is Noetherian, there is a finite $R$-module filtration $0=M_0\subset M_1\subset \dots \subset M_n =R$ such that $M_i/M_{i-1}\cong R/\mathfrak{p}_i$ with some prime ideal $\mathfrak{p}_i$ for all $1\leq i \leq n$. We will prove by induction that the homomorphisms
\[ O_R[G]\otimes_R M_i \rightarrow (U^{ev}_R\otimes_R M_i)^{fin}\]
are isomorphisms for all $1\leq i \leq n$. The base case $i=1$ holds by \eqref{eq: iso over R/p}. Let us do the induction step. Since $O_R[G]$ and $U^{ev}_R(\g)$ are free over $R$, we have the commutative diagram
\[\begin{tikzcd}0 \arrow[r]& O_R[G]\otimes_R M_{i-1}\arrow[d] \arrow[r]& O_R[G]\otimes_R M_i \arrow[d] \arrow[r] &O_R[G]\otimes_R R/\mathfrak{p}_i \arrow[d]\arrow[r] &0\\
0\arrow[r] &U^{ev}_R\otimes_R M_{i-1}\arrow[r]& U^{ev}_R(\g) \otimes_R M_i \arrow[r]& U^{ev}_R(\g) \otimes_R R/\mathfrak{p}_i \arrow[r] &0  
\end{tikzcd}
\]
which gives us the commutative diagram
\begin{equation}\label{eq: big diagram}\begin{tikzcd}0 \arrow[r]& O_R[G]\otimes_R M_{i-1}\arrow[d] \arrow[r]& O_R[G]\otimes_R M_i \arrow[d] \arrow[r] &O_R[G]\otimes_R R/\mathfrak{p}_i \arrow[d]\arrow[r] &0\\
0\arrow[r] &(U^{ev}_R\otimes_R M_{i-1})^{fin}\arrow[r]& (U^{ev}_R(\g) \otimes_R M_i)^{fin} \arrow[r]& (U^{ev}_R(\g) \otimes_R R/\mathfrak{p}_i)^{fin} & 
\end{tikzcd}
\end{equation}
The first vertical arrow  in \eqref{eq: big diagram} is an isomorphism by induction hypothesis while the third vertical arrow in \eqref{eq: big diagram} is an isomorphism by \eqref{eq: iso over R/p}. Therefore, the middle vertical morphism is an isomorphism. This finishes the induction. In particular, when $i=n$,  we have the desired isomorphism $O_R[G]\cong U^{fin}_R$. 
\end{proof}

\subsection{Harish-Chandra center}
\label{sec: HC center}

\begin{defi}{\it The Harish-Chandra center } $Z_{HC} \subset U^{ev}_q(\g)$ is defined to be the $\dU_q(\g)$-invariant part of $U^{ev}_q(\g)$.
\end{defi}
\begin{Lem}$Z_{HC}$ is central in $U^{ev}_q(\g)$ and $Z_{HC}\subset U^{ev}_q(\g)_0$, the zero weight space of $U^{ev}_q(\g)$.
\end{Lem}
\begin{proof}
Since $Z_{HC}=\Hom_{\dU_q(\g)}(W_q(0), U^{ev}_q(\g))$, the second statement holds. On the other hand,
\[ \ad'(K^{\lambda})(x)=K^{\lambda}xK^{-\lambda}, \qquad \ad'(\tE_i)(x)=[\tE_i,x]K^{\zlambda_i}, \qquad \ad'(\tF_i)(x)=[\tF_i, x]K^{-\zmu_i},\]
for $x\in U^{ev}_q(\g)$. Therefore, for $z\in Z_{HC}$, we must have $z$ commutes with $\tE_i, \tF_i, K^\lambda$ for $1\leq i \leq r$ and $\lambda \in 2P$, hence, $z$ is central in $U^{ev}_q(\g)$.
\end{proof}
The zero weight space $U^{ev}_q(\g)_0$ has a decomposition:
\[ U^{ev}_q(\g)_0=U^{ev0}_q\oplus \bigoplus_{\mu \in Q_+}U^{ev<}_{q,-\mu}U^{ev0}_qU^{ev>}_{q, \mu}.\]
Using the commutation relations between $\tE_i, \tF_i$ and $K^\lambda$, one can easily see that $\bigoplus_{\mu \in Q_+}U^{ev<}_{q,-\mu}U^{ev0}_qU^{ev>}_{q, \mu}$ is a two-sided ideal in $U^{ev}_q(\g)_0$. Therefore, we have an algebra homomophism
\[ \pi: Z_{HC}\hookrightarrow U^{ev}_q(\g)_0 \twoheadrightarrow U^{ev0}_q(\g).\]
\begin{Thm}\label{thm: HC-center}The homomorphism $\pi$ defines an algebra isomorphism between $Z_{HC}$ and the following subalgebra of $U^{ev0}_q(\g)$:
\begin{equation}\label{eq: W-inv of Uev0} R\Big[K^{\pm 2\w_1}, \dots, K^{\pm 2\w_r}\Big]^{W_\bullet},
\end{equation}
where we use the dot-action of the Weyl group $W$ on $U^{ev0}_q(\g)$ extended by linearity from:
\[ w_\bullet(K^\mu)= q^{(w^{-1}\rho-\rho, \mu)}K^{w(\mu)}\qquad\textnormal{for all $x\in W,$ $\mu \in 2P$}.\]
\end{Thm}
The results on Harish-Chandra center for the usual De Concini-Kac forms at generic $q$ can be found in \cite[Theorem 6.25]{j}. See \cite[$\mathsection 6.2$]{dckp} for results  when $q$ is an odd order root of unity, while see \cite[Proposition 3.3]{t2} for results  when $q$ is an even order root of unity.
\begin{proof}By Lemma \ref{lem: basis of invariant of R[U]}, $O_q[G]^{\dU_q}$ is a free $R$-module with a basis $c_{\lambda}=\sum_{i} c_{v^*_i, K^{-2\rho} v_i}$ for all $\lambda\in P_+$, where $\{v_i\}$ is a weight basis of the Weyl module $W_q(\lambda)$. Hence, $Z_{HC}$ is a free $R$-module with basis $\{ \hiota(c_\lambda)|\lambda \in P_+\}$ by the isomorphism $\hiota: O_q[G] \cong U_q^{fin}$.

Let us compute $\pi\circ \hiota(c_\lambda)$. This element is uniquely determined by 
\[ \Big\<\pi \circ \hiota(c_\lambda), u\Big\>=c_\lambda(u),\]
for all $u \in \dU^0_q$. On the other hand, let $P_{+,\lambda}$ be the set of dominant weights in $W_q(\lambda)$, then it is straightforward to see that
\[ c_\lambda(u)=\sum_{\mu \in P_{+, \lambda}} \text{rank}\big(W_q(\lambda)_\mu \big) \sum_{\mu' \in W\mu} q^{(-2\rho, \mu')} \hat{\chi}_{\mu'}(u),\]
for all $u \in \dU^0_q$. Therefore,
\[ c_\lambda(u)=\Big\< \sum_{\mu \in P_{+, \lambda}} \text{rank}\big(W_q(\lambda)_\mu \big) \sum_{\mu'\in W\mu} q^{(2\rho, \mu')}K^{2\mu'}, u\Big\>\]
for all $u \in \dU^0_q$. It follows that 
\begin{equation}\label{eq: image of clambda under HC} \pi \circ \hiota(c_\lambda)=\sum_{\mu \in P_{+, \lambda}} \text{rank}\big(W_q(\lambda)_\mu\big)\sum_{\mu'\in W\mu}q^{(\rho, 2\mu')}K^{2\mu'}.
\end{equation}
It is an easy exercise to show that the collection of elements in the right hand side of \eqref{eq: image of clambda under HC} over all $\lambda \in P_+$ forms a $R$-basis of the free $R$-module \eqref{eq: W-inv of Uev0}. This implies that $\pi$ defines an algebra homomorphism between $Z_{HC}$ and the  algebra \eqref{eq: W-inv of Uev0}.
\end{proof}


\section{Poisson Geometry of center $Z$ of $U^{ev}_\e(\g)$ over $\BC$}\label{sec:PoissonGeometry}

The goal of this section is to generalize the results about the center and Azumaya locus of De Concini-Kac forms at odd order roots of unity to the even part algebra $U^{ev}_\e(\g)$ at both odd and even order roots of unity. Most of the arguments closely follow~\cite{dck, dckp, dckp2}.

\subsection{Poisson bracket on $Z$ and $Z_{Fr}$}\

Let $q=\e\in \BC$ be a root of unity of order $\ell$ such that $\ell_i > \max\{ 2, 1-a_{ij}\}_{1\leq j \leq r}$. The center $Z$ of $U^{ev}_\e(\g)$ naturally carries a Poisson structure via quantization. 

Consider 
\begin{equation}\label{eq: A'}
  \CA'=\BC[v,v^{-1}]\Big[\frac{1}{v^{2k}-1}\Big]_{1\leq k\leq \max\{\sd_i\}}, 
\end{equation}
and let $\varphi_\e\colon U^{ev}_{\CA'}(\g)\rightarrow U^{ev}_\e(\g)$ denote the natural epimorphism sending $v$ to $\e$. Then, any $x\in \varphi^{-1}_\e(Z)$ naturally gives rise to a derivation of the algebra $U^{ev}_\e(\g)$ as follows: for any $u \in U^{ev}_\e(\g)$, pick its arbitrary lift $\hat{u}$ in $U^{ev}_{\CA'}(\g)$ and define
\begin{equation}\label{eq: derivation}
  \{x,u\}:=\varphi_{\e}\left(\frac{[x,\hat{u}]}{v-\e}\right).
\end{equation}

\begin{Rem}\label{rem: Der and Poisson}
(a) We note that $\{x,u\}$ is clearly independent of the choice of a lift $\hat{u}\in \varphi^{-1}_\e(u)$.

\noindent
(b) The above construction also gives rise to a  Poisson bracket in $Z$ as follows: given $u, w\in Z$, pick any lifts $\hat{u}\in \varphi^{-1}_\e(u)$ and $\hat{w}\in \varphi^{-1}_\e(w)$, and define:
\begin{equation}\label{eq: poisson bracket}
  \{u,w\}:=\varphi_\e\left(\frac{[\hat{u}, \hat{w}]}{v-\e}\right).    
\end{equation}

\noindent
(c) We note that the derivation $\{x,-\}$ of the entire algebra $U^{ev}_\e(\g)$, defined in \eqref{eq: derivation}, does actually depend on $x$ and not just on its specialization $\varphi_\e(x)\in Z$, in contrast to (b).
\end{Rem}

Recall the ring homomorphism $\sigma\colon \CA'\rightarrow \BC$ sending $v$ to $\e$. We define a nonzero scalar $\sA_i\in \BC$ via:
\begin{equation}
  \sA_i:=\sigma\left(\frac{(\ell_i)_{v_i}!}{v-\e}\right).
\end{equation}

\begin{Prop}\label{prop: ad and poisson bracket} 
For any $u\in Z$, we have the following equalities:
\begin{equation}\label{eq: ad and poisson bracket}
  \ad'\Big(\sA_i\tE_i^{(\ell_i)}\Big)(u)=\{\tE_i^{\ell_i}, u\}K^{\ell_i \zlambda_i}\qquad \textnormal{and} \qquad \ad'\Big(\sA_i\tF_i^{(\ell_i)}\Big)=\{\tF_i^{\ell_i}, u\}K^{-\ell_i \zmu_i},
\end{equation}
where the Poisson brackets $\{\tE_i^{\ell_i}, u\}$ and $\{\tF_i^{\ell_i}, u\}$ are defined by \eqref{eq: poisson bracket} in Remark \ref{rem: Der and Poisson}.
\end{Prop}

\begin{proof}
We shall prove only the first equality. According to \eqref{eq:twisted Hopf on divided powers}, we have in $U^{ev}_{\CA'}(\g)$:
\begin{align*}
  &\Delta'(\tE_i^r)=\sum_{c=0}^{r}v_i^{2(r-c)c}\binom{r}{c}_{v_i} \tE_i^{r-c}\otimes K^{-(r-c)\zlambda_i}\tE_i^c,\\
  &S'(K^{-(r-c)\zlambda_i}\tE_i^c)=(-1)^c v_i^{c(c-1)}\tE_i^c K^{r\zlambda_i}.
\end{align*}
Therefore, for any $r\in \BN$ and $y\in U^{ev}_{\CA'}(\g)$, we have:
\begin{multline*}
  \ad'(\tE_i^r)(y)=\sum_{c=0}^r (-1)^c v_i^{c(2r-c-1)}\binom{r}{c}_{v_i} \tE_i^{r-c}y \tE_i^c K^{r\zlambda_i}=\\
  \left(\tE_i^ry+\sum_{c=1}^r(-1)^c v_i^{c(2r-c-1)}\binom{r}{c}_{v_i} y\tE_i^r+\sum_{c=1}^{r-1}(-1)^c v_i^{c(2r-c-1)}\binom{r}{c}_{v_i}[\tE_i^{r-c},y]\tE_i^c\right)K^{r\zlambda_i}.
\end{multline*}
Evoking the identity
\begin{equation*}
    \sum_{c=0}^r (-1)^c v_i^{c(2r-c-1)}\binom{r}{c}_{v_i} = \, 
    \sum_{c=0}^r (-1)^c v_i^{c(r-1)} \bmat{r\\c}_{v_i} = \, 0,
\end{equation*}
we thus obtain:
\begin{equation}\label{eq: ad(E^r)(y)}
  \ad'(\tE_i^r)(y)=[\tE^r_i,y]K^{r\zlambda_i} + 
  \sum_{c=1}^{r-1}(-1)^c v_i^{c(2r-c-1)}\binom{r}{c}_{v_i} [\tE_i^{r-c},y]\tE_i^c K^{r\zlambda_i}.
\end{equation}
Applying the formula \eqref{eq: ad(E^r)(y)} to $r=\ell_i$ and $y\in \varphi^{-1}_\e(u)$ with $u\in Z$, we then obtain:
\begin{equation*}
  \ad'\left(\sA_i\tE_i^{(\ell_i)}\right)(u) = 
  \varphi_\e\left(\frac{[\tE_i^{\ell_i},y]}{v-\e}K^{\ell_i \zlambda_i} + \sum_{c=1}^{\ell_i-1}(-1)^c v_i^{c(2\ell_i-c-1)}\frac{\binom{\ell_i}{c}_{v_i}}{v-\e}[\tE_i^{\ell_i-c},y]\tE^c_iK^{\ell_i\zlambda_i}\right).
\end{equation*}
As $\varphi_\e([\tE^k_i,y])=[\tE^k_i,u]=0$, we finally get
\[ 
  \ad'\Big(\sA_i \tE_i^{(\ell_i)}\Big)(u)=\{\tE_i^{\ell_i}, u\}K^{\ell_i \zlambda_i} \qquad \forall\, u \in Z.
\]
This completes our proof of the first equality in \eqref{eq: ad and poisson bracket}.
\end{proof}

It turns out that the Frobenius center $Z_{Fr}$ is a Poisson subalgebra of $Z$:

\begin{Prop}\label{prop: Possion structure of ZFr}
(a) The center $Z_{Fr}$ is closed under the Poisson bracket of $Z$.

\noindent
(b) $Z_{Fr}$ is generated by $\Big\{K^\lambda, \tE_i^{\ell_i}, \tF_i^{\ell_i}\Big\}$ for $1\leq i \leq r$ and $\lambda \in 2P^*$ as a Poisson algebra.

\noindent
(c) Let us recall the identification $Z_{Fr} \simeq \BC[G^d_0]$ of Proposition~\ref{Ze and Bruhat cell}. The symplectic leaves of $\Spec Z_{Fr} \cong G^d_0$ are the intersections of conjugacy classes of $G^d$ with $G^d_0$. 
\end{Prop}

\begin{proof}
First, we will prove parts (a) and (b) simultaneously.

\noindent
{\it Step 1:} We show that $Z_{Fr}$ is closed under the derivations $\{\tE_i^{\ell_i}, -\}$, $\{ \tF_i^{\ell_i},-\}$ and $\{K^\lambda, -\}$ for all $1\leq i \leq r$ and $\lambda \in 2P^*$. By Proposition \ref{prop: ad and poisson bracket}, we have: 
\begin{equation}\label{eq: dev of E and F}
  \ad'(\sA_i \tE^{(\ell_i)})(u)=\{\tE^{\ell_i}, u\}K^{\ell_i\zlambda_i}, \qquad 
  \ad'(\sA_i \tF^{(\ell_i)})(u)=\{\tF_i^{\ell_i}, u\}K^{-\ell_i \zmu_i}
\end{equation}
for any $u\in Z$. Since $Z_{Fr}$ is closed under the adjoint action of $\dU_\e(\g)$ (see Remark~\ref{actions of U(gd)}) and $K^{-\ell_i\zlambda_i}, K^{\ell_i\zmu_i}\in Z_{Fr}$, it follows that $Z_{Fr}$ is closed under the derivations $\{\tE_i^{\ell_i}, -\}$ and $\{\tF_i^{\ell_i},-\}$.

Let $\displaystyle\theta:=\frac{v^\ell-1}{v-\e}\Big|_\e$. Then for any homogeneous element $u\in Z_{Fr}$ and $\lambda \in 2P^*$, we have
\begin{equation}\label{eq: dev of K}
  \{K^\lambda, u\}=\theta\frac{(\lambda, \weight(u))}{\ell}uK^\lambda.
\end{equation}
Therefore, $Z_{Fr}$ is also closed under the derivations $\{K^\lambda,-\}$ for any $\lambda \in 2P^*$.

\noindent
{\it Step 2:} Let $\mathcal{P}$ denote the Poisson subalgebra  of $Z$ generated by $\{\tE_i^{\ell_i}, \tF_i^{\ell_i}, K^\lambda\}_{1\leq i \leq r}^{\lambda \in 2P^*}$. Then by Step 1, $Z_{Fr}$ contains $\mathcal{P}$. Let us now show that $\mathcal{P}$ contains $Z_{Fr}$, which implies parts (a) and~(b). By Lemma \ref{lem: Zfin in field F} with the notations thereof, we have 
\[ 
  Z_{Fr}=Z^{fin}_{Fr}[K^{\lambda_0}], \qquad Z^{fin}_{Fr}=\bigoplus_{\lambda \in P^*_+} \ad'(\dU_\e(\g))K^{-2\lambda}.
\]
Thus, it is enough to show that $\ad'(\dU_\e(\g))K^{-2\lambda}\in \mathcal{P}$. The action of $\dU_\e(\g)$ on $Z_{Fr}$ factors through $\cFr\colon \dU_\e(\g)\rightarrow \dU_\BC(\g^d)$ and $\dU_\BC(\g^d)$ is generated by $e_i, f_i$ as an algebra. Therefore, $\ad'(\dU_\e(\g))K^{-2\lambda}$ is $\BC$-linearly spanned by elements of the form
\[ 
  \ad'(x_1\dots x_n)K^{-2\lambda} \ \text{ with  }\ \lambda \in 2P^* \ \mathrm{and} \ 
  x_j=\tE^{(\ell_i)}_i \text{ or } \tF^{(\ell_i)}_i\ \text{ for }1\leq i \leq r,
\]
but these elements are contained in $\mathcal{P}$.

\noindent
c) Let $\upsilon_{\tE_i}, \upsilon_{\tF_i}, \upsilon_{K^\lambda}$ be the vector fields on $\Spec{Z_{Fr}}\cong G^d_0$ generated by the derivations $\{\tE_i^{\ell_i},-\}$, $\{ \tF_i^{\ell_i},-\}$, $\{ K^\lambda,-\}$, respectively. Let $e_i, f_i ,\lambda$ be the vector fields on $G^d_0$ coming from the adjoint action of $G^d$ on itself. Therefore, from 
\eqref{eq: dev of E and F} and \eqref{eq: dev of K}, we have that 
\[ 
  \upsilon_{\tE_i}=\frac{\sA_i}{\e^*_i} K^{-\ell_i \zlambda_i} e_i, \qquad 
  \upsilon_{\tF_i}=\sA_i K^{\ell_i\zmu_i} f_i, \qquad 
  \upsilon_{K^\lambda}=\frac{\theta}{\ell}K^\lambda \lambda.
\]
Since $Z_{Fr}$ is a Poisson algebra generated by $\{\tE_i^{\ell_i}, \tF_i^{\ell_i}, K^\lambda\}_{1\leq i \leq r}^{\lambda \in 2P^*}$, we see that the distribution of the Hamiltonian vector fields  of $\Spec{Z_{Fr}}\cong G^d_0$ is contained in the distribution of vector fields arising from the adjoint action of $\g^d$ on $G^d_0$.

On the other hand, for any function $f,g$ and vector fields $\upsilon_1,\upsilon_2$ on $G^d_0$ we have
\[ 
  [f\upsilon_1,g\upsilon_2]= fg[\upsilon_1,\upsilon_2]+f\upsilon_1(g) \upsilon_2-g\upsilon_2(f)\upsilon_1.
\]
Hence we have 
\[ 
  [\upsilon_{\tE_{i_1}}, \cdots [\upsilon_{\tE_{i_{k-1}}}, \upsilon_{\tE_{i_k}}] \cdots] = 
  aK^b[e_{i_1},\cdots[e_{i_{k-1}},e_{i_k}] \cdots] + f
\]
for some $b\in 2P^*$, $a\in \BC^{\x}$, where $f$ is a $\mathbb{C}[G^d_0]$-linear combination of vector fields  corresponding to root vectors attached to roots $\a\in \Delta_+$ such that $\het (\a) < \het(\a_{i_1}+\dots +\a_{i_k})$. There are similar formulas for $\upsilon_{\tF_i}$. From this we see that the distribution of vector fields arising from the adjoint action of $\g^d$ on $\Spec Z_{Fr}\cong G^d_0$ is contained in (and thus is equal to) the distribution of the Hamiltonian vector fields. This implies part (c).
\end{proof}

We now pick a $\BC$-basis $\{z_\jmath\}$ of $Z_{Fr}$ and lift it to elements $\{\hat{z}_\jmath\}$ with $\hat{z}_\jmath\in \varphi^{-1}_\e(z_\jmath)$. Consider the $\BC$-linear map 
\[ 
  D\colon Z_{Fr}\rightarrow \text{Der}_{\BC}(U^{ev}_\e(\g))
\]
determined by $z_\jmath\mapsto \{\hat{z}_\jmath,-\}$ for all $\jmath$. This makes $(Z_{Fr}, U^{ev}_\e(\g))$ into a {\it Poisson order} as defined in \cite[$\mathsection 2.2$]{bg}. By Theorem $4.2$ in \cite{bg}, we have:

\begin{Prop}\label{prop: fiber over symplectic leaves} 
The fibers of $U^{ev}_\e(\g)$ over two points in the same $G^d$-conjugacy classes in $\Spec(Z_{Fr})$, equivalently, in the same  symplectic leaves, are isomorphic as $\BC$-algebras.
\end{Prop}

\subsection{Degeneration of $U^{ev}_q(\g)$}\

Let us recall two linear maps $\kappa, \gamma\colon \h^*\rightarrow \h^*$ defined in \eqref{kappa and gamma}.

One of the main tools in \cite{dckp2} is a remarkable degeneration of the quantum group. We will establish a similar degeneration for $U^{ev}_q(\g)$,  $q\in \text{Spec}(\CA')$ where $\CA'$ is defined in \eqref{eq: A'}.

Following Section~\ref{ssec: quantum DJ}, we fix a reduced expression of the longest element $w_0=s_{i_1}\dots s_{i_N}$. This defines the convex ordering on $\Delta_+$ (``convex'' means that if $\a<\b\in \Delta_+$ and $\a+\b\in \Delta_+$ then $\a<\a+\b <\b$):
\[ 
  \b_1=\a_{i_1}, \qquad \b_2=s_{i_1}(\a_{i_2}),\quad \dots \quad,\qquad \b_N=s_{i_1}\dots s_{i_{N-1}}(\a_{i_N}).
\]
We also recall the root vectors of~\eqref{eq:root-generator}.


\begin{Lem}\label{lem: commutators of E,F}
(a) For any $1\leq i<j\leq N$, one has:
\[ 
  F_{\b_i}F_{\b_j}-q^{(\b_i, \b_j)}F_{\b_j} F_{\b_i}=\sum_{\vec{k}\in \BZ^N_{\geq 0}} a_{\vec{k}}F^{\cev{k}},
\]
where $a_{\vec{k}} \in \BC[q, q^{-1}]$ and $a_{\vec{k}}\neq 0$ only if $\vec{k}=(k_1, \dots, k_N)$ satisfies $k_s=0$ for $s\leq i$ and $s \geq j$.

\noindent
(b) For any $1\leq i<j\leq N$, one has:
\[ 
  E_{\b_j} E_{\b_i}-q^{-(\b_i,\b_j)}E_{\b_i} E_{\b_j}=\sum_{\vec{k}\in \BZ^N_{\geq 0}} b_{\vec{k}} E^{\cev{k}},
\]
where $b_{\vec{k}} \in \BC[q, q^{-1}]$ and $b_{\vec{k}} \neq 0$ only if $\vec{k}=(k_1, \dots, k_N)$ satisfies $k_s=0$ for $s\leq i$ and $s \geq j$.
\end{Lem}

\begin{proof}
According to \cite[Lemma 4.2]{dckp2}, for $1\leq i<j\leq N$, one has
\begin{equation}\label{eq: commutators of E} 
  E_{\b_j}E_{\b_i}-q^{-(\b_i, \b_j)}E_{\b_i}E_{\b_j}=\sum_{\vec{k}\in \BZ^N_{\geq 0}} c_{\vec{k}}E^{\vec{k}},
\end{equation}
where $c_{\vec{k}}\in\BC[q,q^{-1}]$ and $c_{\vec{k}}\neq 0$ only if $\vec{k}=(k_1,\dots, k_N)$ satisfies $k_s=0$ for $s\leq i$ and $s \geq j$. We note that the right-hand side of~\eqref{eq: commutators of E} contains $E^{\vec{k}}$ rather than $E^{\cev{k}}$.

Part (a) now follows by applying the anti-automorphism $\tau$ of~\eqref{eq: anti-involution map} to~\eqref{eq: commutators of E} and evoking that $\tau(E_{\beta_s})=F_{\beta_s}$ for all $s$.

Part (b) follows from~\eqref{eq: commutators of E} and the observation that 
\[ 
  E^{\vec{r}}=\sum_{\vec{k}\in \BZ^N_{\geq 0}} d_{\vec{k}}E^{\cev{k}},
\]
where for $\vec{r}=(0, \dots, 0, r_{i}, \dots, r_{j}, 0, \dots, 0)$ one has $d_{\vec{k}}\neq 0$ only if $\vec{k}=(k_1, \dots, k_N)$ satisfies $k_s=0$ for $s<i$ and $s>j$.
\end{proof}

Evoking the elements $\tE_{\b_k},\tF_{\b_k}$ of~\eqref{eq:twisted-root-vector}, we immediately obtain:

\begin{Cor}\label{cor: commutator of tE, tF} 
(a) For any $1\leq i<j\leq N$, one has:
\[ 
  q^{-(\b_j, \kappa(\b_i)}\tF_{\b_i}\tF_{\b_j}-\tF_{\b_j}\tF_{\b_i}=\sum_{\vec{k}\in \BZ^N_{\geq 0}} a'_{\vec{k}} \tF^{\cev{k}},
\]
where $a'_{\vec{k}} \in \BC[q, q^{-1}]$ and $a'_{\vec{k}}\neq 0$ only if $\vec{k}=(k_1, \dots, k_N)$ satisfgies $k_s=0$ for $s\leq i$ and $s\geq j$.

\noindent
(b) For any $1\leq i<j\leq N$, one has:
\[ 
  \tE_{\b_j}\tE_{\b_i}-q^{-(\b_j, \kappa(\b_i))}\tE_{\b_i}\tE_{\b_j}=\sum_{\vec{k}\in \BZ^N_{\geq 0}} b'_{\vec{k}} \tE^{\cev{k}},
\]
where $b'_{\vec{k}} \in \BC[q, q^{-1}]$ and $b'_{\vec{k}}\neq 0$ only if $\vec{k}=(k_1,\dots, k_N)$ satisfies $k_s=0$ for $s\leq i$ and $s\geq j$.
\end{Cor}

For any $\vec{k}, \vec{r}\in \BZ^N_{\geq 0}$ and $u\in U^{ev\, 0}_q(\g)$, let
\[ 
  M_{\vec{k}, \vec{r}, u}=\tF^{\cev{k}}u \tE^{\cev{r}},
\]
where $\tF^{\cev{k}}=\tF_{\b_N}^{k_N}\dots \tF_{\b_1}^{k_1}$ and $\tE^{\cev{r}}=\tE_{\b_N}^{r_N}\dots \tE_{\b_1}^{r_1}$. For any positive root $\b=\sum_i b_i\a_i$, consider its height $\het(\b):=\sum_i b_i$. We define a total height of $M_{\vec{k}, \vec{r}, u}$ via
\begin{equation}\label{eq: height of monomial} 
  \het(M_{\vec{k}, \vec{r}, u})=\sum_i(k_i+r_i)\het(\b_i),
\end{equation}
and a total degree of $M_{\vec{k}, \vec{r}, u}$ by 
\begin{equation}\label{eq: deg of monomial} 
  \deg(M_{\vec{k}, \vec{r}, u})=(r_N, \dots, r_1, k_N, \dots, k_1, \het(M_{\vec{k}, \vec{r}, u}))\in \BZ^{2N+1}_{\geq 0}.
\end{equation}
This gives rise to a $\BZ^{2N+1}_{\geq 0}$-filtration of $U^{ev}_q(\g)$. 


By Corollary \ref{cor: commutator of tE, tF} and the commutator relations between $\tE_i$ and $\tF_j$ in \eqref{eq:gen-rel-twistedDJ}, we obtain the following analogue of~\cite[Proposition 4.2]{dckp2}:

\begin{Prop}\label{prop: grU} 
The associated graded algebra $\gr U^{ev}_q$ of the $\BZ^{2N+1}_{\geq 0}$-filtered algebra $U^{ev}_q(\g)$ is an algebra over $\BC$  with generators
\[ 
  X_\a=\tE_\a K^{\gamma(\a)}, \qquad Y_\a=\tF_\a, \qquad K^\lambda \quad (\lambda \in 2P)
\]
subject to the following relations:
\begin{align*}
    K^{\lambda_1}K^{\lambda_2}&=K^{\lambda_1+\lambda_2}, \qquad K^0=1,\\
    K^{\lambda}X_\a &=q^{(\lambda,\a)}X_\a K^\lambda, \qquad K^\lambda Y_\a =q^{-(\lambda, \a)}Y_\a K^\lambda,\\
    X_\a Y_\b&=Y_\b X_\a,\\
    X_\a X_\b&=q^{-(\kappa(\a),\b)}X_\b X_\a \qquad \text{if $\a>\b$ in the convex ordering},\\
    Y_\a Y_\b&=q^{-(\a, \kappa(\b))}Y_\b Y_\a \qquad \text{if $\a>\b$ in the convex ordering},    
\end{align*}
in which $\lambda_1, \lambda_2\in 2P$ and $\a, \b\in \Delta_+$.
\end{Prop}

\begin{Rem}[cf.~Remark 4.2(a) of~\cite{dckp2}]\label{rem: chain of filtered algebras} 
Considering the degree by total height $d_0$, we obtain a $\BZ_{\geq 0}$-filtration of $U^{ev}_q(\g)$; let $U^{(0)}_q$ be the associated graded algebra. Letting $d_1(M_{\vec{k}, \vec{r}, u})=k_1$, we obtain a $\BZ_{\geq 0}$-filtration of $U^{(0)}_q$, let $U^{(1)}_q$ be the associated graded algebra. We proceed by defining $\BZ_{\geq 0}$-filtrations on the algebras $U^{(i)}_q$ and getting their associated graded algebras $U^{(i+1)}_q$.  It is clear that at the last step we get the algebra $\gr U^{ev}_q \cong U^{(2N)}_q$.
\end{Rem}

The following corollary is proved in the same way as~\cite[Corollary $1.8$]{dck}:

\begin{Cor}\label{cor: grU has nonzero divisor}
The algebra $\gr U^{ev}_q$ has  no zero divisors. Hence all algebras $U^{(i)}_q$ in Remark~\ref{rem: chain of filtered algebras} have no zero divisors, in particular, $U^{ev}_q(\g)$ has no zero divisors.
\end{Cor}

Now we restrict to the case when $q=\e$ is a root of unity of order $\ell$. Let $A$ be  an algebra with no zero divisors. Let $Z$ be the center of $A$ and $Q(Z)$ be the quotient field  of $Z$, and let $Q(A)=Q(Z)\otimes_Z A$. The algebra $A$ is called {\em integrally closed} if for any subring $B$ of $Q(A)$  such that $A\subset B \subset z^{-1}A$ for some $z\neq 0 \in Z$, we have $B=A$. The following result follows from \cite[Theorem $6.5$]{dp}:

\begin{Prop}\label{prop: Uev is integrally closed}
All algebras $U^{(i)}_\e$ in Remark \ref{rem: chain of filtered algebras} are integrally closed. In particular,  both  $U^{ev}_\e(\g)$ and $\gr U^{ev}_\e$  are integrally closed.
\end{Prop}

\subsection{Finite dimensional representations of $U^{ev}_\e(\g)$}\

All of the following facts about orders are in \cite[$\mathsection 1$]{dckp2}.
\begin{defi} (a) Let $R$ be a Noetherian domain over $\BC$ and $K$ be its fraction field.  An $R$-algebra $A$ is called {\it an order} if $A$ is finitely generated and torison free over $R$, and $A_K:=A\otimes_R K$ is a finite dimensional algebra over $K$.

\noindent
(b) An $R$-order $A$ is {\it a maximal order} if there is no finitely generated $R$-subalgebra $B$ of $A_K$ that strictly contains $A$ (if such $B$ exists, then it is an $R$-order).
\end{defi}

We will be interested in the $R$-orders $A$ such that $R=Z$ the center of $A$ and $A_{K}$ is a central simple algebra over $K$. Note that in that case $K=Q(Z)$,  the fraction field of $Z$.  We will call such $A$ {\it an order in a central simple algebra}.  Let $\overline{Q(Z)}$ be the algebraic closure of $Q(Z)$ then $A\otimes_Z \overline{Q(Z)}\cong \text{Mat}_d(\overline{Q(Z)})$ for some $d\in \BN$.
\begin{defi}\label{defi: deg d} Let $A$ be an order in a central simple algebra.  The number $d$ as above is called {\it the degree} of $A$.
\end{defi}

\begin{Rem}(see \cite{ir}) \label{rem: order with normal center}
We have the so called reduced  trace map $\textnormal{tr}: A\otimes_Z Q(Z)\rightarrow Q(Z)$. If $Z$ is normal then $\textnormal{tr}(A) \subset Z$. In particular,  this holds for maximal orders since the center of a maximal order is automatically normal.

\end{Rem}

\begin{Thm}[Theorem 1.1 \cite{dckp2}]\label{thm: finite dim rep of order} Suppose $A$ is a finitely generated algebra over $\BC$ and $A$ is  an order closed under the trace (i.e., $\textnormal{tr}(A) \subset Z$) in a  central simple algebra then 

\noindent
(a) $Z$ is a finitely generated algebra over $\BC$.


\noindent
(b) The points of $\Spec(Z)$ parametrize equivalence classes of $d$-dimensional semisimple representations of $A$.

\noindent
(c) Let us consider the natural map $\chi:  \Irr(A)\rightarrow \Spec(Z)$ from the set of irreducible representations of $A$ to the corresponding central characters. This map is surjective. Each fiber consists of the irreducible representations that occur in the corresponding semisimple representation. In particular, each irreducible representation of $A$ has dimension at most $d$.

\noindent
(d) The set 
\[ \Omega_A=\{ a\in \Spec Z~|\; \text{the corresponding semisimple representation is irreducible}\}\]
is a nonempty Zariski open subset of $\Spec Z$ (dense since $Z$ is a domain).

\noindent
(e) Suppose $Z$ is a finitely generated module over a subalgebra $Z_0$. Consider the map 
\[ \Irr(A) \xrightarrow[]{\chi} \Spec Z \xrightarrow[]{\tau} \Spec Z_0.\]
Then the set
\[ \Omega^0_A=\{ a\in \Spec Z_0~|~ (\tau \circ \chi)^{-1}(a) ~\text{ contains all irreducible representations of dimension $d$}\}\]
 is a nonempty Zariski open (dense) subset of $\Spec(Z_0)$.
\end{Thm}

Since $U^{ev}_\e(\g)$ has no zero divisors and finitely generated over $Z$ by Corollary \ref{U^ev is free over Z_e} and  \ref{cor: grU has nonzero divisor}, we see that $U^{ev}_\e(\g)$ is an order in a central simple algebra. Using the fact  that $U^{ev}_\e(\g)$ is integrally closed, Proposition \ref{prop: Uev is integrally closed}, one can easily show that $U^{ev}_\e(\g)$ is a maximal order hence  $U^{ev}_\e(\g)$ is closed under the trace. Therefore, Theorem \ref{thm: finite dim rep of order} can be applied to the triple ($U^{ev}_\e(\g),Z, Z_0=Z_{Fr}$). Combining this with Proposition \ref{prop: fiber over symplectic leaves}, we have
\begin{Prop}\label{prop: Omega0 is union of symplective leaves} $\Omega^0_{U^{ev}_\e}$ is the union of some symplectic leaves of $\Spec Z_{Fr}$.
\end{Prop}

We will now determine the degree $d$ of $U^{ev}_\e(\g)$. The character of the $Z_{Fr}$-action on an irreducible representations of $U^{ev}_\e$ will be called the {\it Frobenius central character} of this representation. We need the following proposition which will be proved in the next section.
\begin{Prop}\label{prop: deg of grU} The degree of $\gr U^{ev}_\e$ is $\prod_{\a\in \Delta_+} \ell_\a$.
\end{Prop}

For any point $p \in \Spec \BC[K^{2\lambda}]_{\lambda \in P}\cong T$,  we define the diagonal module $M_{p,0}$  as follows:
\[ M_{p,0}=U^{ev}_\e/U^{ev}_\e(\tE_i, K^{2\lambda}-p(K^{2\lambda}), \tF_\a^{\ell_\a}),\]
where $\a \in \Delta_+$, $1\leq i \leq r$, and $q$ viewed as the algebra morphism $\BC[K^{2\lambda}]_{\lambda \in P}\rightarrow \BC$.

We recall the identification in Proposition \ref{Ze and Bruhat cell}
\[ Z_{Fr}=\BC[\tE_\a^{\ell_\a} K^{\ell_\a \gamma(\a)}]_{\a\in \Delta_+}\otimes \BC[K^{2\lambda}]_{\lambda \in P^*} \otimes \BC[\tF_\a^{\ell_\a} K^{\ell_\a \kappa(\a)}]_{\a\in \Delta_+}\cong \BC[U^d_-]\otimes \BC[T^d]\otimes \BC[U^d_+]\]
Let $p_\ell$ denote  the point in $T^d$ defined by $K^{2\lambda} \mapsto p(K^{2\lambda})$ for all $\lambda \in P^*$. The following lemma is straightforward:
\begin{Lem}Dimension of $M_{p,0}$ is $\prod_{\a\in \Delta_+}\ell_\a$. The Frobenius central character of $M_{p,0}$ corresponds to the point $(1,p_\ell,1) \in U^d_-\x T^d\x U^d_+$. 
\end{Lem}
\begin{Cor} \label{cor: module at xi} For any point $\upsilon \in 1\x T^d \x 1 \subset \Spec Z_{Fr}$, there are modules $M_{p,0}$ with the Frobenius central character corresponding to the point $\upsilon$.
\end{Cor}
Now we can compute the degree of $U^{ev}_\e(\g)$.
\begin{Prop}\label{prop: deg of U} (a) $d=\prod_{\a\in \Delta_+} \ell_\a$.

\noindent
(b) $\dim_{Q(Z_{Fr})}Q(Z)=\prod_{i=1}^r \ell_i$.    
\end{Prop}
\begin{proof}(a) Since degeneration does not increase the degree \cite[Remark 1.3]{dckp2}, from Proposition \ref{prop: deg of grU}, we have
\[ d=\deg(U^{ev}_\e(g))\geq \prod_{\a\in \Delta_+} \ell_\a.\]
Let $G^{d, diag}$ be the union of intersections of all conjugacy classes of $T^d$ with $\Spec Z_{Fr}$ then $G^{d, diag} \cap \Omega^0_{U^{ev}_\e}$ is non-empty since both subsets are open and Zariski dense. Since $\Omega^0_{U^{ev}_\e}$ is the union of symplectic leaves, we must have 
\[ T^d \cap \Omega^0_{U^{ev}_\e} \neq \emptyset.\]
Let $\upsilon \in T^d \cap \Omega^0_{U^{ev}_\e}$. Since $\upsilon \in \Omega^0_{U^{ev}_\e}$, all irreducible representations with Frobenius central character $\upsilon$ must have dimension $d$. On the other hand, by Corollary \ref{cor: module at xi}, there are some irreducible representations with Frobenius central character $\upsilon$ and dimension not exceeding $\prod_{\a\in \Delta_+} \ell_\a$. Hence
\[ d\leq \prod_{\a\in \Delta_+} \ell_\a.\]
Therefore, $d=\prod_{\a \in \Delta_+}\ell_\a$.

\noindent
(b) Since $Z$ is a finitely generated module over $Z_{Fr}$ and both are integral domains, it follows that $Q(Z)=Z\otimes_{Z_{Fr}} Q(Z_{Fr})$ so that $Q(U^{ev}_\e(\g))=U^{ev}_\e(\g) \otimes_{Z_{Fr}} Q(Z_{Fr})$. Then part $(b)$ holds thanks to the following equalities
\begin{align*}
    \dim_{Q(Z)} Q(U^{ev}_\e(\g))&= d^2 =\Big(\prod_{\a\in \Delta_+} \ell_\a\Big)^2\\
    \Dim_{Q(Z_{Fr})} Q(U^{ev}_\e(\g))&= \Big(\prod_{\a\in \Delta_+} \ell_\a\Big)^2 \cdot \prod_{1\leq i \leq r} \ell_i\\
    \dim_{Q(Z_{Fr})}Q(U^{ev}_\e(\g))&=\Dim_{Q(Z_{Fr})} Q(Z) \cdot \Dim_{Q(Z)}Q(U^{ev}_\e(\g)).
\end{align*}
The first equality follows from Definition \ref{defi: deg d} applied to $U^{ev}_\e(\g)$. The second equality follows from Corollary \ref{U^ev is free over Z_e}.
\end{proof}

\subsection{The degree of $\gr U^{ev}_\e$}\

The algebra $\gr U^{ev}_\e$ is a {\em twisted (Laurent) polynomial algebra} in the sense of  \cite[$\mathsection 2$]{dckp2}. Let us recall the definition of the later. Let $q\in \BC^\x$. Given an $n \x n$ skew-symmetric matrix $H=(h_{ij})$ over $\BZ$, the twisted polynomial algebra $\BC_H[x_1,\dots,x_n]$ is the algebra with generators $x_1,\dots,x_n$ and defining relations:
\[ x_ix_j=q^{h_{ij}}x_j x_i  \qquad (i,j=1,\dots, n).\]
The twisted Laurent polynomial algebra $\BC_H[x_1^{\pm}, \dots ,x_n^{\pm}]$ is defined similarly. 
\begin{Lem}[Lemma 2.2 \cite{dckp2}]\label{lem: x acts as 0 or invertible}  In any irreducible $\BC_H[x_1, \dots, x_n]$-module each element $x_i$ acts either by $0$ or by an invertible operator.
\end{Lem}

Let $q=\e$ be a root of unity of order $\ell$. The matrix $H$ defines a homomorphism 
\[\uH :\BZ^n \rightarrow (\BZ/\ell \BZ)^n.\]
Let $h$ be the cardinality of the image of the homomorphism $\uH$. Let $K$ be the kernel of $\uH$. We will need the following result in \cite[Proposition 2]{dckp2} (for more discussions of twisted polynomial rings and their modules, see \cite[$\mathsection 2$]{dckp2}).
\begin{Prop} \label{prop: deg of twisted polyalg}(a) The elements $\textbf{\textnormal{x}}^a=x_1^{a_1}\dots x_n^{a_n}$ with $a\in K \cap \BZ^n_{\geq 0}$ (resp. $a \in K$) form a basis of the center of $\BC_H[x_1, \dots, x_n]$ (resp. $\BC_H[x_1^{\pm}, \dots, x_n^{\pm}]$).

\noindent
(b) $\deg \BC_H[x_1,\dots, x_n]=\deg \BC_H[x_1^{\pm}, \dots, x_n^{\pm}]=\sqrt{h}$.

\noindent
(c) The algebra $\BC_H[x^{\pm}_1, \dots , x^{\pm}_n]$ is Azumaya over its center.
\end{Prop}
Let $S \subset \{ 1, \dots, n\}$ and $H_S$ be the matrix obtained from $H$ by removing all rows and columns labeled by indices not  in $S$. 
\begin{Lem}\label{lem: intersection with Omega0} Let $Z$ be the center of $\BC_H[x_1, \dots, x_n]$. Assume $Z$ is a finitely generated module  over a subalgebra $Z_0$ and $\{x_i^{m_i}|1\leq i \leq n\} \subset Z_0$ for some $m_i\neq 0$. Let $\mathscr{O}_0$ be the subset in $\Spec Z_0$ of solutions to the equations $x_i^{m_i}=0, i \not \in S$. Assume $ \deg \BC_{H_S}[x_i]_{ i\in  S} =\deg \BC_H[x_1,\dots, x_n]$ then 
\begin{equation}\label{eq: intersection with Omega0} \Omega^0_{\BC_H[x_1, \dots, x_n]} \cap \mathscr{O}_0 \; \text{is a nonempty  Zariski open subset in $\mathscr{O}_0$.}
\end{equation}
\end{Lem}
\begin{proof} Since $\Omega^0_{\BC_H[x_1, \dots, x_n]}$ is a Zariski open subset of $\Spec Z_0$, it is enough to show that 
\[\Omega^0_{\BC_H[x_1,\dots, x_n]} \cap \mathscr{O}_0 \neq \emptyset.\]

Let $Z_0[x_j^{-m_j}]_{j \in S}$ be the localization of $Z_0$ at $\{ x_j^{m_j}|j \in S\}$, then $\Spec(Z_0[x_j^{-m_j}]_{j \in S})$ is an nonempty open subset in $\Spec Z_0$. On any  irreducible $\BC_H[x_1, \dots, x_n]$-module with central character contained in $\Spec (Z_0[x_j^{-m_j}]_{j\in S})$, the actions of $x_j, j \in S,$ will be by invertible operators due to Lemma \ref{lem: x acts as 0 or invertible}. Therefore, irreducible $\BC_H[x_1, \dots, x_n]$-modules whose central characters are  contained in $\Spec(Z_0[x_j^{-m_j}]_{j \in S})$ are in bijection with the irreducible $\BC_H[x_i, x_j^{\pm}]_{i \not \in S, j \in S}$-modules.

Since $\{ x_i^{m_i}|1\leq i \leq n\}\subset Z_0$, the intersection $\mathscr{O}_0 \cap \Spec (Z_0[x_j^{-m_j}]_{j \in S})$ is nonempty. By Lemma \ref{lem: x acts as 0 or invertible}, the elements $x_i, i\not\in S,$ act by $0$ on any irreducible $\BC_H[x_1, \dots, x_n]$-module with central character in $\mathscr{O}_0$. Therefore the irreducible modules of $\BC_H[x_1, \dots, x_n]$ with central characters in $\mathscr{O}_0 \cap \Spec(Z_0[x_j^{-m_j}]_{ j \in S})$ are in bijection with the irreducible of $\BC_{H_S}[x_j^{\pm}]_{j \in S}$-modules. Then by the assumption that $\deg \BC_{H_S}[x_j]_{j \in S}=\deg \BC_H[x_1, \dots, x_n]$ and Lemma \ref{prop: deg of twisted polyalg}(c), it follows that  all irreducible  $\BC_H[x_1, \dots, x_n]$-modules with central characters contained in $\mathscr{O}_0 \cap \Spec Z_0[x_j^{-m_j}]_{j \in S}$ are of maximal possible dimension. Therefore, $\emptyset \neq \mathscr{O}_0 \cap \Spec(Z_0[x_j^{-m_j}]_{j \in S} \subset \mathscr{O}_0 \cap \Omega^0_{\BC_H[x_1, \dots, x_n]}$. This finishes the proof.
\end{proof}

By Proposition \ref{prop: grU}, $\gr U^{ev}_\e$ is the algebra 
\[ \BC_H[K^{\pm 2\w_1}, \dots , K^{\pm 2\w_r}, X_{\b_1}, \dots, X_{\b_N}, Y_{\b_1}, \dots, Y_{\b_N}]\]
where the skew-symmetric matrix $H$ is of the form:
\begin{equation}\label{eq: matrix H of grU}
     H=\begin{bmatrix} B & 0& -A^T\\0&C& A^T \\A&-A &0 \end{bmatrix}
\end{equation}
Let us describe $H$: 
\begin{itemize}
    \item The columns from left to right are labeled by $X_{\b_1},\dots, X_{\b_N}, Y_{\b_1}, \dots, Y_{\b_N}, K^{2\w_1}, \dots, K^{2\w_r}$. The rows from top to bottom are labeled by the same  set. 
    \item The matrix $B=(b_{ij})$ is an $N \x N$ skew-symmetric matrix with    $b_{ij}=(\b_i, \kappa(\b_j))$ for $i<j$. 
    \item The  matrix $C=(c_{ij})$ is an $N \x N$ skew-symmetric matrix with $c_{ij}=(\kappa(\b_i), \b_j)$ for $i<j$
    \item The matrix $A=(a_{ij})$ is of size $r \x N$ with $a_{ij}=(2\w_i, \b_j)$.
\end{itemize}

To compute the cardinality of the image of $\uH$, we need some results about root systems in \cite[$\mathsection 3$]{dckp2}. Given a  simple root $\a$, let 
\begin{equation}\label{eq: the set Ia}
    I_\a=\{1\leq j \leq N|s_{i_j}= s_\a\}, \qquad \CI_\a=\{ \b_{k_1}, \dots, \b_{k_t}| k_j \in I_{\a}\},
\end{equation}

\begin{Lem}[Lemma 3.2, \cite{dckp2}]\label{lem: properties of root systems} Fix a simple root $\a$ and let $\w$ be the corresponding fundamental weight. Let $k_1< \dots <k_r$ be all elements of the set $I_\a$. For $t\in \BZ$, $1\leq t\leq r$, let 
\[ \lambda_t=s_{\b_{k_t}} \dots s_{\b_{k_1}}(\w), \qquad \mu_t=-s_{\b_{k_{r-t+1}}} \dots s_{\b_{k_r}}(\w^*),\]
here $\w^*=-w_0(\w)$. Let $\b^\vee_j:=2 \b_j/(\b_j, \b_j)$ for all $1\leq j \leq N$. Then:

(a) $\lambda_t=s_{i_1}\dots s_{i_{k_t}}(\w);$ in particular $\lambda_r=-\w^*$.

(b) If $k_t<j<k_{t+1},$ then $( \lambda_t, \b^\vee_j)=0$.

(c)  $( \lambda_t, \b^\vee_{k_{t+1}})=1$.

(d)  $\lambda_t=\w-\sum_{i=1}^t \b_{k_i}.$

Similarly:

(a') $\mu_t=\lambda_{r-t};$ in particular $\mu_r=\w$.

(b') If $k_{r-t}<j< k_{r-t+1}$, then $( \mu_t, \b^\vee_j)=0$.

(c') $( \mu_t, \b^\vee_{k_{r-t}})=-1$.

(d') $-s_{\b_{k_1}} \dots s_{\b_{k_r}}(w^*)=-\w^*+\sum_{i=1}^r \b_{k_i}=\w$.
\end{Lem}
\begin{Cor}[Corollary 3.2, \cite{dckp2}]\label{cor: property of root system}
(a) $(\b_l, \w)= \sum_{k_i<l} (\b_l, \b_{k_i})$ if $l \not \in I_\a$.

\noindent
(b) $(\b_{k_t}, 2\w)= (\b_{k_t}, \b_{k_t})+\sum_{i<t}(2\b_{k_1}, \b_{k_i})$ for $k_t \in I_\a$.
\end{Cor}
For each set $I_{\a_i}$, pick $\b_{m_i} \in \CI_{\a_i}$. Let 
\begin{equation}\label{eq: set S}
    S= \begin{cases}\{  K^{2\w_1}, \dots K^{2\w_r}, Y_{\b_i}, X_{\b_i}\}\setminus\{ X_{\b_{m_1}}, \dots, X_{\b_{m_r}}\}\\
    \{K^{2\w_1}, \dots K^{2\w_r}, Y_{\b_i}, X_{\b_i}\}\setminus \{Y_{\b_{m_1}}, \dots, Y_{\b_{m_r}}\}
    \end{cases}
\end{equation}
 Let $H_S$ be the $2N \x 2N$ submatrix of $H$ consisting of columns and rows labeled by the set $S$. So we have the following group homomorphisms
 \begin{equation}\label{eq: the map H and Hs} \uH: \BZ^{2N+r}\rightarrow (\BZ/\ell\BZ)^{2N+r}, \qquad \uH_{S}: \BZ^{2N}\rightarrow (\BZ/\ell \BZ)^{2N}.
 \end{equation}

 \begin{Prop}\label{prop: size images of H and HS} The cardinalities of $\Img \uH$ and $\Img \uH_S$ are equal to $\Big(\prod_{\a\in \Delta_+} \ell_\a\Big)^2$.
 \end{Prop}
 This proposition is proved in \cite[Lemma 3.3]{dckp2} when $\ell$ is odd, but the arguments in {\em loc.cit.} do not work for even $\ell$. 
\begin{proof} Note that the following operations do not change the cardinalities of $\Img \uH$ and $\Img \uH_S$:

$\bullet$ Swapping rows or columns in these matrices.

$\bullet$ Adding an integral multiple of one  row (resp. column) to another row (resp. column).

$\bullet$ Multiply rows (columns) by $\pm 1$.

Using these operators, we will transform $H$ and $H_S$ into diagonal matrices whose non-zero entries are $(\b_i, \b_i), 1\leq i \leq N$, each of which occurs twice. This will imply the lemma. Indeed, 
\[ \Img \uH\cong \bigoplus_{i=1}^N \Big(\big[(\b_i, \b_i)\BZ +\ell \BZ\big]/\ell \BZ\Big)^{\oplus 2} \cong \bigoplus_{i=1}^N \Big( \textnormal{gcd}\big[(\b_i, \b_i), \ell\big] \BZ/\ell \BZ\Big)^{\oplus 2}.\]
Hence $\sharp \Img \uH= \left(\prod_{\a\in \Delta_+} \ell_\a\right)^2$. The case of $H_S$ is the same.

Now let us describe how to transform $H, H_S$ into such diagonal matrices. Recall the conventions in \eqref{eq: matrix H of grU}.

{\bf Computations for $H$.} We will first do the case of $H$. The following equalities are frequently used, again $\kappa$ and $\gamma$ are defined in \eqref{kappa and gamma}:
\begin{equation}\label{eq: some identities} (\gamma(\a), \b)=(\a, \kappa(\b)), \qquad \gamma(\a)+\kappa(\a)=2\a, \qquad (\gamma(\a), \a)=(\a, \kappa(\a))=(\a,\a).
\end{equation}

{\it Step 1: }We will show how to use the last $r$ columns of $H$ to transform $\bmat{B \\0\\A}$ and how to use the last $r$ rows of $H$ to transform $\bmat{ 0& C& A^T}$.

{\it Step 1.1:}  The column of $\bmat{B\\0}$ indexed by
$X_{\b_j}$ is 
\[\bmat{(\b_1, \kappa(\b_j))&\dots& (\b_{j-1}, \kappa(\b_j))&0&-(\kappa(\b_{j+1}), \b_j)&\dots &-(\kappa(\b_N), \b_j)&0&\dots&0}^T.\]
By adding the following vector, which is an integral linear combination of columns in $\bmat{-A^T\\A^T}$
\[ \bmat{-(\b_1, \kappa(\b_j))& \dots & -(\b_N, \kappa(\b_j)) & (\b_1, \kappa(\b_j))& \dots & (\b_N, \kappa(\b_j))}^T,\]
we get 
\begin{equation}\label{eq: modified columns} \bmat{0& \dots &0 & -(\b_j, \b_j)& -2(\b_{j+1}, \b_j)& \dots -2(\b_N, \b_j)& (\b_1, \kappa(\b_j) & \dots (\b_N, \kappa(\b_j))}^T
\end{equation}
here the entries $-2(\b_{k}, \b_j)$ and $(\b_j, \b_j)$ show up because by \eqref{eq: some identities}
\[ -(\kappa(\b_k), \b_j)-(\b_k, \kappa(\b_j))=-(\kappa(\b_k)+\gamma(\b_k), \b_j)=-2(\b_k, \b_j), \quad (\b_j, \kappa(\b_j))=(\b_j, \b_j).\]

{\it Step 1.2:} Similarly, the row of $\bmat{0&C}$ indexed by $Y_{\b_i}$ is 
\[ \bmat{0 &-(\b_i, \kappa(\b_1))&  \dots & -(\b_i, \kappa(\b_{i-1}))& 0&(\kappa(\b_i), \b_{i+1}) & \dots & (\kappa(\b_i), \b_N)}.\]
 By adding the vector 
 \[\bmat{-(\gamma(\b_i), \b_1)&\dots & -(\gamma(\b_i), \b_N) & (\gamma(\b_i), \b_1) & \dots & (\gamma(\b_i), \b_N)},\]
 we get 
 \begin{equation}\label{eq: modified rows} \bmat{-(\gamma(\b_i), \b_1) & \dots & -(\gamma(\b_i), \b_N) & 0& (\b_i, \b_i)& 2(\b_i, \b_{i+1})& \dots & 2(\b_i, \b_N)}.
 \end{equation}

Note that the last entries   of columns \eqref{eq: modified columns} cancel the first entries of rows \eqref{eq: modified rows}: $(\b_i, \kappa(\b_j))-(\gamma(\b_i), \b_j)=0$. Therefore, we have transformed $H$ into the following form:
\[ H_1 =\bmat{B_1 & 0 & -A^T\\0 & C_1 & A^T\\ A& -A & 0},\]
where:\\
$\bullet$ $B_1$ is a lower triangular matrix whose column indexed by $X_{\b_j}$ is 
\[\bmat{0& \dots &0& -(\b_j, \b_j) & -2(\b_{j+1}, \b_j) & \dots &-2(\b_N, \b_j)}^T.\]
$\bullet$ $C_1$ is an upper triangular matrix whose row indexed by $Y_{\b_i}$ is 
\[\bmat{0 & \dots & 0& (\b_i, \b_i) & 2(\b_i, \b_{i+1}) & \dots & 2(\b_i, \b_N)}.\]
Let 
\begin{align*}
    v_{\b_i}&:=\bmat{0& \dots &0 & (\b_i, \b_i)& 2(\b_i, \b_{i+1}) & \dots & 2(\b_i, \b_N)}\\
    V_{\b_j}&:=\bmat{2(\b_1, \b_j)& \dots & 2(\b_{j-1}, \b_j)&(\b_j, \b_j)& 0&\dots &0}^T.
\end{align*}
Then 
\begin{align*}
    B_1=\bmat{-v^T_{\b_1}& \dots & -v^T_{\b_N}}, \qquad C_1=\bmat{V_{\b_1}&\dots &V_{\b_N}}= \bmat{v_{\b_1}\\ \dots \\ v_{\b_N}}=-B_1^T.
\end{align*}

{\it Step 2:} We recall $\CI_{\a_l}=\{\b^{w_l}_{k_1}, \dots , \b^{w_l}_{k_{s_l}}\}$. Here the cardinality of $\CI_{\a_l}$ is $s_l$ and we put the superscript $w_l$  to distinguish the elements in different sets $\CI_{\a_l}$. Then 
\begin{align}
  \label{eq: row with wl}  \bmat{(2\w_l, \b_1)& \dots &(2\w_l, \b_N)}&=v_{\b^{\w_l}_{k_1}}+\dots v_{\b^{\w_l}_{k_{s_l}}}\\
 \label{eq: column with wl}   \bmat{(2\w^*_l, \b_1)& \dots & (2\w^*_l, \b_N)}^T&=V_{\b^{\w_l}_{k_1}}+\dots +V_{\b^{\w_l}_{k_{s_l}}}.
\end{align}
Indeed, by Corollary \ref{cor: property of root system}, we have 
\begin{equation*}
    (2\w_l, \b_j)=\begin{cases} \sum_{i=1}^{t-1} 2(\b^{\w_l}_{k_i}, \b_j), \qquad k_{t-1} <j< k_t\\
    \sum_{i=1}^{t-1} 2(\b^{\w_l}_{k_i}, \b_{k_t})+(\b_{k_t}, \b_{k_t}), \qquad j=k_t
    \end{cases}
\end{equation*}
\begin{equation*}
    (2\w^*_l, \b_j)= (2\sum_{i=1}^{s_l}\b^{\w_l}_{\b_{k_i}}, \b_j)-(2\w_l, \b_j)=\begin{cases} \sum_{i=t}^{s_l} 2(\b^{\w_l}_{k_i}, \b_j), \qquad k_{t-1}< j< k_t\\
        \sum_{i=t+1}^{s_l}2(\b^{\w_l}_{k_i}, \b_{k_t})+(\b_{k_t}, \b_{k_t}), \qquad j=k_t
        \end{cases}
\end{equation*}

{\it Step 3:} Using Step 2,  we can use the columns of $B_1$  to eliminate top right submatrix $-A^T$. Meanwhile, we can use the rows of $C_1$ to eliminate the middle submatrix $-A$ in the last row of $H_1$. We will get the new matrix 
\begin{equation*}
    H_2=\bmat{B_1&0&0 \\0&C_1& A^T\\A&0& D+E},
\end{equation*}
where

$\bullet$ $D=(d_{ij})$ is a $r\x r$ matrix coming from eliminating $-A^T$, and $d_{ij}=-(2\w_i, \sum_{l=1}^{s_j} \b^{\w_j}_{k_l})$

$\bullet$ $E=(e_{ij})$ is an $r\x r$-matrix coming from eliminating $-A$, and $e_{ij}=(\sum_{l=1}^{s_i} \b^{\w_i}_{k_l}, 2\w_j)=-d_{ji}$
On the other hand, $D$ is a symmetric matrix. Indeed, note that $\w_i+\w^*_i=\sum_{l=1}^{s_i} \b^{\w_i}_{k_l}$, hence
\[ d_{ij}=-2(\w_i, \w_j+\w_j^*)=-2(\w_i+\w^*_i, \w_j)=d_{ji}.\]
Therefore, $D+E=0$

{\it Step 4:} We observe that  in  any row  or column of $H_2$ with nonzero diagonal entries, all other entries in these rows or columns are divisible by the diagonal entries. So we can use the diagonal entries of $B_1$ to eliminate the other entries in $\bmat{B_1\\0\\A}$, meanwhile, we can use the diagonal entries of $C_1$ to eliminate the other entries in $\bmat{0& C_1& A^T}$. So we can transform $H_2$, hence $H$, to the desired diagonal matrix.

{\bf Computations on $H_S$.} Now we consider $H_S$ with 
\[ S=\{X_{\b_1}, \dots, X_{\b_N}, Y_{\b_1}, \dots Y_{\b_N},  K^{2\w_1}, \dots, K^{2\w_r}\}/\{ Y_{\b_{m_1}}, \dots Y_{\b_{m_r}}\}.\]
We still perform as in Step $1$-$2$ but in Step $3$, since some rows and columns  of $\bmat{0 & C& A^T}$ were deleted, we  get the following matrix
\[ H_{S,1}=\bmat{B_1&0&0\\0&\underline{C}_1& \underline{A}^T\\ A & \underline{F}_1& \underline{F}_2},\]
where

$\bullet$ $\bmat{\underline{C}_1& \underline{A}^T}$ is obtained from $\bmat{C_1& A^T}$ by deleting rows and columns indexed by $Y_{\b_{m_1}}, \dots, Y_{\b_{m_r}}$. 

$\bullet$ Row at $K^{2\w_i}$ of $\bmat{\underline{F}_1&\underline{F}_2}$ is $-$ row at $Y_{\b_{m_i}}$ of $\bmat{C_1&A^T}$ after deleting columns indexed by $Y_{\b_{m_1}}, \dots, Y_{\b_{m_r}}$

So we can multiply rows of $\bmat{A & \underline{F}_1 &\underline{F}_2}$ by $-1$, then  permute rows of $H_{S,1}$ to get
\[ H_{S,2}=\bmat{B_1&0&0\\\ast& C_2 &A^T},\]
where  $C_2$ is obtained from $C_1$ by deleting columns at $Y_{\b_{m_1}}, \dots, Y_{\b_{m_r}}$.

Now by using \eqref{eq: column with wl} in Step 2, we can use $C_2$ to transform $A^T$ into a matrix whose columns are all mising columns of $C_1$ in $C_2$. Then permuting columns and rows , we obtain
\[ H_{S,3}=\bmat{B_1&0\\\ast & C_1},\]
 which can be  tranformed into the desired diagonal matrix  as in Step $4$.
\end{proof}

\begin{Cor} Proposition \ref{prop: deg of grU} holds.
\end{Cor}

Let $\a^*_i=-w_0(\a_i)$ and $\w^*_i =-w_0(\w_i)$ for $1\leq i \leq r$. Let  \[ \CI_{\a_i}=\{ \b_{k_1}, \dots, \b_{k_r}\}, \qquad \CI_{\a^*_i}=\{ \b_{l_1}, \dots, \b_{l_r}\} \quad \text{$1\leq i  \leq r$}.\]
where $\CI_?$ is introduced in \eqref{eq: the set Ia}. The following proposition is similar to \cite[Proposition 5.2]{dckp2}. We provide the proof adapting  to our setup for reader's convenience.
\begin{Prop}\label{prop: center of grU} The center of $\gr U^{ev}_\e$ is generated by the elements
$K^{\pm 2\ell_i \w_i}, X_\a^{\ell_a}, Y_\a^{\ell_\a}$ as well as $u_i$ for $1\leq i \leq r, \a \in \Delta_+$, where 
\[ u_i= K^{-2\w_i +\kappa(\w_i+\w^*_i)} \prod X_{\b_{k_i}} \prod Y_{\b_{l_j}}.\]
\end{Prop}
\begin{proof}Label the standard basis $\{(\dots, 1, \dots)\}$ of the lattice $\BZ^{2N+r}$ by the set 
\[\{X_{\b_N}, \dots, X_{\b_1}, Y_{\b_N},\dots, Y_{\b_1}, K^{2\w_1}, \dots, K^{2\w_r}\}.\]
 Then each monomial in $\gr U_\e^{ev}$ corresponds to a vector in $\BZ^{2N+r}$, see description of $\gr U_\e^{ev}$ in Proposition \ref{prop: grU}.

{\it Step 1.} We show that $u_i$ is central in $\gr U^{ev}_\e$. Let $\ul{u_i}$ denote the coordinate of $u_i$ in the lattice $\BZ^{2N+r}$. Then it is enough to show that $\ul{u_i} H=0$, where $H$ is given by \eqref{eq: matrix H of grU}.
The product of the row $\ul{u_i}$ with the column of $H$ labeled by $X_{\b_j}$ is 
\begin{align*}
   & \sum_{1\leq k \leq r}(-2\w_i+\kappa(\w_i+\w^*_i), \a^\vee_k/2)(2\w_k, \b_j)+\sum_{k_m <j} (\b_{k_m}, \kappa(\b_j)) -\sum_{k_m >j} (\kappa(\b_{k_m}), \b_j)\\
   & =-(2\w_i, \b_j)+(\kappa(\w_i+\w^*_i), \b_j)+\sum_{k_m <j} (\b_{k_m}, \kappa(\b_j)) -\sum_{k_m >j} (\kappa(\b_{k_m}), \b_j)\\
   &=-(2\w_i, \b_j)+\sum_{k_m<j}(\b_{k_m}, 2\b_j)+ \begin{cases} (\b_j, \b_j) \qquad & \text{if $j\in I_{\a_i}$}\\
    0& \text{otherwise}
    \end{cases}\\
& =0 \qquad \text{by Corrollary \ref{cor: property of root system}}
\end{align*}
In the second equality we use $\w_i+\w^*_i=\b_{k_1}+\dots +\b_{k_r}$ (see Lemma \ref{lem: properties of root systems}), $(\b_{k_m},  \kappa(\b_j))=(\gamma(\b_{k_m}), \b_j)=(\b_{k_m}, 2\b_j)-(\kappa(\b_{k_m}), \b_j)$ and $(\kappa(\b_j),\b_j)=(\b_j, \b_j)$. 

The product of the row $\ul{u_i}$ with the column of $H$ labeled by $Y_{\b_j}$ is computed in a similar way. The product of $\ul{u_i}$ with the column of $H$ labeled by $K^{2\w_j}$ is 
\begin{align*}
    \sum_{k_m} (-2\w_j, \b_{k_m})+\sum_{l_n}(2\w_j, \b_{l_n})=(-2\w_j, \w_i+\w^*_i)+(2\w_j, \w_i+\w^*_i)=0,
\end{align*}
here we use $\w_i+\w^*_i=\b_{k_1}+\dots+\b_{k_r}=\b_{l_1}+\dots +\b_{l_r}$, see Lemma \ref{lem: properties of root systems}.

{\it Step 2.}
We will show that $K^{\pm 2\ell_i \w_i}, X_\a^{\ell_\a}, Y_\e^{\ell_\a}$ are central in $\gr U^{ev}_\e$. Let us do this for $X_\a^{\ell_\a}$ only. Let $\ul{X}_\a^{\ell_\a}$  denote the coordinate of $X_\a^{\ell_\a}$ in the lattice  $\BZ^{2N+r}$.

The product of $\ul{X}_\a^{\ell_\a}$ with the column labeled by $X_{\b_j}$ of $H$ in \eqref{eq: matrix H of grU} is 
\[ \begin{cases} \ell_\a(\b_k, \kappa(\a))=2\ell_\a d_\a(\gamma(\b_k)/2, \a^\vee) \qquad &\text{if $\b_k <\a$}\\
0 & \text{if $\b_k=\a$}\\
-\ell_\a(\kappa(\b_k), \a)=2\ell_\a d_\a(\kappa(\b_k)/2, \a^\vee) & \text{if $\b_k>\a$}
\end{cases}
\]
here $d_\a=(\a, \a)/2$. All these integer values are divisible by $\ell$ since $\gamma(\b_k), \kappa(\b_k) \in 2P$. Doing the same computations for the rest of the column vectors in $H$, we arrive at our conclusion that $\ul{X}_\a^{\ell_\a}$ is contained in the kernel of the map $\ul{H}$ in \eqref{eq: the map H and Hs}, hence $X_\a^{\ell_\a}$ is central.

{\it Step 3.}
Let us use $\ul{K}^{2\ell_i \w_i}, \ul{X}_\a^{\ell_\a}, \ul{Y}_\a^{\ell_\a}, \ul{u_i}$ to denote the coordinates of the corresponding elements in $\BZ^{2N+r}$. Let $K'$ denote the sublattice in $\BZ^{2N+r}$ generated by $\ul{K}^{2\ell_i \w_i}, \ul{X}_\a^{\ell_\a}, \ul{Y}_\a^{\ell_\a}, \ul{u_i}$. 

For each set $\CI_{\a_i}$ choose one element $\b'_i$, we replace  the basis  vector $(\dots, 1, \dots)$ corresponding to  $X_{\b'_i}$ by $\ul{u_i}$. Then we still get a basis for the lattice $\BZ^{2N+r}$. Moreover, $\ell_{\b'_i}=\ell_i$. By these observations, one can easily show that the cardinality of $\BZ^{2N+r} /K' $ is less than or equal to $ (\prod_{\a\in \Delta_+} \ell_\a)^2$. On the other hand,  $K'$ is contained in the kernel $K$ of $\uH$ and the cardinality of $\BZ^{2N+r}/K$ is $(\prod_{\a\in \Delta_+} \ell_\a)^2$ by Proposition \ref{prop: size images of H and HS}. Therefore $K'=K$.

{\it Step 4.} The proposition follows by the description of the centers of the twisted polynomial rings in Proposition \ref{prop: deg of twisted polyalg}.
\end{proof}
\subsection{The center $Z$ of $U^{ev}_\e(\g)$}\

This section describes $Z$ in terms of $Z_{HC}$ and $Z_{Fr}$, see \cite[Theorem 6.4]{dckp} and  \cite[Theorem 5.4]{t2} for related results of original De Concini-Kac forms.

Recall the Harish-Chandra isomorphism in Theorem \ref{thm: HC-center}:
\[\pi: Z_{HC}\rightarrow \BC[K^{2\lambda}]_{\lambda \in P}^{W_\bullet},\]
here $\BC[K^{2\lambda}]_{\lambda \in P}$ is the group algebra the lattice $2P$.

The algebra automorphism $\gamma_{-\rho}: \BC[K^{2\lambda}]_{\lambda\in P}\rightarrow \BC[K^{2\lambda}]_{\lambda \in P}$ is defined by $\gamma(K^{2\lambda})=\e^{(-\rho, 2\lambda)}K^{2\lambda}$ for all $\lambda \in P$ (here we fix an element $\e^{1/\sN}$). Then, by Theorem \ref{thm: HC-center}, we have 
\begin{equation}\label{eq: HC map}\gamma_{-\rho}\circ \pi: Z_{HC}\xrightarrow{\sim} \BC[K^{2\lambda}]_{\lambda \in P}^W,
\end{equation}
where $W$-action is as follows: $w(K^\mu)=K^{w(\mu)}$ for all $w\in W, \mu \in 2P$. 

Let us define
\[ Z_\cap:=Z_{Fr}\cap Z_{HC}=(Z_{Fr})^{\dU_\e}\]
\begin{Lem} $Z_\cap\cong \BC[K^{2\lambda}]_{\lambda \in P^*}^W \subset \BC[K^{2\lambda}]^W_{\lambda \in P} \cong Z_{HC}$. 
\end{Lem}
\begin{proof} Recall the isomorphism $\varphi: O_\e[G]\cong U^{fin}_\e$ in Theorem \ref{prop: REA-Ufin general case}. Thanks to Lemma \ref{lem: Oe[G^d] and Z^ev}, under this isomorphism we have $O_\e[G^d]\cong Z^{fin}_{Fr}$. Hence $O_\e[G^d]^{\dU_\e}\cong (Z_{Fr})^{\dU_\e} =Z_\cap$.

For any $\lambda \in P^*$, let  $W^d_\e(\lambda)$ be the Weyl module in $\Rep(\dU^*_\e)$. Let $c^d_\lambda =\sum c_{(v^d_i)^*, K^{-2\rho}v^d_i}\in O_\e[G^d]$, here $\{v^d_i\}$ is a weight basis of $W^d_\e(\lambda)$ and we sum over the indexing set of the basis. Then $\{c^d_\lambda\}_{\lambda\in P^*_+}$ form a $\BC$-basis of $O_\e[G^d]^{\dU_\e}$  because $\Rep(\dU^*_\e)$ is equivalent to $\Rep(\dU_\BC(\g^d))$ (cf. \eqref{eq: Rep(U*) vs Rep(U(gd))}),  hence semisimple, and then $O_\e[G^d]\cong \bigoplus_{\lambda \in P^*_+} W^d_\e(\lambda)\otimes (W^d_\e(\lambda))^*$. Now compute the image of $\{ \varphi(c^d_\lambda)\}_{\lambda \in P^*_+}$ under the Harish-Chandra map \eqref{eq: HC map} as in the proof of Theorem \ref{thm: HC-center}, one can see that $\{\gamma_{-\rho}\circ \pi\circ \varphi(c^d_\lambda)\}_{\lambda \in P^*_+}$ forms a basis of $\BC[K^{2\lambda}]^W_{\lambda \in P^*}$. This implies the lemma.
\end{proof}

\begin{Rem}The inclusion $Z_\cap \hookrightarrow Z_{HC}$ gives a rise to a map $\bullet^\ell: T/W\rightarrow T^d/W$. Under the identification $Z_{Fr}\cong \BC[G^d_0]$, this inclusion  corresponds to the categorical quotient $G^d_0 \hookrightarrow G^d \rightarrow G^d\sslash G^d \cong T^d/W$.  Hence 
\[ \Spec\; Z_{Fr}\otimes_{Z_\cap} Z_{HC}  \cong G^d_0 \x_{T^d/W} T/W.\]
\end{Rem}
\begin{Thm}[cf. Theorem $6.4$ \cite{dckp}]The natural map $Z_{Fr}\otimes_{Z_\cap} Z_{HC}\rightarrow Z$ is an isomorphism.
\end{Thm}
\begin{proof} The proof closely follows \cite[Theorem 6.4]{dckp}. Let $\tilde{Z}:=Z_{Fr}\otimes_{Z_\cap} Z_{HC}$

{\it Step 1:} We will show that  the algebra $Z_{Fr}\otimes_{Z_\cap} Z_{HC}$ is normal. 
By the construction in Section \ref{SSS:root_data}, $P^*$ is the weight lattice of a root system. In particular, $G^d$ is simply connected. It follows that $W$ acts on $T^d$ as a reflection group. Hence the quotient morphism $T^d\rightarrow T^d/W$ is flat. On the other hand, $T/W$ is smooth because $G$ is simply connected. It follows that the morphism $T/W\rightarrow T^d/W$ is flat as any finite dominant morphism between smooth varieties. Hence the morphism $G^d\times_{T^d/W}T/W\rightarrow G^d$ is finite and flat. Since $G^d$ is smooth, $G^d\times_{T^d/W}T/W$ is Cohen-Macaulay, hence so is $G^d_0 \times_{T^d/W} T/W$. In particular, thanks to the Serre normality criterium, it's enough to check that $\tilde{Z}=\mathbb{C}[G^d_0\times_{T^d/W}T/W]$ is regular in codimension $1$.

Since $G^d$ is simply connected,  we can apply results in \cite[$\mathsection 3.8$]{St}. By \cite[Theorem 1, $\mathsection 3.8$]{St}, $G^d_0=G_1 \cup G_2$, where $G_1$ is the open set of regular elements and $G_2$ is the closed subvariety of codimension at least $2$. Furthermore, By \cite[Theorem 3, $\mathsection 3.8$]{St}, the map $G_1\rightarrow T^d/W$ is smooth, hence the map $G_1 \x_{T^d/W} T/W \rightarrow T/W$ is smooth. Meanwhile $T/W$ is smooth, hence $G_1 \x_{T^d/W} T/W$ is smooth. On the other hand, $\bullet^\ell: T/W \rightarrow T^d/W$ is a finite map, hence $G_2 \x_{T^d/W} T/W$ has codimension at least $2$. So, we have proved that $\tilde{Z}$ is regular in codimension $1$.

{\it Step 2:} We have the following diagram
\[ \begin{tikzcd}
     \Spec \; Z\arrow[rr] \arrow[dr] && \Spec \; \tilde{Z}\arrow[dl]&\\
     & \Spec\; Z_{Fr}&&
\end{tikzcd}
\]
$\Spec \tZ$ is irreducible by the following reasons: the map $G^d_0 \times_{T^d/W} \rightarrow T/W$ is flat and dominant, the fibers are irreducible and the base $T/W$ is irreducible. On the other hand, two diagonal maps are finite and dominant. Therefore, the horizonal map is dominant, equivalently, the map $\tilde{Z}\rightarrow Z$ is injective.

{\it Step 3:} One can see that $\dim_{Q(Z_{Fr})} Q(\tilde{Z}) =\prod_{i=1}^r \ell_i$, hence $\dim_{Q(Z_{Fr})} Q(\tilde{Z})=\dim_{Q(Z_{Fr})} Q(Z)$ by Proposition \ref{prop: deg of U}. Therefore, the map $Q(\tilde{Z})\hookrightarrow Q(Z)$ is an isomorphism, equivalently, $\vartheta: \Spec\tilde{Z}\rightarrow \Spec Z$ is birational. So $\vartheta$ is a finite birational map to a normal variety, hence $\vartheta$ is an isomorphism. This completes the proof.   
\end{proof}

The rest of this section is about technical results used in Lemma \ref{lem: center is normal}. They are similar to \cite[Theorem 5.2]{dckp2}. However, our set up is slightly different and the proof in {\it loc.cit.} is rather sketchy. So we provide a detail proof for reader's convenience. 
\begin{Prop}[cf. Theorem 5.2 \cite{dckp2}]\label{prop: chain of centers} Let $Z^{(i)}$ be the center of the algebra $U^{(i)}_\e$ in Remark \ref{rem: chain of filtered algebras}. Then $\gr Z^{(i)}=Z^{(i+1)}$.
\end{Prop}

\begin{proof}The proposition will follow if we  show that  $\gr Z= Z_{\gr U^{ev}_\e}$. But then it is enough to prove that the generators $K^{\pm 2\ell_i \w_i}, X_\a^{\ell_\a}, Y_\a^{\ell_\a}, u_i$  of $Z_{\gr U^{ev}_\e}$ from Proposition \ref{prop: center of grU} are contained in $\gr Z$. On the other hand, the elements $K^{\pm 2\ell_i \w_i}, X_\a^{\ell_\a}, Y_\a^{\ell_\a}$ in $U^{ev}_\e(\g)$ are contained in the Frobenius center $Z_{Fr}$. Furthermore, their images under the symbol map in $\gr U^{ev}_\e$ are $K^{\pm 2\ell_i \w_i}, X_\a^{\ell_\a}, Y_\e^{\ell_\a}$, respectively. It follows that $K^{\pm 2\ell_i \w_i}, X_\a^{\ell_\a}, Y_\a^{\ell_\a}$ are contained in $\gr Z$. Therefore, it remains to  show that $u_i \in \gr Z$. This follows from Lemma \ref{lem: gr Harish-Chandra center} below.
\end{proof}
Let us consider  the Weyl module $W_\e(\w_i)$ for the fundamental weight $\w_i$. Let $v_1, \dots, v_t$ be the weight basis of $W_\e(\w_i)$ such that if $\weight(v_i)> \weight(v_j)$ then $i>j$. In particular, $v_1$ is the lowest weight vector of $W_\e(\w_i)$ and $v_t$ is the highest weight vector of $W_\e(\w_i)$.  The image of the element $c_{\w_i}=\sum_j c_{v^*_j, K^{-2\rho} v_j} \in O_\e[G]$ in $ U^{ev}_\e$ is contained in the Harish-Chandra center $Z_{HC}$ of $U^{ev}_\e(\g)$, see Section \ref{sec: HC center}. 
\begin{Lem}\label{lem: gr Harish-Chandra center} $\gr (c_{\w_i})=a  u_i$ for some $a\in \BC^\x$.
\end{Lem}
To prove Lemma \ref{lem: gr Harish-Chandra center}, we need some technical results. 
\begin{Lem}\label{lem: technical lem 1} Let $M \in \Rep(\dU_\e(\g))$. Let $0\neq m \in M_\lambda$. For each $\a\in \Delta_+$, set $r_\a=\max\{r| \tE_\a^{(r)}m \neq 0\}$ and $s_\a=\max\{ s|\tF_\a^{(s)}m \neq 0\}$. Then $s_\a-r_\a=(\lambda, \a^\vee)$, where $\a^\vee=2\a/(\a, \a)$.
\end{Lem}
\begin{proof} We will prove the similar statement where $\tE_\a^{(n)}, \tF_\a^{(n)}$ are replaced by  the standard roots generators $E_\a^{(n)}, F_\a^{(n)}$ because we want to use the Lusztig braid group actions on rational representations in $\Rep(\dmU_\e(\g)) \cong \Rep(\dU_\e(\g))$. The result for our version of the quantum group will follow since the twisted root generators $\tE_\a^{(n)}, \tF_\a^{(n)}$ of $\dU_\e(\g)$ are different from the standard root generators only by  factors of the form $K^?$.

When $\a=\a_i$ for some simple root $\a_i$, the statement is proved in \cite[Lemma 1.11]{APW1}. Note that in \emph{loc.~cit.} they proved the result over the localization $\BZ[v, v^{-1}]_{\m_0}$ of $\BZ[v,v^{-1}]$ at the ideal $\m_0=(p, v-1)$. However, their proof works over an arbitrary ring $R$, see Remark \ref{rem: 1rem about APW}.. 

In general, $\a=s_{i_1} \dots s_{i_{k-1}}(\a_{i_k})$ so that $E^{(n)}_\a=T_{i_1} \dots T_{i_{k-1}}(E^{(n)}_{i_k}), F^{n}_\a=T_{i_1}\dots T_{i_{k-1}}(F^{(n)}_{i_k})$. There are compatible braid group actions on the rational representations in $\Rep(\dmU_\e(\g))$ \cite[ $\mathsection 41.2$]{l-book}. Consider $m'=T^{-1}_{i_{k-1}}\dots T^{-1}_{i_1} m$, a weight vector of weight $s_{i_{k-1}}\dots s_{i_1}(\lambda)$. Then 
\[ E_\a^{(n)}m =T_{i_1}\dots T_{i_{k-1}}(E^{(n)}_{i_k}m'), \qquad F_\a^{(n)}m=T_{i_1}\dots T_{i_{k-1}}(F^{(n)}_{i_k}m').\]
Hence the statement for general $\a$ follows from the case when $\a=\a_i$ a simple root.
\end{proof}
Recall the sets $\CI_{\a_i}=\{ \b_{k_1}, \dots, \b_{k_r}\}, \CI_{\a^*_i}=\{ \b_{l_1}, \dots, \b_{l_r}\}$. Let 
\[ \lambda_t=s_{\b_{k_t}}\dots s_{\b_{k_1}}(\w_i), \qquad \lambda'_t=s_{\b_{k_t}}\dots s_{\b_{k_1}}(-\w^*_i)\qquad \forall 1\leq t\leq r.\]
In the weight basis $v_1, \dots, v_m$  of $W_\e(\w_i)$ as above, let $v_{\lambda_t}, v_{\lambda'_t}$ be the unique weight vectors of weights $\lambda_t, \lambda'_t$, respectively. 
\begin{Lem}\label{lem: technical lem 2} (a) $\tE_{\b_j}^{(n)}(v_{\lambda_t})=0$ for all $n \geq 1$ and $k_t< j \leq k_{t+1}$. 

\noindent
(b) $\tF_{\b_j}^{(n)}v_{\lambda'_t}=0$ for all $n \geq 1$ and $l_t< j \leq l_{t+1}$.
\end{Lem}
\begin{proof}Again, we will prove the statements for the standard root generators $E_\a^{(n)}, F_\a^{(n)}$ instead of $\tE_\a^{(n)}, \tF_\a^{(n)}$.

We will prove part $(a)$ only since the proof of part $(b)$ is similar.  Using Lusztig's braid group action, we get
\[ E_{\b_j}^{(n)}=T_{s_{i_1}}\dots T_{s_{i_{j-1}}}(E^{(n)}_{i_k}), \qquad v_{\lambda_t}=a T_{s_{i_1}} \dots T_{s_{i_{k_t}}}(v_{\w_i}), \qquad a\in \BC^\x.\]
here we use $\lambda_t=s_{i_1}\dots s_{i_{k_t}} (\w_i)$ by Lemma \ref{lem: properties of root systems}.
Hence
\[ E_{\b_j}^{(n)}v_{\lambda_t}=aT_{s_{i_1}}\dots T_{s_{i_{j-1}}}(E_{i_{k}}^{(n)} T^{-1}_{s_{i_{j-1}}}\dots T^{-1}_{s_{i_k+1}}(v_{\w_i}).\]
Since $s_{i_m}(\w_i)=\w_i$ for all $k_t< m< k_{t+1}$, it follows that $T^{-1}_{s_{i_m}}(v_{\w_i})$ is a nonzero weight vector of weight $w_i$ in $W_\e(\w_i)$. Therefore, 
\[ T^{-1}_{s_{i_{j-1}}}\dots T^{-1}_{s_{i_k+1}}(v_{\w_i})=b v_{\w_i}, \qquad b\in \BC^{\x}\]
Hence 
\[ E^{(n)}_{\b_j}(v_{\lambda_t})=ab T_{s_{i_1}} \dots T_{s_{i_{j-1}}}(E^{(n)}_{i_k} v_{\w_i})=0.\]
\end{proof}

\begin{proof}[Proof of Lemma \ref{lem: gr Harish-Chandra center}]We note that $c_{\w_i}=\sum_j \e^{-(2\rho, \weight(v_j))} c_{v^*_j, v_j}$. We view $c_{\w_j}$ as  an element of $U_{\e}^{ev}(\mathfrak{g})$ via the embedding $O_\e[G]\hookrightarrow U^{ev}_\epsilon(\mathfrak{g})$. We recall the definitions of the height and degree of monomials in $U_\e^{ev}$ in  \eqref{eq: height of monomial}-\eqref{eq: deg of monomial} as well as the $\BZ^{2N+1}_{\geq 0}$-filtration on $U_\e^{ev}$.

We need to show that $u_i$ is the highest monomial of $c_{\w_i}$ in the $\BZ^{2N+1}_{\geq 0}$-filtration. This follows from the following two steps.

{\it Step 1.} We will show that the highest degree monomial in $c_{v^*_1, v_1}$  is 
\[ (\tF_{\b_{l_s}} \dots \tF_{\b_{l_1}}K^{\kappa(\w_i+\w_i^*)}) K^{-2\w_i}(\tE_{\b_{k_t}}\dots\tE_{\b_{k_1}} K^{\gamma(\w_i+\w^*_i)}),\]
which is equal to $b u_i$ for some $b\in \BC^\x$. The height of this monomial is $2\w_i+2\w^*_i$.

By Lemma \ref{lem: matrix coeff in Uev}, $c_{v^*_1, v_1}$ contains the monomial 
\begin{equation}\label{eq: monomial}(\tF^{\cev{r}} K^{\kappa(-\weight(\tF^{\cev{r}}))})K^{-2\weight(\tE^{\cev{r}})-2\weight(v_1)} (\tE^{\cev{k}} K^{\gamma(\weight(\tE^{\cev{k}}))})
\end{equation}
if and only if
\[  v^*_1(\tF^{(\cev{k})} \tE^{(\cev{r})} v_1) \neq 0.\]

$\bullet$ The height of monomial \eqref{eq: monomial} is $ \weight(\tE^{\cev{k}})-\weight(\tF^{\cev{r}}) =\weight(\tE^{(\cev{r})})-\weight(\tF^{(\cev{k})})$. 
This height is the largest possible  (equal to $2\w_i +2\w^*_i$) only when $\tE^{(\cev{r})}v_1= a v_t$ and $\tF^{(\cev{k})}v_t =b v_1$ for some $a, b \in \BC^\x$. 

$\bullet$ To keep track of the degree of \eqref{eq: monomial}, we put the degree $( k_N, \dots, k_1)$ on $\tF^{(\cev{k})}$  and the degree $(r_N, \dots, r_1)$ on $\tE^{(\cev{r})}$ via the lexicographic orders:  $(k_N, 0, \dots, 0)< \dots< (0,\dots, 0, k_1)$ and $(r_N,0, \dots, 0)< \dots < (0, \dots, 0, r_1)$

$\bullet$ We will show that the highest degree monomial $\tE^{(\cev{r})}$ such that $\tE^{(\cev{r})} v_1=av_t$ for some $a\in \BC^\x$ is $\tE_{\b_{l_s}}\dots \tE_{\b_{l_1}}$. Note that $v_1=v_{-\w^*_i}, v_t=v_{\w_i}$.


For $1\leq j < l_1$, $\tF_{\b_j}^{(n)}v_{-\w^*_i}=0$ for all $n \geq 1$ , and $( -\w^*_i, \b^\vee_j)=0$. Then by Lemma \ref{lem: technical lem 1}, $\tE^{(n)}_{\b_j} v_{-\w^*_i}=0$ for all $n\geq 1$ and $1\leq j < l_1$.\\
Moreover, $\tF^{(n)}_{\b_{l_1}} v_{-\w^*_i}=0$ for all $n\geq 1$, and $(-\w^*_i, \b^\vee_{l_1})=1$. Then by Lemma \ref{lem: technical lem 1}, 
\[\begin{cases} 
\tE_{\b_{l_1}}v_{-\w^*_i} \neq 0 \Rightarrow \tE_{\b_{l_1}} v_{-\w^*_i} =a v_{\lambda'_1}, \qquad a\neq 0\\
\tE^{(n)}_{\b_{l_1}}v_{-\w^*_i}=0 \qquad \forall n\geq 2.
\end{cases}\]
So we can assume $\tE^{(\cev{r})}=\tE^{(r_N)}_{\b_N}\dots \tE^{(r_{l_1}+1)}_{\b_{l_1+1}} \tE_{\b_{l_1}}$.

For $l_1<j<l_2$, $\tF^{(n)}_{\b_j} v_{\lambda'_1}=0$, and $(\lambda'_1, \b^\vee_j)=0$. Then by Lemma \ref{lem: technical lem 1}, $\tE^{(n)}_{\b_j}\tE_{\b_{l_1}} v_{-\w^*_i}=0$.\\
Moreover, $\tE^{(n)}_{\b_{l_2}}v_{\lambda'_1}=0$ for $n \geq 1$, and $(\lambda'_1, \b^\vee_{l_1})=1$ then by Lemma \ref{lem: technical lem 1}
\[\begin{cases} \tE_{\b_{l_2}} \tE_{\b_{l_1}} v_{-\w^*_i} \neq 0 \Rightarrow \tE_{\b_{l_2}}\tE_{\b_{l_1}} v_{-\w^*_i}= bv_{\lambda'_2}, \;\; b\neq 0\\
\tE^{(n)}_{\b_{l_2}} \tE_{\b_{l_1}} v_{-\w^*_i}=0 \qquad \forall n \geq 2
\end{cases}\]
So we can assume $\tE^{(\cev{r})}=\tE^{(r_N)}_{\b_N}\dots \tE^{(r_{l_2}+1)}_{\b_{l_2+1}} \tE_{\b_{l_2}}\tE_{\b_{l_1}}$. 

Keep doing, we see that the highest degree monomial is $\tE_{\b_{l_r}}\dots \tE_{\b_{l_1}}$. 

$\bullet$ Similarly, the highest degree monomial $\tF^{(\cev{k})}$ such that $\tF^{(\cev{k})} v_t=bv_1$ for some $b\in \BC^\x$ is $\tF_{\b_{k_t}} \dots \tF_{\b_{k_1}}$.

{\it Step 2.} All monomials in $c_{v^*_i, v_i}$ with $i>1$ have lower degree than $u_i$. Indeed, the heights of all monomials in $c_{v^*_i, v_i}$ are at most $2\w_i -2\weight(v_i)< 2\w_i +2\w^*_i$.   
\end{proof}

\subsection{The Azumaya locus of $U^{ev}_\e(\g)$ over $\Spec Z$}\ \label{ssec: Azumaya locus}

Let us recall the following set introduced in Theorem \ref{thm: finite dim rep of order}:
\[ \Omega_{U^{ev}_\e} =\{ p\in \Spec Z~|~ \text{the corresponding semisimlpe representation is irreducible}\}.\]
Since $U^{ev}_\e(\g)$ is a completely prime (i.e, without zero divisors) Noetherian algebra, by \cite[Proposition 3.1]{bg2}, 
\[ \Omega_{U^{ev}_\e}=\{ p\in \Spec Z~|~ U^{ev}_\e \otimes_Z Z_{\m_p}\; \text{is Azumaya over}\; Z_{\m_p}\}.\]
Hence, $\Omega_{U^{ev}_\e}$ is called {\em the Azumaya locus} of $U^{ev}_\e(\g)$. In this section, we construct a large open subset of $\Omega_{U^{ev}_\e}$. We follow the arguments in \cite{dckp2}. Let us state  two  key lemmas.

For a filtered algebra $A$, let $\gr A=\bigoplus_i A_{i+1}/A_i$ be the associated graded algebra and $\mathscr{R}(A)=\bigoplus_i A_i t^i$ be the Rees algebra of $A$.
\begin{Lem}[Proposition 1.4 \cite{dckp2}]\label{lem: regular sequence} Let $A$ be a commutative filtered algebra and let $a_1, \dots, a_n \in A$ be such that $\bar{a}_1, \dots, \bar{a}_n$ is a regular sequence of $\gr A$. Let $I=(a_1, \dots, a_n)$ be the ideal of $A$ generated by $a_1, \dots, a_n$. Then 
\begin{enumerate}[label=\alph*)]
\item $a_1, \dots, a_n$ is a regular sequence in $A$.
\item The ideal $\gr I$ of $\gr A$ is generated by the elements $\bar{a}_1, \dots, \bar{a}_n$.
\end{enumerate}
\end{Lem}

\begin{Lem}[Lemma 1.5 \cite{dckp2}]\label{key lemma about filtered algebra} Let $A$ be a finitely generated algebra over $\BC$ equipped with a filtration such that 

$\bullet$ $A$, $\gr A$ and $\mathscr{R}(A)$ are orders closed under trace in  central simple algebras, and

$\bullet$ $\gr A$ is a finitely generated order of the same degree as $A$.

Let $Z_0$ be a central subalgebra of $A$  such that $\gr Z_0$ is a finitely generated algebra over $\BC$ and $\gr A$ is a finitely generated module over $\gr Z_0$. Let $I$ be an ideal of $Z_0$ and $\gr I$ be the associated graded ideal of $\gr Z_0$. Let $\mathscr{O}$ (resp. $\mathscr{O}_1$) be the set of zeroes of $I$ (resp. $\gr I$) in $\Spec Z_0$ (resp. $\Spec (\gr Z_0)$). Suppose that $\mathscr{O}_1\cap \Omega^0_{\gr A} \neq 0$. Then $\mathscr{O}\cap \Omega^0_A \neq 0$.
\end{Lem}
\begin{Rem}The condition of being closed under trace on $\mathscr{R}(A)$ is not mentioned in \cite[Lemma 1.5]{dckp2}. But we believe this condition is important so that the set $\Omega^0_{\mathscr{R}(A)}$ is open dense  in $\Spec \mathscr{R}(Z_0)$ by Theorem \ref{thm: finite dim rep of order}. The openess of $\Omega^0_{\mathscr{R}(A)}$ is crucial in the last argument of the proof of \cite[Lemma 1.5]{dckp2}.
\end{Rem}
Recall the $\dU_\BC(\g^d)$-equivariant isomorphism $Z_{Fr}\cong \BC[G^d_0]$ in Proposition \ref{Ze and Bruhat cell}. For any $\lambda\in P^*_+$, let $V^d(\lambda)$ be the irreducible representation of $G^d$ with  highest weight $\lambda$. Consider the morphism 
\[\pi_\lambda: \Spec Z_{Fr}\cong G^d_0\subset G^d \rightarrow GL(V^d(\lambda)).\]
Let 
\[ \phi_{\lambda}(u)=\textnormal{tr}_{V^d(\lambda)}(\pi_\lambda(u)), \qquad \forall ~ p \in \Spec Z_{Fr}.\]
Then the following lemma is standard
\begin{Lem}\label{lem: trace element phi(wi)}(a) The elements $\Big\{\phi_{\lambda}|\lambda \in P^*_+ \Big\}$ are a $\BC$-basis of $Z_\cap= \BC[G^d_0]^{\dU_\BC(\g^d)}$.

\noindent
(b) $Z_\cap$ is a polynomial algebra over $\BC$ on the generators $\{ \phi_{\w^d_1}, \dots, \phi_{w^d_r}\}$, here $\{\w^d_i=\ell_i \w_i\}_{1\leq i\leq r}$ are fundamental weights of $\g^d$.
\end{Lem}
The set $\{\a^d_i=\ell_i \a_i\}_{1\leq i \leq r}$ is the set of simple roots of $\g^d$. Recall that the Weyl group $W^d$ of $G^d$ is the same as $W$ via identification $s_{\a_i}\mapsto s_{\a^d_i}$. Recall the fixed reduced expression $w_0=s_{i_1}\dots s_{i_N}$ so that we have the set of positive roots $\{ \b^d_k=s_{i_1}\dots s_{i_{k-1}}(\a^d_{i_k})=\ell_{i_k}\b_k\}$.  For each simple root $\a^d$, let 
\begin{equation}\label{eq: the set Ida}
I_{\a^d}:=\{ 1\leq j \leq N| s_{i_j}=s_{\a^d}\}, \qquad \CI_{\a^d}:=\{ \b^d_{k_1}, \dots,\b^d_{k_t}|k_j \in I_{\a^d}\}.
\end{equation}
It is not hard to see that $I_{\a^d}=I_\a$, the set introduced in \eqref{eq: the set Ia}. Set
\[ x_{\b^d}=\tE^{\ell_{\b}}_{\b} K^{\ell_\b \gamma(\b)}, \qquad y_{\b^d}=\tF^{\ell_\b}_\b K^{\ell_\b \kappa(\b)}, \qquad z_{\w^d_i}=K^{2\ell_i\w_i}.\]
where $\b$ is a positive root of $\g$ and $\b^d=\ell_\b \b$ is the positive root of $\g^d$.

\begin{Lem}[cf. Lemma 4.6 \cite{dckp2}] \label{lem: image in grZ} Recall the element $\phi_{\w^d_i} \in Z_{Fr}\cong \BC[G^d_0]$ in Lemma \ref{lem: trace element phi(wi)}.b . Then the image $\overline{\phi}_{\w^d_i} \in \gr Z_{Fr}$ is 
\begin{equation}\label{eq: graded image of phi}\overline{\phi}_{\w^d_i}=z_{-\w^d_i} x_{\b^d_{k_r}} \dots x_{\b^d_{k_1}} y_{\b^{d}_{l_r}}\dots y_{\b^{d}_{l_1}},
\end{equation}
here $\CI_{\a^d_i}=\{ \b^d_{k_1}, \dots, \b^d_{k_r}\}$ and $\CI_{-w_0(\a^d_i)}=\{\b^{d}_{l_1},\dots, \b^{d}_{l_r}\}$ as in \eqref{eq: the set Ida}.
\end{Lem}

\begin{proof} The proof is exactly the same as the proof of Lemma \ref{lem: gr Harish-Chandra center}. We just want to make a remark. In the proof of Lemma \ref{lem: gr Harish-Chandra center}, the orthogonality of pairings $(\tF^{\cev{k}}, \tE^{(\cev{r})})'$ and $(\tF^{(\cev{k})}, \tE^{\cev{r}})'$ allows us to see whether  the monomial $\tF^{\cev{k}} K^{?} \tE^{\cev{r}}$ shows up in $c_{f,v} \in U^{ev}_\e(\g)$ by seeing whether $f(\tF^{(\cev{r})} \tE^{(\cev{k})} v)$ is  zero or not. In this lemma, we used  the orthogonality of pairing $(\tF^{\cev{k}}, e^{(\cev{r})})'$ in $Z_{Fr}^< \x \dU^>(\g^d)\rightarrow \BC$ and pairing $(f^{(\cev{k})}, \tE^{\cev{r}})'$ in $\dU^<(\g^d) \x Z^>_{Fr}\rightarrow \BC$ in Lemma \ref{lem: perfect pairings}.
\end{proof}

We are now going to construct the promised open subset of $\Omega_{U^{ev}_\e}$. Let us consider the projection
\[ \sA_1: \Spec Z \cong G^d_0 \x_{T^d/W} T/W \rightarrow G^d_0 \hookrightarrow G^d.\]
Let $G^{d, reg}$ be the set of regular elements in $G^d$. 
\begin{Thm}[cf. Theorem 5.1 \cite{dckp2}]\label{thm: open set of Azumaya locus} The set $\Omega_{U^{ev}_\e}$ contains the open set $\sA_1^{-1}(G^{d, reg})$.
\end{Thm}

Before proceeding to prove this theorem, we need to show that Lemma \ref{key lemma about filtered algebra} can be applied to the chain of filtered algebras $U^{(i)}_\e$ in Remark \ref{rem: chain of filtered algebras}. To this end, this is enough to prove that $\mathscr{R}(U^{(i)})$ is closed under the trace for all $0\leq i \leq 2N$. By Remark \ref{rem: order with normal center}, it is sufficient to prove 
\begin{Lem} \label{lem: center is normal} The Rees algebra $\mathscr{R}(Z^{(i)})$ is normal for all $i$. 
\end{Lem}
\begin{proof} Since $\mathscr{R}(Z^{(i)})/(t)=\gr Z^{(i)}\cong Z^{(i+1)}$, a finitely generated  algebra over $\BC$, $\mathscr{R}(Z^{(i)})$ is a finitely generated graded algebra  over $\BC$.  

The fiber of $\mathscr{R}(Z^{(i)})$ at $0\in \operatorname{Spec}(\mathbb{C}[t])$ is $\gr Z^{(i)}\cong Z^{(i+1)}$ and the  fiber at $a\in \operatorname{Spec}(\mathbb{C}[t^{\pm 1}])$ is $Z^{(i)}$. Since $U^{(i)}_\e$ and $U^{(i+1)}_\e$ are integrally closed  by Proposition \ref{prop: Uev is integrally closed}, both $Z^{(i)}$ and $Z^{(i+1)}$ are normal. Hence all fibers of $\mathscr{R}(Z^{(i)})$ over $\BC[t]$ are normal. On the other hand, $\mathscr{R}(Z^{(i)})$ is faithfully flat over $\BC[t]$ and $\BC[t]$ is normal. Hence by \cite[Corollary 22.E]{ma}, $\mathscr{R}(Z^{(i)})$ is normal.
\end{proof}
The proof of Theorem \ref{thm: open set of Azumaya locus} follows the strategy of the proof of \cite[Theorem 5.1]{dckp2}. However, some details are modified. 
\begin{itemize}
\item In the {\it loc.cit.}  they used the notion "trace filtration", see \cite[Definition 5.1]{dckp2},  in order to use Lemma \ref{key lemma about filtered algebra} ( \cite[Lemma 1.5]{dckp2}). However, they did not verify that the filtration on the De Concini-Kac form constructed in \cite[Proposition 4.2]{dckp2} is a trace filtration, indeed, it seems to be difficult to verify this. Here we use Lemma \ref{lem: center is normal} in order to apply Lemma \ref{key lemma about filtered algebra}.
\item Some results used in the proof of \cite[Theorem 5.1]{dckp2} require adaptations in our setup or different proofs including Proposition \ref{prop: size images of H and HS} and Lemma \ref{lem: image in grZ}, etc.
\end{itemize}
So we provide the proof of Theorem \ref{thm: open set of Azumaya locus} here for the reader's convenience.
\begin{proof}
{\it Step 1:} Let $\chi_0: \Irr(U^{ev}_\e(\g) )\rightarrow \Spec Z_{Fr}$ be the map that sends irreducible representations of $U^{ev}_\e$ to the restriction of the central characters to $Z_{Fr}$. Let us recall 
\[\Omega^0_{U^{ev}_\e}:=\{ p \in \Spec Z_{Fr}~|~ \text{all representations of $\chi^{-1}_0(p)$ have the maximal dimension }\}.\]
Then it is enough to show that $G^{d, reg}_0:=G^d_0 \cap G^{d, reg} \subset \Omega^0_{U^{ev}_\e}$. 

Let $\mathscr{O}$ be a conjugacy class of regular elements in $G^d$ and $\bar{\mathscr{O}}$ be the Zariski closure of $\mathscr{O}$ in $G^d$. So we need to show that $\mathscr{O}\cap G^d_0 \subset \Omega^0_{U^{ev}_\e}$. By Proposition \ref{prop: Omega0 is union of symplective leaves},  $\Omega^0_{U^{ev}_\e}$ is a union of symplectic  leaves. The symplectic leaves of $\Spec Z_{Fr}=G^d_0$ are exactly the intersections of conjugacy classes of $G^d$ with $G^d_0$ by Proposition \ref{prop: Possion structure of ZFr}. Therefore, it is enough to show that $\mathscr{O}\cap \Omega^0_{U^{ev}_\e} \neq 0$. Since $\Omega^0_{U^{ev}_\e}$ is open in $G^d_0$,  it is then enough to show that $\bar{\mathscr{O}} \cap \Omega^0_{U^{ev}_\e} \neq 0$.

{\it Step 2:} It is well-known that $\bar{\mathscr{O}}\subset G^d$ is given by the equations
\[  \textnormal{tr}_{V^d(\w^d_i)}(g) =c_i, \qquad i=1, \dots, n,\; \text{and}\; g \in G^d.\]
hence the set $\bar{\mathscr{O}}\cap \Spec Z_{Fr}$ is given by the equations
\[ \phi_{\w^d_i}=c_i, \qquad i=1, \dots, n.\]
We want to apply Lemma \ref{key lemma about filtered algebra} to the ideal $I$ of $Z_{Fr}$ generated by $\{ \phi_{\w^d_i}-c_i\}$. For this first of all, we need to show that each of the algebra $U^{(j)}$ introduced in Remark \ref{rem: chain of filtered algebras} has degree $d=\prod_{\a \in \Delta_+} \ell_\a$. Since the degree can only decrease in each step, it is enough to show that $d$ is the degree of $U^{(2N)}=\gr U^{ev}_\e$, but this follows by Proposition \ref{prop: deg of grU}. By easy induction with Lemma \ref{key lemma about filtered algebra}, we only need to show that if $\gr I$ is the associated graded ideal of $I$ in $\gr Z_{Fr}$, and $\mathscr{O}_1$ is its set of zeros, then 
\begin{equation}\label{eq: intersection is nonzero}
    \mathscr{O}_1  \cap \Omega^0_{\gr U^{ev}_\e} \neq 0
\end{equation}

{\it Step 3:} The elements $\overline{\phi}_{\w^d_i}, i =1, \dots, r$ in \eqref{eq: graded image of phi} form a regular sequence of $\gr Z_{Fr}$ since they are monomials in disjoint sets of indeterminates. Hence, by Lemma \ref{lem: regular sequence}, $\gr I$  is generated by $\{ \overline{\phi}_{\w^d_i}, i=1, \dots, r\}$.  So $\mathscr{O}_1$ is the union of the following subvarieties. From each monomial $\overline{\phi}_{\w^d_i}$ we choose a factor in $\{ x_{\b^d_{k_r}}, \dots, x_{\b^d_{k_1}}, y_{\b^d_{l_r}}, \dots y_{\b^d_{l_1}}\}$. The subvarieties of interests are obtained as the sets of zeroes of these factors. So it is enough to show that one of these subvarieties intersects $\Omega^0_{U^{ev}_\e}$ nontrivially.

Let us recall the set $\CI_{\a_i}$ and pick $\b_{m_i} \in \CI_{\a_i}$. As in \eqref{eq: set S}, let
\[ S=\{ K^{2\w_1}, \dots, K^{2\w_r}, Y_{\b_i}, X_{\b_i}\} /\{ Y_{\b_{m_1}}, \dots, Y_{\b_{m_r}}\}.\]
Then $Y_{\b_{m_i}}^{\ell_{\b_{m_i}}} \in \gr Z_{Fr}$. The closed subvariety $\mathscr{Y}$  defined by equations $Y_{\b_{m_i}}^{\ell_{\b_{m_i}}} =0, i=1, \dots, r$  is a component of $\mathscr{O}_1$ (note that $y_{\b^d_{m_i}}=Y_{\b_{m_i}}^{\ell_{\b_{m_i}}} K^{\ell_{\b_{m_i}}\gamma(\b_{m_i})}$ and $K^{\ell_{\b_{m_i}} \gamma(\b_{m_i})}$ is invertible in $\gr Z_{Fr}$). By  Proposition \ref{prop: size images of H and HS}, we can apply Lemma \ref{lem: intersection with Omega0} to obtain $\mathscr{Y} \cap \Omega^0_{\gr U^{ev}_\e} \neq 0$. Therefore, \eqref{eq: intersection is nonzero} holds. This finishes the proof.
\end{proof}

\appendix

\section{Equivariant modules}\label{sec:equiv}

This appendix is devoted to establish the basic properties of (left, right, or bi) modules over an algebra with a compatible action of a Hopf algebra.


\subsection{$H$-invariants}\label{ssec:H-inv}
Let $H$ be a Hopf algebra over a field $k$. For any left $H$-module $V$, the {\it $H$-invariant part} of $V$ is defined by 
\[
  V^H=\big\{v\in V \,|\, hv=\varepsilon(h)v \; \text{for all $h\in H$} \big\} \,.
\]
For any left $H$-modules $M$ and $N$, we have the following two actions of $H\otimes H$ on $\Hom_k(M,N)$:
\begin{align*}
    (x\otimes y)(m)&= xf(S(y)m),\\
    (x\otimes y)(m) &= yf(S^{-1}(x)m).
\end{align*}
Composing these actions with the coproduct $\Delta\colon H\rightarrow H\otimes H$, we obtain two left actions of $H$ on $\Hom_k(M,N)$ 
that will be denoted by $(H, \cdot_1)$ and $(H, \cdot_2)$, respectively.

\begin{Lem}\label{H-invariant} 
$\Hom_H(M,N)=\Hom_k(M,N)^{(H, \cdot_1)}=\Hom_k(M,N)^{(H, \cdot_2)}$, where the superscript mean that we  take the $H$-invariants  with respect to the corresponding $H$-actions.
\end{Lem}
The proof is standard, c.f \cite[$\mathsection 3.10$]{j}.

\subsection{$H$-equivariant $A$-modules}

We shall now also assume that $A$ is an algebra over $k$. We start with the standard definition:

\begin{defi}
An $H$-module algebra $A$ (assumed to be unital, with unit $1_A$) is an algebra $A$ with a left 
$H$-module structure such that 
\[
  h(1_A)=\varepsilon(h)1_A, \qquad h(ab)=\sum (h_{(1)}a)(h_{(2)}b) \qquad \forall\, h \in H,\, a,b \in A \,.
\]
\end{defi}

In this setup, one naturally introduces the categories of $H$-equivariant left/right $A$-modules:

\begin{defi}\label{A Mod}
Let $A$ be an $H$-module algebra. 

\noindent
\medskip
(a) An $H$-equivariant left $A$-module $M$ is a left $A$-module $M$ with a left $H$-module structure such that 
\begin{equation}\label{eq:left-equiv}
  h(am)=\sum (h_{(1)}a)(h_{(2)}m) \qquad \forall\, h \in H\,,\, a\in A,\, m\in M \,.
\end{equation}
Let $A\Rmod^H$ denote the category of $H$-equivariant left $A$-modules. A morphism $\eta\colon M\rightarrow N$ 
in $A\Lmod^H$ is a morphism of left $A$-modules which is also a morphism of left $H$-modules.

\medskip
\noindent
(b) An $H$-equivariant right $A$-module $M$ is a right $A$-module $M$ with a left $H$-module structure such that 
\begin{equation}\label{eq:right-equiv}
  h(mb)=\sum (h_{(1)}m)(h_{(2)}b) \qquad \forall\, h\in H,\, b\in A,\, m \in M \,.
\end{equation}
Let $A\Lmod^H$ denote the category of $H$-equivariant right $A$-modules. A morphism $\eta\colon M\rightarrow N$ 
in $A\Rmod^H$ is a morphism of right $A$-modules which is also a morphism of left $H$-modules.

\medskip
\noindent
(c) An $H$-equivariant $A$-bimodule $M$ is a $A$-bimodule $M$ with a left $H$-module structure such that 
\begin{equation}\label{eq:bi-equiv}
  h(amb)=\sum (h_{(1)}a)(h_{(2)}m)(h_{(3)}b)  \qquad \forall\, h \in H,\, a,b \in A,\, m \in M \,.
\end{equation}
Let $A\Bimod^H$ denote the category of $H$-equivariant $A$-bimodules. A morphism $\eta\colon M\rightarrow N$ 
in $A\Bimod^H$ is a morphism of $A$-bimodules which is also a morphism of left $H$-modules.
\end{defi}

We note that the condition~\eqref{eq:bi-equiv} is equivalent to the combination of~\eqref{eq:left-equiv} 
and~\eqref{eq:right-equiv}.

\begin{Lem}\label{Hom in H-equiv A-Rmod}
(a) For any $M,N\in A\Rmod^H$, the space $\Hom_{A\Rmod}(M,N)$ carries a natural structure of a left $H$-module via 
\begin{equation} \label{H-action on Hom}
  (hf)(m)=\sum h_{(1)}f(S(h_{(2)})m) 
\end{equation}
for $h\in H, f\in \Hom_{A\Rmod}(M,N), m\in M$. Furthermore, we have
\[
  \Hom_{A\Rmod^H}(M,N)=\Hom_{A\Rmod}(M,N)^{(H, \cdot_1)} \,. 
\]

\medskip
\noindent
(b) For any $M\in A\Bimod^H$ and $N\in A\Rmod^H$, the space $\Hom_{A\Rmod}(M,N)$ is naturally an object in $A\Rmod^H$ 
with the left $H$-module structure as in \eqref{H-action on Hom} and the right $A$-module structure defined by 
$(fa)(m)=f(am)$ for all $a\in A, f\in \Hom_{A\Rmod}(M,N), m\in M$.

\medskip
\noindent
(c) For any $M\in A\Rmod^H$ and $N\in A\Bimod^H$, the space $\Hom_{A\Rmod}(M,N)$ is naturally an object in $A\Lmod^H$ 
with the left $H$-module structure as in \eqref{H-action on Hom} and the left $A$-module structure defined by 
$(af)(m)=af(m)$ for all $a\in A, f\in \Hom_{A\Rmod}(M,N), m\in M$.

\medskip
\noindent
(d) For any $M,N\in A\Bimod^H$, the space $\Hom_{A\Rmod}(M,N)$ is naturally an object in $A\Bimod^H$ with 
the left $H$-module structure as in \eqref{H-action on Hom} and the $A$-bimodule structure defined by 
$(afb)(m)=af(bm)$ for $a,b\in A, f\in \Hom_{A\Rmod}(M,N), m\in M$.
\end{Lem}
\begin{proof}This is straightforward.
\end{proof}

\begin{Rem}\label{rem: A-Rmod vs A-Lmod}
(a) If $A$ is an $H$-module algebra, then $A^{\op}$ is an $H^{\ccop}$-module algebra with the same action, whereas $H^\ccop\simeq H$ as algebras.

\noindent
(b) If $M,N\in A\Lmod^H$ then \eqref{H-action on Hom} does not equip $\Hom_{A\Lmod}(M,N)$ with an $H$-action in general. 
However, evoking (a), we note that any $M\in A\Lmod^H$ is naturally an object in $A^{\op}\Rmod^{H^{\ccop}}$, where the left 
$H^{\ccop}$-module structure on $M$ is unchanged and the right $A^{\op}$-module structure coincides with the left $A$-module structure. 
Thus, we have $\Hom_{A\Lmod}(M,N)=\Hom_{A^{\op}\Rmod}(M,N)$. By Lemma \ref{Hom in H-equiv A-Rmod}(a), the right-hand side has 
a left $H^{\ccop}$-action defined by $(hf)(m)=\sum h_{(2)}f(S^{-1}(h_{(1)})m)$ for $h\in H^{\ccop}, f\in \Hom_{A^\op \Rmod}(M,N), m\in M$. 
This implies that for any $M, N\in A\Lmod^H$, we have a left $H$-action on $\Hom_{k}(M,N)$ given by 
\begin{equation}\label{H-action on Hom 2}
  (hf)(m)=\sum h_{(2)}f(S^{-1}(h_{(1)})m) \qquad \forall\, h\in H,\, f\in \Hom_{A\Lmod}(M,N),\, m\in M \,.
\end{equation}
\end{Rem}

The following result follows from Lemma~\ref{Hom in H-equiv A-Rmod} and Remark~\ref{rem: A-Rmod vs A-Lmod} 


\begin{Lem} \label{Hom in H-equiv A-Lmod}
(a) For any $M,N \in A\Lmod^H$, the space $\Hom_{A\Lmod}(M,N)$ carries a natural structure of a left $H$-module 
via~\eqref{H-action on Hom 2}.
Furthermore, we have
\[
  \Hom_{A\Lmod^H}(M,N)=\Hom_{A\Lmod}(M,N)^{(H,\cdot_2)} \,.
\]

\noindent
(b) For any $M\in A\Bimod^H$ and $N\in A\Lmod^H$, the space $\Hom_{A\Lmod}(M,N)$ is naturally an object in $A\Lmod^H$ 
with the left $H$-module structure as in \eqref{H-action on Hom 2} and the left $A$-module structure defined by 
$(af)(m)=f(ma)$ for $a\in A,  f\in \Hom_{A\Lmod}(M,N), m \in M$.

\medskip
\noindent
(c) For any $M\in A\Lmod^H$ and $N\in A\Bimod^H$, the space $\Hom_{A\Lmod}(M,N)$ is naturally an object in $A\Rmod^H$ 
with the left $H$-module structure as in \eqref{H-action on Hom 2} and the right $A$-module structure defined by 
$(fa)(m)=f(m)a$ for $a\in A, f\in \Hom_{A\Lmod}(M,N), m\in M$.

\medskip
\noindent
(d) For any $M,N\in A\Bimod^H$, the space $\Hom_{A\Lmod}(M,N)$ is a naturally an object in $A\Bimod^H$ with 
the left $H$-module structure as in \eqref{H-action on Hom 2} and the $A$-bimodule structure defined by 
$(afb)(m)=f(ma)b$ for $a,b\in A, f\in \Hom_{A\Lmod}(M,N), m\in M$.
\end{Lem}

In the next lemma, we shall equip $\Hom_{A\Rmod}(\cdot,\cdot)$ with the left $H$-action via \eqref{H-action on Hom}, 
while $\Hom_{A\Lmod}(\cdot,\cdot)$ shall be equipped with the left $H$-action via \eqref{H-action on Hom 2}.

\begin{Lem}[Tensor-Hom Adjunction]\label{tensor-hom-adj}
(a) For any $M,N \in A\Rmod^H$ and $P\in A\Bimod^H$, we have a natural isomorphism of left $H$-modules:
\begin{equation}\label{Tensor-Hom}
  \Hom_{A\Rmod}(N\otimes_A P,M) \simeq \Hom_{A\Rmod}(N, \Hom_{A\Rmod}(P,M)) \,.
\end{equation}
Here, $\Hom_{A\Rmod}(P,M)\in A\Rmod^H$ by Lemma \ref{Hom in H-equiv A-Rmod}(b) and $N\otimes_A P$ is naturally in $A\Rmod^H$.

\medskip
\noindent
(b) Let $P,M\in A\Rmod^H$ be such that $P$ is a finitely generated projective right $A$-module. Note that $A$ is naturally 
an object in $A\Bimod^H$. Then we have an isomorphism of left $H$-modules: 
\begin{equation}\label{Hom-Dual}
  \Hom_{A\Rmod}(P,M) \simeq M\otimes_A \Hom_{A\Rmod}(P,A) \,.
\end{equation}
If we assume further that $P\in A\Bimod^H$, then \eqref{Hom-Dual} is an isomorphism in $A\Rmod^H$.

\medskip
\noindent
(c) For any $P\in A\Lmod^H$, $\Hom_{A\Lmod}(P,A)$ is an object in $A\Rmod^H$ by Lemma \ref{H-action on Hom 2}(c), and  
$\Hom_{A\Rmod}(\Hom_{A\Lmod}(P,A),A))$ is an object in $A\Lmod^H$ by Lemma \ref{H-action on Hom}(c). Define the map 
\begin{equation}
  \pi\colon P \longrightarrow \Hom_{A\Rmod}(\Hom_{A\Lmod}(P,A),A))
\end{equation}
via $\pi(p)(f)=f(p)$ for $p\in P, f\in \Hom_{A\Lmod}(P,A)$. Then $\pi$ is a morphism in $A\Lmod^H$. 
If additionally $P$ is a finitely generated projective left $A$-module, then $\pi$ is an isomorphism.
\end{Lem}
\begin{proof}
    The proof is straightforward.
\end{proof}

\end{document}